\documentclass[a4paper]{amsart}

\usepackage[left=2.5cm,right=2.5cm,top=2.5cm,bottom=2.5cm]{geometry}

\usepackage{amsmath}
\usepackage{amsthm}
\usepackage{amssymb}
\usepackage{mathtools}
\usepackage{enumerate}
\usepackage{hyperref}
\usepackage{color}
\usepackage{graphicx}
\usepackage{import}
\usepackage{xcolor}

\theoremstyle{plain}
\newtheorem{introthm}{Theorem}

\newtheorem{theorem}{Theorem}[section]
\newtheorem{proposition}[theorem]{Proposition}
\newtheorem{lemma}[theorem]{Lemma}
\newtheorem{corollary}[theorem]{Corollary}

\newtheorem*{corollary*}{Corollary}

\theoremstyle{definition}

\newtheorem{remark}[theorem]{Remark}
\newtheorem{example}[theorem]{Example}
\newtheorem{question}[theorem]{Question}

\numberwithin{equation}{section}
\numberwithin{figure}{section}

\newcommand{\BA}{\mathbb{A}}

\newcommand{\BC}{\mathbb{C}}
\newcommand{\BD}{\mathbb{D}}

\newcommand{\BQ}{\mathbb{Q}}
\newcommand{\BR}{\mathbb{R}}

\newcommand{\BT}{\mathbb{T}}

\newcommand{\BZ}{\mathbb{Z}}

\newcommand{\MD}{\mathcal{D}}
\newcommand{\ME}{\mathcal{E}}
\newcommand{\MF}{\mathcal{F}}

\newcommand{\ML}{\mathcal{L}}

\newcommand{\MQ}{\mathcal{Q}}

\newcommand{\MS}{\mathcal{S}}
\newcommand{\MT}{\mathcal{T}}
\newcommand{\MU}{\mathcal{U}}
\newcommand{\MV}{\mathcal{V}}

\newcommand{\MX}{\mathcal{X}}
\newcommand{\MY}{\mathcal{Y}}

\title{Packing stability and the subleading asymptotics of symplectic Weyl laws}
\author{Oliver Edtmair}

\begin{document}
\maketitle

\begin{abstract}
We prove that symplectic ball packing stability holds for every compact, connected symplectic $4$-manifold with smooth boundary. This follows from a stronger result: the full volume of any such manifold can be filled by a single symplectic ellipsoid. As an application, we obtain estimates---with sharp exponents---for the error terms in the symplectic Weyl laws for embedded contact homology capacities, periodic Floer homology spectral invariants, and link spectral invariants. We also construct an example of a star-shaped domain in $\BR^4$, arbitrarily $C^1$ close to the unit ball and with boundary of regularity just below $C^2$ and smooth away from a single point, for which packing stability fails. Our proofs reveal a close connection between symplectic packing stability in the presence of smooth boundary and the algebraic structure of Hamiltonian diffeomorphism groups, particularly Banyaga's simplicity results.
\end{abstract}

\tableofcontents

\section{Introduction}

\subsection{Ball packing stability}

Consider a connected symplectic manifold $(M^{2n},\omega)$ of finite volume, which may be open or have boundary. For $k \geq 1$, the \textit{$k$-th ball packing number of $M$}, denoted by $p_k(M)$, quantifies how much of the volume of $M$ can be covered by $k$ disjoint, symplectically embedded balls of equal size. More precisely, the numbers $p_k(M)$ are defined by
\begin{equation*}
p_k(M) \coloneqq \sup_a \frac{k \cdot \operatorname{vol}(B(a))}{\operatorname{vol}(M)} \in [0,1],
\end{equation*}
where the supremum is taken over all $a>0$ such that the disjoint union of $k$ copies of the closed $2n$-dimensional ball $B(a) = B^{2n}(a)$ of symplectic width $a$ admits a symplectic embedding into the interior of $M$. The ball packing numbers serve as a fundamental quantitative measure of symplectic rigidity. Since Gromov’s seminal work \cite{gro85}, which revealed the first non-trivial obstructions to symplectic embeddings, these invariants have attracted considerable attention. We refer to \cite{sch, hut11b, chls07} for surveys.

In \cite[Remark 1.5.G]{mp94}, McDuff and Polterovich observed that $\lim_{k\rightarrow \infty} p_k(M) = 1$ for every symplectic manifold of finite volume. This implies that symplectic ball packings become asymptotically flexible as the number of balls increases. In \cite{bir97, bir99}, Biran proved a much stronger and deeper flexibility result for closed, connected symplectic $4$-manifolds $(M^4,\omega)$ which are rational, meaning that the symplectic form $\omega$ represents a rational cohomology class $[\omega] \in H^2(M;\BQ)$. He showed that there exists a finite threshold $k_0$, depending on $M$, such that $p_k(M) = 1$ for all $k\geq k_0$. This phenomenon is referred to as \textit{symplectic ball packing stability}.

A natural problem is to determine which connected symplectic manifolds of finite volume, beyond closed, rational symplectic $4$-manifolds, exhibit packing stability. This question is explicitly raised in the survey article \cite[\S 3.2, Problem 4]{chls07} by Cieliebak, Latschev, Hofer, and Schlenk, as well as in \cite[Conjecture 13.2]{sch} by Schlenk. Several results have been established in this direction. It was proved by Buse and Hind \cite{bh11, bh13} that packing stability holds for closed, rational symplectic manifolds of arbitrary dimension. In dimension four, Buse, Hind, and Opshtein \cite{bho16} eliminated the rationality assumption, showing that packing stability holds for all closed, connected symplectic $4$-manifolds. In higher dimensions, however, the general case remains open. Moreover, packing stability has been verified for certain open manifolds with highly symmetric boundary, including balls and ellipsoids \cite{bh13}, $4$-dimensional polydisks and pseudoballs \cite{bho16}, and $4$-dimensional rational convex toric domains \cite{chmp25}.

In contrast to the positive results described above, Cristofaro-Gardiner and Hind \cite{ch} recently identified the first examples of open symplectic manifolds of finite volume for which packing stability fails. They constructed open domains $U \subset \BR^4$, diffeomorphic to the open $4$-ball, such that $p_k(U) < 1$ for all $k\geq 1$. The failure of packing stability in these examples is connected to the domains' ``wild'' boundary behaviour; see Subsection \ref{subsec:limits_packing_stability} below for further discussion. This leaves a significant gap between these counterexamples and the known cases of packing stability, i.e. closed manifolds and special toric domains. It is an intriguing problem to more precisely identify and understand the transition between packing stability and its failure.

An important question is whether packing stability holds for compact symplectic manifolds with smooth boundary; see \cite[Question 7.1]{ch}. The known cases involving toric boundary provide limited insight into what one might expect in the general setting. For example, the symplectic $4$-ball $B^4(1)$ is symplectomorphic to the complement of a line in $\BC P^2$, equipped with a suitably scaled Fubini--Study symplectic form. More generally, any $4$-dimensional convex toric domain can be approximated by toric domains which are symplectomorphic to the complements of nodal divisors in closed symplectic $4$-manifolds; see \cite{cri19}. In other words, the known instances of packing stability in the presence of boundary essentially reduce to the case of nodal divisor complements and thus do not substantially extend beyond the closed case. In particular, no examples are known of packing stability for symplectic manifolds whose boundary carries a characteristic foliation with chaotic behaviour. In fact, general symplectic embedding results for manifolds with dynamically rich boundaries are rare; see \cite{edt24} for an exception.

A major motivation for studying ball packings in symplectic manifolds with general smooth boundary arises from a close connection to the asymptotic behaviour of certain sequences of symplectic spectral invariants; see Subsection \ref{subsec:subleading_asymptotics} below. These invariants encode important dynamical information, making manifolds with boundary---where the dynamics occurs---central to our investigation. Our first main result establishes that packing stability holds for all compact, connected symplectic $4$-manifolds with smooth boundary.

\begin{introthm}[Ball packing stability]
\label{thm:ball_packing_stability}
Let $(M^4,\omega)$ be a compact, connected symplectic $4$-manifold, possibly with smooth boundary. Then symplectic ball packing stability holds for $M$. This means that there exists a positive integer $k_0$ such that $p_k(M) = 1$ for all $k \geq k_0$.
\end{introthm}

We note that in Theorem \ref{thm:ball_packing_stability} we only assume that the boundary is smooth and in particular do not impose any contact type assumptions.

\subsection{Ellipsoid embedding stability}

Recall that the \textit{symplectic ellipsoid} $E(a_1,\dots, a_n)$ with widths $a_i>0$ is defined by
\begin{equation*}
E(a_1,\dots,a_n) \coloneqq \left\{ z\in \BC^n \enspace\Big|\enspace \sum_j\frac{\pi|z_j|^2}{a_j} \leq 1 \right\},
\end{equation*}
where $\BR^{2n} \cong \BC^n$ is equipped with the standard symplectic form $\omega_0 = \sum_i dx_i\wedge dy_i$. Let $(M^{2n},\omega)$ be a connected symplectic manifold of finite volume. For every real number $a\geq 1$, the \textit{ellipsoid embedding number} $p_a^E(M)$ quantifies how much of the volume of $M$ can be filled by a single symplectic ellipsoid obtained from $E(1,\dots,1,a)$ by scaling. It is defined by
\begin{equation*}
p_a^E(M) \coloneqq \sup_\lambda \frac{\operatorname{vol}(E(\lambda,\dots,\lambda,\lambda a))}{\operatorname{vol}(M)}\in [0,1],
\end{equation*}
where the supremum is taken over all positive numbers $\lambda>0$ such that the ellipsoid $E(\lambda,\dots,\lambda,\lambda a)$ admits a symplectic embedding into the interior of $M$; see \cite[\S 1.3.2]{sch05b}. Since the ellipsoid $E(a,\dots,a)$ is simply the ball $B(a)$, we have $p_1^E(M) = p_1(M)$. It is well known that, for every positive integer $k\geq 1$, the ellipsoid $E(1,\dots,1,k)$ can be fully packed by $k$ balls of equal size. This implies that $p_k^E(M) \leq p_k(M)$ for all $k\geq 1$.

Schlenk proved in \cite[\S 1.3.2, Theorem 3]{sch05b} that, for every connected symplectic manifold $M$ of finite volume, we have $\lim_{a\rightarrow \infty} p_a^E(M) = 1$. In other words, symplectic embeddings of ellipsoids into $M$ become asymptotically flexible as the ellipsoids become increasingly skinny. We say that \textit{ellipsoid embedding stability} holds for $M$ if there exists a finite threshold $a_0$ such that $p_a^E(M) = 1$ for all $a\geq a_0$.

Since ball packing stability holds for every ellipsoid, the condition $p_a^E(M) = 1$ for even a single value of $a$ implies ball packing stability. In particular, ellipsoid embedding stability implies ball packing stability. Moreover, it was shown in \cite[Theorem 1.1]{bh13} that ellipsoid embedding stability holds for every ellipsoid, so knowing that $p_a^E(M) = 1$ for one value of $a$ also implies ellipsoid embedding stability.

It follows from \cite{bir01} and \cite{ops13}---see also \cite[Theorem 1.3]{bh13}---that every closed, connected, rational symplectic manifold $M$ can be fully filled by one single ellipsoid; that is $p_a^E(M) = 1$ for some $a\geq 1$. This means that ellipsoid embedding stability holds for such manifolds. While ball packing stability is known to hold for all closed symplectic $4$-manifolds \cite{bho16}, the corresponding statement for ellipsoid embedding stability is open; see \cite[Conjecture 1]{bho16} and \cite[Conjecture 13.8]{sch}.

We prove that our ball packing stability result (Theorem \ref{thm:ball_packing_stability}) can, in fact, be strengthened to ellipsoid embedding stability. Although this result is new even in the closed case, the primary case of interest concerns compact symplectic manifolds whose boundary exhibits rich dynamical behaviour.

\begin{introthm}[Ellipsoid embedding stability]
\label{thm:ellipsoid_embedding_stability}
Let $(M^4,\omega)$ be a compact, connected symplectic $4$-manifold, possibly with smooth boundary. Then symplectic ellipsoid embedding stability holds for $M$. This means that there exists a number $a_0$ such that $p_a^E(M) = 1$ for all $a \geq a_0$.
\end{introthm}

\begin{remark}
Let $(M^4,\omega)$ be a compact, connected symplectic $4$-manifold with smooth boundary. Let $a_0$ denote the threshold from Theorem~\ref{thm:ellipsoid_embedding_stability}. Fix $a>a_0$ and consider the symplectic ellipsoid $E = E(r,r a)$ where the parameter $r>0$ is chosen so that $\operatorname{vol}(M) = \operatorname{vol}(E)$. This means that the closed symplectic ellipsoid $\lambda E$ symplectically embeds into the interior of $M$ for all $\lambda \in (0,1)$. This naturally raises the question of whether the interior of $E$ itself admits a symplectic embedding into the interior of $M$. In light of \cite[Theorem~4.4]{pv15}, this would follow if one could arrange the symplectic embeddings of $\lambda E$ into the interior of $M$ to depend smoothly on $\lambda$ as $\lambda$ approaches $1$. It is conceivable that, with additional care, such a smooth dependence can be achieved using the symplectic embedding constructions developed in this paper. However, we will not pursue this direction further here and leave it for future work.
\end{remark}

In the literature, the ellipsoid embedding numbers $p_a^E(M)$ of a symplectic $4$-manifold $(M,\omega)$ are often encoded in terms of the \textit{ellipsoid embedding function} $c_M:[0,\infty)\rightarrow \BR$; see, for example, \cite{chmp25}. This function is defined by
\begin{equation*}
c_M(a) \coloneqq \inf\left\{ r \mid E(1,a) \overset{s}{\hookrightarrow} (M,r\omega) \right\},
\end{equation*}
where $X\overset{s}{\hookrightarrow}Y$ denotes the existence of a symplectic embedding from $X$ into $Y$. The function $c_M(a)$ encodes the same information as the ellipsoid embedding numbers $p_a^E(M)$.

The ellipsoid embedding function $c_M$ is notoriously difficult to compute. In a landmark paper, McDuff and Schlenk \cite{ms12} succeeded in computing the ellipsoid embedding function $c_{B(1)}$ of the $4$-dimensional ball---arguably the simplest nontrivial target manifold. Their result is an intricate yet explicit function featuring the so-called infinite Fibonacci staircase. A further key feature of $c_{B(1)}$ is that for sufficiently large $a$, the function stabilizes to the volume curve:
\begin{equation*}
c_{B(1)}(a) = \sqrt{\frac{a}{2\operatorname{vol}(B(1))}} = \sqrt{a}.
\end{equation*}
This stabilization is equivalent to the assertion that ellipsoid embedding stability holds for $B(1)$.

For a general compact, connected symplectic $4$-manifold $M$ with smooth boundary, fully computing the function $c_M$ appears entirely out of reach. Nevertheless, Theorem \ref{thm:ellipsoid_embedding_stability} shows that $c_M$ always stabilizes to the volume curve.

\subsection{Limits of packing stability}
\label{subsec:limits_packing_stability}

As mentioned above, Cristofaro-Gardiner and Hind \cite{ch} constructed the first examples of symplectic manifolds of finite volume that do not exhibit ball packing stability. For a discussion of the failure of packing stability in the context of packings by certain generalized convex toric domains, rather than by balls, we refer the reader to \cite{cmm}.

The examples in \cite{ch} are unbounded concave toric domains of finite volume. More precisely, let $f: [0,\infty) \rightarrow (0,1]$ be a convex function such that $\int_0^\infty f(x) dx < \infty$. Define $\Omega\subset \BR^2$ to be the region in the first quadrant lying below the graph of $f$, and set
\begin{equation}
\label{eq:unbounded_toric_domain}
X_f \coloneqq \left\{ (z_1,z_2)\in \BC^2 \mid \pi(|z_1|^2,|z_2|^2) \in \Omega \right\} \subset \BC^2.
\end{equation}
This domain has finite volume, and it is shown in \cite{ch} that packing stability fails for $X_f$.

Moreover, for suitable choices of $f$, \cite{ch} proves that one can symplectically fold the infinite ``tail'' of $X_f$ into a bounded region of $\BC^2$, thereby producing an open, bounded domain $U$ that is diffeomorphic to the open ball and symplectomorphic to the interior of $X_f$. Packing stability must then also fail for $U$. However, the folding construction yields ``wild'' boundary behaviour: in particular $\partial U$ is not a topological submanifold of $\BC^2$. To more precisely locate the transition between packing stability and its failure, it is natural to seek counterexamples with more regular boundary. This is the content of the following result.

\begin{introthm}[$C^{2-\varepsilon}$ failure of packing stability]
\label{thm:failure_packing_stability}
Arbitrarily $C^1$ close to the unit ball in $\BR^4$, there exists a star-shaped domain $X\subset \BR^4$ with the following properties:
\begin{enumerate}
\item The boundary of $X$ is in the regularity class $\bigcap_{\alpha\in (0,1)}C^{1,\alpha}$, but not in $C^{1,1}$. Here $C^{1,\alpha}$ denotes the H\"{o}lder space of differentiable functions with $C^{0,\alpha}$ H\"{o}lder continuous derivative.
\item The boundary of $X$ is smooth on the complement of a single point.
\item Ball packing stability and ellipsoid embedding stability fail for $X$. In fact, we have the following more general assertion: Let $W\subset \BR^4$ be an arbitrary compact domain with smooth boundary, possibly disconnected. Consider the number
\begin{equation*}
p^W(X) \coloneqq\sup_r \frac{\operatorname{vol}(rW)}{\operatorname{vol}(X)} \in [0,1],
\end{equation*}
where the supremum is taken over all $r>0$ such that $rW$ admits a symplectic embedding into $X$. Then $p^W(X)$ is strictly less than $1$.
\end{enumerate}
\end{introthm}

We do not currently know whether counterexamples to packing stability exist for symplectic manifolds with boundary regularity of class $C^2$ or higher. We expect that our methods for proving packing stability in the smooth case can be extended to manifolds whose boundary has finite differentiability. However, we do not expect that our present techniques can reach the $C^2$ threshold; see Subsection \ref{subsec:open_problems} for further discussion.

It is a classical problem, dating back to work of Eliashberg and Hofer \cite{eh92} (see also \cite{cfhw96, cie97}), to understand when an open symplectic manifold can be given a boundary, and to what extent the interior of a symplectic manifold with boundary encodes information about the boundary. Questions of this nature have also been studied in \cite{ch}. For example, it is shown there that, under suitable decay conditions on the function $f$, the interior of the domain $X_f$ defined in equation \eqref{eq:unbounded_toric_domain} is not symplectomorphic to the interior of any compact symplectic manifold with smooth boundary.

Note that if the interior of a symplectic $4$-manifold $M$ is symplectomorphic to the interior of a compact symplectic manifold with smooth boundary, then packing stability must hold for each component of $M$ by Theorem \ref{thm:ball_packing_stability}. Thus Theorem \ref{thm:failure_packing_stability} has the following consequence:

\begin{corollary*}[$C^{2-\varepsilon}$ domains with non-smoothable interior]
There exists a star-shaped domain $X\subset \BR^4$ which is $C^1$ close to the unit ball, whose boundary is in the regularity class $\bigcap_{\alpha \in (0,1)} C^{1,\alpha}$ and smooth on the complement of a single point, and whose interior is not symplectomorphic to the interior of a compact symplectic manifold with smooth boundary.
\end{corollary*}

\subsection{Subleading asymptotics of symplectic Weyl laws}
\label{subsec:subleading_asymptotics}

Given a symplectic $4$-manifold $(M^4,\omega)$, possibly open or with boundary, Hutchings \cite{hut11} defines a sequence of numbers
\begin{equation*}
0 = c_0(M) < c_1(M) \leq c_2(M) \leq \cdots \leq \infty
\end{equation*}
called \emph{embedded contact homology (ECH) capacities}; see \cite{hut14} for a survey. For every compact domain $X\subset \BR^4$ with piecewise smooth boundary, he proves the following remarkable asymptotic formula:
\begin{equation}
\label{eq:weyl_law_ech}
c_k(X) = 2\sqrt{\operatorname{vol}(X) k} + o(k^{1/2}) \qquad (k\rightarrow \infty).
\end{equation}
This identity shows that ECH capacities asymptotically recover the symplectic volume of the domain. By analogy with the classical Weyl law for the spectrum of the Laplace operator, formula \eqref{eq:weyl_law_ech} is often referred to as the \textit{ECH Weyl law}. We refer to \cite{chr15} for a version of the Weyl law for ECH spectral invariants associated with a contact form on a $3$-manifold. See also \cite{iri15} for a major application of the ECH Weyl law concerning the smooth closing lemma for $3$-dimensional Reeb flows.

It is a natural and compelling problem to study the error terms
\begin{equation*}
e_k(X) \coloneqq c_k(X) - 2\sqrt{\operatorname{vol}(X) k}
\end{equation*}
in the ECH Weyl law.
In the case of the classical Weyl law for the Laplace spectrum, the corresponding error terms have been extensively investigated over the past century; see, for example, \cite{cou20, lev52, dg75, ivr80, cg23}, among many others. In the symplectic setting, Sun \cite{sun19} proved that $e_k(X) = O(k^{125/252})$ for every star-shaped domain $X\subset \BR^4$ with smooth boundary. This was improved by Cristofaro-Gardiner and Savale \cite{cs20}, who established the bound $e_k(X) = O(k^{2/5})$. Later Hutchings \cite[Theorem 1.1]{hut22} showed that in fact $e_k(X) = O(k^{1/4})$ for every compact domain $X \subset \BR^4$ with smooth boundary.

On the other hand, in all known examples where the capacities $c_k(X)$ have been explicitly computed, one finds that $e_k(X) = O(1)$; see, for instance, \cite{wor23}. These examples comprise different classes of toric domains and are somewhat limited in scope, as the Reeb dynamics on the boundary of a toric domain is exceptionally simple and fails to reflect the complex, chaotic dynamics that can occur on the boundary of a general domain.

Nevertheless, Hutchings \cite[Conjecture 1.5]{hut22} conjectured that, for a generic star-shaped domain $X\subset \BR^4$ with smooth boundary, the error terms satisfy
\begin{equation}
\label{eq:hutchings_conjecture}
\lim_{k\rightarrow \infty} e_k(X) = -\frac{1}{2}\operatorname{Ru}(X),
\end{equation}
where $\operatorname{Ru}(X)$ denotes the \textit{Ruelle invariant} \cite{rue85} associated to the natural Reeb flow on the boundary of $X$; see \cite[\S 1.2]{hut22} for more details. This conjecture can be viewed as a symplectic analogue of the well-known Weyl conjecture on the error terms in the classical Weyl law; see, for instance, \cite[Conjecture 3.3.7]{lmp23}.

The following theorem is an application of our symplectic ball packing stability result, Theorem \ref{thm:ball_packing_stability}. It establishes that the error terms in the ECH Weyl law are bounded for all compact domains in $\BR^4$ with smooth boundary. This significantly improves upon the best previously known bound of $O(k^{1/4})$, and can also be viewed as a step towards Hutchings' conjecture \eqref{eq:hutchings_conjecture}.

\begin{introthm}[Subleading asymptotics ECH]
\label{thm:subleading_asymptotics_ech}
Let $X\subset \BR^4$ be a compact domain with smooth boundary. Then the ECH capacities satisfy
\begin{equation*}
c_k(X) = 2 \sqrt{\operatorname{vol}(X) k} + O(1) \qquad (k\rightarrow \infty).
\end{equation*}
In other words, $e_k(X) = O(1)$.
\end{introthm}

The exponent $0$ in the bound $O(1)$ on the error terms $e_k(X)$ is sharp. There are examples of smooth domains $X$, for example the ball $B(a)$, such that $e_k(X)$ does not converge to zero; see, for instance, \cite[Example 1.2]{hut22}.

\begin{remark}
\label{rem:elementary_ech_capacities}
Recently, Hutchings \cite{hut22b} defined an alternative sequence of capacities $c_k^{\operatorname{alt}}(X)$ which retain most of the formal properties of the ECH capacities $c_k(X)$, thereby recovering many key applications. Unlike the ECH capacities, however, their construction avoids the sophisticated machinery of embedded contact homology and is instead considerably more elementary. The proof of Theorem \ref{thm:subleading_asymptotics_ech} carries over essentially unchanged to this setting and implies that the error terms in the Weyl law for these alternative capacities are also bounded. More precisely, for every compact domain $X\subset \BR^4$ with smooth boundary, one has
\begin{equation*}
c_k^{\operatorname{alt}}(X) = 2 \sqrt{\operatorname{vol}(X) k} + O(1) \qquad (k\rightarrow \infty).
\end{equation*}
\end{remark}

A variant of embedded contact homology, known as \emph{periodic Floer homology (PFH)}, was introduced in \cite{hut02, hs05}. This theory allows one to extract numerical invariants---called \textit{PFH spectral invariants}---from area-preserving surface diffeomorphisms. It was recently established independently by Cristofaro-Gardiner, Prasad, and Zhang \cite{cpz}, and by Hutchings and the author \cite{eh}, that these invariants satisfy a Weyl-type asymptotic formula, analogous to the ECH Weyl law. Important applications of the PFH Weyl law include smooth closing lemmas for area preserving surface maps established in \cite{cpz, eh} and the resolution of the long-standing simplicity conjecture in \cite{chs24}.

As in the ECH case, understanding the asymptotic behavior of the associated error terms is an interesting problem; see, for instance, \cite[Problem 2.2]{aim}. Our ball packing stability result (Theorem~\ref{thm:ball_packing_stability}) implies that the error terms in the PFH Weyl law are bounded. More precisely, we obtain the following result. For the relevant notation and terminology, we refer to \cite[Theorem 8.1]{eh}.

\begin{introthm}[Subleading asymptotics PFH]
\label{thm:subleading_asymptotics_pfh}
Consider a closed surface $\Sigma$ equipped with an area form $\omega$ of total area $A$. Fix a rational area preserving diffeomorphism $\phi\in\operatorname{Diff}(\Sigma,\omega)$ and let $Y_\phi$ denote its mapping torus. Let $\gamma_i\subset Y_\phi$ be a sequence of monotone reference cycles whose degrees $d_i$ diverge to $+\infty$ and let $\sigma_i\in HP(\phi,\gamma_i)$ be a sequence of non-zero PFH classes. Then, for every pair of Hamiltonians $H_\pm\in C^\infty(Y_\phi)$, we have
\begin{equation*}
c_{\sigma_i}(\phi,\gamma_i,H_+) - c_{\sigma_i}(\phi,\gamma_i,H_-)+\int_{\gamma_i} (H_+-H_-)dt = d_iA^{-1}\int_{Y_\phi} (H_+-H_-)dt\wedge\omega_\phi + O(1) \qquad (k\rightarrow \infty).
\end{equation*}
\end{introthm}

\begin{remark}
In Theorem \ref{thm:subleading_asymptotics_pfh}, we set the subgroups $G_i \subset \operatorname{ker}(\omega_\phi)$ appearing in \cite[Theorem 8.1]{eh} all equal to $\operatorname{ker}(\omega_\phi)$. As explained in \cite[Remark 8.2]{eh} and \cite[Remark 7.3]{eh}, this choice, combined with the assumption that $\phi$ is rational and the reference cycles $\gamma_i$ are monotone, ensures that the $U$-cycle assumption in \cite[Theorem 8.1]{eh} is automatically satisfied. This was proved in \cite{cppz}. We also emphasize that Theorem \ref{thm:subleading_asymptotics_pfh} is not vacuous: there exists a sequence of monotone reference cycles $\gamma_i$ with diverging degrees, along with a sequence of non-zero classes $\sigma_i \in HP(\phi,\gamma_i)$; see again \cite[Remark 8.2]{eh}.
\end{remark}

\begin{remark}
In \cite{edta}, we introduced an alternative to PFH spectral invariants using more elementary methods; see also Remark \ref{rem:elementary_ech_capacities}.
These elementary PFH spectral invariants likewise satisfy a Weyl law, and again our methods show that the associated error terms are $O(1)$. More precisely, in the notation of \cite[Theorem 1.6]{edta}, we have
\begin{equation*}
c_{d_i}(H_+,H_-) = d_i A^{-1} \int_{Y_\phi} (H_+-H_-)dt\wedge\omega_\phi + O(1) \qquad (i\rightarrow \infty)
\end{equation*}
for every diverging sequence of positive integers $d_i$ such that $c_{d_i}$ is finite for each $i$.
\end{remark}

In \cite{chmss22}, Cristofaro-Gardiner, Humili\`{e}re, Mak, Seyfaddini, and Smith introduced another sequence of spectral invariants, called \emph{link spectral invariants}, which also satisfy a Weyl law. As in the previous cases, it is natural to investigate their subleading asymptotics; see \cite[Problem 2.1]{aim}. Let $(\Sigma,\omega)$ be a closed surface equipped with an area form of total area $A$. A \emph{Lagrangian link} $\underline{L}\subset \Sigma$ is a disjoint union of smoothly embedded circles in $\Sigma$. For links $\underline{L}$ satisfying a certain monotonicity condition, the associated link spectral invariant is a function
\begin{equation*}
c_{\underline{L}} : C^\infty([0,1]\times \Sigma) \rightarrow \BR.
\end{equation*}
It is shown in \cite{chmss22} that for an equidistributed sequence of $d$-component links $\underline{L}_d$, one has
\begin{equation*}
c_{\underline{L}_d}(H) = A^{-1} \int_{[0,1]\times \Sigma} H dt\wedge \omega + o(1) \qquad (d\rightarrow \infty).
\end{equation*}
In subsequent work \cite{chmss}, the same authors established the sharp bound $O(d^{-1})$ on the subleading asymptotics of this Weyl law in the case where $\Sigma$ is the 2-sphere. For higher genus surfaces, Mak and Trifa \cite{mt} obtained a weaker bound of $O(d^{-1/2})$.

Using Theorem \ref{thm:subleading_asymptotics_pfh} and results of Chen \cite{chea,cheb} relating PFH spectral invariants to link spectral invariants, we are able to obtain the sharp bound $O(d^{-1})$ for arbitrary genus. We emphasize, however, that the class of links considered in Chen’s work differs slightly from that in \cite{chmss22}. Our result applies to the former class.

We say that a link $\underline{L}\subset \Sigma$ is \textit{admissible} if the complement $\Sigma \setminus \underline{L}$ consists of open disks and a single planar surface. We say that $\underline{L}$ is \textit{monotone} if each connected component of $\Sigma \setminus \underline{L}$ has the same area.

\begin{introthm}[Subleading asymptotics link spectral invariants]
\label{thm:subleading_asymptotics_link_invariants}
Let $(\Sigma,\omega)$ be a closed connected surface of total area $A$ and consider a sequence $\underline{L}_d$ of admissible, monotone $d$-component links. Then, for every $H \in C^\infty([0,1]\times \Sigma)$, we have
\begin{equation*}
c_{\underline{L}_d}(H) = A^{-1} \int_{[0,1]\times \Sigma} H dt \wedge \omega + O(d^{-1}) \qquad (d\rightarrow \infty).
\end{equation*}
\end{introthm}

\subsection{Overview of the proofs and outline of the paper}
\label{subsec:discussion}

We now highlight some key ideas underlying the proofs of our main results. We begin with Theorems \ref{thm:ball_packing_stability} and \ref{thm:ellipsoid_embedding_stability}, which concern ball packing and ellipsoid embedding stability, respectively. Let $(M,\omega)$ be a compact, connected symplectic $4$-manifold with smooth boundary. As noted above, since both ball packing and ellipsoid embedding stability are known to hold when the target is a symplectic ellipsoid, it suffices to fill the volume of $M$---up to arbitrarily small error---by a single symplectic ellipsoid of fixed aspect ratio. In the case where $(M,\omega)$ is closed and rational, Opshtein \cite[Proposition 1.3]{ops13} observed that such an ellipsoid filling can be deduced from the existence of a Donaldson divisor \cite{don96}. This is closely related to the well-known Biran decomposition of a symplectic manifold \cite{bir01}. In contrast, our approach does not appeal to Donaldson's theorem. Instead, we rely crucially on deep results of Banyaga concerning the simplicity and perfectness of the group of Hamiltonian diffeomorphisms \cite{ban78, ban97}. We refer to \cite{bs} for another recent flexibility result related to Banyaga's theorems.

To illustrate how the algebraic structure of Hamiltonian diffeomorphism groups enters into our symplectic embedding constructions, we consider the following simple toy example. Let $\sigma$ be an area form on the $2$-sphere $S^2$ of total area $1$. Let $\BT \coloneqq \BR/\BZ$ denote the circle, and define the open $4$-manifold $W \coloneqq \BR \times \BT \times S^2$, equipped with the symplectic form $\Omega \coloneqq ds \wedge dt + \sigma$, where $(s,t)$ are coordinates on the cylinder $\BR \times \BT$. Let $H : \BT \times S^2 \rightarrow \BR$ be a $1$-periodic, mean-normalized Hamiltonian. We naturally view its graph as a codimension-$1$ hypersurface $\operatorname{graph}(H) \subset (W,\Omega)$. The characteristic foliation on $\operatorname{graph}(H)$ induced by $\Omega$ lifts the Hamiltonian flow $\varphi_H^t$ on $(S^2,\sigma)$ generated by $H$. In particular, $\operatorname{graph}(H)$ can be interpreted as the mapping torus of $\varphi_H^1$. Let $C<0$ be a constant strictly less than the minimum of $H$, and define $M$ to be the compact region in $W$ bounded from below by $\{C\} \times \BT \times S^2$ and from above by $\operatorname{graph}(H)$. We view $(M, \Omega)$ as a toy model for a compact symplectic $4$-manifold with smooth boundary. Since $H$ is arbitrary, the dynamics on the boundary component $\operatorname{graph}(H)$ can be complicated. In what follows, we sketch how to decompose $M$, for sufficiently negative values of $C$, into much simpler pieces---namely finitely many balls and polydisks.

In suitable polar coordinates $(z,\theta) \in (0,1)\times \BT$ on $S^2$, the area form $\sigma$ can be written as $\sigma = dz \wedge d\theta$. Define the autonomous Hamiltonian
\begin{equation*}
F : S^2 \rightarrow \BR \qquad F(z,\theta) \coloneqq \frac{1}{2} z - \frac{1}{4}.
\end{equation*}
This Hamiltonian is mean-normalized and its time-$1$ flow $\varphi_F^1$ is a half-rotation of $S^2$. Consider the normal closure of the subgroup $\left\{ \operatorname{id},\varphi_F^1 \right\}\subset \operatorname{Ham}(S^2,\sigma)$ generated by $\varphi_F^1$. Since $\operatorname{Ham}(S^2,\sigma)$ is a simple group by Banyaga's theorem, this normal closure must equal the entire group $\operatorname{Ham}(S^2)$. This implies that every Hamiltonian diffeomorphism of $(S^2,\sigma)$ can be expressed as a finite composition of conjugates of the half-rotation $\varphi_F^1$. In particular, we can find finitely many Hamiltonian diffeomorphisms $\psi_1, \dots, \psi_n \in \operatorname{Ham}(S^2,\sigma)$ such that
\begin{equation}
\label{eq:proof_overview_factorization}
\varphi_H^1 = (\psi_n^{-1}\circ\varphi_F^1\circ\psi_n) \circ \cdots \circ (\psi_1^{-1}\circ\varphi_F^1\circ\psi_1) = \varphi_{\psi_n^*F}^1 \circ \cdots \circ \varphi_{\psi_1^*F}^1.
\end{equation}
Moreover, this identity continues to hold on the level of the universal cover of $\operatorname{Ham}(S^2,\sigma)$, possibly after adding two additional factors of $\varphi_F^1$.

We now use the factorization \eqref{eq:proof_overview_factorization} of $\varphi_H^1$ into half-rotations $\varphi_{\psi_i^*F}^1$ to construct a decomposition of $M$ into balls and polydisks. To this end, we define a special Hamiltonian $G:\BT\times S^2 \rightarrow \BR$ with the same time-$1$ map as $H$, that is $\varphi_G^1 = \varphi_H^1$. This Hamiltonian is constructed such that the path $(\varphi_G^t)_{t\in [0,1]}$ is the concatenation of the $n$ paths $(\varphi_{\psi_i^*F}^t)_{t\in [0,1]}$. We divide the circle $\BT$ into intervals $I_i \coloneqq [\frac{i-1}{n},\frac{i}{n}]$ for $1 \leq i \leq n$. Let $\rho_n : \BT \rightarrow [0,\infty)$ be a non-negative function such that $\rho_n|_{I_i}$ is compactly supported in the interior of $I_i$, and $\int_{I_i}\rho_n dt = 1$. We set
\begin{equation*}
G(t,p) \coloneqq \rho_n(t) \psi_i^*F(p) \qquad \text{for $p \in S^2$ and $t \in I_i$}.
\end{equation*}
Here the role of $\rho_n$ is to reparametrize the paths $(\varphi_{\psi_i^*F}^t)_{t\in [0,1]}$ in such a way that $G$ is smooth.

Now assume that the constant $C<0$ is smaller than the minimum of $G$. We then consider the compact region $N$ in $W$ bounded below by $\left\{ C \right\}\times \BT \times S^2$ and bounded above by $\operatorname{graph}(G)$. It is an elementary but important observation that $(M,\Omega)$ and $(N,\Omega)$ are symplectomorphic, provided $H$ can be connected to $G$ through a smooth family $H^\lambda$ of mean-normalized Hamiltonians, strictly greater than $C$, and satisfying $\varphi_{H^\lambda}^1 = \varphi_H^1 = \varphi_G^1$ for all $\lambda$; see Lemma \ref{lem:isotopies_of_subgraphs}. Since the factorization \eqref{eq:proof_overview_factorization} is assumed to hold on the level of the universal cover, such a connecting family of Hamiltonians exists whenever $C$ is sufficiently negative. We have therefore reduced the task of decomposing $(M,\Omega)$ into balls and polydisks to the analogous task for $(N,\Omega)$.

We now cut $(N,\Omega)$ into $n$ pieces of the form $N \cap (\BR \times I_i \times S^2)$. It is straightforward to verify that each of these $n$ pieces is symplectomorphic to the domain $K\subset \BR \times [0,1] \times S^2$ bounded below by $\left\{ \frac{C}{n} \right\} \times [0,1] \times S^2$ and above by the graph of $F$. Here, we equip $\BR \times [0,1] \times S^2$ with the symplectic form $ds\wedge dt + \sigma$ and regard $F$ as a Hamiltonian $F: [0,1]\times S^2 \rightarrow \BR$. Hence it suffices to decompose $K$ into balls and polydisks. This can be done explicitly, since $F$ is a simple height function on $S^2$.

We remove the poles of $S^2$ and cut open the resulting annulus $(0,1)\times \BT$ along the line segment $\left\{ \theta = 0 \right\}$. This reveals that the open domain
\begin{equation*}
U \coloneqq \left\{ (s,t,z,\theta) \in \BR \times (0,1)^3 \mid \frac{C}{n} < s < \frac{1}{2}z - \frac{1}{4} \right\} \subset (\BR^4,ds\wedge dt + dz\wedge \theta)
\end{equation*}
can be naturally regarded as an open subset of $K$ of full volume. The domain $U$ can be viewed as the Lagrangian product of the right trapezoid
\begin{equation*}
\left\{ (s,z)\in \BR \times (0,1) \mid  \frac{C}{n} < s < \frac{1}{2}z - \frac{1}{4}\right\}
\end{equation*}
and the square $(0,1)^2$. We can cut the right trapezoid into a rectangle and a triangle. This yields a decomposition of $U$ into two pieces: a Lagrangian product of a rectangle and a square, which is symplectomorphic to a polydisk, and a Lagrangian product of a right triangle and a square, which is symplectomorphic to an ellipsoid. In our case, this ellipsoid is $E(1,\frac{1}{2})$, which can in turn be decomposed into two disjoint copies of the ball $B(\frac{1}{2})$. This concludes our construction of a decomposition of $(M,\Omega)$ into balls and polydisks. We carry out a variation of this construction in more detail in Section \ref{sec:strat_of_subgraphs}.

A significant amount of work is required to extend this simple toy example to obtain a filling of a general compact, connected symplectic $4$-manifold by a single ellipsoid. One immediate challenge is to upgrade a decomposition into finitely many balls and polydisks to an ellipsoid filling. Very roughly speaking, our strategy is as follows. Since ellipsoid embedding stability is already known to hold for balls and polydisks, we can fill each of the balls and polydisks in our decomposition with a skinny ellipsoid. We assume that the smaller of the two widths of each skinny ellipsoid is the same across all the ellipsoids. We then connect these skinny ellipsoids through tunnels at the interfaces between the balls and polydisks in the decomposition, ultimately forming a single ellipsoid that fills the entire volume. A detailed implementation of this strategy is provided in Sections \ref{sec:symplectic_ribbons} and \ref{sec:tame_packings}. It is interesting to point out that these arguments, combined with the proof of ball packing stability for general, possibly irrational, closed symplectic $4$-manifolds given in \cite{bho16}, are already sufficient to prove ellipsoid embedding stability for general closed symplectic $4$-manifolds. This special case of Theorem \ref{thm:ellipsoid_embedding_stability} resolves \cite[Conjecture 1]{bho16}.

Another challenge is that a general symplectic $4$-manifold $M$ is not necessarily symplectomorphic to the subgraph of some Hamiltonian on a surface. Our strategy is to decompose $M$ into pieces that are symplectomorphic to subgraphs. To achieve this, it becomes necessary to replace the $2$-sphere $S^2$ with the annulus $\BA$ since $M$ could be exact and, in particular, not contain any closed symplectic surfaces. Working on the annulus requires versions of Banyaga's results for Hamiltonian diffeomorphisms on a surface with boundary, where the diffeomorphisms are not required to restrict to the identity on the boundary. We establish such generalizations in a separate paper \cite{edt}. They are reviewed in Section \ref{sec:ham_diffeo_surface_with_boundary}. The decomposition of a general manifold $M$ into subgraphs of Hamiltonians on the annulus $\BA$ is somewhat technical in nature and is carried out in Sections \ref{sec:symplectic_frusta}, \ref{sec:proof_ellipsoid_embedding_stability}, and \ref{sec:proof_stratification_lemma}.

A more subtle issue with our toy example above is the role of the negative constant $C$. Recall that we only explained how to decompose the domain $(M,\Omega)$ inside $W$ into balls and polydisks for sufficiently negative $C$. The problem is that the threshold of how negative $C$ has to be depends on $n$, the number of half-rotations in the factorization of $\varphi_H^1$. Specifically, the oscillation of the Hamiltonian $G$ is of the same order of magnitude as $n$, meaning that $|C|$ must also be of the order of $n$. However, the number $n$ of half-rotations in the factorization of $\varphi_H^1$ can be arbitrarily large. In fact, this number is bounded from below by the autonomous norm of $\varphi_H^1$, which is defined as the least number of autonomous Hamiltonian diffeomorphisms required in a factorization of $\varphi_H^1$. The autonomous norm on Hamiltonian diffeomorphism groups is known to be unbounded in many cases; see, for instance, \cite[\S 6.3]{gg04}, \cite{bk13}, \cite{bks18}. This is problematic because we would like to decompose $(M,\Omega)$ into balls and polydisks for a fixed value of $C$, not just for sufficiently negative values. In order to resolve this, we establish certain quantitative refinements of Banyaga's results in \cite{edt}. These refinements are of independent interest and are reviewed in Section \ref{sec:ham_diffeo_surface_with_boundary}. As a consequence, we show in Section \ref{sec:quant_decomp_rot} that the number of rotations needed to factor a Hamiltonian diffeomorphism is bounded in some $C^\infty$ open neighbourhood of the identity. In Section \ref{sec:strat_of_subgraphs}, we use this quantitative factorization result to decompose subgraphs of Hamiltonians on $\BA$ that are $C^\infty$ small perturbations of affine Hamiltonians into balls and polydisks.

In Section \ref{sec:bounding_subleading_asymptotics}, we prove Theorems \ref{thm:subleading_asymptotics_ech}, \ref{thm:subleading_asymptotics_pfh}, and \ref{thm:subleading_asymptotics_link_invariants}, establishing bounds on the error terms in the Weyl laws for ECH capacities, PFH spectral invariants, and link spectral invariants. The proof of the ECH Weyl law for domains in $\BR^4$ given in \cite{hut11} relies on packing domains with balls up to arbitrarily small error, where the number of balls depends on the error. The more efficient ball packing provided by Theorem \ref{thm:ball_packing_stability} allows us to upgrade the estimates in \cite{hut11} and obtain $O(1)$ bounds on the error terms. Similarly, in the case of PFH, the computations in \cite{eh}, which also rely on ball packings, can be improved to yield bounded error terms. Finally, the bounds on the error terms in the Weyl law for link spectral invariants are derived from the PFH bounds via Chen’s works \cite{chea,cheb}, which relate the two sets of spectral invariants.

Finally, in Section \ref{sec:packing_failure}, we construct a star-shaped domain $X\subset \BR^4$ with regularity slightly less than $C^2$ for which packing stability fails. To prove the failure of packing stability, we estimate the error terms in the ECH Weyl law (strictly speaking, the Weyl law for alternative ECH capacities, see Remark \ref{rem:elementary_ech_capacities}) and show that they diverge to negative infinity. It is noteworthy that our domain $X$ is not toric, so we cannot appeal to known computational techniques for ECH capacities of toric domains. Instead, our estimates rely on translating the problem of estimating the ECH capacities of $X$ into estimating the link spectral invariants of a certain Hamiltonian, which, in our situation, is much easier. This approach to estimating ECH capacities may be of independent interest; see \cite[Problem 2.4 \& 2.5]{aim}.

\subsection{Discussion and open problems}
\label{subsec:open_problems}

One of the main insights of this work is the connection between symplectic packing stability for manifolds with smooth boundary and the algebraic structure of Hamiltonian diffeomorphism groups, particularly their simplicity and perfectness properties. This reflects a broader pattern: several deep questions about Hamiltonian diffeomorphism groups---especially those related to Hofer geometry---have counterparts in the theory of symplectic embeddings. Illustrative examples include recent parallel developments concerning the large-scale geometry of the Hofer metric on Hamiltonian diffeomorphism groups and the symplectic Banach--Mazur metric on spaces of symplectic domains; see, e.g., \cite{ps23, chs24b, chb} and the references therein. Another instance is provided by closely related results on the local structure of geodesics with respect to the Hofer distance \cite{bp94, lm95} and the symplectic Banach--Mazur distance \cite{abe}.

In this light, the failure of packing stability for domains with rough boundary (Theorem \ref{thm:failure_packing_stability}) may be viewed as an analogue of recent non-simplicity theorems for various groups of Hamiltonian homeomorphisms, beginning with the inital breakthrough \cite{chs24} resolving the long-standing \textit{simplicity conjecture}. It is proved in \cite{chs24b} that the group of area preserving homeomorphisms of the $2$-sphere, denoted $\overline{\operatorname{Ham}}(S^2)$, is not simple---contrasting the simplicity of the smooth group $\operatorname{Ham}(S^2)$. This non-simplicity result does not directly correspond to failure of packing stability for domains with rough boundary, since the proof relies, loosly speaking, on detecting area preserving homeomorphisms with ``infinite'' Hofer energy and Calabi invariant, which would correspond to symplectic domains of infinite volume, for which the ball packing numbers are not even defined. However, $\overline{\operatorname{Ham}}(S^2)$ contains a natural subgroup $\operatorname{Hameo}(S^2)$ of \textit{Hamiltonian hameomorphisms}, introduced by Oh and M\"{u}ller \cite{om07}, whose elements have finite Hofer energy. It was shown in \cite{chmss} that $\operatorname{Hameo}(S^2)$ is not a simple group either, and this result serves as a close analogue of Theorem \ref{thm:failure_packing_stability}.

Given a compact, connected symplectic $4$-manifold $(M, \omega)$ with smooth boundary, a natural problem is to estimate the threshold $k_0 = k_0(M)$ in Theorem \ref{thm:ball_packing_stability} at which the ball packing numbers $p_k(M)$ stabilize to $1$. This appears to be a difficult question in general, as the threshold is highly sensitive to the geometry of the boundary of $M$. For example, if $X_n\subset \BR^4$ is a sequence of smooth domains converging to a rough domain $X_*\subset \BR^4$ as in Theorem \ref{thm:failure_packing_stability}, for which packing stability fails, then the stability thresholds $k_0(X_n)$ tend to infinity. Interestingly, for the toy model domains $(M,\Omega)$ inside $W$ considered in Subsection \ref{subsec:discussion}, the stability threshold $k_0(M)$ can be estimated in terms of the minimal number of rotations needed to factor $\varphi_H^1$. Closely related quantities---such as the \textit{autonomous norm} of a Hamiltonian diffeomorphism---have been studied by various authors; see, for instance, \cite{gg04, bk13,bks18}. It would be interesting to further investigate this relationship.

Another natural problem is to determine explicit bounds on the error terms in the ECH Weyl law. As before, this is challenging due to the strong dependence of the error terms on the boundary of the domain. Notably, the problem of bounding these error terms is closely tied to that of estimating ball packing stability thresholds. Indeed, as seen in our proof of Theorem \ref{thm:subleading_asymptotics_ech} in Section \ref{sec:bounding_subleading_asymptotics}, one can obtain explicit bounds on the error terms $e_k(X)$ for a smooth domain $X\subset \BR^4$ in terms of the stability thresholds $k_0(X)$ and $k_0(B\setminus X)$, where $B\subset \BR^4$ is a sufficiently large ball containing $X$.

The Ruelle invariant $\operatorname{Ru}(X)$, which appears in Hutchings' conjecture \eqref{eq:hutchings_conjecture} on the error terms in the ECH Weyl law, can be defined for star-shaped domains $X\subset \BR^4$ with $C^2$ boundary and it is continuous with respect to the $C^2$ topology on the space of star-shaped domains \cite[Proposition 2.13 c]{ce22}. In light of Hutchings' conjecture, one would therefore expect that the error terms $e_k(X)$ are $O(1)$ for compact domains with $C^2$ boundary. However, we do not yet know how to prove this. We show (Proposition \ref{prop:divergent_error_terms_ech}) that the domains in Theorem \ref{thm:failure_packing_stability}, which have regularity just below $C^2$ and for which packing stability fails, have unbounded error terms $e_k^{\operatorname{alt}}$ in the Weyl law for the alternative ECH capacities (see Remark \ref{rem:elementary_ech_capacities}). The same is expected to hold for the error terms $e_k$ in the usual ECH Weyl law, but we strictly speaking do not prove this since we rely on results from \cite{eh} which are only stated for the alternative ECH capacities. This leads to the following natural question, which could serve as a test for Hutchings' conjecture.

\begin{question}
Is $C^2$ the regularity threshold at which the error terms in the ECH Weyl law start to be $O(1)$?
\end{question}

If we knew packing stability for domains with $C^2$ boundary, this would imply $O(1)$ subleading asymptotics for such domains. As already noted above, we expect that our methods yield packing stability for domains in some finite differentiability class, but we do not expect that our present techniques can reach $C^2$ regularity. The reason for this is that our packing construction uses (quantitative versions of) Banyaga's simplicity results, which ultimately rely on a KAM result due to Herman \cite[Annexe, Theorem 2.2]{her70}. This result says that a $C^\infty$ small Hamiltonian perturbation of a Diophantine rotation of the $2$-dimensional torus $T^2$ is still conjugated to the original rotation. For sufficiently Diophantine rotation vectors, this continues to hold for $C^{3+\varepsilon}$ perturbations, but fails for $C^{3-\varepsilon}$ perturbations, see e.g. \cite{her83, cw13}. Since $C^{4+\varepsilon}$ Hamiltonians generate $C^{3+\varepsilon}$ Hamiltonian diffeomorphisms, it seems plausible that our methods can show packing stability with $C^{4+\varepsilon}$ boundary, but they likely break down for lower regularity.

\begin{question}
What is the regularity threshold at which packing stability starts to hold? Does it agree with the regularity threshold for $O(1)$ subleading asymptotics in the ECH Weyl law?
\end{question}

The star-shaped domain $X\subset \BR^4$ of regularity just below $C^2$ in our counterexample to packing stability (Theorem~\ref{thm:failure_packing_stability}) is not convex. In fact, there does not even exist a symplectomorphism of $\BR^4$ mapping $X$ to a convex domain. The reason is that the sectional curvature of the boundary of $X$ (away from the single non-smooth point) is unbounded from below. A more interesting question is whether the interior of $X$ is symplectomorphic to the interior of a convex domain. It is reasonable to conjecture that this is not the case. Indeed, by \cite{ce22} the Ruelle invariant of a smooth convex domain can be bounded from above and below by constants only depending on the volume and the action of a systole of the domain. Since volume and action of a systole are continuous on the space of convex domains with respect to the Hausdorff topology, Hutchings' conjecture~\eqref{eq:hutchings_conjecture} suggests that the error terms in the ECH Weyl law of any (not necessarily smooth) convex domain should be bounded. Since the error terms of the domain $X$ in Theorem~\ref{thm:failure_packing_stability} diverge to negative infinity, one expects the interior of $X$ to not be symplectomorphic to the interior of a convex domain.

In view of this discussion, ECH capacities conjecturally do not yield obstructions to symplectic packing stability for general convex domains and we are led to the following question:

\begin{question}
Does packing stability hold for the interiors of arbitrary, not necessarily smooth, convex domains?
\end{question}

Let $(Y^3,\xi)$ be a closed connected contact $3$-manifold. Let $\Gamma \in H_1(Y;\BZ)$ be a homology class such that $c_1(\xi)+2\operatorname{PD}(\Gamma)$ is a torsion element of $H_2(Y;\BZ)$ and let $I$ be an absolute $\BZ$-grading of the embedded contact homology $ECH(Y,\xi,\Gamma)$. Then the ECH Weyl law in \cite[Theorem 1.3]{chr15} states that, for every sequence of non-zero homogeneous ECH classes $\sigma_k \in ECH(Y,\xi,\Gamma)$ with grading $I(\sigma_k)$ diverging to $+\infty$ and for every contact form $\lambda$ defining $\xi$, the ECH spectral invariants $c_{\sigma_k}(Y,\lambda)$ satisfy the asymptotic formula
\begin{equation}
\label{eq:general_ech_weyl_law}
c_{\sigma_k}(Y,\lambda) = \sqrt{\operatorname{vol}(Y,\lambda) I(\sigma_k)} + o(I(\sigma_k)^{1/2}) \qquad (k\rightarrow \infty).
\end{equation}
In view of Theorem \ref{thm:subleading_asymptotics_ech}, the following question is natural:

\begin{question}
Are the error terms in the ECH Weyl law \eqref{eq:general_ech_weyl_law} bounded as well?
\end{question}

In this direction, our packing stability result Theorem \ref{thm:ball_packing_stability} together with the arguments in \cite[\S 3]{chr15} imply the inequality
\begin{equation*}
c_{\sigma_k}(Y,\lambda) \geq \sqrt{\operatorname{vol}(Y,\lambda) I(\sigma_k)} + O(1) \qquad (k\rightarrow \infty).
\end{equation*}
Moreover, if the error terms are $O(1)$ for one single contact form $\lambda_0$ defining $\xi$, then they are $O(1)$ for every contact form $\lambda$ defining $\xi$.

Finally, it is natural to explore packing stability in higher dimensions. Since Banyaga's results apply in arbitrary dimensions, it seems promising to attempt to extend the $4$-dimensional symplectic embedding techniques developed in this paper to higher dimensions. However, the question of whether there exist higher-dimensional domains with rough boundaries for which packing stability fails remains wide open. Given the connections between symplectic packing and the algebraic structure of Hamiltonian diffeomorphism groups established in this paper, this question can be viewed as a variant of the higher-dimensional version of the simplicity conjecture, which asks whether the groups of compactly supported Hamiltonian homeomorphisms $\overline{\operatorname{Ham}}(B^{2n}(1))$ and hameomorphisms $\operatorname{Hameo}(B^{2n}(1))$ of the open ball $B^{2n}(1)$ in dimension $2n \geq 4$ are simple or not; see, for instance, \cite{mss}. We expect these questions to be of the same level of difficulty.\\

\paragraph{\textbf{Acknowledgements}}

The author thanks Dan Cristofaro-Gardiner, Richard Hind, and Michael Hutchings for helpful discussions and Felix Schlenk for his interest and comments on a draft of this paper. The author is supported by Dr.\ Max R\"{o}ssler, the Walter Haefner Foundation, and the ETH Z\"{u}rich Foundation.

\section{Hamiltonian diffeomorphisms of surfaces with boundary}
\label{sec:ham_diffeo_surface_with_boundary}

Given a symplectic surface with boundary, we introduce Hamiltonian diffeomorphisms which are not required to restrict to the identity on the boundary. This is a slightly non-standard notion, but has appeared before in the literature in similar forms; see, for instance, \cite[\S 2]{abhs18}. Moreover, we state Theorem \ref{thm:quantitative_perfectness} proved in \cite{edt}, which is a quantitative refinement of Banyaga's perfectness and simplicity results for Hamiltonian diffeomorphism groups \cite{ban78, ban97}.\\

Let $(\Sigma,\omega)$ be a $2$-manifold equipped with an area form. We allow $\Sigma$ to have boundary and do not require it to be compact. Let $\operatorname{Diff}_c(\Sigma,\omega)$ denote the group of all compactly supported diffeomorphisms of $\Sigma$ which preserve the area form $\omega$. Diffeomorphisms $\varphi\in \operatorname{Diff}_c(\Sigma,\omega)$ are not required to restrict to the identity on the boundary $\partial\Sigma$. Consider an isotopy $\varphi^t$ in $\operatorname{Diff}_c(\Sigma,\omega)$. Let $X_t$ be the time dependent vector field generating $\varphi^t$. Since $\varphi^t$ is a compactly supported area preserving isotopy, the vector field $X_t$ is compactly supported, tangent to the boundary $\partial\Sigma$, and satisfies $\ML_{X_t}\omega = 0$. Conversely, any such vector field $X_t$ generates a compactly supported isotopy in $\operatorname{Diff}_c(\Sigma,\omega)$. Vector fields $X_t$ satisfying these properties are in bijective correspondence with families of compactly supported closed $1$-forms $\alpha_t$ on $\Sigma$ which vanish on vectors tangent to the boundary $\partial \Sigma$. This correspondence is characterized by the identity $\alpha_t = \iota_{X_t} \omega$.

Consider a time dependent Hamiltonian $H:[0,1]\times \Sigma \rightarrow \BR$. The induced family of closed $1$-forms $dH_t$ is compactly supported and vanishes on vectors tangent to the boundary $\partial\Sigma$ if and only if, for every fixed time $t$, the function $H_t$ is locally constant on the boundary $\partial\Sigma$ and outside some compact subset of $\Sigma$. If this is the case, we call the Hamiltonian $H$ \textit{$\partial$-admissible}. If $H$ is $\partial$-admissible, then the vector field $X_{H_t}$ characterized by $dH_t = \iota_{X_{H_t}}\omega$ induces an isotopy $\varphi_H^t$ in $\operatorname{Diff}_c(\Sigma,\omega)$.

If $H$ and $G$ are $\partial$-admissible Hamiltonians, we define new Hamiltonians $H\#G$ and $\overline{H}$ by
\begin{equation}
\label{eq:hamiltonian_of_composition}
(H\#G)_t \coloneqq H_t + G_t \circ (\varphi_G^t)^{-1} \qquad \text{and} \qquad \overline{H}_t \coloneqq - H_t \circ \varphi_H^t.
\end{equation}
The Hamiltonians $H\#G$ and $\overline{H}$ are also $\partial$-admissible. They generate the isotopies $\varphi_H^t\circ \varphi_G^t$ and $(\varphi_H^t)^{-1}$, respectively.

We call an isotopy $\varphi^t$ in $\operatorname{Diff}_c(\Sigma,\omega)$ a \textit{compactly supported Hamiltonian isotopy} if it is of the form $\varphi^t = \varphi_H^t$ for a $\partial$-admissible $H$ which in addition is compactly supported and vanishes on the boundary $\partial \Sigma$. We call a diffeomorphism $\varphi\in \operatorname{Diff}_c(\Sigma,\omega)$ a \textit{compactly supported Hamiltonian diffeomorphism} if there exists a compactly supported Hamiltonian isotopy $(\varphi^t)_{t\in [0,1]}$ starting at $\varphi^0=\operatorname{id}$ and ending at $\varphi^1 = \varphi$. The compactly supported Hamiltonian diffeomorphisms form a subgroup
\begin{equation*}
\operatorname{Ham}_c(\Sigma,\omega)\subset \operatorname{Diff}_c(\Sigma,\omega).
\end{equation*}
We emphasize that since $H$ is only required to vanish on the boundary $\partial\Sigma$ but not in a neighbourhood of it, Hamiltonian diffeomorphisms need not restrict to the identity on the boundary. We caution that while every $\partial$-admissible Hamiltonian $H$ induces an isotopy in $\operatorname{Diff}_c(\Sigma,\omega)$, this isotopy is not necessarily a Hamiltonian isotopy.

Suppose now that $\Sigma$ is connected and not closed. Let $\widetilde{\operatorname{Ham}}_c(\Sigma,\omega)$ denote the universal cover of $\operatorname{Ham}_c(\Sigma,\omega)$. The \textit{Calabi homomorphism}
\begin{equation*}
\operatorname{Cal} : \widetilde{\operatorname{Ham}}_c(\Sigma,\omega) \rightarrow \BR
\end{equation*}
is defined as follows: Given an element $\widetilde{\varphi}\in \widetilde{\operatorname{Ham}}_c(\Sigma,\omega)$, choose a Hamiltonian isotopy $(\varphi^t)_{t\in [0,1]}$ representing $\widetilde{\varphi}$. Let $H$ be the unique Hamiltonian generating $\varphi^t$ which is compactly supported and vanishes on the boundary $\partial\Sigma$. Then $\operatorname{Cal}(\widetilde{\varphi})$ is defined by
\begin{equation*}
\operatorname{Cal}(\widetilde{\varphi}) \coloneqq \int_{[0,1]\times \Sigma} H dt\wedge \omega.
\end{equation*}
This integral turns out to be independent of the choice of representative $\varphi^t$ of $\widetilde{\varphi}$. We define
\begin{equation*}
\operatorname{Ham}_c^0(\Sigma,\omega)\subset \operatorname{Ham}_c(\Sigma,\omega)
\end{equation*}
to be the subgroup consisting of all Hamiltonian diffeomorphisms $\varphi\in \operatorname{Ham}_c(\Sigma,\omega)$ which possess a lift $\widetilde{\varphi}\in \widetilde{\operatorname{Ham}}_c(\Sigma,\omega)$ which is contained in the kernel of the Calabi homomorphism $\operatorname{Cal}$.

For closed surfaces $\Sigma$, it will be convenient to set $\operatorname{Ham}_c^0(\Sigma,\omega) \coloneqq \operatorname{Ham}(\Sigma,\omega)$. In the case that $\Sigma$ is disconnected, we define $\operatorname{Ham}_c^0(\Sigma,\omega)$ to be the subgroup consisting of all Hamiltonian diffeomorphisms $\varphi\in \operatorname{Ham}_c(\Sigma,\omega)$ such that, for every connected component $\Sigma'$ of $\Sigma$, the restriction of $\varphi$ to $\Sigma'$ belongs to $\operatorname{Ham}_c^0(\Sigma',\omega)$.

It was proved by Banyaga in \cite{ban78} (see also \cite[\S 4.3]{ban97}) that if $(\Sigma,\omega)$ does not have boundary, then the group $\operatorname{Ham}_c^0(\Sigma,\omega)$ is perfect, i.e. it agrees with its commutator subgroup. If in addition $\Sigma$ is connected, then $\operatorname{Ham}_c^0(\Sigma,\omega)$ is also simple, i.e. it does not have any non-trivial normal subgroups. If $\Sigma$ has non-empty boundary, then the group $\operatorname{Ham}_c^0(\Sigma,\omega)$ is never simple because the diffeomorphisms restricting to the identity on the boundary form a non-trival normal subgroup. However, we prove in \cite{edt} that $\operatorname{Ham}_c^0(\Sigma,\omega)$ is still perfect. In the context of general diffeomorphism groups, an analogous result is obtained in \cite{ryb98}. However, the methods there do not adapt well to the conservative setting and our proofs in \cite{edt} rely on different tools.

For our purposes, the following refinement of the perfectness of $\operatorname{Ham}_c^0(\Sigma,\omega)$ proved in \cite{edt} will be essential. In the non-conservative setting, similar results were obtained in \cite{hrt13}, but it is unclear whether the proofs there can be adapted to the conservative settings and our arguments in \cite{edt} are rather different.

\begin{theorem}[Quantitative perfectness \cite{edt}]
\label{thm:quantitative_perfectness}
There exists a positive integer $m>0$ such that the following is true. Let $(\Sigma,\omega)$ be a symplectic surface, possibly with boundary and not necessarily compact, and let $U\Subset V\Subset \Sigma$ be relatively compact open subsets. Consider the map
\begin{equation*}
\Phi : \operatorname{Ham}_c^0(V)^{2m} \rightarrow \operatorname{Ham}_c^0(V) \quad (u_1,v_1,\dots,u_m,v_m) \mapsto \prod\limits_{i=1}^m [u_i,v_i].
\end{equation*}
Let us equip both $\operatorname{Ham}_c^0(V)$ and $\operatorname{Ham}_c^0(U)$ with the topology of uniform $C^\infty$ convergence on the compact set $\overline{V}$. Then there exist a $C^\infty$ open neighbourhood $\MU\subset \operatorname{Ham}_c^0(U)$ of the identity and a continuous map
\begin{equation*}
\Psi : \MU \rightarrow \operatorname{Ham}_c^0(V)^{2m}
\end{equation*}
which is a right inverse of $\Phi$, i.e. which satisfies $\Phi\circ \Psi = \operatorname{id}_{\mathcal{U}}$. Moreover, we can choose $\MU$ and $\Psi$ such that $\Psi(\operatorname{id})$ is arbitrarily $C^\infty$ close to the tuple $(\operatorname{id},\dots,\operatorname{id})$.
\end{theorem}

\begin{remark}
The right inverse $\Psi$ in Theorem \ref{thm:quantitative_perfectness} can in fact be chosen to be smooth in a suitable sense, see \cite{edt}. However, we will not need this here.
\end{remark}

\section{Quantitative factorization into rotations}
\label{sec:quant_decomp_rot}

The main result of this section is Theorem \ref{thm:rot_decomp}, which allows us to write Hamiltonian diffeomorphisms of the annuls which are $C^\infty$ close to the identity and have vanishing Calabi invariant as a composition of finitely many smooth conjugates of rotations in a quantitatively controlled way. A key ingredient is the quantitative perfectness result for Hamiltonian diffeomorphism groups Theorem \ref{thm:quantitative_perfectness}. Some of our arguments are inspired by the proof that perfectness implies simplicity for certain groups of homeomorphisms; see \cite{eps70} and \cite[Theorem 2.1.7]{ban97}.\\

Let $\BT\coloneqq \BR/\BZ$ denote the circle and consider the annulus $\BA\coloneqq [0,1]\times \BT$ equipped with coordinates $(x,y)$ and the area form $\omega\coloneqq dx\wedge dy$. For every real number $\alpha\in\BR$, we define the autonomous Hamiltonian
\begin{equation*}
H_\alpha:[0,1] \times \BA\rightarrow \BR \qquad H_\alpha(t,x,y)\coloneqq \alpha x.
\end{equation*}
The Hamiltonian $H_\alpha$ is $\partial$-admissible and the Hamiltonian vector field of $H_\alpha$ is given by $X_{H_\alpha} = -\alpha\partial_y$. The time-$1$ map $\varphi_{H_\alpha}^1$ of the induced flow is the rotation
\begin{equation*}
R_\alpha : \BA\rightarrow \BA \qquad R_\alpha(x,y) \coloneqq (x, y-\alpha).
\end{equation*}

\begin{theorem}
\label{thm:rot_decomp}
There exists an integer $\ell>0$ divisible by $4$ such that the following statement is true for every $\alpha\in \BR\setminus \frac{1}{2}\BZ$ and for every $C^\infty$ neighbourhood $\MU\subset \operatorname{Ham}^0(\BA)$ of the identity:

For every Hamiltonian diffeomorphism $\varphi \in \operatorname{Ham}^0(\BA)$ which is sufficiently $C^\infty$ close to the identity, there exist Hamiltonian diffeomorphisms $\psi_1,\dots,\psi_\ell \in \MU$ depending continuously on $\varphi$ with respect to the $C^\infty$ topology and satisfying the following property: For $1\leq i \leq \ell$, define $\varepsilon_i\coloneqq +1$ if $i$ is congruent to $0$ or $1$ modulo $4$ and $\varepsilon_i\coloneqq -1$ otherwise. Moreover, define $H_i \coloneqq \varepsilon_i \psi_i^* H_\alpha$. Then
\begin{equation*}
\varphi = \prod\limits_{i=1}^\ell \varphi_{H_i}^1.
\end{equation*}
\end{theorem}

Let us abbreviate the two components of $\partial\BA$ by
\begin{equation*}
\partial_+\BA\coloneqq \left\{ 1 \right\}\times \BT \quad \text{and} \quad \partial_-\BA \coloneqq \left\{ 0 \right\}\times \BT.
\end{equation*}

\begin{corollary}
\label{cor:rot_decomp}
Let $\ell >0$ be an integer as in Theorem \ref{thm:rot_decomp}. Then for every $\alpha\in \BR\setminus \frac{1}{2}\BZ$, every $C^\infty$ neighbourhood $\MU \subset \operatorname{Ham}^0(\BA)$ of the identity, and every $C^\infty$ neighbourhood $\MV \subset C^\infty([0,1]\times \BA)$ of zero, the following statement holds:

For every $\partial$-admissible Hamiltonian $H:[0,1]\times \BA \rightarrow \BR$ which is sufficiently $C^\infty$ close to $0$ and satisfies
\begin{equation}
\label{eq:corollary_rot_decomp_hamiltonian}
\int\limits_{[0,1]\times \BA} H dt\wedge \omega = 0 \quad \text{and} \quad \int\limits_{[0,1]} H(t,\partial_\pm\BA) dt = 0,
\end{equation}
there exists a $1$-parameter family of $\partial$-admissible Hamiltonians $(H^\lambda)_{\lambda\in [0,1]}$ starting at $H^0 = H$ and satisfying the following properties:
\begin{enumerate}
\item The time-$1$-map $\varphi_{H^\lambda}^1$ is independent of $\lambda \in [0,1]$.
\item For all $\lambda \in [0,1]$, we have
\begin{equation*}
\int\limits_{[0,1]\times \BA} H^\lambda dt\wedge\omega = 0 \quad \text{and} \quad \int_{[0,1]} H^\lambda(t,\partial_\pm\BA) dt = 0.
\end{equation*}
\item For all $\lambda \in [0,1]$, we have $H^\lambda \in \MV$.
\item There exist Hamiltonian diffeomorphisms $\psi_i \in \MU$ such that
\begin{equation*}
H^1 = \#_{i=1}^\ell \varepsilon_i \psi_i^*H_\alpha,
\end{equation*}
where the signs $\varepsilon_i$ are defined as in Theorem \ref{thm:rot_decomp}.
\end{enumerate}
\end{corollary}

\begin{proof}
Note that equation \eqref{eq:corollary_rot_decomp_hamiltonian} implies that $\varphi \coloneqq \varphi_H^1 \in \operatorname{Ham}^0(\BA)$. Indeed, define a $\partial$-admissible Hamiltonian $K$ by
\begin{equation*}
K : [0,1]\times \BA \rightarrow \BR \qquad K(t,x,y) \coloneqq -x H(t,\partial_+\BA) - (1-x) H(t,\partial_-\BA).
\end{equation*}
Then $K\# H$ clearly vanishes on $\partial\BA$. Moreover, it follows from equation \eqref{eq:corollary_rot_decomp_hamiltonian} that $\varphi_K^1 = \operatorname{id}_\BA$ and that $\int_{[0,1]\times\BA} K dt \wedge \omega = 0$. This implies that $\varphi_{K\#H}^1 = \varphi_H^1$ and that $\int_{[0,1]\times \BA} K\# H dt \wedge \omega = 0$. We deduce that $\varphi_H^1 \in \operatorname{Ham}^0(\BA)$.

We can therefore apply Theorem \ref{thm:rot_decomp} to $\varphi$. Proving Corollary \ref{cor:rot_decomp} then boils down to proving the following elementary assertion:

Let $G^0$ and $G^1$ be two $C^\infty$ small $\partial$-admissible Hamiltonians on $\BA$. Assume that $\varphi_{G^0}^1 = \varphi_{G^1}^1$ and that
\begin{equation}
\label{eq:corollary_rot_decomp_proof}
\int\limits_{[0,1]\times \BA} G^\lambda dt\wedge\omega = 0 \quad \text{and} \quad \int_{[0,1]} G^\lambda(t,\partial_\pm\BA) dt = 0
\end{equation}
for $\lambda \in \left\{ 0,1 \right\}$. Then there exists a family $(G^\lambda)_{\lambda\in [0,1]}$ of $C^\infty$ small $\partial$-admissible Hamiltonians on $\BA$ connecting $G^0$ to $G^1$ and satisfying $\varphi_{G^\lambda}^1 = \varphi_{G^0}^1$ and identity \eqref{eq:corollary_rot_decomp_proof} for all $\lambda \in [0,1]$.

In order to see this, first note that we can connect $G^0$ to a Hamiltonian which vanishes on $\partial\BA$ as follows. Define
\begin{equation*}
K^\lambda : [0,1]\times \BA \rightarrow \BR \qquad K^\lambda(t,x,y) \coloneqq -\lambda (x G^0(t,\partial_+\BA) + (1-x) G^0(t,\partial_-\BA)).
\end{equation*}
As already observed above for the Hamiltonian $K$, we have $\varphi_{K^\lambda}^1 = \operatorname{id}_\BA$ for all $\lambda$. Thus $(K^\lambda \# G^0)_{\lambda\in [0,1]}$ connects $G^0$ to the Hamiltonian $K^1\#G^0$, which vanishes on $\partial\BA$. Via an analogous construction, we connect $G^1$ to a Hamiltonian vanishing on $\partial\BA$.

We may therefore replace $G^0$ and $G^1$ by Hamiltonians vanishing on $\partial\BA$. In this situation, the Hamiltonians $G^j$ for $j \in \left\{ 0,1 \right\}$ generate paths $(\varphi_{G^j}^t)_{t\in [0,1]}$ in $\operatorname{Ham}(\BA)$ starting at the identity and ending at $\varphi$. The desired claim follows from the fact that $\operatorname{Ham}(\BA)$ is locally contractible with respect to the $C^\infty$ topology. As in the usual case of Hamiltonian diffeomorphisms on symplectic manifolds without boundary, this can be proved using generating functions. We refer to \cite[\S 2]{abhs18} for a discussion of generating functions in the context of Hamiltonian diffeomorphisms which do not necessarily restrict to the identity on the boundary.
\end{proof}

Before turning to the proof of Theorem \ref{thm:rot_decomp}, we establish some preliminary lemmas.

\begin{lemma}
\label{lem:commutators}
Let $X$ be a topological space and let $U\subset X$ be an open subset. Let $u$ and $v$ be homeomorphisms of $X$ supported in $U$. Let $\varphi$ and $\psi$ be a homeomorphisms of $X$ such that the three sets $U$, $\varphi(U)$, and $\psi(U)$ are pairwise disjoint. Then the following identity holds:
\begin{equation}
\label{eq:commutators}
[u,v] = [ [u,\varphi], [v,\psi]].
\end{equation}
\end{lemma}

\begin{proof}
We observe that
\begin{equation}
\label{eq:commutators_proof_a}
\operatorname{supp}([u,\varphi]) \subset U \cup \varphi(U) \quad \text{and} \quad [u,\varphi]|_U = u|_U.
\end{equation}
Similarly, we have
\begin{equation}
\label{eq:commutators_proof_b}
\operatorname{supp}([v,\psi]) \subset U\cup \psi(U) \quad \text{and} \quad [v,\psi]|_U = v|_U.
\end{equation}
Since any two homeomorphisms with disjoint supports commute, the desired identity \eqref{eq:commutators} is an immediate consequence of \eqref{eq:commutators_proof_a} and \eqref{eq:commutators_proof_b}.
\end{proof}

\begin{lemma}
\label{lem:from_commutators_to_rotations}
Let $U\subset \BA$ be an open subset of the annulus. Let $\alpha\in \BR$ and assume that the sets $U$, $R_\alpha(U)$, and $R_\alpha^{-1}(U)$ are pairwise disjoint. Let $u,v\in \operatorname{Ham}_c(U)$ be compactly supported Hamiltonian diffeomorphisms. Then we have
\begin{equation}
\label{eq:from_commutators_to_rotations}
[u,v] = \varphi_{(u^{-1})^*H_\alpha }^1 \circ
\varphi_{-H_\alpha}^1 \circ
\varphi_{-(v^{-1})^*H_\alpha }^1 \circ
\varphi_{H_\alpha}^1 \circ \varphi_{H_\alpha}^1 \circ
\varphi_{-(u^{-1})^*H_\alpha }^1 \circ
\varphi_{-H_\alpha}^1 \circ
\varphi_{(v^{-1})^*H_\alpha }^1.
\end{equation}
\end{lemma}

\begin{proof}
We apply Lemma \ref{lem:commutators} with $\varphi=R_\alpha$ and $\psi=R_\alpha^{-1}$ and evaluate the right hand side of identity \eqref{eq:commutators}.
\end{proof}

The following fragmentation lemma is proved in \cite{edt}.

\begin{lemma}
\label{lem:fragmentation_calabi}
Let $U_1,\dots,U_n$ be an open covering of the annulus $\BA$. Then, for every $\varphi \in \operatorname{Ham}^0(\BA)$ which is sufficiently $C^1$ close to the identity, there exists a fragmentation
\begin{equation*}
\varphi = \varphi_1\circ \cdots \circ \varphi_n\qquad \text{with}\enspace \varphi_i\in \operatorname{Ham}^0_c(U_i).
\end{equation*}
The diffeomorphisms $\varphi_i$ depend continuously on $\varphi$ with respect to the $C^\infty$ topology on $\operatorname{Ham}^0(\BA)$ and if $\varphi = \operatorname{id}$, then $\varphi_i=\operatorname{id}$ for all $i$.
\end{lemma}

For every $p \in \BA$ and $r>0$, let $B_r(p)\subset \BA$ denote the open ball of radius $r$ centered at $p$, where we equip $\BA$ with the distance function induced by the Riemannian metric $dx^2 + dy^2$.

\begin{lemma}
\label{lem:colorings}
For every positive integer $k>0$, we define the lattice
\begin{equation*}
P^k \coloneqq (\frac{1}{k}\BZ \cap [0,1]) \times (\frac{1}{k}\BZ / \BZ) \subset \BA.
\end{equation*}
Then there exists a positive integer $n>0$ with the following property: For every $\alpha \in \BR\setminus \frac{1}{2}\BZ$ and for every positive integer $k$ such that $\frac{2}{k} < \operatorname{dist}(\alpha,\frac{1}{2}\BZ)$, there exists an $n$-coloring of $P^k$ such that, for any two distinct points $p\neq q \in P^k$ of the same color, the six sets
\begin{equation*}
B_{\frac{1}{k}}(p),\enspace R_\alpha(B_{\frac{1}{k}}(p)),\enspace R_\alpha^{-1}(B_{\frac{1}{k}}(p)),\enspace B_{\frac{1}{k}}(q), \enspace R_\alpha(B_{\frac{1}{k}}(q)),\enspace R_\alpha^{-1}(B_{\frac{1}{k}}(q))
\end{equation*}
are pairwise disjoint.
\end{lemma}

\begin{proof}
Let $\alpha \in \BR\setminus \frac{1}{2}\BZ$ and let $k$ be a positive integer such that $\frac{2}{k} < \operatorname{dist}(\alpha,\frac{1}{2}\BZ)$. Note that this condition implies that, for every point $p \in \BA$, the three sets $B_{\frac{1}{k}}(p)$, $R_\alpha(B_{\frac{1}{k}}(p))$, and $R_\alpha^{-1}(B_{\frac{1}{k}}(p))$ are pairwise disjoint. We construct a graph $G$ with vertex set $P^k$ as follows: For any two distinct $p\neq q \in P^k$, we insert an edge between $p$ and $q$ if and only if the intersection
\begin{equation*}
(B_{\frac{1}{k}}(p) \cup R_\alpha(B_{\frac{1}{k}}(p)) \cup R_\alpha^{-1}(B_{\frac{1}{k}}(p))) \cap (B_{\frac{1}{k}}(q) \cup R_\alpha(B_{\frac{1}{k}}(q)) \cup R_\alpha^{-1}(B_{\frac{1}{k}}(q)))
\end{equation*}
is non-empty. Note that vertex colorings of $G$, i.e. colorings of the vertex set of $G$ such that no two vertices of the same color are connected by an edge, are precisely colorings of the lattice $P^k$ satisfying the property in the statement of the lemma. Hence our task is to show that the chromatic number $\gamma(G)$, i.e. the least number of colors needed to color $G$, admits an upper bound independent of $\alpha$ and $k$. Let $\Delta(G)$ denote the maximal vertex degree of $G$, i.e. the maximal number of incident edges of any vertex of $G$. By Brooks' theorem \cite{bro41}, the chromatic number of any graph satisfies the bound $\gamma(G) \leq \Delta(G) + 1$. We observe that, for each $p \in P^k$, there exist at most $9$ distinct points $q\in P^k$ such that $B_{\frac{1}{k}}(p)$ and $B_{\frac{1}{k}}(q)$ have non-empty intersection. This implies that $\Delta(G)$ admits an upper bound independent of $\alpha$ and $k$. We deduce that the same is true for $\gamma(G)$, which concludes the proof of the lemma.
\end{proof}

\begin{proof}[Proof of Theorem \ref{thm:rot_decomp}]
Pick a positive integer $m>0$ such that the conclusion of Theorem \ref{thm:quantitative_perfectness} holds. Moreover, pick a positive integer $n>0$ such that the conclusion of Lemma \ref{lem:colorings} holds. Our goal is to show that Theorem \ref{thm:rot_decomp} holds with
\begin{equation*}
\ell \coloneqq 8mn.
\end{equation*}

Let $\alpha \in\BR\setminus \frac{1}{2}\BZ$. Fix a positive integer $k>0$ such that $\frac{2}{k} < \operatorname{dist}(\alpha,\frac{1}{2}\BZ)$. Pick an $n$-coloring of the lattice $P^k\subset \BA$ as in Lemma \ref{lem:colorings}. Fix a number $\frac{\sqrt{2}}{2} < r < 1$. For each color $1\leq i \leq n$, we define
\begin{equation*}
U_i \coloneqq \bigcup\limits_{\substack{p\in P^k\\ \text{$p$ has color $i$}}} B_{\frac{r}{k}}(p) \qquad \text{and} \qquad
V_i \coloneqq \bigcup\limits_{\substack{p\in P^k\\ \text{$p$ has color $i$}}} B_{\frac{1}{k}}(p)
\end{equation*}
Both unions are disjoint and $U_i \Subset V_i\Subset \BA$. By the property of the coloring of $P^k$ in the statement of Lemma \ref{lem:colorings}, the sets $V_i$, $R_\alpha(V_i)$, and $R_\alpha^{-1}(V_i)$ are pairwise disjoint. Moreover, note that $(U_i)_i$ forms an open cover of $\BA$.

Suppose that $\varphi\in \operatorname{Ham}_c^0(\BA)$ is $C^\infty$ close to the identity. By Lemma \ref{lem:fragmentation_calabi}, there exists a fragmentation
\begin{equation*}
\varphi = \varphi_1 \circ \cdots \circ \varphi_n \quad \text{with $\varphi_i \in \operatorname{Ham}_c^0(U_i)$}.
\end{equation*}
The diffeomorphisms $\varphi_i$ depend continuously on $\varphi$ and are $C^\infty$ close to the identity. By Theorem \ref{thm:quantitative_perfectness}, we can write
\begin{equation*}
\varphi_i = [u_i^1,v_i^1] \circ \cdots \circ [u_i^m,v_i^m] \quad \text{with $u_i^j,v_i^j \in \operatorname{Ham}_c^0(V_i)$}.
\end{equation*}
The diffeomorphisms $u_i^j$ and $v_i^j$ also depend continuously on $\varphi$ and are $C^\infty$ close to the identity.

By Lemma \ref{lem:from_commutators_to_rotations}, we have
\begin{equation*}
[u_i^j,v_i^j] = \varphi_{((u_i^j)^{-1})^*H_\alpha }^1 \circ
\varphi_{-H_\alpha}^1 \circ
\varphi_{-((v_i^j)^{-1})^*H_\alpha }^1 \circ
\varphi_{H_\alpha}^1 \circ \varphi_{H_\alpha}^1 \circ
\varphi_{-((u_i^j)^{-1})^*H_\alpha }^1 \circ
\varphi_{-H_\alpha}^1 \circ
\varphi_{((v_i^j)^{-1})^*H_\alpha }^1.
\end{equation*}
For an appropriate choice of Hamiltonian diffoemorphisms $\psi_1,\dots,\psi_\ell$ from the set
\begin{equation*}
\left\{ \operatorname{id}, (u_i^j)^{-1}, (v_i^j)^{-1} \right\}
\end{equation*}
we then have
\begin{equation*}
\varphi = \prod\limits_{i=1}^\ell \varphi_{\varepsilon_i \psi_i^* H_\alpha}^1.
\end{equation*}
Since the diffeomorphisms $u_i^j$ and $v_i^j$ depend continuously on $\varphi$ and are $C^\infty$ close to the identity, the same if true for the diffeomorphisms $\psi_i$.
\end{proof}

\section{Stratification of subgraphs}
\label{sec:strat_of_subgraphs}

In this section we prove Theorem \ref{thm:stratification_perturbed_affine_subgraph}, which says that the subgraph of a Hamiltonian on the annulus which is a $C^\infty$ small perturbation of an affine Hamiltonian can be decomposed into domains symplectomorphic to subgraphs of affine Hamiltonians. This decomposition induces a decomposition into symplectic ellipsoids and polydisks; see Remark \ref{rem:stratification_perturbed_affine_subgraph_ellipsoids_polydisks}.\\

It will be useful to introduce the following notation. Given a set $X$ and functions $f,g:X\rightarrow \BR$, we let $D(f,g)\subset \BR\times X$ denote the intersection between the supergraph of $f$ and the subgraph of $g$, i.e. we set
\begin{equation*}
D(f,g)\coloneqq\left\{ (s,x) \in \BR\times X \mid f(x) \leq s \leq g(x) \right\}. 
\end{equation*}
Moreover, we abbreviate $D(f) \coloneqq D(0,f)$. This is the subgraph of $f$, truncated at $0$.

We are particularly interested in the case that the set $X$ is of the form $X=[0,1]\times \Sigma$ for a symplectic surface $(\Sigma,\omega)$. In this case, we equip $M\coloneqq \BR\times [0,1]\times \Sigma$ with the symplectic form $\Omega \coloneqq ds\wedge dt + \omega$, where $(s,t)$ are the coordinates on $\BR\times [0,1]$. This turns the domains $D(f,g)\subset (M,\Omega)$ into symplectic manifolds.

Recall from Section \ref{sec:quant_decomp_rot} that $\BA = [0,1]\times \BT$ denotes the annulus. Moreover, recall that the Hamiltonian $H_\alpha : [0,1]\times \BA \rightarrow \BR$ is defined by $H_\alpha(t,x,y) = \alpha x$ for $\alpha \in \BR$. Given a real number $a>0$, we let $\BA(a)$ denote the annulus of area $a$. More precisely, we set $\BA(a) \coloneqq [0,a]\times \BT$ and equip it with the area form $\omega\coloneqq dx\wedge dy$.

\begin{theorem}
\label{thm:stratification_perturbed_affine_subgraph}
Let $h:[0,1]\times \BA \rightarrow \BR$ be a Hamiltonian of the form $h(t,x,y) = a + bx$ for real numbers $a$ and $b$. Assume that $h$ is strictly positive. Let $H:[0,1]\times \BA \rightarrow \BR$ be a $\partial$-admissible Hamiltonian. If $H$ is sufficiently $C^\infty$ close to $h$, then there exists a stratification $\MS$ of $D(H)$ such that the closure $\overline{S}$ of any top-dimensional stratum $S$ of $\MS$ is symplectomorphic either to $D(H_\alpha|_{[0,1]\times \BA(c)})$ for a rational number $\alpha=p/q>0$ and a real number $c>0$ or to $D(C|_{[0,1]\times \BA(c)})$ for a positive constant function $C$ and a real number $c>0$.

There exists an upper bound on the denominator $q$ of $\alpha$ which can be chosen uniformly among all $H$ sufficiently $C^\infty$ close to a fixed $h$. Moreover, there exists a uniform positive lower bound on the constants $c$ and $C$.
\end{theorem}

\begin{remark}
\label{rem:stratification_perturbed_affine_subgraph_ellipsoids_polydisks}
The full volume of $D(H_\alpha|_{[0,1]\times \BA(c)})$ can be filled by a single copy of the ellipsoid $E(c,\alpha c)$; see Section \ref{sec:tame_packings}. Similarly, $D(C|_{[0,1]\times \BA(c)})$ can be fully filled by the polydisk $P(c,C)$ of widths $c$ and $C$. Theorem \ref{thm:stratification_perturbed_affine_subgraph} therefore says that every sufficiently small perturbation of an affine subgraph can be fully packed by some finite collection of polydisks and rational ellipsoids.
\end{remark}

We begin with some preliminary lemmas.

\begin{lemma}
\label{lem:symplectomorphisms_of_subgraphs}
Let $\Sigma$ be a compact surface, possibly with boundary. Let $\omega$ be an area form on $\Sigma$. Let $H:[0,1]\times \Sigma\rightarrow \BR$ be a $\partial$-admissible Hamiltonian. Then the map
\begin{equation*}
f_H:\BR\times [0,1]\times \Sigma \rightarrow \BR\times [0,1]\times \Sigma \qquad (s,t,p) \mapsto (s+H(t,\varphi_H^t(p)),t,\varphi_H^t(p))
\end{equation*}
is a symplectomorphism of $(M,\Omega)$. If $F,G:[0,1]\times \Sigma\rightarrow \BR$ are $\partial$-admissible Hamiltonians, then $f_H$ maps $D(F,G)$ to $D(H\#F, H\#G)$.
\end{lemma}

\begin{proof}
Both assertions follow from a straightforward direct computation.
\end{proof}

\begin{lemma}
\label{lem:isotopies_of_subgraphs}
Let $\Sigma$ be a compact connected surface, possibly with boundary. Let $\omega$ be an area form on $\Sigma$. Let $H_{\pm}^\lambda:[0,1]\times \Sigma\rightarrow\BR$ be two smooth families of $\partial$-admissible Hamiltonians on $\Sigma$ parametrized by $\lambda\in [0,1]$. We assume that the strict inequality $H_+^\lambda > H_-^\lambda$ holds for all $\lambda\in [0,1]$. Then the following two statements are equivalent:
\begin{enumerate}
\item \label{item:isotopies_of_subgraphs_symplectomorphic_domains} There exists a smooth family $(\psi^\lambda)_{\lambda\in [0,1]}$ of symplectomorphisms of $(M,\Omega)$ starting at the identity such that $\psi^\lambda(D(H_-^0,H_+^0))=D(H_-^\lambda,H_+^\lambda)$.
\item \label{item:isotopies_of_subgraphs_constant_quantities} The following three assertions hold:
\begin{enumerate}[(i)]
\item \label{item:isotopies_of_subgraphs_return_map} There exists a smooth isotopy $(\chi^\lambda)_{\lambda\in [0,1]}$ of area preserving diffeomorphisms of $\Sigma$ starting at the identity such that
\begin{equation*}
(\varphi_{H_-^\lambda}^1)^{-1}\circ \varphi_{H_+^\lambda}^1 = \chi^\lambda \circ ((\varphi_{H_-^0}^1)^{-1}\circ \varphi_{H_+^0}^1) \circ (\chi^\lambda)^{-1}.
\end{equation*}
\item \label{item:isotopies_of_subgraphs_area} For all boundary components $S$ of $\Sigma$, the quantity
\begin{equation*}
\int_{[0,1]} (H_+^\lambda(t,S) - H_-^\lambda(t,S)) dt
\end{equation*}
is independent of $\lambda\in [0,1]$.
\item \label{item:isotopies_of_subgraphs_volume} The volume
\begin{equation*}
\operatorname{vol}(D(H_-^\lambda,H_+^\lambda)) = \int_{[0,1]\times \Sigma} (H_+^\lambda - H_-^\lambda) dt\wedge\omega
\end{equation*}
is independent of $\lambda\in[0,1]$.
\end{enumerate}
\end{enumerate}
\end{lemma}

\begin{proof}
Let us abbreviate $A^\lambda \coloneqq D(H_-^\lambda,H_+^\lambda)$. Note that $A^\lambda$ is a manifold with boundary and corners. The boundary $\partial A^\lambda$ has the following top dimensional strata:
\begin{itemize}
\item $\partial_\pm A^\lambda \coloneqq \operatorname{graph}(H_\pm^\lambda)$
\item $\partial_j A^\lambda \coloneqq (\BR\times \left\{ j \right\}\times \Sigma)\cap A^\lambda$ for $j\in \left\{ 0,1 \right\}$
\item $\partial_S A^\lambda \coloneqq (\BR\times [0,1]\times S)\cap A^\lambda$ where $S$ is a component of $\partial\Sigma$
\end{itemize}
Restricting the natural projection $M\rightarrow [0,1]$ to $\partial_\pm A^\lambda$ yields a fibration of $\partial_\pm A^\lambda$ over the interval $[0,1]$ with fiber $\Sigma$. The strata $\partial_j A^\lambda$ for $j\in \left\{ 0,1 \right\}$ also fiber over the interval with fiber $\Sigma$. We can choose the fibration such that the fibers are transverse to the vertical vector field $\partial_s$. Since the Hamiltonians $H_\pm^\lambda$ are assumed to be $\partial$-admissible, there exists, for every boundary component $S$ of $\partial\Sigma$, a subset $Q_S^\lambda\subset \BR\times [0,1]$ such that $\partial_SA^\lambda = Q_S^\lambda\times S$. The set $Q_S^\lambda$ is a smoothly embedded strip in $\BR\times [0,1]$ connecting $\BR\times \left\{ 0 \right\}$ and $\BR \times \left\{ 1 \right\}$. The area of $Q_S^\lambda$ is given by the integral
\begin{equation*}
\operatorname{area}(Q_S^\lambda) = \int_{[0,1]} (H_+^\lambda(t,S)-H_-^\lambda(t,S))dt.
\end{equation*}
The symplectic form $\Omega$ on $M$ induces a characteristic foliation on $\partial A^\lambda$. Let us orient this foliation such that the orientation of the foliation followed by the coorientation induced by the restriction of $\Omega$ yields the boundary orientation on $\partial A^\lambda$. On the strata $\partial_\pm A^\lambda$, the characteristic foliation is positively tangent to the vector field
\begin{equation}
\label{eq:isotopies_of_subgraphs_proof_characteristic_foliation}
\pm (\partial_t H_\pm^\lambda \cdot \partial_s + \partial_t + X_{H_\pm^\lambda}).
\end{equation}
In particular, it is transverse to the fibers of $\partial_\pm A^\lambda\rightarrow [0,1]$. The boundary of the boundary strata $\partial_\pm A^\lambda$ consists of two types of strata: the fibers over $0$ and $1$ of the projection $\partial_\pm A^\lambda\rightarrow [0,1]$ and the annuli of the form $\partial_\pm A^\lambda \cap \partial_S A^\lambda$ where $S$ is a component of $\partial\Sigma$. The characteristic foliation is transverse to the former type. Since the Hamiltonians $H_\pm^\lambda$ are assumed to be $\partial$-admissible, the characteristic foliation is tangent to the latter type. On the strata $\partial_j A^\lambda$, the characteristic foliation is positively tangent to $(-1)^j \partial_s$. For our choice of fibration $\partial_j A^\lambda\rightarrow [0,1]$, it is therefore transverse to the fibers. The union
\begin{equation}
\label{eq:isotopies_of_subgraphs_proof_union}
T^\lambda\coloneqq \partial_+A^\lambda\cup \partial_1A^\lambda\cup \partial_-A^\lambda \cup \partial_0 A^\lambda
\end{equation}
is not smooth at the corners $\partial_\pm A^\lambda \cap \partial_j A^\lambda$. By the above discussion, the characteristic foliation is transverse to these corners. Moreover, it is tangent to the boundary of $T^\lambda$. Therefore, the characteristic foliation induces a flow on $T^\lambda$ which is well-defined up positive reparametrization. This flow is transverse to the fibers of the fibrations of the strata $\partial_\pm A^\lambda$ and $\partial_jA^\lambda$. In fact, each fiber is a surface of section of the flow. Consider the fiber over $0$ of the projection $\partial_+ A^\lambda\rightarrow [0,1]$. Let $\iota^\lambda:\Sigma\hookrightarrow M$ denote the parametrization of this fiber which has the property the composition with the natural projection $M\rightarrow \Sigma$ agrees with the identity map of $\Sigma$. With respect to the parametrization $\iota^\lambda$, the first return map of the flow is given by $(\varphi_{H_-^\lambda}^1)^{-1} \circ \varphi_{H_+^\lambda}^1$. In other words, we can regard $T^\lambda$ as a (non-smooth) mapping torus of the diffeomorphism $(\varphi_{H_-^\lambda}^1)^{-1} \circ \varphi_{H_+^\lambda}^1$. On the strata $\partial_S A^\lambda$, the characteristic foliation is tangent to the circles $\left\{ s \right\}\times \left\{ t \right\}\times S$. Thus we can regard $\partial_SA^\lambda$ as the mapping torus of the identity map on the strip $Q_S^\lambda$.\\

It is an easy consequence of the above discussion that statement (\ref{item:isotopies_of_subgraphs_symplectomorphic_domains}) in Lemma \ref{lem:isotopies_of_subgraphs} implies statement (\ref{item:isotopies_of_subgraphs_constant_quantities}): A family of symplectomorphisms $\psi^\lambda$ as in \eqref{item:isotopies_of_subgraphs_symplectomorphic_domains} clearly has to preserve $\operatorname{vol}(A^\lambda)$ and $\operatorname{area}(Q_S^\lambda)$. Moreover, it maps the characteristic foliation on the boundary of $A^0$ to the characteristic foliation on the boundary of $A^\lambda$. Thus it induces a familiy of conjugating diffeomorphisms for the first return maps of the mapping tori $T^0$ and $T^\lambda$.\\

Conversely, assume that statement \eqref{item:isotopies_of_subgraphs_constant_quantities} in Lemma \ref{lem:isotopies_of_subgraphs} holds. Our construction of the family of symplectomorphisms $\psi^\lambda$ proceeds in several steps.\\

\emph{Step 1:} We reduce ourselves to the case that $\chi^\lambda = \operatorname{id}$ for all $\lambda\in [0,1]$. To this end, we define Hamiltonians $\widetilde{H}_\pm^\lambda \coloneqq H_\pm^\lambda \circ (\chi^\lambda)^{-1}$. It follows from assumption \eqref{item:isotopies_of_subgraphs_return_map} that
\begin{equation*}
(\varphi_{\widetilde{H}_-^\lambda}^1)^{-1}\circ \varphi_{\widetilde{H}_+^\lambda}^1 = (\varphi_{\widetilde{H}_-^0}^1)^{-1}\circ \varphi_{\widetilde{H}_+^0}^1.
\end{equation*}
In other words, the Hamiltonians $\widetilde{H}_\pm^\lambda$ satisfy assumption \eqref{item:isotopies_of_subgraphs_return_map} with $\widetilde{\chi}^\lambda = \operatorname{id}$. Moreover, the Hamiltonians $\widetilde{H}_\pm^\lambda$ satisfy assumptions \eqref{item:isotopies_of_subgraphs_area} and \eqref{item:isotopies_of_subgraphs_volume}. Let $\widetilde{A}^\lambda\subset M$ denote the region between the graphs of $\widetilde{H}_\pm^\lambda$. It is not hard to see that if statement \eqref{item:isotopies_of_subgraphs_symplectomorphic_domains} holds for the family of sets $\widetilde{A}^\lambda$, then it also holds for $A^\lambda$. Indeed, given a family of symplectomorphisms $\widetilde{\psi}^\lambda$ of $M$ starting at the identity and mapping $\widetilde{A}^0$ to $\widetilde{A}^\lambda$, simply define
\begin{equation*}
\psi^\lambda \coloneqq (\operatorname{id}_{\BR\times [0,1]}\times \chi^\lambda)^{-1} \circ \widetilde{\psi}^\lambda.
\end{equation*}
The symplectomorphism $\psi^\lambda$ maps $A^0 = \widetilde{A}^0$ to $A^\lambda$. After replacing $H_\pm^\lambda$ by $\widetilde{H}_\pm^\lambda$, we can therefore assume that $\chi^\lambda = \operatorname{id}$.\\

\emph{Step 2:} We define a family of homeomorphisms
\begin{equation*}
\psi^\lambda : T^0 \rightarrow T^\lambda
\end{equation*}
with the following properties:
\begin{enumerate}[(a)]
\item \label{item:isotopies_of_subgraphs_proof_psi_on_T_start} $\psi^0=\operatorname{id}$.
\item \label{item:isotopies_of_subgraphs_proof_psi_on_T_base_surface} $\psi^\lambda$ maps the fiber over $0$ of the fibration $\partial_+ A^0\rightarrow [0,1]$ to the fiber over $0$ of $\partial_+A^\lambda\rightarrow [0,1]$. Consider the parametrizations $\iota^0$ and $\iota^\lambda$ of these fibers defined above. With respect to these parametrizations, the restriction of $\psi^\lambda$ to these fibers is given by the identity map $\operatorname{id}$.
\item \label{item:isotopies_of_subgraphs_proof_psi_on_T_circle_fibers} For every $\lambda$, the boundary of $T^\lambda$ fibers over $\bigcup_S \partial Q_S^\lambda$ with circle fibers of the form $\left\{ s \right\}\times \left\{ t \right\} \times S$. The map $\psi^\lambda$ maps circle fibers of $\partial T^0$ to circle fibers of $\partial T^\lambda$.
\item \label{item:isotopies_of_subgraphs_proof_psi_on_T_diff_pm} The restriction of $\psi^\lambda$ to $\partial_\pm A^0$ is a diffeomorphism onto $\partial_\pm A^\lambda$.
\item \label{item:isotopies_of_subgraphs_proof_psi_on_T_diff_j} For $j\in \left\{ 0,1 \right\}$, the restriction of $\psi^\lambda$ to $\partial_jA^0$ is a diffeomorphism onto $\partial_jA^\lambda$.
\item \label{item:isotopies_of_subgraphs_proof_psi_on_T_Omega} $\psi^\lambda$ is compatible with $\Omega$, i.e. $(\psi^\lambda)^* \Omega|_{T^\lambda} = \Omega|_{T^0}$. In fact, $\psi^\lambda$ extends to a family of symplectomorphisms between open neighbourhoods of $T^0$ and $T^\lambda$ in $(M,\Omega)$.
\newcounter{counter_strat_subgraphs}
\setcounter{counter_strat_subgraphs}{\value{enumi}}
\end{enumerate}
The construction uses assumption \eqref{item:isotopies_of_subgraphs_return_map}. Here are the details. Let $(\varphi^\lambda_t)_t$ be a parametrization of the flow on $T^\lambda$ induced by the characteristic foliation. We can choose $\varphi^\lambda_t$ such that the time it takes a flow line to traverse each of the strata $\partial_\pm A^\lambda$ and $\partial_jA^\lambda$ is exactly $1$. Moreover, we can choose $\varphi_t^\lambda$ such that, for each fixed $t$, circle fibers on the boundary of $T^\lambda$ are mapped to circle fibers. It follows from assumption \eqref{item:isotopies_of_subgraphs_return_map} and the reduction in Step 1 that there exists a unique family of homeomorphisms $\psi^\lambda:T^0\rightarrow T^\lambda$ such that $\psi^\lambda$ is a conjugacy between the flows $(\varphi^0_t)_t$ and $(\varphi^\lambda_t)_t$, i.e. such that
\begin{equation*}
(\psi^\lambda)^{-1} \circ \varphi^\lambda_t \circ \psi^\lambda = \varphi^0_t,
\end{equation*}
and such that property \eqref{item:isotopies_of_subgraphs_proof_psi_on_T_base_surface} holds. The family of homeomorphisms $\psi^\lambda$ clearly satisfies properties \eqref{item:isotopies_of_subgraphs_proof_psi_on_T_start} - \eqref{item:isotopies_of_subgraphs_proof_psi_on_T_diff_j}. Moreover, we have $(\psi^\lambda)^* \Omega|_{T^\lambda} = \Omega|_{T^0}$. If $\varphi_t^\lambda$ is chosen with some care near the non-smooth locus of $T^\lambda$, one can extend $\psi^\lambda$ to a symplectomorphism between open neighbourhoods of $T^0$ and $T^\lambda$ and therefore also guarantee property \eqref{item:isotopies_of_subgraphs_proof_psi_on_T_Omega}.\\

\emph{Step 3:} We extend the family of homeomorphism $\psi^\lambda$ from Step 2 to a family of homeomorphisms
\begin{equation*}
\psi^\lambda : \partial A^0 \rightarrow \partial A^\lambda
\end{equation*}
such that:
\begin{enumerate}[(a)]
\setcounter{enumi}{\value{counter_strat_subgraphs}}
\item $\psi^0=\operatorname{id}$
\item $\psi^\lambda$ restricts to a diffeomorphism between $\partial_S A^0$ and $\partial_S A^\lambda$ for all $S$.
\item $\psi^\lambda$ is compatible with $\Omega$ and extends to a symplectomorphism between open neighbourhoods of $\partial A^0$ and $\partial A^\lambda$ in $(M,\Omega)$.
\end{enumerate}
The construction makes use of assumption \eqref{item:isotopies_of_subgraphs_area}. Recall that $\partial_S A^\lambda = Q_S^\lambda\times S$. The intersection of $T^\lambda$ with $\partial_S A^\lambda$ is equal to
\begin{equation*}
\partial ( \partial_S A^\lambda) = \partial Q_S^\lambda\times S.
\end{equation*}
By property \eqref{item:isotopies_of_subgraphs_proof_psi_on_T_circle_fibers}, the restriction of $\psi^\lambda$ to $\partial Q_S^0\times S$ maps circle fibers of $\partial Q_S^0\times S$ to circle fibers of $\partial Q_S^\lambda\times S$. In particular, it descends to a homeomorphism $\overline{\psi}^\lambda : \partial Q_S^0\rightarrow \partial Q_S^\lambda$ which maps smooth strata diffeomorphically to smooth strata. Since $Q_S^0$ and $Q_S^\lambda$ have the same total area by assumption \eqref{item:isotopies_of_subgraphs_area}, we can extend to an area preserving diffeomorphism $\overline{\psi}^\lambda : Q_S^0 \rightarrow Q_S^\lambda$. We define the extension of $\psi^\lambda$ to $Q_S^0\times S$ to be a lift $Q_S^0\times S\rightarrow Q_S^\lambda\times S$ of $\overline{\psi}^\lambda$ which agrees with the map $\psi^\lambda$ already defined on the boundary $\partial Q_S^0\times S$. By construction, this extension restricts to a diffeomorphism between $\partial_S A^0$ and $\partial_S A^\lambda$. Moreover, it is compatible with $\Omega$ since it lifts the area preserving diffeomorphism $\overline{\psi}^\lambda$. By property \eqref{item:isotopies_of_subgraphs_proof_psi_on_T_start} in Step 2 we can arrange $\psi^0$ to be the identity. For a careful choice of extension and lift, the family $\psi^\lambda$ extends to a family of symplectomorphisms between open neighbourhoods of $A^0$ and $A^\lambda$.\\

\emph{Step 4:} Using assumption \eqref{item:isotopies_of_subgraphs_volume}, we extend the family of homeomorphisms $\psi^\lambda$ form Step 3 to the desired family of symplectomorphisms of $(M,\Omega)$. First, extend $\psi^\lambda$ to a family of symplectomorphisms between open neighbourhoods of $\partial A^0$ and $\partial A^\lambda$. It is not hard to further extend to a family of symplectic embeddings
\begin{equation*}
\psi^\lambda : \operatorname{nb}(M\setminus A^0) \rightarrow M
\end{equation*}
defined on a neighbourhood of the complement of $A^0$ and starting at the inclusion $\psi^0$. The subtle part is to extend $\psi^\lambda$ over the interior of $A^0$. This makes use of assumption \eqref{item:isotopies_of_subgraphs_volume}. Let $X^\lambda$ denote the family of symplectic vector fields generating the family of symplectic embeddings $\psi^\lambda: \operatorname{nb}(M\setminus A^0)\hookrightarrow M$. Here the vector field $X^\lambda$ is defined on a neighbourhood of $M\setminus A^\lambda$. Our goal is to extend $X^\lambda$ to a family of symplectic vector fields defined on all of $M$. The desired extension of $\psi^\lambda$ to all of $M$ will simply be the flow generated by $X^\lambda$. Let $\alpha^\lambda$ be the unique family of closed $1$-forms characterized by the identity $\iota_{X^\lambda}\Omega = \alpha^\lambda$. The closed $1$-form $\alpha^\lambda$ is defined on $\operatorname{nb}(M\setminus A^\lambda)$ and we need to extend to a closed $1$-form on all of $M$. Let us first choose an arbitrary extension of $\alpha^\lambda$ to a family of not necessarily closed $1$-forms $\alpha_0^\lambda$ defined on $M$. The differential $d\alpha_0^\lambda$ is a closed $2$-form with compact support contained in the interior $\operatorname{int}(A^\lambda)$. We need to show that there exists a $1$-form $\beta^\lambda$ compactly supported in $\operatorname{int}(A^\lambda)$ such that $d\beta^\lambda = d\alpha_0^\lambda$. Once we have such $\beta^\lambda$, we can take $\alpha^\lambda\coloneqq \alpha_0^\lambda - \beta^\lambda$ to be the desired closed extension of $\alpha^\lambda$. Showing the existence of $\beta^\lambda$ is equivalent to showing that the cohomology class $[d\alpha_0^\lambda]$ vanishes in $H^2_c(\operatorname{int}(A^\lambda);\BR)$. This in turn is equivalent to showing that the homomorphism
\begin{equation*}
\langle \cdot, d\alpha_0^\lambda  \rangle : H_2(A^\lambda, \partial A^\lambda) \rightarrow \BR
\end{equation*}
vanishes. The homomorphism $\langle \cdot , d\alpha_0^\lambda\rangle$ is given by the composition of the boundary homomorphism $\partial:H_2(A^\lambda,\partial A^\lambda)\rightarrow H_1(\partial A^\lambda)$ with the homomorphism
\begin{equation*}
\langle \cdot, \alpha^\lambda \rangle : H_1(\partial A^\lambda) \rightarrow \BR.
\end{equation*}
Thus we need to verify that $\langle \cdot ,\alpha^\lambda\rangle$ vanishes on the image of $\partial$. Let $\gamma$ be any loop in $\partial A^\lambda$. After homotoping and possibly inverting the loop $\gamma$, we can assume that it is given by a finite concatenation of path segments of one of the following two types:
\begin{enumerate}[(A)]
\item \label{item:isotopies_of_subgraphs_proof_path_type1} Path segments contained in $\iota^\lambda(\Sigma)$, i.e. the fiber over $0$ of the projection $\partial_+ A^\lambda\rightarrow [0,1]$.
\item \label{item:isotopies_of_subgraphs_proof_path_type2} Flow line segments of $\varphi_t^\lambda$ starting at a point contained in the interior of $\iota^\lambda(\Sigma)$ and ending at the next intersection point with $\iota^\lambda(\Sigma)$.
\end{enumerate}
Our goal is to show that the integral of $\alpha^\lambda$ over paths of type \eqref{item:isotopies_of_subgraphs_proof_path_type1} or \eqref{item:isotopies_of_subgraphs_proof_path_type2} vanishes. It follows from this assertion that $\langle \cdot, \alpha^\lambda\rangle$ vanishes on all loops $\gamma$ in $\partial A^\lambda$. Let us begin with paths of type \eqref{item:isotopies_of_subgraphs_proof_path_type1}. It follows from property \eqref{item:isotopies_of_subgraphs_proof_psi_on_T_base_surface} that the restriction of the vector field $X^\lambda$ to $\iota^\lambda(\Sigma)$ is parallel to $\partial_s$. Therefore, the pull-back of $\alpha^\lambda$ to $\iota^\lambda(\Sigma)$ vanishes. In particular, the integral of $\alpha^\lambda$ over type \eqref{item:isotopies_of_subgraphs_proof_path_type1} paths is zero. Let us now turn to path segments of type \eqref{item:isotopies_of_subgraphs_proof_path_type2}. We define
\begin{equation*}
u^\lambda:[0,4]\times \Sigma \rightarrow T^\lambda \qquad u^\lambda(t,p) \coloneqq \varphi^\lambda_t(\iota^\lambda(p)).
\end{equation*}
Clearly, $u_p^\lambda\coloneqq u^\lambda(\cdot,p)$ is a path of type \eqref{item:isotopies_of_subgraphs_proof_path_type2} for every $p\in \Sigma$. Since $\alpha^\lambda$ is closed in a neighbourhood of $\partial A^\lambda$ and the restriction of $\alpha^\lambda$ to $\iota^\lambda(\Sigma)$ vanishes, the integral $\int_{[0,4]} (u_p^\lambda)^* \alpha^\lambda$ is independent of $p\in \Sigma$. This implies that, for all $p\in \Sigma$, we have
\begin{equation}
\label{eq:isotopies_of_subgraphs_proof_type2a}
\operatorname{area}(\Sigma) \int_{[0,4]} (u_p^\lambda)^* \alpha^\lambda = \int_{[0,4]\times\Sigma} (u^\lambda)^*\alpha^\lambda\wedge \omega.
\end{equation}
Using that $(u^\lambda)^* \Omega = \omega$, we obtain
\begin{eqnarray}
(u^\lambda)^*\alpha^\lambda\wedge \omega &=& (u^\lambda)^* (\alpha^\lambda\wedge \Omega)\label{eq:isotopies_of_subgraphs_proof_type2b}\\
&=& (u^\lambda)^* (\iota_{X^\lambda}\Omega \wedge \Omega)\nonumber\\
&=& (u^\lambda)^* \iota_{X^\lambda}\left( \frac{1}{2}\Omega\wedge\Omega \right).\nonumber
\end{eqnarray}
We compute
\begin{eqnarray}
0 &=& \frac{d}{d\lambda} \operatorname{vol}(A^\lambda) \label{eq:isotopies_of_subgraphs_proof_type2c}\\
&=& \int_{\partial A^\lambda} \iota_{X^\lambda}\left( \frac{1}{2}\Omega\wedge\Omega \right)\nonumber\\
&=& \int_{T^\lambda} \iota_{X^\lambda}\left( \frac{1}{2}\Omega\wedge\Omega \right)\nonumber\\
&=& \int_{[0,4]\times \Sigma} (u^\lambda)^* \iota_{X^\lambda}\left( \frac{1}{2}\Omega\wedge\Omega \right)\nonumber\\
&=& \int_{[0,4]\times \Sigma} (u^\lambda)^*\alpha^\lambda\wedge \omega\nonumber\\
&=& \operatorname{area}(\Sigma) \int_{[0,4]} (u_p^\lambda)^* \alpha^\lambda.\nonumber
\end{eqnarray}
Here the first equality follows from assumption \eqref{item:isotopies_of_subgraphs_volume}. The second equality is a consequence of Stokes' theorem. The vector field $X^\lambda$ is tangent to $\partial_S A^\lambda$. Thus the integral of $\iota_{X^\lambda}\left( \frac{1}{2}\Omega\wedge\Omega \right)$ over $\partial_S A^\lambda$ vanishes, which implies the third equality. The forth equality follows because $u^\lambda$ is a parametrization of $T^\lambda$. The fifth equality is immediate from equation \eqref{eq:isotopies_of_subgraphs_proof_type2b}. Finally, the sixth equality follows from identity \eqref{eq:isotopies_of_subgraphs_proof_type2a}. Since equation \eqref{eq:isotopies_of_subgraphs_proof_type2c} holds for all $p\in \Sigma$, we conclude that the integral of $\alpha^\lambda$ over all paths of type $\eqref{item:isotopies_of_subgraphs_proof_path_type2}$ vanishes. This concludes Step 4 and therefore the proof of the lemma.
\end{proof}

\begin{proof}[Proof of Theorem \ref{thm:stratification_perturbed_affine_subgraph}]
It will be useful to introduce the following notation. For every $\partial$-admissible Hamiltonian $G:[0,1]\times \BA \rightarrow \BR$, we set
\begin{equation*}
V(G) \coloneqq \int\limits_{[0,1]\times \BA} G dt\wedge \omega \quad \text{and} \quad A_\pm(G) \coloneqq \int_0^1 G(t,\partial_\pm\BA) dt
\end{equation*}
where $\partial_+\BA = \left\{ 1 \right\}\times \BT$ and $\partial_-\BA = \left\{ 0 \right\}\times \BT$.

\emph{Step 1:} Let us begin with the case that $h$ is constant and that the Hamiltonian $H$ satisfies $V(H) = A_\pm(H)$. Fix an integer $\ell>0$ such that the conclusion of Theorem \ref{thm:rot_decomp} holds. Pick a positive rational number $\alpha > 0$ not contained in $\frac{1}{2}\BZ$ such that $\ell \alpha < \frac{1}{2}\min h$. Note that if $H$ is sufficiently $C^\infty$ close to $h$, then the Hamiltonian $G\coloneqq H-V(H)$ satisfies the hypotheses of Corollary \ref{cor:rot_decomp}. Therefore, there exists a family of Hamiltonians $(G^\lambda)_{\lambda\in [0,1]}$ starting at $G^0 = G$ and satisfying all the assertions of Corollary \ref{cor:rot_decomp}. Set $H^\lambda \coloneqq G^\lambda + V(H)$. The constant family of Hamiltonians equal to zero and the family $(H^\lambda)_{\lambda\in [0,1]}$ satisfy the hypotheses of Lemma \ref{lem:isotopies_of_subgraphs}. We conclude that $D(H) = D(H^0)$ is symplectomorphic to $D(H^1)$. Thus it suffices to construct a stratification of the latter subgraph.

It follows from Corollary \ref{cor:rot_decomp} that the Hamiltonian $H^1$ is given by
\begin{equation*}
H^1 = V(H) + \#_{i=1}^\ell \varepsilon_i \psi_i^*H_\alpha
\end{equation*}
for signs $\varepsilon_i \in \left\{ \pm 1 \right\}$ and Hamiltonian diffeomorphisms $\psi_i\in \operatorname{Ham}(\BA)$. Let us define
\begin{equation*}
H_i \coloneqq \varepsilon_i\psi_i^*H_\alpha + \operatorname{min}(\varepsilon_i\psi_i^*H_\alpha).
\end{equation*}
Note that each $H_i$ is non-negative and vanishes exactly on one of the boundary components of $\BA$. It is straightforward to see that $D(H_i)$ is symplectomorphic to $D(H_\alpha)$. We have
\begin{equation*}
H^1 = V(H) - \ell \alpha /2 + \#_{i=1}^\ell H_i.
\end{equation*}
We inductively define
\begin{equation*}
F_0 \coloneqq V(H) -\ell \alpha/2 \quad \text{and}\quad F_i \coloneqq F_{i-1} \# H_i \enspace\text{for $1\leq i\leq \ell$}.
\end{equation*}
This induces a stratification of $D(H^1) = D(F_\ell)$ with $\ell +1$ top-dimensional strata
\begin{equation*}
D(F_0) \quad \text{and} \quad D(F_{i-1},F_i) \enspace\text{for $1\leq i\leq \ell$}.
\end{equation*}
By Lemma \ref{lem:symplectomorphisms_of_subgraphs}, the stratum $D(F_{i-1},F_i)$ is symplectomorphic to $D(H_i)$, which in turn is symplectomorphic to $D(H_\alpha)$. Thus all of these strata are of the desired form. Since $F_0$ is constant, the stratum $D(F_0)$ is of the desired form as well. Note that the rational number $\alpha =p/q$ was chosen independently of $H$. Hence the uniform upper bound on the denominator $q$ is automatic. Moreover, the parameter $c>0$ in the statement of the theorem is just equal to $1$ for all strata. Finally, the constant $C$ of the stratum $D(F_0)$ clearly admits a positive lower bound independent of $H$. This concludes the proof of Theorem \ref{thm:stratification_perturbed_affine_subgraph} in the case $V(H) = A_\pm(H)$.

\emph{Step 2:} Next, let us assume that
\begin{equation}
\label{eq:stratification_perturbed_affine_subgraph_proof_a}
V(H) = \frac{1}{2} (A_-(H) + A_+(H)) \qquad \text{and} \qquad A_+(H) - A_-(H) \in \BQ.
\end{equation}
If $A_+(H) = A_-(H)$, this is precisely the case treated in Step 1, so we assume that $A_+(H) \neq A_-(H)$. We may in addition assume that $\alpha \coloneqq A_+(H) - A_-(H) > 0$. We stratify $D(H)$ into the following two pieces:
\begin{equation*}
D(H_\alpha) \qquad \text{and}\qquad D(H_\alpha,H).
\end{equation*}
The first of the two pieces is already of the desired form. By Lemma \ref{lem:symplectomorphisms_of_subgraphs}, the second piece is symplectomorphic to $D(\overline{H}_\alpha \# H)$. It follows from assumption \eqref{eq:stratification_perturbed_affine_subgraph_proof_a} and the definition of $\alpha$ that the Hamiltonian $\overline{H}_\alpha\# H$ is precisely of the form considered in Step 1. Thus we may further stratify $D(\overline{H}_\alpha \# H)$ as in Step 1. Note that while the parameters of the substrata of $D(\overline{H}_\alpha \# H)$ admit uniform bounds independent of $H$, this is not true for the stratum $D(H_\alpha)$ because $\alpha = A_+(H) - A_-(H)$ is determined by $H$ and may have arbitrarily large denominator. This will be rectified in the next step of the proof.

\emph{Step 3:} Let us now turn to the general case. Consider numbers $0=x_0 < x_1 < x_2 < x_3=1$ and decompose the annulus $\BA = [0,1]\times \BT$ into three smaller annuli $\BA_i\coloneqq [x_{i-1},x_i]\times \BT$ where $1\leq i \leq 3$. The area of $\BA_i$ is given by $a_i\coloneqq x_i-x_{i-1}$. Or goal is to adjust the numbers $x_i$ and to construct a Hamiltonian $G(t,x,y) = G(x)$ only depending on $x$ such that the following properties hold. Let $G_i\coloneqq G|_{[0,1]\times \BA_i}$ denote the restriction.
\begin{enumerate}
\item $G$ is $C^\infty$ close to $h - \frac{1}{2}\min h$
\item $A_-(H) - A_-(G) = A_+(H) - A_+(G) = V(H) - V(G)$
\item We have
\begin{equation*}
V(G_i) = \frac{1}{2}a_i (A_-(G_i) + A_+(G_i)) \qquad \text{and} \qquad \alpha_i\coloneqq\frac{A_+(G_i) - A_-(G_i)}{a_i} \in \BQ.
\end{equation*}
\end{enumerate}
Moreover, we would like to have an upper bound on the minimal denominators of the rational numbers $\alpha_i$ and a lower bound on the areas $a_i$, both of which are uniform in $H$.

Before turning to the construction of such a Hamiltonian $G$, let us first describe how we use it to stratify $D(H)$. First, we stratify $D(H)$ into the following four pieces:
\begin{equation*}
D(G,H) \qquad \text{and} \qquad D(G_i) \enspace\text{for $1\leq i \leq 3$}
\end{equation*}
By Lemma \ref{lem:symplectomorphisms_of_subgraphs}, the first piece is symplectomorphic to $D(\overline{G}\# H)$. The Hamiltonian $\overline{G}\# H$ satisfies the assumption in Step 1 and we can therefore further stratify $D(G,H)$ into pieces of the desired form. Each of the pieces $D(G_i)$ is, up to scaling of the symplectic form, of the form considered in Step 2 and can thus be further stratified as well. The result is a stratification of $D(H)$ into pieces of the desired form. The parameters of the resulting strata admit uniform bounds in $H$. For the substrata of $D(G,H)$ this follows from Step 1. For the substrata of $D(G_i)$ this follows from the uniform upper bound on the minimal denominators of $\alpha_i$ and the uniform lower bound on $a_i$.

It remains to construct $G$. Recall that the Hamiltonian $h$ is given by $h(t,x,y) = a + bx$. Let is pick a rational number $\alpha_+$ slightly bigger than $b$ and a rational number $\alpha_-$ slightly smaller than $b$. Note that if $H$ is sufficiently $C^\infty$ close to $h$, then $\alpha\coloneqq A_+(H) - A_-(H)$ will be contained in the interval $(\alpha_-,\alpha_+)$ and be uniformly bounded away from its entpoints. Let $0<w<\frac{\alpha_+-\alpha}{\alpha_+-\alpha_-}$ be a parameter which is to be determined. Define
\begin{equation*}
x_1 \coloneqq w \quad \text{and} \quad x_2 \coloneqq w + \frac{\alpha-\alpha_-}{\alpha_+-\alpha_-}
\end{equation*}
and let $g:[0,1]\rightarrow \BR$ be the unique continuous piecewise linear function such that
\begin{enumerate}
\item $g(0) = A_-(H)$
\item $g(1) = A_+(H)$
\item $g$ has slope $\alpha_-$ on the intervals $[x_0,x_1]$ and $[x_2,x_3]$
\item $g$ has slope $\alpha_+$ on the interval $[x_1,x_2]$.
\end{enumerate}
By taking $H$ sufficiently $C^\infty$ close to $h$, we can make the difference between $V(H)$ and $\frac{1}{2}(A_-(H) + A_+(H))$ arbitrarily small. In particular, the integral $\int_0^1 g(x) dx$ will be bigger than $V(H)$ for $w$ near $0$ and smaller than $V(H)$ for $w$ near $\frac{\alpha_+-\alpha}{\alpha_+-\alpha_-}$. Thus there exists a unique value $w$ such that $\int_0^1 g(x) dx = V(H)$. For $H$ sufficiently $C^\infty$ close to $h$, this value $w$ is uniformly bounded away from the endpoints of the interval $(0,\frac{\alpha_+-\alpha}{\alpha_+-\alpha_-})$. Let us pick a smoothing $\tilde{g}$ of $g$ which has the following properties:
\begin{enumerate}
\item $\tilde{g}$ agrees with $g$ at the points $x_0$, $x_1$, $x_2$ and $x_3$
\item The integrals of $\tilde{g}$ and $g$ agree on each of the intervals $[x_0,x_1]$, $[x_1,x_2]$ and $[x_2,x_3]$.
\item $\tilde{g}$ is $C^\infty$ close to $h$. Note that we can control this $C^\infty$ distance purely in terms of the inital pick of the rational numbers $\alpha_\pm$ close to $b$ and uniformly among all $H$ sufficiently close to $h$.
\end{enumerate}
Now define $G(t,x,y) \coloneqq \tilde{g}(x) - \frac{1}{2}\min h$. It is straightforward to check that this Hamiltonian has all desired properties.
\end{proof}

\section{Symplectic ribbon complexes}
\label{sec:symplectic_ribbons}

We introduce the notion of a \textit{symplectic ribbon complex}. The main result of this section, Theorem \ref{thm:ellipsoid_packing_of_ribbon_complex}, states that ellipsoid embedding stability holds for all connected symplectic ribbon complexes consisting of balls and polydisks. This result plays a crucial role in our proof of Theorem \ref{thm:ellipsoid_embedding_stability}, as it allows us to deduce ellipsoid embedding stability of a symplectic $4$-manifold from the existence of certain decompositions into balls and polydisks; see Section \ref{sec:tame_packings}.\\

Recall that the $4$-dimensional symplectic ellipsoid $E(a,b)$ of widths $a,b >0$ is defined by
\begin{equation*}
E(a,b)\coloneqq \left\{ z\in \BC^2 \enspace\Big|\enspace \frac{\pi |z_1|^2}{a} + \frac{\pi|z_2|^2}{b} \leq 1 \right\}.
\end{equation*}
The $4$-dimensional ball of width $a>0$ is the ellipsoid $B^4(a) \coloneqq E(a,a)$. The polydisk of widths $a,b >0$ is given by
\begin{equation*}
P(a,b) \coloneqq \left\{ z\in \BC^2 \mid \pi |z_1|^2 \leq a \enspace\text{and}\enspace \pi|z_2|^2\leq b \right\}.
\end{equation*}

We define the \textit{symplectic ribbon} $R(a)$ of width $a>0$ to be
\begin{equation*}
R(a) \coloneqq [0,1] \times B^2(a)
\end{equation*}
where $B^2(a)$ denotes the $2$-dimensional disk of area $a$. The area form on $B^2(a)$ induces a closed $2$-form $\omega$ on $R(a)$ which is characterized by the requirements that the restriction of $\omega$ to the fibers $\left\{ * \right\}\times B^2(a)$ agrees with the area form on $B^2(a)$ and that the contraction $\iota_{\partial_t}\omega$ vanishes, where $t$ denotes the coordinate of the $[0,1]$ factor. The leaves of the characteristic foliation on $R(a)$ induced by $\omega$ are therefore simply the fibers $[0,1]\times \left\{ * \right\}$. The ribbon $R(a)$ has two ends $\left\{ j \right\}\times B^2(a)$, where $j\in \left\{ 0,1 \right\}$. We think of $R(a)$ as not carrying a preferred orientation, i.e. we ignore the obvious orientation of the interval $[0,1]$.

Now consider a smooth $4$-dimensional star-shaped domain $X$ and a symplectic ribbon $R$. Let $B$ be one of its ends. An \textit{attaching map} from this end to $X$ is simply a smooth embedding $\iota:B \rightarrow \partial X$ such that the pull-back of the symplectic form on $X$ via $\iota$ agrees with the standard area form on $B$. We would also like to allow certain non-smooth star-shaped domains $X$ such as polydisks. In this case, we always assume that the image of an attaching map is disjoint from the non-smooth locus of the boundary of the domain.

Consider a finite collection of star-shaped domains and ribbons of variable widths. For some ribbon ends, fix an attaching map to one of the star-shaped domains. We require that the images of these attaching maps are pairwise disjoint. Moreover, we require that, for each ribbon, we have an attaching map for at least one of its ends. Take the disjoint union of all the star-shaped domains and ribbons and then take the quotient by gluing ribbon ends to the star-shaped domains via the chosen attaching maps. We call the resulting object a \textit{symplectic ribbon complex}. Note that not all ribbon ends are required to be attached to some star-shaped domain. We call such an end a \textit{free end}. The corresponding ribbon is called a \textit{free ribbon}.

We can alternatively think of a symplectic ribbon complex as the following datum:
\begin{itemize}
\item a collection $\MX$ of finitely many star-shaped domains;
\item a collection $\MD$ of finitely many disjoint symplectic disks in the boundaries of domains contained in $\MX$;
\item a partition of $\MD$ into individual disks and pairs of disks;
\item for each pair of disks, a diffeomorphism between them compatible with the restriction of the symplectic form to the disks.
\end{itemize}
Here the individual disks in the partition correspond to ribbons with a free end. The area-preserving diffeomorphism between a pair of disks corresponds to the map obtained by traversing a ribbon following the characteristic foliation. A \textit{symplectomorphism} between two symplectic ribbon complexes is a symplectomorphism between the disjoint unions of the underlying shar-shaped domains which is compatible with the collections of disks, their partitions, and the transition maps between pairs of disks.

\begin{example}
\label{ex:simple_ribbon_complex}
Consider an ellipsoid $E(a,b)$. For every $0<c<b$ and for every angle $\theta\in \BR/2\pi\BZ$, there exists a unique map
\begin{equation}
\label{eq:special_attaching_map}
\iota(\theta): B^2(c)\rightarrow \partial E(a,b)
\end{equation}
satisfying the following properties: For $j\in \left\{ 1,2 \right\}$, let $\operatorname{pr}_j:\BC^2\rightarrow \BC$ be the projection onto the $j$-th factor. Then the composition $\operatorname{pr}_2\circ \iota(\theta)$ is simply the inclusion of $B^2(c)$ into $B^2(b)$ and the image of the composition $\operatorname{pr}_1\circ \iota(\theta)$ is contained in the ray $\BR_{\geq 0}\cdot e^{i\theta}$. An easy computation shows that the pull-back of the standard symplectic form on $\BC^{2} $ via $\iota(\theta)$ agrees with the standard area form on $B^2(c)$. For each positive integer $n$, let $E(a,b;c,n)$ be the symplectic ribbon complex built from one copy of the ellipsoid $E(a,b)$ and $n$ copies of the ribbon $R(c)$ as follows: For each $0\leq j < n$, attach a copy of $R(c)$ to $\partial E(a,b)$ on one of its ends via the attaching map $\iota(2\pi j/n)$. The complex $E(a,b;c,n)$ has $n$ free ends. By construction, each of these free ends comes with a preferred identification with $B^2(c)$. This preferred identification is not part of the datum of a symplectic ribbon complex, but we will occasionally make use of it below.

If we replace the collection of attaching angles $2\pi j/n$ by any other collection of $n$ distinct attaching angles $\theta_0<\dots<\theta_{n-1}$, the resulting symplectic ribbon complex is symplectomorphic to $E(a,b;c,n)$. Indeed, there exists a symplectomorphism of $E(a,b)$ which, for every $0\leq j < n$, maps the image of $\iota(\theta_j)$ to the image of $\iota(2\pi j/n)$. Note, however, that this isomorphism is not compatible with the preferred identifications of the free ribbon ends with $B^2(c)$ mentioned above.

If $E(a,a)=B^4(a)$ is a ball, let us abbreviate the resulting complex by $B^4(a;c,n)$.
\end{example}

\begin{example}
\label{ex:complex_from_ribbon_graph}
Consider a ribbon graph $G$, i.e. a graph such that the incident edges of each vertex are cyclically ordered. Assume that the vertices of $G$, except possibly for some of the leaves of $G$, are labelled by positive real numbers. Let $A$ denote the labelling of $G$. Fix real numbers $0<c<b$. Given these data, we can construct a symplectic ribbon complex $K(G,A,b,c)$ as follows: For each labelled vertex $v$ of $G$, let $a_v$ denote its label, $n_v$ the number of edges incident to $v$, and take a copy of $E(a_v,b;c,n_v)$. Note that the free ends of $E(a_v,b;c,n_v)$ are naturally cyclically ordered. To each incident edge of $v$, assign a free end of $E(a_v,b;c,n_v)$ respecting the cyclic order. Now glue together any pair of free ends assigned to the same edge via the identity map of $B^2(c)$ to obtain a ribbon complex $K(G,A,b,c)$. The free ends of this ribbon complex correspond precisely to the unlabelled vertices of $G$.

We observe that in contrast to Example \ref{ex:simple_ribbon_complex}, the angles $\theta_j$ at which ribbons are attached to the ellipsoid $E(a_v,b)$ can no longer be changed freely if the graph $G$ has cycles: Indeed, consider the simplest example of just one ribbon $R(c)$ attached to an ellipsoid $E(a,b)$ at both of its ends via $\iota(\theta_0)$ and $\iota(\theta_1)$ for two angles $\theta_0 < \theta_1$. The disk $\operatorname{im}\iota(\theta_1)$ can be obtained from $\operatorname{im}\iota(\theta_0)$ by flowing forward via the Reeb flow on $\partial E(a,b)$ for some amount of time. This induces a transition map $\operatorname{im}\iota(\theta_0)\rightarrow \operatorname{im}\iota(\theta_1)$ which assigns to each point in $\operatorname{im}\iota(\theta_0)$ the first intersection point of the trajectory starting at that point with $\operatorname{im}\iota(\theta_1)$. Traversing the ribbon following the characteristic foliation induces a map $\operatorname{im}\iota(\theta_1) \rightarrow \operatorname{im}\iota(\theta_0)$. Composing these two maps yields an area preserving diffeomorphism of $\operatorname{im}\iota(\theta_0)$ which is a rotation by angle $\frac{a(\theta_1-\theta_0)}{b}$. Clearly, this number is invariant under symplectomorphisms of ribbon complexes. On the other hand, if $G$ is a tree, then only the cyclic order of the ribbons attached to an ellipsoid matters and the angles can be varied freely.
\end{example}

Let $(M^4,\omega)$ be a $4$-dimensional symplectic manifold. A {\it symplectic embedding} of a symplectic ribbon complex $K$ into $M$ is an injective continuous map $\varphi: K\overset{s}{\hookrightarrow} M$ satisfying the following properties:
\begin{itemize}
\item The restriction $\varphi_X$ of $\varphi$ to each domain $X$ of $K$ is a smooth symplectic embedding.
\item The restriction $\varphi_R$ of $\varphi$ to each ribbon $R$ of $K$ is a smooth embedding which pulls back the symplectic form on $M$ to the closed $2$-from $\omega$ on $R$.
\item Suppose that $R$ is attached to $\partial X$ at one of its ends. Then $\varphi_R$ meets $\varphi_X(\partial X)$ transversely, i.e. $\partial_t \varphi_R$ is no-where tangent to $\varphi_X(\partial X)$ along the image of the gluing region of $X$ and $R$.
\end{itemize}

\begin{remark}
\label{rem:neighbourhood_theorem_for_ellipsoid_ribbon_complexes}
Every symplectic ribbon complex $K$ admits a symplectic embedding into some symplectic $4$-manifold. There is a neighbourhood theorem for symplectic embeddings of symplectic ribbon complexes: Let $K$ be a symplectic ribbon complex and let $\varphi:K\overset{s}{\hookrightarrow} M$ and $\varphi':K\overset{s}{\hookrightarrow} M'$ be two symplectic embeddings into symplectic $4$-manifolds $M$ and $M'$. If the manifolds have bounary, we assume that the images of the embeddings are contained in the interior. Then there exist open neighbourhoods $U$ of $\varphi(K)$ in $M$ and $U'$ of $\varphi'(K)$ in $M'$ and a symplectomorphism $\psi: U\rightarrow U'$ such that $\varphi' = \psi\circ \varphi$. This can be proved using Gotay's neighbourhood theorem for coisotropic submanifolds \cite{got82}.
\end{remark}

Let $K$ be a non-empty symplectic ribbon complex. The \textit{volume} $\operatorname{vol}(K)$ of $K$ is defined to be the sum of the volumes of its domains. Let $(a_i)_i$ be the collection of all the Gromov widths of the domains of $K$ and let $(b_j)_j$ be the collection of all widths of ribbons of $K$. We define the \textit{width} of $K$ to be the quantity
\begin{equation*}
w(K) \coloneqq \min\left\{ \min_i a_i, \min_j b_j \right\}.
\end{equation*}
Moreover, we define the \textit{volume normalized width} by
\begin{equation*}
\overline{w}(K) \coloneqq\frac{w(K)}{\sqrt{\operatorname{vol}(K)}}. 
\end{equation*}
Note that this quantity is invariant under scalings of the symplectic form on $K$. For every $a\geq 1$, we define the ellipsoid embedding number $p^E_a(K)$ of $K$ to be
\begin{equation*}
p^E_a(K) \coloneqq \sup_\lambda \frac{\operatorname{vol}(E(\lambda,\lambda a))}{\operatorname{vol}(K)} \in [0,1],
\end{equation*}
where the supremum is taken over all $\lambda >0$ such that, for every symplectic embedding $\varphi : K \overset{s}{\hookrightarrow} M$ into the interior of a symplectic $4$-manifold $M$ and for every open neighbourhood $U$ of $\varphi(K)$ in $M$, the ellipsoid $E(\lambda,\lambda a)$ symplectically embeds into $U$.

\begin{theorem}
\label{thm:ellipsoid_packing_of_ribbon_complex}
Let $K$ be a non-empty connected symplectic ribbon complex built using only balls and polydisks. Then ellipsoid embedding stability holds for $K$, i.e. there exists a finite number $a_0$ such that $p^E_a(K) = 1$ for all $a\geq a_0$. Moreover, for every $\delta >0$, the number $a_0$ can be chosen uniformly among all $K$ with $\overline{w}(K) \geq \delta$.
\end{theorem}

The remainder of this section is concerned with the proof of Theorem \ref{thm:ellipsoid_packing_of_ribbon_complex}.

\begin{lemma}
\label{lem:skinny_ellipsoid_embeddings}
For every $e_0\geq 1$, there exists $a_0$ such that
\begin{equation*}
p^E_a(E(1,e)) = p^E_a(P(1,e)) = 1
\end{equation*}
for all $1\leq e \leq e_0$ and all $a \geq a_0$.
\end{lemma}

\begin{proof}
The statement about $p^E_a(E(1,e))$ is an immediate consequence of \cite[Theorem 1.4]{bh13}. The statement about $p^E_a(P(1,e))$ can be proved via the same method. Indeed, a $4$-dimensional ellipsoid $E$ symplectically embeds into a $4$-dimensional polydisk $P$ if and only if $c_k(E) \leq c_k(P)$ for all $k>0$ where $c_k$ denote the ECH capacities, see \cite{cri19}. One can therefore prove the assertion by analysing the formulas for ECH capacities of ellipsoids and polydisks. Alternatively, the statement also follows from \cite[Theorem 3]{bho16} in combination with the fact that the problem of embedding an ellipsoid into a polydisk is equivalent to embedding a certain collection of balls associated to the weight sequence of the ellipsoid, see \cite{cri19}. Note that \cite[Theorem 3]{bho16} applies to target manifolds which are pseudoballs, see \cite[Definition 1.3]{bho16}. While polydisks are not pseudoballs, they are contained in the closure of the space of pseudoballs. The proof of \cite[Theorem 3]{bho16} also applies to polydisks without changes.
\end{proof}

\begin{lemma}
\label{lem:special_ribbon_complex_neighbourhood_packing}
Let $(G,A)$ be a connected labelled ribbon graph and let $0<c<b$ be two positive numbers. Let $K = K(G,A,b,c)$ be the ribbon complex constructed in Example \ref{ex:complex_from_ribbon_graph}. Define $\alpha\coloneqq \sum_v a_v$ where the sum runs over all labelled vertices of $G$. Let $n$ denote the number of free ends of $K$ and let $K'$ denote the ribbon complex obtained from $K$ by deleting all ribbons with a free end. Then for every symplectic embedding $\varphi: K\overset{s}{\hookrightarrow}M$ into a symplectic $4$-manifold $M$, possibly with boundary, such that $\varphi(K')\subset \operatorname{int}(M)$, every open neighbourhood $U$ of $\varphi(K')$ in $M$, and every $0<\gamma<\beta\leq c$, there exists a symplectic embedding $\psi: E(\alpha,\beta;\gamma,n) \overset{s}{\hookrightarrow} M$ such that the following is true:
\begin{enumerate}
\item The image $\psi(E(\alpha,\beta))$ is contained in $U$.
\item \label{item:special_ribbon_complex_neighbourhood_packing_free_ribbons} For each free ribbon $R$ of $K$, there exists exactly one free ribbon $S$ of $E(\alpha,\beta;\gamma,n)$ such that $\psi(S)\subset \varphi(R)$. Moreover, the image of the free end of $S$ under $\psi$ is contained in the image of the free end of $R$ under $\varphi$.
\end{enumerate}
\end{lemma}

\begin{proof}
After possibly deleting some ribbons of $K$, we can assume that $G$ is a tree. Let us begin by constructing a special symplectic embedding $\varphi: K \overset{s}{\hookrightarrow} \BC^2$. By Remark \ref{rem:neighbourhood_theorem_for_ellipsoid_ribbon_complexes}, it will be sufficient to prove the lemma for this specific embedding. Let us begin by embedding the ribbon tree $G$ into the complex plane $\BC$ such that all edges are straight line segments. For each labelled vertex $v$, take a copy of $B^2(a_v)$ centered at $v$. After dilating the embedding of the tree $G$ if necessary, we can assume that the disks centered at two distinct vertices are disjoint and that no unlabelled vertex is contained in any of the disks. Moreover, we can assume that each edge only meets the disks centered at the vertices it is incident to. It is straightforward to construct a symplectic embedding $\varphi:K \overset{s}{\hookrightarrow}\BC^2$ such that the image $\varphi(K)$ is the union of one copy of the ellipsoid $E(a_v,b)$ for each labelled vertex $v$, translated such that it is centered at $\left\{ v \right\}\times \left\{ 0 \right\}\in \BC^2$, and all the sets of the form $e\times \left\{ B^2(c) \right\}$ where $e$ is an edge. Note that the projection $\operatorname{pr}_1(\varphi(K))$ to the first factor of $\BC^2$ is precisely given by the above configuration of disks and line segments. In order to simplify notation, let us identify $K$ with its image under $\varphi$.

Fix an arbitrary open neighbourhood $U$ of $K'$ inside $\BC^2$ and numbers $0<\gamma<\beta\leq c$. Let $D\subset \BC \cong \BC \times \left\{ 0 \right\}$ be a topological disk which contains $\operatorname{pr}_1(K')$ in its interior and which is contained in $U\cap (\BC\times \left\{ 0 \right\})$. For $\varepsilon>0$, Let $K'_\varepsilon$ denote the ribbon complex obtained from $K'$ by replacing each domain $E(a_v,b)$ by the smaller domain $E(a_v,c+\varepsilon)$ while leaving the widths of all ribbons unchanged. If $\varepsilon>0$ is sufficiently small, it is possible to find a smooth non-negative function $f:D\rightarrow \BR_{\geq 0}$ with the following properties:
\begin{enumerate}
\item $f$ vanishes on the boundary $\partial D$.
\item $f$ has precisely one local maximum and no other critical points.
\item The set
\begin{equation*}
A(D,f) \coloneqq \left\{ (z_1,z_2)\in \BC^2 \mid z_1 \in D, \enspace \pi |z_2|^2 \leq f(z_1) \right\}
\end{equation*}
is contained in $U$ and contains $K'_\varepsilon$ in its interior.
\end{enumerate}
Choose an area preserving diffeomorphism $\psi$ of $\BC$ such that $\tilde{D}\coloneqq \psi(D)$ is a round disk cetered at the origin and such that $\tilde{f}\coloneqq f\circ \psi^{-1}$ is rotation invariant. Clearly, $\psi\times \operatorname{id}_\BC$ is a symplectomorphism mapping $A(D,f)$ onto $A(\tilde{D},\tilde{f})$. Our goal is to show that $A(\tilde{D},\tilde{f})$ contains the ellipsoid $E(\alpha,\beta)$. Once we know this, $(\psi\times \operatorname{id}_\BC)^{-1}$ gives a symplectic embedding of $E(\alpha,\beta)$ into $U$.

Consider a general ellipsoid $E(p,q)$. Then we have
\begin{equation}
\label{eq:special_ribbon_complex_neighbourhood_packing_proof_ellipsoid_area}
\operatorname{area} (\left\{ z\in B^2(p) \mid \operatorname{area}(E(p,q) \cap \left\{ z \right\}\times \BC) \geq s \right\}) = \frac{p(q-s)}{q}.
\end{equation}
It follows from this identity, the fact that $K'_\varepsilon$ is contained in $A(D,f)$, and the fact that $\psi$ is area preserving that
\begin{equation*}
\operatorname{area}\left(\left\{ z \in \tilde{D} \mid \tilde{f}(z) \geq s \right\}\right) = \operatorname{area}(\left\{ z \in D \mid f(z)\geq s \right\}) \geq \sum_{v} \frac{a_v(c+\varepsilon-s)}{c+\varepsilon} = \frac{\alpha(c+\varepsilon-s)}{c+\varepsilon} \geq \frac{a(\beta-s)}{\beta}.
\end{equation*}
Again using identity \eqref{eq:special_ribbon_complex_neighbourhood_packing_proof_ellipsoid_area} and the fact that $\tilde{f}$ is rotation invariant and non-increasing along the radial direction, we see that $E(\alpha,\beta)$ is contained in $A(\tilde{D},\tilde{f})$.

Finally, we observe that it is straightforward to attach $n$ copies of the ribbon $R(\gamma)$ to $(\psi\times \operatorname{id}_\BC)^{-1}(E(\alpha,\beta))$ in such a way that property \eqref{item:special_ribbon_complex_neighbourhood_packing_free_ribbons} in the statement of the lemma holds. The existence of the desired embedding of $E(\alpha,\beta;\gamma,n)$ is an immediate consequence.
\end{proof}

Let $0<c<a$. Consider the ball $B^4(a)\subset \BC^2$ and let $\rho:\partial B^4(a)\rightarrow \BC P^1(a)$ be the Hopf map. Here $\BC P^1(a)$ denotes the $2$-sphere of area $a$. We call a symplectic embedding $\psi: B^2(c) \hookrightarrow \partial B^4(a)$ \textit{free} if the composition $\rho\circ \psi: B^2(c)\rightarrow \BC P^1(a)$ is an embedding. We call a finite collection $(\psi_j)_j$ of symplectic embeddings $\psi_j:B^2(c)\hookrightarrow \partial B^4(a)$ \textit{free} if the compositions $\rho\circ \psi_j: B^2(c)\rightarrow \BC P^1(a)$ are pairwise disjoint embeddings.

\begin{lemma}
\label{lem:free_attaching_maps_ball}
Let $0<c<a$ and $n>0$. Then the symplectomorphism group of $B^4(a)$ acts transitively on free $n$-tuples $(\psi_j)_j$ of symplectic embeddings $\psi_j:B^2(c)\hookrightarrow \partial B^4(a)$.
\end{lemma}

\begin{proof}
Consider two free $n$-tuples of symplectic embeddings $(\psi_j)_j$ and $(\psi_j')_j$. First note that the symplectomorphism group of $\BC P^1(a)$ acts transitively on $n$-tuples of pairwise disjoint symplectic embeddings of $B^2(c)$. Any symplectomorphism of $B^4(a)$ maps Hopf fibers on the boundary $\partial B^4(a)$ to Hopf fibers and therefore induces a symplectomorphism of $\BC P^1(a)$. Conversely, any symplectomorphism of $\BC P^1(a)$ lifts to a symplectomorphism of $B^4(a)$. Thus we can find a symplectomorphism $\chi'$ of $B^4(a)$ such that $\rho\circ \psi_j' = \rho\circ \chi'\circ \psi_j$ for all $j$. Consider two symplectic embeddings $\psi,\psi':B^2(c)\hookrightarrow \partial B^4(a)$ which are both free. If $\rho\circ \psi = \rho\circ \psi'$, then $\psi$ and $\psi'$ are related by a symplectomorphism of $B^4(a)$ whose support can be taken to be contained in an arbitrarily small neighbourhood of the preimage $\rho^{-1}(\operatorname{im}(\rho\circ \psi))$. Using this observation, we can correct $\chi'$ to a symplectomorphism $\chi$ of $B^4(a)$ which satisfies $\psi_j' = \chi \circ \psi_j$ for all $j$.
\end{proof}

\begin{proposition}
\label{prop:special_embedding_into_ball}
Let $B\subset \BC^2$ be a ball and let $n$ be a non-negative integer. Then the following is true for every sufficiently small $b>0$: Let $0<a$ such that $\operatorname{vol}(E(a,b)) < \operatorname{vol}(B)$. Let $0<c<b$. For $0\leq j < n$, let $\psi_j:B^2(c)\hookrightarrow \partial B$ be a symplectic embedding and assume that the collection $(\psi_j)_j$ is free. Then there exists a symplectic embedding $\varphi: E(a,b;c,n) \overset{s}{\hookrightarrow} B$ whose restriction to the $j$-th free ribbon end of $E(a,b;c,n)$ agrees with the embedding $\psi_j$. Here we use the preferred identification of the free ribbon ends with $B^2(c)$, see Example \ref{ex:simple_ribbon_complex}. Moreover, the image $\operatorname{im}(\varphi)$ intersects $\partial B$ precisely at the free ends and the intersection is transverse.
\end{proposition}

Before turning to the proof of Proposition \ref{prop:special_embedding_into_ball}, we need to establish some preliminary results.

\begin{lemma}
\label{lem:special_embedding_into_CP2}
Let $a,b,\beta>0$ and suppose that the ellipsoid $E=E(a,b) \subset \BC^2$ symplectically embeds into $M=\BC P^2(\beta)$. Let $C\subset M$ be a complex line. If $b<\beta$, then there exists a symplectic embedding $\varphi: E\overset{s}{\hookrightarrow} M$ such that $\varphi^{-1}(C) = E \cap (\left\{ 0 \right\}\times \BC)$.
\end{lemma}

\begin{proof}
This is a special case of \cite[Lemma 1.6]{bho16}.
\end{proof}

\begin{lemma}
\label{lem:connecting_ribbon_to_core_plane}
Let $a,b>0$ and consider the ellipsoid $E(a,b)$. Let $0<c<b$ and $r>1$. Then there exists a symplectic embedding $\varphi:E(a,b;c,1)\overset{s}{\hookrightarrow} E(ra,rb)$ such that the restriction of $\varphi$ maps the free end of $E(a,b;c,1)$ into $\left\{ 0 \right\}\times \BC$ and such that the image of $\varphi$ does not intersect the set $\left\{ 0 \right\}\times \BC$ in any other point.
\end{lemma}

\begin{proof}
Consider the $3$-dimensional half space $H\coloneqq \BR_{\geq 0}\times \BC$. One can find an explicit Hamiltonian diffeomorphism $\psi$ compactly supported in the interior of $E(ra,rb)$ such that the image of $\left\{ 0 \right\}\times \BC$ is disjoint from $E(a,b)$, contained in $H$, and of the form
\begin{equation*}
\psi(\left\{ 0 \right\}\times \BC) = \left\{ (f(z),z) \mid z\in B^2(rb)\right\}
\end{equation*}
for some compactly supported non-negative function $f:B^2(rb)\rightarrow \BR_{\geq 0}$. We can connect the boundary of $E(a,b)$ to the image $\psi(\left\{ 0 \right\}\times \BC)$ via a ribbon $R(c)$ contained in the half space $H$. Finally, apply $\psi^{-1}$ to obtain the desired embedding of the ribbon complex $E(a,b;c,1)$.
\end{proof}

\begin{proof}[Proof of Proposition \ref{prop:special_embedding_into_ball} in the case $n\leq 1$]
After scaling, we may assume that $B = B^4(1)$ has symplectic width equal to $1$.

The case $n=0$ follows immediately from the well-known statement that a sufficiently skinny ellipsoid symplectically embeds into a ball if and only if the volume constraint is satisfied, see Lemma \ref{lem:skinny_ellipsoid_embeddings}.

The case $n=1$ of Proposition \ref{prop:special_embedding_into_ball} follows from Lemmas \ref{lem:special_embedding_into_CP2} and \ref{lem:connecting_ribbon_to_core_plane}. Indeed, let $b>0$ be small and consider $0<c<b$ and $0<a$ such that $\operatorname{vol}(E(a,b)) < \operatorname{vol}(B)$. If $b$ is sufficiently small, then $E(a,b)$ symplectically embeds into $\BC P^2 = \BC P^2(1)$ by the case $n=0$. We can pick $r>1$ sufficiently close to $1$ such that $E(ra,rb)$ still embeds into $\BC P^2$. Fix a complex line $S\subset \BC P^2$. Combining Lemmas \ref{lem:special_embedding_into_CP2} and \ref{lem:connecting_ribbon_to_core_plane}, we see that there exists a symplectic embedding $E(a,b;c,1) \overset{s}{\hookrightarrow} \BC P^2$ which maps the free ribbon end into $S$ and is otherwise disjoint from $S$. Consider a map $B\rightarrow \BC P^2$ which maps the interior of $B$ symplectomorphically onto the complement of $S$ and whose restriction to the boundary of $B$ is simply the Hopf map $\rho: \partial B\rightarrow S$. Via this map, we can pull back the symplectic embedding $E(a,b;c,1) \overset{s}{\hookrightarrow} \BC P^2$ to a symplectic embedding $\varphi:E(a,b;c,1)\overset{s}{\hookrightarrow} B$. Clearly, $\varphi$ embeds the free ribbon end of $E(a,b;c,1)$ into $\partial B$ and the composition with the Hopf map $\rho$ is still an embedding. Since any other free embedding of $B^2(c)$ into $\partial B$ is related to this embedding via a symplectomorphism of $B$ by Lemma \ref{lem:free_attaching_maps_ball}, this concludes the proof of Proposition \ref{prop:special_embedding_into_ball} in the case $n=1$.
\end{proof}

\begin{lemma}
\label{lem:special_embedding_into_ball_dependent_attaching_maps}
Let $B\subset \BC^2$ be a ball and let $n$ be a non-negative integer. Then for all sufficiently small $b>0$ the following is true: Consider $0<c<b$ and $0<a$ such that $\operatorname{vol}(E(a,b))<\operatorname{vol}(B)$. Let $\psi:B^2(c)\hookrightarrow \partial B$ be a free embedding. Let $D_1, \dots, D_{n-1} \subset \partial B$ be embedded disks such that, for each $j$, the restriction of the Hopf map $\rho$ to $D_j$ is an embedding with the same image as $\rho\circ \psi$. Assume that the disks $\operatorname{im}(\psi), D_1,\dots, D_{n-1}$ are pairwise disjoint. Then there exists a symplectic embedding $\varphi: E(a,b;c,n)\overset{s}{\hookrightarrow} B$ such that the restriction of $\varphi$ to one free ribbon end of $E(a,b;c,n)$ is given by $\psi$ and each $D_j$ for $1\leq j\leq n-1$ is the image of a free ribbon end of $E(a,b;c,n)$.
\end{lemma}

\begin{proof}
By the special case $n=1$ of Proposition \ref{prop:special_embedding_into_ball} we just proved, we can find a symplectic embedding $\varphi: E(a,b;c,1)\overset{s}{\hookrightarrow} B$ such that the restriction of $\varphi$ to the free ribbon end of $E(a,b;c,1)$ is given by $\psi$. Let $R = R(c)$ denote the ribbon of $E(a,b;c,1)$. For $\varepsilon>0$, we define the thickened ribbon 
\begin{equation*}
R_\varepsilon \coloneqq [0,\varepsilon] \times R = [0,\varepsilon]\times [0,1] \times B^2(c)
\end{equation*}
and equip it with the symplectic form $ds\wedge dt + \omega$, where $(s,t)$ are the coordinates on $[0,\varepsilon]\times [0,1]$ and $\omega$ is the standard area form on $B^2(c)$. If $\varepsilon>0$ is sufficiently small, we can find a symplectic embedding $\alpha : R_\varepsilon \overset{s}{\hookrightarrow} B$ such that the restriction of $\alpha$ to $\left\{ 0 \right\}\times R \cong R$ agrees with the restriction of $\varphi$ to $R$ and $\alpha$ maps the two boundary strata $[0,\varepsilon]\times \left\{ j \right\} \times B^2(c)$ for $j\in \left\{ 0,1 \right\}$ into $\partial \varphi(E(a,b))$ and $\partial B$, respectively. By the assumption that $\rho(D_j) = \operatorname{im}(\rho\circ \psi)$, up to applying a symplectomorphism of $B$ and relabelling the disks $D_j$, we can assume that $D_j = \alpha(\left\{ j/(n-1) \right\}\times \left\{ 1 \right\} \times B^2(c))$ for $1 \leq j \leq n-1$. For each $j$, let us attach the ribbon $\alpha(\left\{ j/(n-1) \right\}\times R)$ to $\varphi(E(a,b))$. As observed in Example \ref{ex:simple_ribbon_complex}, the resulting ribbon complex is symplectomorphic to $E(a,b;c,n)$. It is the image of a symplectic embedding $E(a,b;c,n)\overset{s}{\hookrightarrow} B$ satisfying all the desired properties at the free ribbon ends.
\end{proof}

Consider the triangle $T\subset \BR^2$ with vertices $(0,0)$, $(1,0)$ and $(0,1)$ and the square $Q\subset \BR^2$ with vertices $(0,0)$, $(1,0)$, $(1,1)$ and $(0,1)$. It is well-known that the interior of the Lagrangian product $T\times_L Q$ is symplectomorphic to the interior of the ball $B^4(1)$ (see \cite[Proposition 5.2]{tra95}). Even though $T\times_L Q$ is not smooth everywhere, one can still make sense of a generalized Reeb flow on its boundary. The dynamics of this generalized Reeb flow is equivalent to the Minkowski billard dynamics of $T$ and $Q$ (see \cite{gt02} and \cite[\S 2.4]{ao14}). In the following, we give a detailed description of the generalized Reeb flow. For this purpose, let us abbreviate the edges of $T$ by $t_l$, $t_b$ and $t_d$, as indicated in Figure \ref{fig:triangle_and_square}. Moreover, let $q_l$, $q_b$, $q_r$ and $q_t$ denote the edges of $Q$.
\begin{figure}[htpb]
\centering
\def\svgwidth{0.6\textwidth}
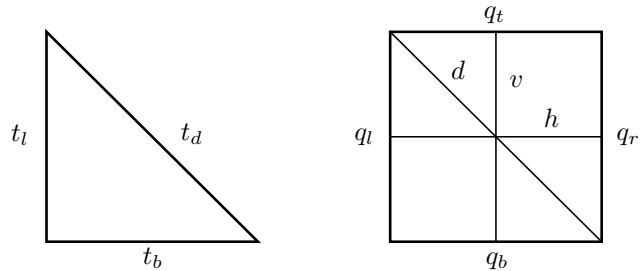
\caption{The triangle $T$ and the square $Q$.}
\label{fig:triangle_and_square}
\end{figure}

The boundary of $T\times_L Q$ has seven $3$-dimensional faces, three of the form $t_*\times Q$ and four of the form $T\times q_*$. On each $3$-dimensional face, the characteristic foliation is very simple: It is parallel to $J_0\nu$, where $\nu$ is the unit outer normal vector of the face. There are twelve $2$-dimensional faces $t_*\times q_*$. Four of these $2$-dimensional faces, $t_l\times q_l$, $t_l\times q_r$, $t_b\times q_b$ and $t_b\times q_t$, are Lagrangian, i.e. the restriction of the symplectic form vanishes on them. The remaining eight $2$-dimensional faces are symplectic. At each of the symplectic $2$-dimensional faces, the Reeb flow simply transitions between the two $3$-dimensional faces meeting there. At the Lagrangian $2$-dimensional faces, the generalized Reeb flow is discontinuous. The Reeb flow admits a very simple description, see Figure \ref{fig:generalized_reeb_dynamics}.
\begin{figure}[htpb]
\centering
\def\svgwidth{1.0\textwidth}
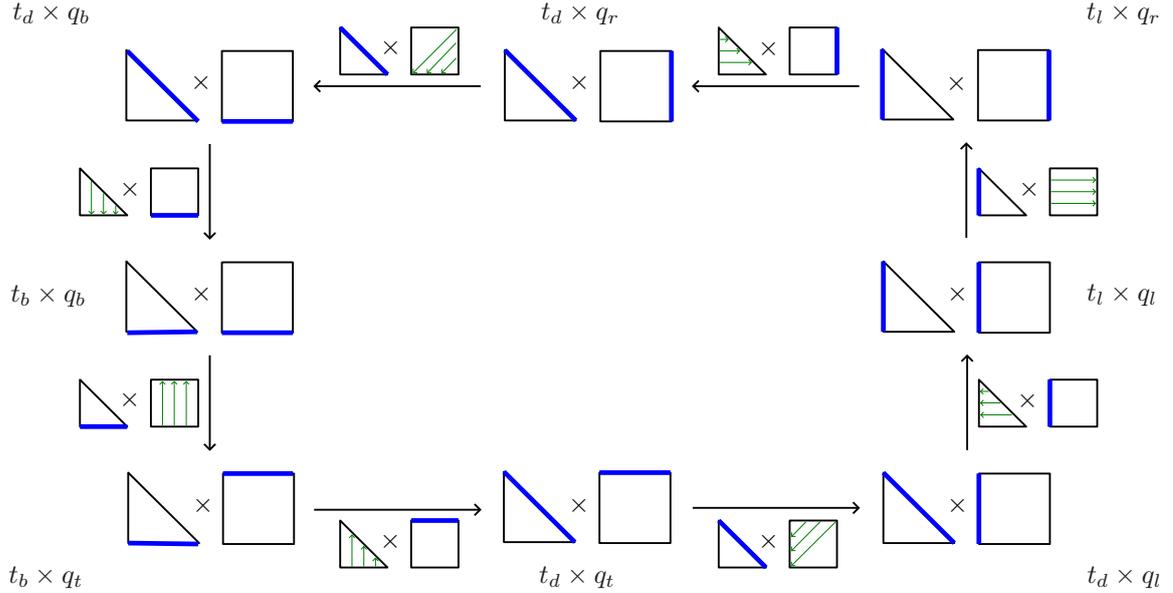
\caption{Reeb dynamics on the boundary of $T\times_L Q$.}
\label{fig:generalized_reeb_dynamics}
\end{figure}
Let us for example start at the symplectic $2$-face $t_l\times q_l$. At this $2$-face, the Reeb flow enters the $3$-face $t_l\times Q$. It traverses $t_l\times Q$ horizontally from left to right, until it meets the $2$-face $t_l\times q_r$, where it leaves $t_l\times Q$ and enters $T\times q_r$. Clearly, the transition map from $t_l\times q_l$ to $t_l\times q_r$ is affine linear. As indicated in Figure \ref{fig:generalized_reeb_dynamics}, the Reeb flow passes through all the symplectic $2$-faces and then ends up back at $t_l\times q_l$. The return map is simply given by the identity.

Let $0<c<1$. Let $D\subset t_l\times q_l$ be a smooth disk of area $c$ obtained by first taking a square inside $t_l\times q_l$ of area $c$ and with the same center as $t_l\times q_l$ and then slightly smoothing its corners. Using the Reeb flow, we produce copies of $D$ contained in the interiors of $t_*\times Q$ as follows. Let $v$ and $h$ denote the vertical and horizontal line segments dividing the square $Q$ into two rectangles of equal areas. Moreover, let $d$ denote the diagonal of $Q$ connecting the vertices $(0,1)$ and $(1,0)$. If we start at $t_l\times q_l$, the Reeb flow meets each of the cross sections $t_l\times v$ and $t_b\times h$ exactly once before it returns to $t_l\times q_l$. Similarly, it meets the cross section $t_d\times d$ exactly twice, each time at a different half of $t_d\times d$. We define $\tilde{D}$ to be the intersection of all Reeb trajectories starting at $D$ with the union of the cross sections $t_l\times v$, $t_b\times h$ and $t_d\times d$. The set $\tilde{D}$ consists of four copies of $D$.

\begin{lemma}
\label{lem:embedding_into_lagrangian_product}
Let $c<a<1$. Then there exists a symplectic embedding $B^4(a;c,4) \overset{s}{\hookrightarrow} T\times_L Q$ which maps the four free ribbon ends of $B^4(a;c,4)$ exactly to the four components of $\tilde{D}$.
\end{lemma}

\begin{proof}
It is well known that $B^4(a)$ admits a symplectic embedding into the interior of $T\times_L Q$, see \cite{tra95}. What we need to show is that there exists an embedding whose restriction to the boundary of $B^4(a)$ is sufficiently tame such that we can insert four ribbons of width $c$ starting at the boundary of $B^4(a)$ and ending at the four components of $\tilde{D}$. In order to see that this is possible, let us recall the following symplectic embedding construction from \cite[\S 2]{sch03} and \cite[\S 4.1]{sch05}. Let $U\subset \BR^2$ be the open ball of area $1$ centered at $0$ and let $V\subset \BR^2$ be an open axis parallel square or area $1$. A family $\ML$ of loops in $V$ is called \textit{admissible} if there exists a diffeomorphism $\beta : U\setminus \left\{ 0 \right\} \rightarrow V\setminus \left\{ p \right\}$ for some point $p\in V$ such that circles in $U$ centered at $0$ are mapped to members of $\ML$ and the restriction of $\beta$ to a neighbourhood of $0$ is a translation. Schlenk proves that given an admissible family $\ML$, there exists a symplectomorphism $\varphi:U \rightarrow V$ mapping circles centered at $0$ to members of $\ML$, see \cite[Lemma 2.5]{sch03} or \cite[Lemma 4.2]{sch05}. One can arrange the family $\ML$ such that the image of the restriction of the symplectomorphism $\varphi\times \varphi : U\times U \rightarrow V\times V$ to $B^4(a)$ is an arbitrarily close approximation of $aT\times_L Q$, see Figure \ref{fig:admissible_curve_family}. Here $aT$ is the scaling of $T$ by a factor $a$.
\begin{figure}[htpb]
\centering
\def\svgwidth{0.3\textwidth}
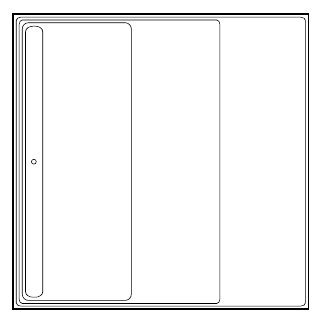
\caption{A special admissible family of loops $\ML$.}
\label{fig:admissible_curve_family}
\end{figure}

For a careful choice of the family of loops $\ML$, the generalized Reeb orbits on the boundary of $aT\times_L Q$ away from the Lagrangian $2$-faces are approximated arbitrarily well by images of Hopf circles on the boundary of $B^4(a)$. Note that $aT\times_L Q$ is symplectomorphic to $\sqrt{a}T \times_L \sqrt{a}Q$. Let us translate $\sqrt{a}T \times_L \sqrt{a}Q$ such that it is contained in the interior of $T\times_L Q$. Then the above embedding construction yields an embedding of $B^4(a)$ into the interior of $T\times_L Q$ which approximates $\sqrt{a}T\times_L \sqrt{a}Q$ arbitrarily well and has the property that generalized Reeb orbits on $\sqrt{a}T\times_L \sqrt{a}Q$ away from the Lagrangian $2$-strata are approximated by images of Hopf circles.

Inside the boundary of $\sqrt{a}T\times_L\sqrt{a}Q$ we have a collection $\tilde{E}$ of four disks of area $c$ each which is constructed in the same way as $\tilde{D}$. The disks $\tilde{E}$ can be constructed in such a way that it is possible to insert four copies of the ribbon $R(c)$ in the complement of $\sqrt{a}T\times_L \sqrt{a}Q$ connecting the components of $\tilde{E}$ to the corresponding components of $\tilde{D}$. Since the image of $B^4(a)$ under our symplectic embedding approximates $\sqrt{a}T\times_L\sqrt{a}Q$ arbitrarily well and generalized Reeb orbits on the boundary of $\sqrt{a}T\times_L\sqrt{a}Q$ which intersect $\tilde{E}$ are approximated by images of Hopf circles, we can modify the four ribbons in such a way that they connect the components of $\tilde{D}$ to the boundary of the image of $B^4(a)$ and give rise to the desired embedding of $B(a;c,4)$.
\end{proof}

\begin{proof}[Proof of Proposition \ref{prop:special_embedding_into_ball} in the case $n\geq 2$]
We assume that $B=B^4(1)$. For simplicity, let us first treat the case $n=2$. The ball $B$ admits a full packing by four copies of $B^4(1/2)$. Such a packing can be seen explicitly as follows. As mentioned above, the interior of $B$ is symplectomorphic to the interior of the Lagrangian product $T\times_L Q$. We divide $T$ into four triangles $T_0, T_1, T_2, T_3$ as indicated in Figure \ref{fig:four_packing_moment_image}. This induces a decomposition of $T\times_L Q$ into four pieces $T_j\times_L Q$, the interior of each of which is symplectomorphic to the interior of $B^4(1/2)$.
\begin{figure}[htpb]
\centering
\def\svgwidth{0.3\textwidth}
\begingroup%
  \makeatletter%
  \providecommand\color[2][]{%
    \errmessage{(Inkscape) Color is used for the text in Inkscape, but the package 'color.sty' is not loaded}%
    \renewcommand\color[2][]{}%
  }%
  \providecommand\transparent[1]{%
    \errmessage{(Inkscape) Transparency is used (non-zero) for the text in Inkscape, but the package 'transparent.sty' is not loaded}%
    \renewcommand\transparent[1]{}%
  }%
  \providecommand\rotatebox[2]{#2}%
  \newcommand*\fsize{\dimexpr\f@size pt\relax}%
  \newcommand*\lineheight[1]{\fontsize{\fsize}{#1\fsize}\selectfont}%
  \ifx\svgwidth\undefined%
    \setlength{\unitlength}{170.07874016bp}%
    \ifx\svgscale\undefined%
      \relax%
    \else%
      \setlength{\unitlength}{\unitlength * \real{\svgscale}}%
    \fi%
  \else%
    \setlength{\unitlength}{\svgwidth}%
  \fi%
  \global\let\svgwidth\undefined%
  \global\let\svgscale\undefined%
  \makeatother%
  \begin{picture}(1,1)%
    \lineheight{1}%
    \setlength\tabcolsep{0pt}%
    \put(0,0){\includegraphics[width=\unitlength,page=1]{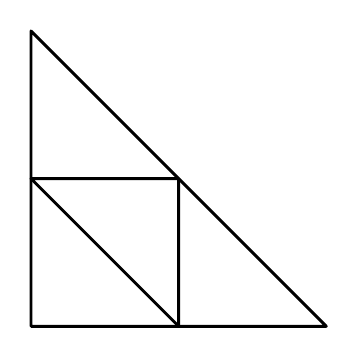}}%
    \put(0.37901738,0.3568813){\color[rgb]{0,0,0}\makebox(0,0)[t]{\lineheight{1.25}\smash{\begin{tabular}[t]{c}$T_0$\end{tabular}}}}%
    \put(0.62901738,0.19021457){\color[rgb]{0,0,0}\makebox(0,0)[lt]{\lineheight{1.25}\smash{\begin{tabular}[t]{l}$T_1$\end{tabular}}}}%
    \put(0.21235072,0.19021457){\color[rgb]{0,0,0}\makebox(0,0)[lt]{\lineheight{1.25}\smash{\begin{tabular}[t]{l}$T_3$\end{tabular}}}}%
    \put(0.21235072,0.60688124){\color[rgb]{0,0,0}\makebox(0,0)[lt]{\lineheight{1.25}\smash{\begin{tabular}[t]{l}$T_2$\end{tabular}}}}%
    \put(0,0){\includegraphics[width=\unitlength,page=2]{four_packing_moment_image.pdf}}%
  \end{picture}%
\endgroup%

\caption{Division of $T$ into four triangles $T_j$.}
\label{fig:four_packing_moment_image}
\end{figure}

Recall that we introduced, for every $0<\gamma<1$, a subset $\tilde{D}$ of the boundary of $T\times_L Q$ consisting of four disks of symplectic area $\gamma$ each. The same construction yields, for every $j$ and every $0<\gamma<1/2$, a subset $\tilde{D}_j$ of the boundary of $T_j\times_L Q$ consisting of four disks of area $\gamma$ each. These disks can be arranged such that, for $1\leq j \leq 3$, we have
\begin{equation*}
\tilde{D}_0 \cap ((T_0\cap T_j)\times_L Q) = \tilde{D}_j \cap ((T_0\cap T_j)\times_L Q),
\end{equation*}
i.e. the components of $\tilde{D}_0$ and $\tilde{D}_j$ match at the interface between $T_0\times_L Q$ and $T_j\times_L Q$.

Let $b>0$ be a small positive number. If necessary, we will further shrink $b$ over the course of the proof. Let $0<c<b$ be arbitrary and let $a>0$ be such that $\operatorname{vol}(E(a,b)) < \operatorname{vol}(B)$. Our goal is to construct a symplectic embedding $E(a,b;c,2) \overset{s}{\hookrightarrow} B$ mapping the two free ends into the boundary $\partial B$ such that the resulting collection of two disk embeddings is free.

If $b$ is sufficiently small, we may pick positive numbers $0<\gamma<\alpha$ such that $\operatorname{vol}(E(a,b)) < 4 \operatorname{vol}(B^4(\alpha)) < 1$ and $\gamma > b$. For $0\leq j \leq 3$, let $\tilde{D}_j$ be the subset of the boundary of $T_j\times_L Q$ mentioned above consisting of four disks of area $\gamma$ each. By Lemma \ref{lem:embedding_into_lagrangian_product}, we may embed, for each $0\leq j \leq 3$, a copy $L_j$ of $B^4(\alpha;\gamma,4)$ into $T_j\times_L Q$ such that the four free ends of $L_j$ are mapped to the four disk components of $\tilde{D}_j$. Note that $L_0$ and $L_j$ have one matching pair of free ends if $j\in \left\{ 1,2 \right\}$ and two matching pairs of free ends if $j=3$. We build a ribbon complex $L$ embedded in $T\times_L Q$ by taking the union $\bigcup_j L_j$ and gluing all matching pairs of free ends.

By Lemma \ref{lem:special_embedding_into_ball_dependent_attaching_maps}, after possibly shrinking $b$, we can find $f>b$ and a symplectic embedding $E(a/4,f;b,4)\overset{s}{\hookrightarrow} B^4(\alpha)$ which maps the $j$-th free end of $E(a/4,f;b,4)$ into the image of the map $\iota(2\pi j/4): B^2(\gamma) \rightarrow \partial B^4(\alpha)$ for $0\leq j \leq 3$. This implies that we can also embed $E(a/4,f;b,4)$ into $B^4(\alpha;\gamma,4)$ in such a way that the four free ends of $E(a/4,f;b,4)$ are contained in the four free ends of $B^4(\alpha;\gamma,4)$: simply elongate each ribbon of $E(a/4,f;b,4)$ by attaching an appropriate ribbon which is contained in the corresponding ribbon of $B^4(\alpha;\gamma,4)$. Take a copy $K_0$ of $E(a/4,f;b,4)$. By the above discussion, we can embed it into the complex $L$ in such a way that the ellipsoid part of $K_0$ is mapped into the ball component of $L$ contained in $T_0\times_L Q$ and the four free ribbon ends are mapped to the boundaries of the three remaining ball components of $L$. Two free ends of $K_0$ are mapped to the boundry of the ball component of $L$ inside $T_3\times_L Q$ and one free end is mapped to the boundary of the ball component inside $T_j\times_L Q$ for $j\in \left\{ 1,2 \right\}$. We pick one of the two ribbons of $K_0$ whose end is contained in the ball component of $L$ inside $T_3\times_L Q$ and delete it. The resulting ribbon complex is still denoted by the same symbol $K_0$.

By the case $n=1$ of Proposition \ref{prop:special_embedding_into_ball}, we may take a copy $K_3$ of $E(a/4,f;b,1)$ and embed it into the ball component of $L$ contained in $T_3\times_L Q$ in such a way that the free end of $K_3$ is mapped to the boundary of this ball component via exactly the same map as one the the three free ends of $K_0$. Similarly, by Lemma \ref{lem:special_embedding_into_ball_dependent_attaching_maps}, we may take a copy $K_j$ of $E(a/4,f;b,2)$ for $j\in \left\{ 1,2 \right\}$ and embed it into the ball component of $L$ contained in $T_j\times_L Q$ in such a way that one of the two free ends of $K_j$ is mapped to the boundary of the ball component via exactly the same map as one of the free ends of $K_0$ and the image of the other free end of $K_j$ is contained in the attaching region of one of the the ribbons of $L$ connecting to the diagonal part of the boundary of $T_j\times_L Q$. We elongate the latter ribbon of $K_j$ inside the ribbon of $L$ such that its free end is mapped to the boundary of $T_j\times_L Q$.

We form a ribbon complex $K$ by taking the union $\bigcup_j K_j$ and gluing all matching pairs of free ends. See Figure \ref{fig:ribbon_complexes_L_K} for a schematic of $L$ and $K$. The resulting complex $K$ is symplectomorphic to the complex $K(G,A,f,b)$ constructed in Example \ref{ex:complex_from_ribbon_graph} where $(G,A)$ is the labelled ribbon graph displayed in Figure \ref{fig:labelled_graph}.

\begin{figure}[htpb]
\centering
\def\svgwidth{0.4\textwidth}
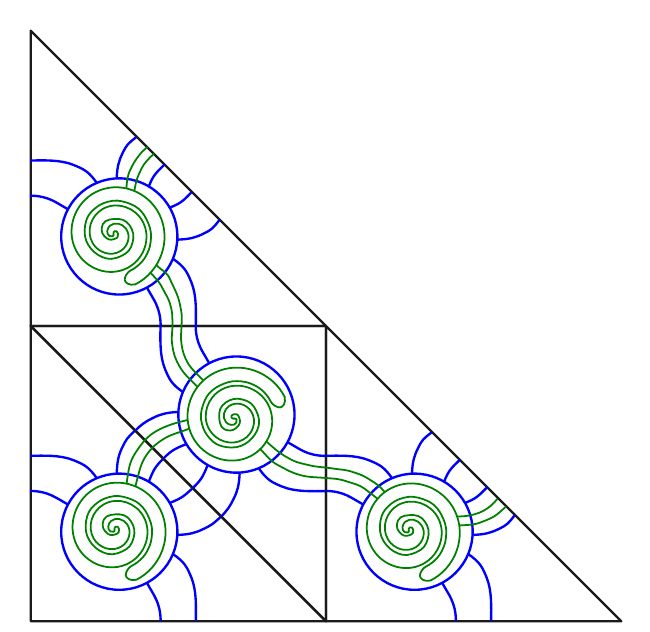
\caption{A schematic of the ribbon complexes $L$ (blue) and $K$ (green) embedded in $T\times_L Q$.}
\label{fig:ribbon_complexes_L_K}
\end{figure}

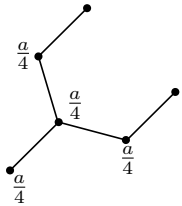
\begin{figure}[htpb]
\centering
\def\svgwidth{0.2\textwidth}
\begingroup%
  \makeatletter%
  \providecommand\color[2][]{%
    \errmessage{(Inkscape) Color is used for the text in Inkscape, but the package 'color.sty' is not loaded}%
    \renewcommand\color[2][]{}%
  }%
  \providecommand\transparent[1]{%
    \errmessage{(Inkscape) Transparency is used (non-zero) for the text in Inkscape, but the package 'transparent.sty' is not loaded}%
    \renewcommand\transparent[1]{}%
  }%
  \providecommand\rotatebox[2]{#2}%
  \newcommand*\fsize{\dimexpr\f@size pt\relax}%
  \newcommand*\lineheight[1]{\fontsize{\fsize}{#1\fsize}\selectfont}%
  \ifx\svgwidth\undefined%
    \setlength{\unitlength}{212.5984252bp}%
    \ifx\svgscale\undefined%
      \relax%
    \else%
      \setlength{\unitlength}{\unitlength * \real{\svgscale}}%
    \fi%
  \else%
    \setlength{\unitlength}{\svgwidth}%
  \fi%
  \global\let\svgwidth\undefined%
  \global\let\svgscale\undefined%
  \makeatother%
  \begin{picture}(1,0.86666667)%
    \lineheight{1}%
    \setlength\tabcolsep{0pt}%
    \put(0,0){\includegraphics[width=\unitlength,page=1]{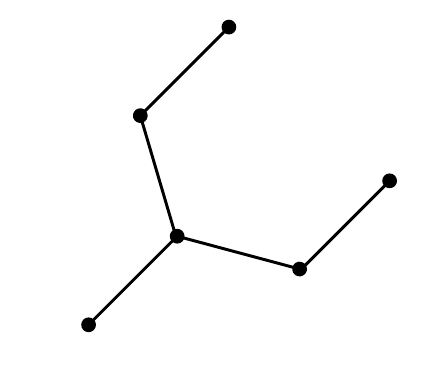}}%
    \put(0.68458175,0.15702712){\color[rgb]{0,0,0}\makebox(0,0)[t]{\lineheight{1.25}\smash{\begin{tabular}[t]{c}$\frac{a}{4}$\end{tabular}}}}%
    \put(0.20165166,0.03642893){\color[rgb]{0,0,0}\makebox(0,0)[lt]{\lineheight{1.25}\smash{\begin{tabular}[t]{l}$\frac{a}{4}$\end{tabular}}}}%
    \put(0.21712924,0.59099864){\color[rgb]{0,0,0}\makebox(0,0)[lt]{\lineheight{1.25}\smash{\begin{tabular}[t]{l}$\frac{a}{4}$\end{tabular}}}}%
    \put(0.43044104,0.38438877){\color[rgb]{0,0,0}\makebox(0,0)[lt]{\lineheight{1.25}\smash{\begin{tabular}[t]{l}$\frac{a}{4}$\end{tabular}}}}%
  \end{picture}%
\endgroup%

\caption{The labelled ribbon graph $(G,A)$.}
\label{fig:labelled_graph}
\end{figure}

By construction, $K$ is embedded in $T\times_L Q$ and its two free ends are mapped into the diagonal part $t_d\times Q$ of the boundary of $T\times_L Q$. We compose the embedding of $K$ into $T\times_L Q$ with the symplectic embedding
\begin{equation*}
T\times_L Q \rightarrow B \qquad (\mu_1,\mu_2,\theta_1,\theta_2) \mapsto (\sqrt{\frac{\mu_1}{\pi}} e^{2\pi i \theta_1}, \sqrt{\frac{\mu_2}{\pi}} e^{2\pi i \theta_2}).
\end{equation*}
This yields a symplectic embedding of $K$ into $B$. It is straightforward to check that the two free ends of $K$ in $\partial B$ form a collection of two free disks. By Lemma \ref{lem:special_ribbon_complex_neighbourhood_packing}, there exists a symplectic embedding of $E(a,b;c,2)$ into $B$ such that images of the two free ends are contained in the images of the two free ends of $K$. This implies that the two free ends of $E(a,b;c,n)$ form a collection of two free disks in $\partial B$. Using Lemma \ref{lem:free_attaching_maps_ball}, the case $n=2$ of Proposition \ref{prop:special_embedding_into_ball} is an immediate consequence.

The method used to deal with the case $n=2$ can be easily adapted to also treat the case $n>2$. Instead of decomposing $T\times_L Q$ into four balls, we decompose into $n^2$ balls, see Figure \ref{fig:fine_decomp}. We omit the details.
\begin{figure}[htpb]
\centering
\def\svgwidth{0.3\textwidth}
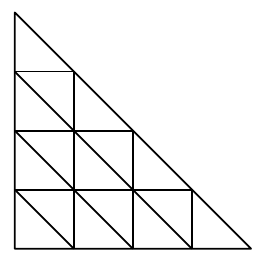
\caption{Decomposition of $B$ into $n^2$ balls for $n=4$.}
\label{fig:fine_decomp}
\end{figure}
\end{proof}

Before turning to the proof of Theorem \ref{thm:ellipsoid_packing_of_ribbon_complex}, we need to state and prove a version of Proposition \ref{prop:special_embedding_into_ball} for polydisks. Consider a polydisk $P = P(a_1,a_2)$. Then the smooth locus of the boundary of $P$ has two components
\begin{equation*}
\partial_1P = \operatorname{int}(B^2(a_1))\times \partial B^2(a_2) \quad \text{and} \quad \partial_2P = \partial B^2(a_1) \times \operatorname{int}(B^2(a_2)).
\end{equation*}
Let $i\in \left\{ 1,2 \right\}$. Note that we have a projection $\operatorname{pr}_i:\partial_iP\rightarrow \operatorname{int}(B^2(a_i))$. Let $c>0$ be a positive real number. A symplectic embedding $\psi: B^2(c) \hookrightarrow \partial_i P$ is called \textit{free} if the composition $\operatorname{pr}_i\circ \psi : B^2(c)\rightarrow \operatorname{int}(B^2(a_i))$ is an embedding. Let $n$ be a non-negative integer. Consider symplectic embeddings $\psi_j:B^2(c)\hookrightarrow \partial_iP$ for $1\leq j \leq n$. We call the tuple $(\psi_j)_j$ \textit{free} if the compositions $\operatorname{pr}_i\circ \psi_j$ are pairwise disjoint symplectic embeddings. If $(\psi_j)_j$ is a tuple of symplectic embeddings $\psi_j:B^2(c)\hookrightarrow \partial_1P\cup \partial_2P$, we call it \textit{free} if and only if the two subcollections of embeddings into $\partial_1P$ and $\partial_2P$ are free.

\begin{lemma}
\label{lem:free_collections_polydisk}
Let $(\psi_j)_j$ and $(\psi'_j)_j$ be two free collection of symplectic embeddings $\psi_j:B^2(c)\hookrightarrow \partial_1P \cup \partial_2P$. Then there exists a symplectomorphism $\chi$ of $P$ preserving $\partial_1P$ and $\partial_2P$ and satisfying $\psi'_j = \chi\circ \psi_j$ for all $j$ if and only if, for every $j$, the images of $\psi_j$ and $\psi'_j$ are contained in the same component of $\partial_1P\cup \partial_2P$.
\end{lemma}

\begin{proof}
If there exists such a symplectomorphism $\chi$, then clearly the images of $\psi_j$ and $\psi_j'$ must be contained in the same component of $\partial_1P\cup \partial_2P$ for all $j$.

In order to show the converse direction, first note that the compactly supported symplectomorphism group of $\operatorname{int}(B^2(a_i))$ acts transitively on $n$-typles of pairwise disjoint symplectic embedings of $B^2(c)$ into $\operatorname{int}(B^2(a_i))$ for every $n>0$. Since any compactly supported symplectomorphism of $\operatorname{int}(B^2(a_i))$ lifts to a symplectomorphism of $P$ preserving $\partial_1 P$ and $\partial_2 P$, we can find such a symplectomorphism $\chi'$ with the property that $\operatorname{pr}_i\circ\psi_j' = \operatorname{pr}_i\circ \chi'\circ \psi_j$ for every $j$ and $i$ such that the images of $\psi_j$ and $\psi_j'$ are contained in $\partial_i P$. Finally, observe that if $\psi,\psi':B^2(c)\rightarrow \partial_iP$ are to free symplectic embeddings such that $\operatorname{pr}_i\circ \psi = \operatorname{pr}_i\circ \psi'$, then $\psi$ and $\psi'$ are related by a symplectomorphism of $P$ which is supported near $\operatorname{pr}_i^{-1}(\operatorname{im}(\operatorname{pr}_i\circ \psi))$. This shows the existence of the desired symplectomorphism $\chi$ with the property that $\psi_j' = \chi\circ \psi_j$.
\end{proof}

\begin{proposition}
\label{prop:special_embedding_into_polydisk}
Let $P$ be a $4$-dimensional symplectic polydisk and let $n$ be a non-negative integer. Then there exists $b_0>0$ such that the following is true for every $0<b\leq b_0$: Let $0<a$ such that $\operatorname{vol}(E(a,b)) < \operatorname{vol}(P)$. Let $0<c<b$. For $0\leq j < n$, let $\psi_j:B^2(c)\hookrightarrow \partial_1P\cup \partial_2P$ be a symplectic embedding into the smooth locus of $\partial P$ and assume that the collection $(\psi_j)_j$ is free. After possibly permuting the indices of $\psi_j$, there exists a symplectic embedding $\varphi:E(a,b;c,n) \overset{s}{\hookrightarrow} P$ whose restriction to the $j$-th free ribbon end of $E(a,b;c,n)$ agrees with the embedding $\psi_j$. Here we use the preferred identification of the free ribbon ends with $B^2(c)$, see Example \ref{ex:simple_ribbon_complex}. Moreover, the image $\operatorname{im}(\varphi)$ intersects $\partial P$ precisely at the free ends and the intersection is transverse.

For fixed $n>0$ and a compact subset $C \subset \BR_{>0}^2$, one can choose $b_0$ uniformly for all polydisks $P(x,y)$ with $(x,y)\in C$.
\end{proposition}

\begin{lemma}
\label{lem:special_embedding_into_S2timesS2}
Let $a,b,\alpha,\beta>0$ and suppose that the ellipsoid $E = E(a,b) \subset \BC^2$ symplectically embeds into $M = \BC P^1(\alpha)\times \BC P^1(\beta)$. Define $C\coloneqq \left\{ * \right\} \times \BC P^1(\beta)$. If $b<\beta$, then there exists a symplectic embedding $\varphi: E\overset{s}{\hookrightarrow} M$ such that $\varphi^{-1}(C) = E \cap (\left\{ 0 \right\}\times \BC)$.
\end{lemma}

\begin{proof}
In contrast to Lemma \ref{lem:special_embedding_into_CP2}, which is an analogous statement with target manifold $\BC P^2$ instead of $\BC P^1\times \BC P^1$, we cannot directly apply \cite[Lemma 1.6]{bho16} because this result is only stated for $\BC P^2$. However, as we sketch below, the proof of \cite[Lemma 1.6]{bho16}, which relies on the symplectic blow up construction for ellipsoid embeddings introduced in \cite{mcd09} and the singular inflation techniques of \cite{mo15}, readily adapts to show Lemma \ref{lem:special_embedding_into_S2timesS2}.

The strategy of the proof of \cite[Lemma 1.6]{bho16} is the following. For $\tau>0$ sufficiently small, there clearly exists an embedding $\tau E\overset{s}{\hookrightarrow} M$ intersecting the symplectic sphere $C$ in the desired way. Blowing up this ellipsoid embedding yields a nodal symplectic sphere configuration $S$ inside a suitable blowup $\hat{M}$ of $M$. Since $\tau E$ is in special position with respect to $C$, one also obtains a sympelctic surface $\hat{C}$ inside $\hat{M}$ which is a proper transform of $C$ and intersects $S$ positively and transverely. The embedding of the full ellipsoid $E$ is obtained by changing the symplectic areas of the sphere components of $S$ via symplectic inflation relative to the singular symplectic surface $S\cup \hat{C}$. In order to carry out this inflation, one needs an embedded symplectic surface intersecting $S\cup \hat{C}$ positively and transversely and representing a suitable homology class with positive self intersection number. In the proof of \cite[Lemma 1.6]{bho16}, the existence of such a symplectic surface is deduced from \cite[Corollary 1.2.17]{mo15}, which deals with blow ups of $\BC P^2$. In the present setting of $M = \BC P^1(\alpha) \times \BC P^1(\beta)$, one can obtain the desired symplectic surface from \cite[Theorem 1.2.16]{mo15}, which more generally applies to blow ups of arbitrary rational or ruled symplectic manifolds.
\end{proof}

\begin{proof}[Proof of Proposition \ref{prop:special_embedding_into_polydisk}]
The proof is similar to the proof of Proposition \ref{prop:special_embedding_into_ball}. The case $n=0$ is immediate from Lemma \ref{lem:skinny_ellipsoid_embeddings}. The case $n=1$ is deduced from Lemma \ref{lem:connecting_ribbon_to_core_plane} in combination with Lemma \ref{lem:special_embedding_into_S2timesS2}.

Note that the interior of the polydisk $P(x,y)$ is symplectomorphic to the interior of the Lagrangian product $S \times_L Q$ where $S \subset\BR^2$ is a rectangle of side lengths $x$ and $y$ parallel to the axes and $Q$, as before, is a square of side lenth $1$. Again, the generalized Reeb flow on the boundary of $S\times_L Q$ is just given by the Minkowski billiard dynamics of $S$ and $Q$. Let us label the edges of $S$ and $Q$ as in Figure \ref{fig:rectangle_and_square}. The generalized Reeb flow on the boundary of $S\times_L Q$ is depicted in Figure \ref{fig:generalized_reeb_flow_polydisk}.

\begin{figure}[htpb]
\centering
\def\svgwidth{0.6\textwidth}
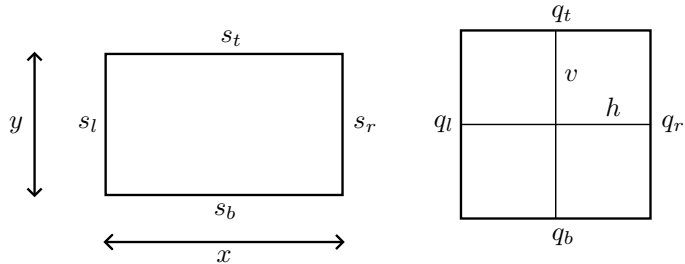
\caption{The rectangle $S$ and the square $Q$.}
\label{fig:rectangle_and_square}
\end{figure}

\begin{figure}[htpb]
\centering
\def\svgwidth{1.0\textwidth}
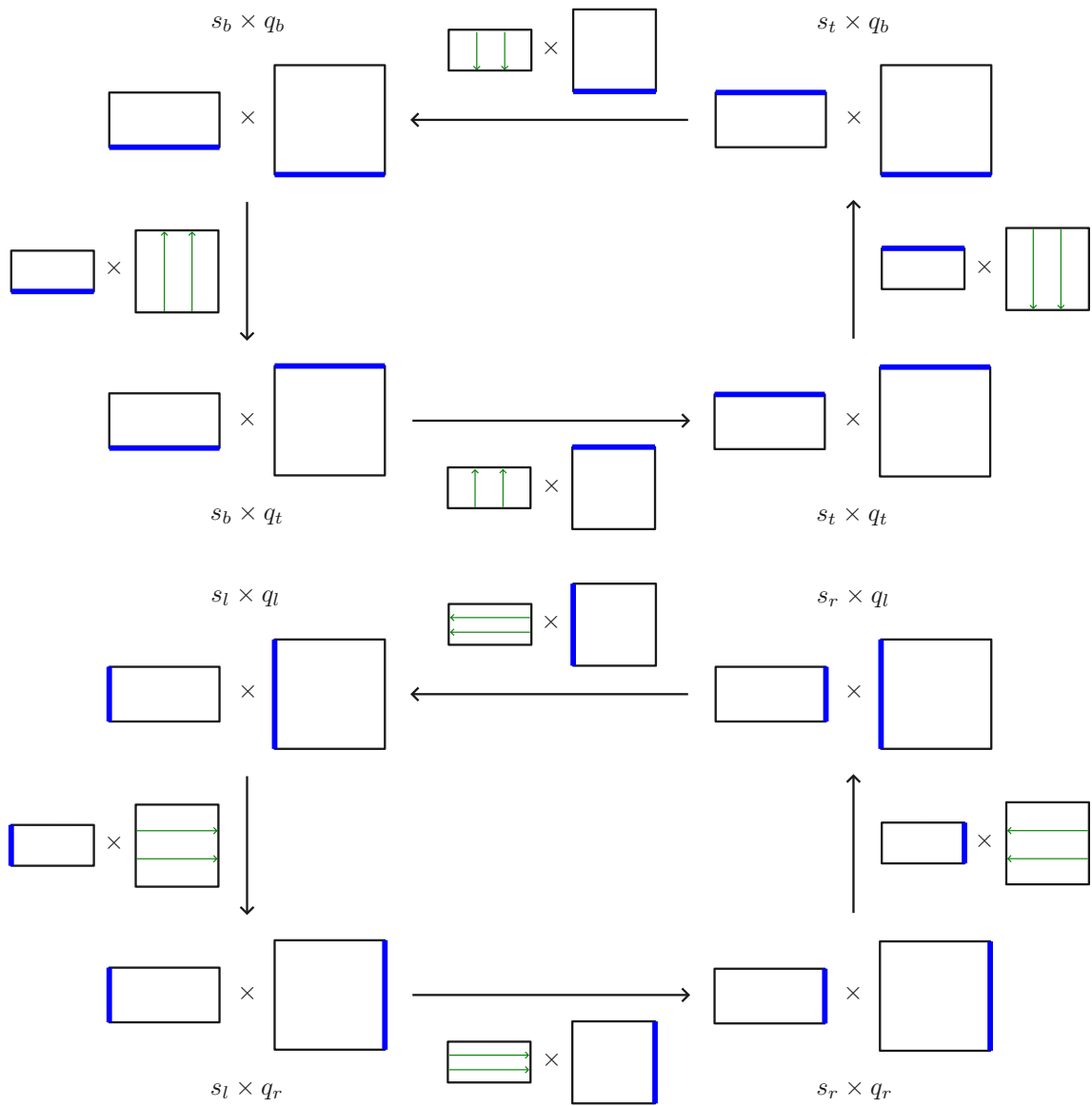
\caption{Reeb dynamics on the boundary of $S\times_L Q$.}
\label{fig:generalized_reeb_flow_polydisk}
\end{figure}

Let $0<\xi<x$. Let $D\subset s_b\times q_b$ be a disk of area $\xi$ obtained by first taking a rectangle of area $\xi$ and then smoothing its corners. Using the Reeb flow, we produce copies of $D$ inside $s_b\times Q$ and $s_t\times Q$. As before, let $h$ (resp. $v$) denote the horizantal (resp. vertical) line segment dividing the square $Q$ into two rectangles of equal area, see Figure \ref{fig:rectangle_and_square}. If we start at $s_b\times q_b$, then the Reeb flow meets both cross sections $s_b \times h$ and $r_t\times h$ exactly once before returning to $r_b\times q_b$. We define $\tilde{D}_h \subset (r_b\cup r_t)\times h$ to be the intersection of all characteristics starting at $D$ with $(r_b\cup r_t)\times h$. Note that we can arrange $\tilde{D}_h$ in such a way that if we place a triangle at the bottom or top edge of $S$ as in Figure \ref{fig:triangle_on_top_rectangle} and if $\tilde{D}\subset\partial(T\times_LQ)$ is the collection  of four disks of area $\xi$ used in the proof of Proposition \ref{prop:special_embedding_into_ball}, then the component of $\tilde{D}_h$ contained in the interface between $S\times_L Q$ and $T\times_LQ$ exactly matches the component of $\tilde{D}$ contained in this interface.
\begin{figure}[htpb]
\centering
\def\svgwidth{0.2\textwidth}
\begingroup%
  \makeatletter%
  \providecommand\color[2][]{%
    \errmessage{(Inkscape) Color is used for the text in Inkscape, but the package 'color.sty' is not loaded}%
    \renewcommand\color[2][]{}%
  }%
  \providecommand\transparent[1]{%
    \errmessage{(Inkscape) Transparency is used (non-zero) for the text in Inkscape, but the package 'transparent.sty' is not loaded}%
    \renewcommand\transparent[1]{}%
  }%
  \providecommand\rotatebox[2]{#2}%
  \newcommand*\fsize{\dimexpr\f@size pt\relax}%
  \newcommand*\lineheight[1]{\fontsize{\fsize}{#1\fsize}\selectfont}%
  \ifx\svgwidth\undefined%
    \setlength{\unitlength}{141.73228346bp}%
    \ifx\svgscale\undefined%
      \relax%
    \else%
      \setlength{\unitlength}{\unitlength * \real{\svgscale}}%
    \fi%
  \else%
    \setlength{\unitlength}{\svgwidth}%
  \fi%
  \global\let\svgwidth\undefined%
  \global\let\svgscale\undefined%
  \makeatother%
  \begin{picture}(1,1.46)%
    \lineheight{1}%
    \setlength\tabcolsep{0pt}%
    \put(0,0){\includegraphics[width=\unitlength,page=1]{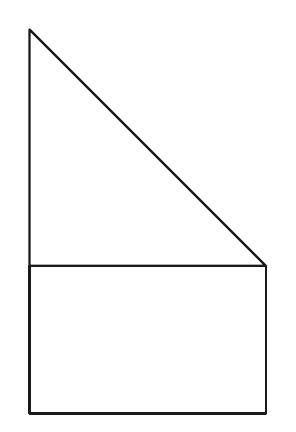}}%
    \put(0.30251112,0.84631527){\color[rgb]{0.10196078,0.10196078,0.10196078}\makebox(0,0)[t]{\lineheight{1.25}\smash{\begin{tabular}[t]{c}$T$\end{tabular}}}}%
    \put(0.50596576,0.28948466){\color[rgb]{0.10196078,0.10196078,0.10196078}\makebox(0,0)[t]{\lineheight{1.25}\smash{\begin{tabular}[t]{c}$S$\end{tabular}}}}%
    \put(0.48722988,0.60002875){\color[rgb]{0.10196078,0.10196078,0.10196078}\makebox(0,0)[t]{\lineheight{1.25}\smash{\begin{tabular}[t]{c}$t_b=s_t$\end{tabular}}}}%
  \end{picture}%
\endgroup%

\caption{A triangle $T$ on top of a rectangle $S$ meeting it at the common edge $t_b = s_t$.}
\label{fig:triangle_on_top_rectangle}
\end{figure}
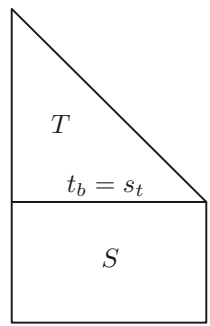
Similarly, for $0<\eta<y$ we can construct $\tilde{D}_v\subset (r_l\cup r_r)\times v$ consisting of two disk components of area $\eta$.

Now consider a polydisk $P = P(x,y)$ and an integer $n\geq 2$. By symmetry, we may reduce ourselves to the case $x\geq y$. Let $S\subset \BR^2$ denote the axis parallel rectangle of width $x$ and height $y$. We divide $S$ into rectangles and triangles as follows: First, we divide $S$ into $n+1$ rectangles of width $x$ and height $y/(n+1)$. Each of these $n+1$ rectangles, we further divide into squares of side length $y/(n+1)$ and a single rectangle of height $y/(n+1)$ whose width is contained in the interval $[y/(n+1), 2y/(n+1))$. Note that since $x\geq y$ by assumption, each of the $n+1$ rectangles from the first division step gives rise to at least $n$ squares. Each of the squares, we further divide into two triangles along one of the diagonals. See Figure \ref{fig:concatenating_ribbon_complexes_polydisk} for an illustration of the resulting decomposition of $S$. Let $S_k$ and $T_j$ denote the rectangles and triangles, respectively, in the decomposition. We obtain an induced decomposition of $S\times_LQ$ into polydisks $S_k\times_L Q$ and balls $T_j\times_L Q$.

\begin{figure}[htpb]
\centering
\def\svgwidth{0.6\textwidth}
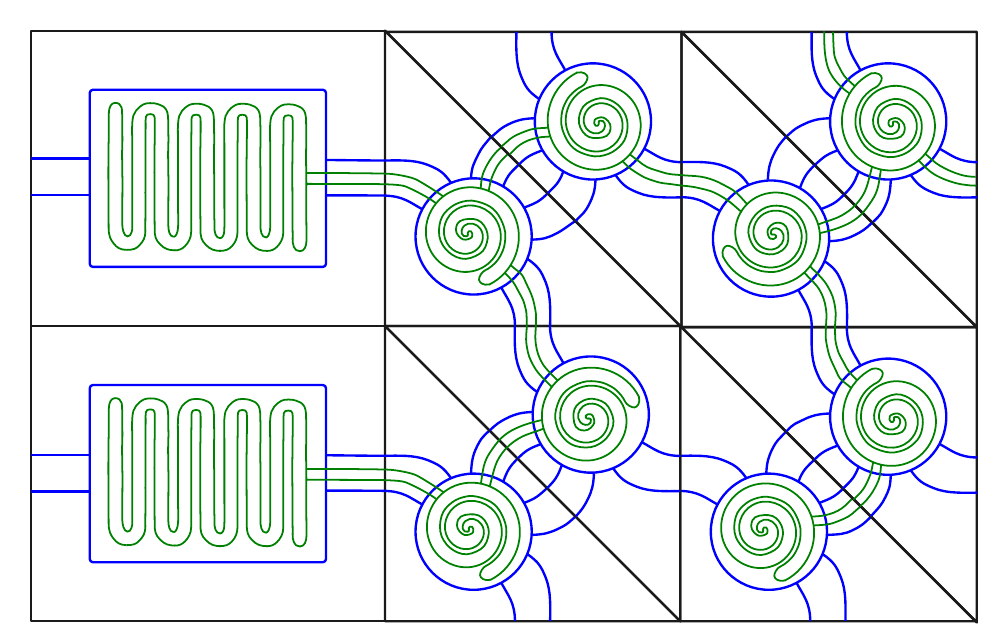
\caption{A decomposition of $S$ into rectangles and triangles and a schematic of the ribbon complexes $L$ (blue) and $K$ (green).}
\label{fig:concatenating_ribbon_complexes_polydisk}
\end{figure}

Let $b>0$ be sufficiently small and $a>0$ such that $\operatorname{vol}(E(a,b)) < \operatorname{vol}(P)$. Moreover, let $0<c<b$ and let $\psi_i:B^2(c)\hookrightarrow \partial_1P\cup \partial_2P$ be a free collection of $n$ symplectic embeddings. We construct an embedding of $E(a,b;c,n)$ into $P$ following the same strategy used in the proof of Proposition \ref{prop:special_embedding_into_ball}.

First, we embed a copy of $B(\beta;\gamma,4)$ for suitable parameters $\beta$ and $\gamma$ into $T_j\times_LQ$ such that the free ends land in the disks $\tilde{D} \subset \partial(T_j\times_L Q)$ introduced in the proof of Proposition \ref{prop:special_embedding_into_ball}. Moreover, we embed a copy of $P(\alpha,\beta;\gamma,2)$ for suitable $\alpha$ into $S_k\times_L Q$ such that the two free ends land in the disks $\tilde{D}_v$ introduced above. Here we use Lemma \ref{lem:embedding_into_lagrangian_product} and its polydisk analogue. Then we connect matching free ribbons at the interfaces between neighbouring cells. A schematic of the resulting ribbon complex $L$ is depicted in Figure \ref{fig:concatenating_ribbon_complexes_polydisk}.

Next, we embed a copy of $K = K(G,A,f,b)$ into $L$ for suitable $f>b$ and a suitable labelled ribbon graph $(G,A)$, see Figure \ref{fig:concatenating_ribbon_complexes_polydisk}. This uses the case $n=1$ of Proposition \ref{prop:special_embedding_into_polydisk} and Lemma \ref{lem:special_embedding_into_ball_dependent_attaching_maps}. Composition with an appropriate embedding $S\times_L Q\overset{s}{\hookrightarrow}P$ yields an embedding of $K$ into $P$ mapping the free ribbon ends into $\partial P$ via a free collection of $n$ embeddings $B^2(c)\hookrightarrow \partial_1P\cup \partial_2P$. We can build $K$ such that the number of free ribbon ends of $K$ which land in $\partial_1P$ (resp. $\partial_2P$) matches the number of embeddings $\psi_i$ with image in $\partial_1P$ (resp. $\partial_2P$). Moreover, we can arrange the parameters of $K$ such that Lemma \ref{lem:special_ribbon_complex_neighbourhood_packing} yields an embedding $E(a,b;c,n)\overset{s}{\hookrightarrow} P$, again mapping the free ribbon ends into $\partial_1P\cup \partial_2P$ via a free collection of maps. By Lemma \ref{lem:free_collections_polydisk} we can assume that these maps precisely agree with the $\psi_i$, after possibly permuting the indices of $\psi_i$.

Finally, observe that if $n>0$ is fixed and $(x,y)$ range over a compact subset $C \subset \BR_{>0}^2$, then the parameters of all polydisks and balls in our division of $S\times_L Q$ are uniformly bounded away from zero. Therefore, it ultimately follows from Lemma \ref{lem:skinny_ellipsoid_embeddings} that the threshold $b_0$ can be chosen uniformly.
\end{proof}

\begin{proof}[Proof of Theorem \ref{thm:ellipsoid_packing_of_ribbon_complex}]
Let us normalized the volume of $K$ to be equal to $1$. In this situation, we have $w(K) = \overline{w}(K)$. Let $\delta >0$ and assume that $w(K) \geq \delta$.

Observe that we can bound the total number of domains of $K$ from above in terms of $\delta$. After possibly removing some ribbons, we can assume that the graph underlying the ribbon complex $K$ is a tree. In this situation, we can in addition bound the total number of ribbons of $K$ from above in terms of $\delta$. After possibly shrinking the widths of the ribbons of $K$, we can assume that, for every domain $X$ of $K$, the collection of all attaching maps of ribbons attached to $X$ is free. Note that we can carry out this simplification in such a way that the new width $w(K)$ can be bounded from below just in terms of $\delta$.

For each domain $X$ of $K$, let $n_X$ denote the number of ribbons attached to $X$. Let us choose $b>0$ sufficiently small such that, for every domain $X$ of $K$, the assertion of Proposition \ref{prop:special_embedding_into_ball} holds for the pair $(X,n_X)$ if $X$ is a ball and the assertion of Proposition \ref{prop:special_embedding_into_polydisk} holds if $X$ is a polydisk. Moreover, we choose $b$ smaller than any width of a ribbon of $K$. Note that this choice of $b$ can be made to depend only on $\delta$. The reason is that the parameters of the balls, polydisks, and ribbons of $K$ are bounded away from zero in terms of $\delta$ and that all $n_X$ are bounded from above in terms of $\delta$. Let $\varepsilon>0$ be arbitrary. Pick a positive number $c$ such that $b-\varepsilon < c < b$. For every domain $X$ of $K$, pick $a_X>0$ such that $\operatorname{vol}(X) - \varepsilon < \operatorname{vol}(E(a_X,b)) < \operatorname{vol}(X)$.

Consider a domain $X$ of $K$. Let $(\psi_j)_{1\leq j\leq n_X}$ denote the collection of attaching maps of all the ribbons attached to $X$. Pick a symplectic embedding $\varphi_X:E(a_X,b;c,n_X)\overset{s}{\hookrightarrow} X$ which maps the free ribbon ends into $\partial X$ via the restricted attaching maps $\psi_j|_{B^2(c)}$. Such a symplectic embedding exists by Proposition \ref{prop:special_embedding_into_ball} if $X$ is a ball and by Proposition \ref{prop:special_embedding_into_polydisk} if $X$ is a polydisk. We connect the complexes $E(a_X,b;c,n_X)$ to a single complex $L$ as follows. Let $R=[0,1]\times B^2(d)$ be a ribbon of $K$. Consider one of its ends $\left\{i \right\}\times B^2(d)$. Let $X$ be the domain of $K$ this end is attached to. The image of one of the free ends of $E(a_X,b;c,n_X)$ is contained in the end $\left\{ i \right\}\times B^2(d)$, which we identify with its image in $\partial X$. The restriction of $\varphi_X$ to the free end $B^2(c)$ of $E(a_X,b;c,n_X)$ is simply given by inclusion $B^2(c)\hookrightarrow \left\{ i \right\}\times B^2(d)$. Let us glue the free end to the end $\left\{ i \right\}\times B^2(c)$ of the ribbon $[0,1]\times B^2(c)\subset R$. We perform this gluing at all the free ends of all the complexes $E(a_X,b;c,n_X)$ and let $L$ denote the resulting complex. By construction, this complex $L$ symplectically embeds into $K$. Moreover, it is a complex of the form considered in Lemma \ref{lem:special_ribbon_complex_neighbourhood_packing}.

Now suppose that $\varphi:K\overset{s}{\hookrightarrow} M$ is a symplectic embedding and let $U$ be an open neighbourhood of $\varphi(K)$. By Lemma \ref{lem:special_ribbon_complex_neighbourhood_packing}, the ellipsoid $E(a,c)$ symplectically embeds into $U$. Letting $\varepsilon$ tend to zero, we see that $p^E_{a/b}(K) = 1$. Note that $a/b$ can be bounded from above just in terms of $\delta$. It therefore follows from Lemma \ref{lem:skinny_ellipsoid_embeddings} that there exists a threshold $a_0$ only depending on $\delta$ such that $p^E_e(K)=1$ for all $e\geq a_0$.
\end{proof}

\section{Tame packings by balls and polydisks}
\label{sec:tame_packings}

In this section, we introduce the notion of a \textit{tame packing by balls and polydisks}. The motivation behind this definition is that the existence of a tame packing by balls and polydisks implies ellipsoid embedding stability. This is an easy corollary of Theorem \ref{thm:ellipsoid_packing_of_ribbon_complex}. We also introduce the stronger notion of a \textit{$\partial$-tame packing by balls and polydisks}. The advantage of this stronger notion is that if a symplectic manifold $X$ can be decomposed into pieces which admit $\partial$-tame packings by balls and polydisks, then so does $X$. Using Theorem \ref{thm:stratification_perturbed_affine_subgraph}, we show that, for every $\partial$-admissible Hamiltonian $H$ on the annulus $\BA$ which is a $C^\infty$ small perturbation of a strictly positive affine Hamiltonian, the subgraph $D(H)$ admits a $\partial$-tame packing by balls and polydisks and therefore satisfies ellipsoid embedding stability; see Corollary \ref{cor:tame_packing_perturbed_affine_subgraph}. We extend this result to Hamiltonians $H$ which are allowed to vanish on one boundary component of $\BA$. This is the content of Theorem \ref{thm:tame_packing_perturbed_linear_subgraph}.\\

Let $M$ be a symplectic $4$-manifold of finite volume. We say that $M$ admits a \textit{tame packing by balls and polydisks} (or simply a \textit{tame packing}) if there exists $\delta>0$ such that, for every $\varepsilon>0$, there exists a connected symplectic ribbon complex $K$, all of whose domains are balls or polydisks, which admits a symplectic embedding into $\operatorname{int}(M)$, and which satisfies $w(K)\geq \delta$ and $\operatorname{vol}(K) \geq \operatorname{vol}(M) - \varepsilon$.

The following is an immediate corollary of Theorem \ref{thm:ellipsoid_packing_of_ribbon_complex}.

\begin{corollary}
\label{cor:from_tame_packing_to_ellipsoid_packing}
Let $M$ be a symplectic $4$-manifold of finite volume. If $M$ admits a tame packing by balls and polydisks, then ellipsoid embedding stability holds for $M$, i.e. $p_a^E(M) = 1$ for all sufficiently large $a$.
\end{corollary}

Let $M$ be a compact symplectic $4$-manifold with boundary and corners. We say that $M$ admits a \textit{$\partial$-tame packing by balls and polydisks} (or simply a \textit{$\partial$-tame packing}) if there exists an open and dense subset $U\subset \partial M$ contained in the smooth locus of the boundary of $M$ with the following property: For every finite collection $\MD$ of pairwise disjoint closed symplectic disks contained $U$, there exists $\delta>0$ such that, for every $\varepsilon>0$, there exists a connected symplectic ribbon complex $K$, all of whose domains are balls or polydisks, which admits a symplectic embedding into $M$ such that the images of the free ribbon ends of $K$ are precisely given by the disks $\MD$, and which satisfies $w(K)\geq \delta$ and $\operatorname{vol}(K)\geq \operatorname{vol}(M) - \varepsilon$.

\begin{lemma}
\label{lem:boundary_tame_packing_stratification}
Let $M$ be a compact connected symplectic $4$-manifold with boundary and corners. Suppose that $M$ admits a stratification $\MT$ such that each top-dimensional stratum of $\MT$ admits a $\partial$-tame packing by balls and polydisks. Then $M$ admits a $\partial$-tame packing by balls and polydisks.
\end{lemma}

\begin{proof}
For each top-dimensional stratum $T$ of $\MT$, pick an open and dense subset $U_T \subset \partial T$ as in the definition of a $\partial$-tame packing. Let $U\subset \partial M$ be the intersection with $\partial M$ of the union of all the sets $U_T$. The set $U$ is open and dense in $\partial M$. Let $\MD \subset U$ be an arbitrary collection of pairwise disjoint symplectic disks. Our goal is to fill the volume of $M$ up to arbitrarily small error by a connected symplectic ribbon complex with free ribbon ends at the disks $\MD$ and with width bounded away from zero. For any pair of distinct top-dimensional strata $T$ and $T'$ of $\MT$ which meet at a stratum $S$ of codimension $1$, pick a closed symplectic disk contained in $S \cap U_T \cap U_{T'}$. Let $\ME$ denote the resulting collection of symplectic disks. For each top-dimensional stratum $T$ of $\MT$, we can fill arbitrarily much of the volume of $T$ by a connected symplectic ribbon complex $K_T$ with free ends at the disks $\partial T \cap (\MD \cup \ME)$ and with width bounded away from zero. Now connect all the ribbon complexes $K_T$ to a single connected ribbon complex $K$ by gluing all pairs of free ribbon ends meeting at disks in $\ME$. The free ends of $K$ are precisely the disks $\MD$. Moreover, we can arrange $K$ to fill the volume of $M$ up to arbitrarily small error while keeping the width $w(K)$ bounded away from zero.
\end{proof}

\begin{example}
\label{ex:boundary_tame_packing_ball}
For $a>0$, let $\Delta(a)\subset \BR^2$ denote the triangle with vertices $(0,0)$, $(a,0)$ and $(0,a)$. Let $Q\coloneqq [0,1]^2$ denote the square. Then the Lagrangian product $\Delta(a)\times_L Q$, whose interior is symplectomorphic to the interior of the ball $B^4(a)$, admits a $\partial$-tame packing. Indeed, recall from the proof of Lemma \ref{lem:embedding_into_lagrangian_product} that the ball $B^4(ra)$ symplectically embeds into $\Delta(a)\times_L Q$ for every $r \in (0,1)$. Moreover, the image of the embedding can be arranged to approximate $\Delta(\sqrt{r}a) \times_L \sqrt{r}Q$, translated into the interior of $\Delta(a)\times_L Q$, arbitrarily well. Let $\MD$ be a collection of symplectic disks in the smooth locus of the boundary of $\Delta(a)\times_L Q$. For every $r\in (0,1)$ sufficiently close to $1$, we can connect the disks $\MD$ to the boundary of the symplectically embedded ball $B^4(ra)$ via a disjoint collection of symplectic ribbons. We can view the result as an embedding of a symplectic ribbon complex $K$ consisting of a single ball and ribbons whose free ends are mapped to the disks $\MD$. As $r$ approaches $1$, the volume of $K$ tends to the volume of $\Delta(a)\times_L Q$, while the width $w(K)$ remains bounded away from $0$, showing that $\Delta(a)\times_L Q$ admits a $\partial$-tame packing.
\end{example}

\begin{example}
\label{ex:boundary_tame_packing_polydisk}
For $a,b>0$, let $\square(a,b)\subset \BR^2$ denote the rectangle with vertices $(0,0)$, $(a,0)$, $(a,b)$ and $(0,b)$. Again, let $Q = [0,1]^2$ denote the square. Then the Lagrangian product $\square(a,b)\times_L Q$, whose interior is symplectomorphic to the interior of the polydisk $P(a,b)$, admits a $\partial$-tame packing. In order to see this, note that we can alternatively view $\square(a,b)\times_L Q$ as the symplectic product $\square(a,1)\times \square(b,1)$. For $r\in (0,1)$, pick a symplectic embedding $B^2(ra) \overset{s}{\hookrightarrow} \square(a,1)$ whose image is a rectangle with slightly rounded corners contained in the interior of $\square(a,1)$. Pick an analogous embedding $B^2(rb)\overset{s}{\hookrightarrow} \square(b,1)$. The product of these two embeddings yields a symplectic embedding of $P(ra,rb)$ into $\square(a,1) \times \square(b,1)$. As in Example \ref{ex:boundary_tame_packing_ball}, we can connect the image of this symplectic embedding to any given collection of symplectic disks in the smooth locus of the boundary of $\square(a,b)\times_L Q$ via disjoint symplectic ribbons. Letting $r$ tend to $1$, we see that $\square(a,b)\times_L Q$ admits a $\partial$-tame packing.
\end{example}

\begin{example}
\label{ex:boundary_tame_packing_rational_ellipsoid}
For $a,b>0$, let $\Delta(a,b)\subset \BR^2$ denote the triangle with vertices $(0,0)$, $(a,0)$ and $(0,b)$. We claim that if $a/b$ is rational, then the Lagrangian product $\Delta(a,b)\times_L Q$, whose interior is symplectomorphic to the interior of the ellipsoid $E(a,b)$, admits a $\partial$-tame packing. First, we argue that $\Delta(a,b)\times_L \BT^2$ has a $\partial$-tame packing, where $\BT^2$ is the $2$-torus obtained from the square $Q$ by identifying opposite sides. Recall that for $a/b$ rational, one can decompose $\Delta(a,b)$ into finitely many triangles $T_i$, each of which can be obtained from a triangle of the form $\Delta(a_i)$ by translations and the action of $\operatorname{SL}(2,\BZ)$ (see e.g. \cite[\S 3.1]{mcd09}). This yields a stratification of $\Delta(a,b)\times_L \BT^2$ with top-dimensional strata given by $T_i\times_L \BT^2$. Each of these strata is symplectomorphic to $\Delta(a_i)\times_L \BT^2$. Recall from Example \ref{ex:boundary_tame_packing_ball} that $\Delta(a_i)\times_L Q$ admits a $\partial$-tame packing. In view of the natural inclusion $\operatorname{int}(\Delta(a_i)\times_L Q) \hookrightarrow \Delta(a_i)\times_L \BT^2$, this implies that $\Delta(a_i)\times_L \BT^2$ and hence $T_i\times_L \BT^2$ has a $\partial$-tame packing. Lemma \ref{lem:boundary_tame_packing_stratification} thus implies that that $\Delta(a,b)\times \BT^2$ admits a $\partial$-tame packing as well.

It is well-known that $\Delta(ra,rb)\times_L\BT^2$ symplectically embeds into $\Delta(a,b)\times_L Q$ for all $r\in (0,1)$. Explicit embeddings can be obtained via a variant of the symplectic embedding construction from \cite[\S 2]{sch03} and \cite[\S 4.1]{sch05} which we reviewed in the proof of Lemma \ref{lem:embedding_into_lagrangian_product}. For $s>0$, let $\BA(s) \coloneqq [0,s]\times \BT$ denote the annulus of area $s$. Note that we can regard the Lagrangian product $\Delta(a,b)\times_L \BT^2$ as a subset of the symplectic product $\BA(a)\times \BA(b)$. Similarly, we can regard the Lagrangian product $\Delta(a,b)\times_L Q$ as a subset of the symplectic product $\square(a,1)\times \square(b,1)$. Consider a family of loops $\ML_a$ in the complement of a line segment $L_a$ in the rectangle $\operatorname{int}(\square(a,1))$ as depicted in Figure \ref{fig:admissible_curve_family_rectangle}. Then there exists a symplectomorphism $\varphi_a : \operatorname{int}(\BA(a)) \rightarrow \operatorname{int}(\square(a,1))\setminus L_a$ which maps circles of the form $\left\{ x \right\}\times \BT$ in $\BA(a)$ to members of $\ML_a$. Let $\varphi_b : \operatorname{int}(\BA(b)) \rightarrow \operatorname{int}(\square(b,1)) \setminus L_b$ be the analogous symplectomorphism arising from a family of loops $\ML_b$ in the rectangle $\operatorname{int}(\square(b,1))$. Now consider the triangle $\Delta(ra,rb)$ and slightly translate it off the axes into the interior of $\Delta(a,b)$. Let $\Delta'(ra,rb)$ denote the resulting triangle. We can arrange the families of loops $\ML_a$ and $\ML_b$ such that the image of $\Delta'(ra,rb)\times_L\BT^2$ under $\varphi_a\times\varphi_b$ approximates $\Delta(ra,rb)\times_L Q$ arbitrarily well. This yields the desired embedding of $\Delta(ra,rb)\times_L \BT^2$ into $\Delta(a,b)\times Q$. Given a collection of symplectic disks $\MD$ in the smooth locus of the boundary of $\Delta(a,b)\times_L Q$, we can connect them to symplectic disks in the boundary of the image of the symplectic embedding of $\Delta(ra,rb)\times_L \BT^2$ via symplectic ribbons. Since we already know that $\Delta(ra,rb)\times_L \BT^2$ admits a $\partial$-tame packing, we can conclude that the same is true for $\Delta(a,b)\times_L Q$.
\end{example}

\begin{figure}[htpb]
\centering
\def\svgwidth{0.4\textwidth}
\begingroup%
  \makeatletter%
  \providecommand\color[2][]{%
    \errmessage{(Inkscape) Color is used for the text in Inkscape, but the package 'color.sty' is not loaded}%
    \renewcommand\color[2][]{}%
  }%
  \providecommand\transparent[1]{%
    \errmessage{(Inkscape) Transparency is used (non-zero) for the text in Inkscape, but the package 'transparent.sty' is not loaded}%
    \renewcommand\transparent[1]{}%
  }%
  \providecommand\rotatebox[2]{#2}%
  \newcommand*\fsize{\dimexpr\f@size pt\relax}%
  \newcommand*\lineheight[1]{\fontsize{\fsize}{#1\fsize}\selectfont}%
  \ifx\svgwidth\undefined%
    \setlength{\unitlength}{303.30708661bp}%
    \ifx\svgscale\undefined%
      \relax%
    \else%
      \setlength{\unitlength}{\unitlength * \real{\svgscale}}%
    \fi%
  \else%
    \setlength{\unitlength}{\svgwidth}%
  \fi%
  \global\let\svgwidth\undefined%
  \global\let\svgscale\undefined%
  \makeatother%
  \begin{picture}(1,0.60747664)%
    \lineheight{1}%
    \setlength\tabcolsep{0pt}%
    \put(0,0){\includegraphics[width=\unitlength,page=1]{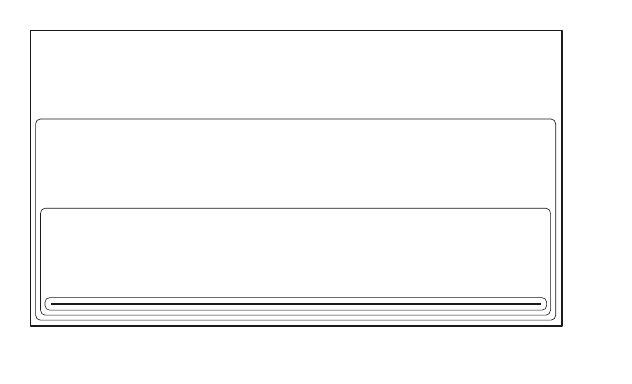}}%
    \put(0.48477071,0.0147833){\color[rgb]{0.10196078,0.10196078,0.10196078}\makebox(0,0)[lt]{\lineheight{1.25}\smash{\begin{tabular}[t]{l}$1$\end{tabular}}}}%
    \put(0.93540806,0.31747553){\color[rgb]{0.10196078,0.10196078,0.10196078}\makebox(0,0)[lt]{\lineheight{1.25}\smash{\begin{tabular}[t]{l}$a$\end{tabular}}}}%
    \put(0.62931707,0.1610261){\color[rgb]{0.10196078,0.10196078,0.10196078}\makebox(0,0)[lt]{\lineheight{1.25}\smash{\begin{tabular}[t]{l}$L_a$\end{tabular}}}}%
  \end{picture}%
\endgroup%

\caption{A special admissible family of loops $\ML_a$ in $\square(a,1)\setminus L_a$.}
\label{fig:admissible_curve_family_rectangle}
\end{figure}
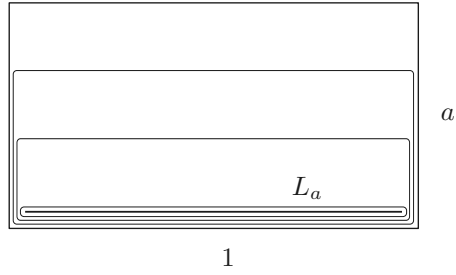

\begin{corollary}
\label{cor:tame_packing_perturbed_affine_subgraph}
In the setting of Theorem \ref{thm:stratification_perturbed_affine_subgraph}, the subgraph $D(H)$ admits a $\partial$-tame packing by balls and polydisks. Therefore, ellipsoid embedding stability holds for $D(H)$.
\end{corollary}

\begin{proof}
Theorem \ref{thm:stratification_perturbed_affine_subgraph} yields a stratification of $D(H)$ into subgraphs of the form $D(H_\alpha|_{[0,1]\times \BA(c)})$ for rational $\alpha$ and $D(C|_{[0,1]\times \BA(c)})$. Cutting the annulus $\BA(c) = [0,c]\times \BT$ along the line segment $[0,c]\times \left\{ 0 \right\}$ turns the former type of stratum into $\Delta(c,\alpha c)\times_L Q$ and the latter type into $\square(c,C)\times_L Q$. By Examples \ref{ex:boundary_tame_packing_rational_ellipsoid} and \ref{ex:boundary_tame_packing_polydisk}, both Lagrangian products admit $\partial$-tame packings. Thus by Lemma \ref{lem:boundary_tame_packing_stratification} the same is true for $D(H)$.
\end{proof}

\begin{theorem}
\label{thm:tame_packing_perturbed_linear_subgraph}
Let $h:[0,1]\times \BA \rightarrow \BR$ be a Hamiltonian of the form $h(t,x,y) = ax$ for a real number $a>0$. Let $H:[0,1]\times \BA \rightarrow \BR$ be a $\partial$-admissible Hamiltonian which vanishes on $[0,1]\times \left\{ 0 \right\}\times \BT$. If $H$ is sufficiently $C^\infty$ close to $h$, then $D(H)$ admits a $\partial$-tame packing by balls and polydisks.
\end{theorem}

\begin{proof}
We divide the annulus $\BA$ into three smaller annuli
\begin{equation*}
\BA_0 \coloneqq [0,\frac{1}{3}]\times \BT \qquad \BA_1 \coloneqq [\frac{1}{3},\frac{2}{3}]\times \BT \qquad \BA_2 \coloneqq [\frac{2}{3},1]\times \BT.
\end{equation*}
Pick a rational number $\alpha$ close to $\frac{a}{2}$. Note that both $H_\alpha$ and $H\# \overline{H}_\alpha$ are $C^\infty$ close to $h/2$. We may therefore chose a Hamiltonian $G:[0,1]\times \BA \rightarrow \BR$ such that
\begin{itemize}
\item $G = H \# \overline{H}_\alpha$ in a neighbourhood of $[0,1]\times \BA_0$
\item $G = H_\alpha$ in a neighbourhood of $[0,1]\times \BA_2$
\item $G$ is $C^\infty$ close to $h/2$.
\end{itemize}
We divide $D(H)$ into two pieces $D(G)$ and $D(G,H)$. By Lemma \ref{lem:boundary_tame_packing_stratification}, it suffices to show that each of these two pieces $D(G)$ and $D(G,H)$ has a $\partial$-tame packing.

By Lemma \ref{lem:symplectomorphisms_of_subgraphs}, $D(G,H)$ is symplectomorphic to $D(\overline{G}\# H)$, so it suffices to find a $\partial$-tame packing for $D(\overline{G} \# H)$. On the annulus $\BA_0$, we have
\begin{equation*}
\overline{G} \# H = H_\alpha \# \overline{H} \# H = H_\alpha.
\end{equation*}
We stratify $D(\overline{G} \# H)$ into two pieces $D(\overline{G} \# H|_{[0,1]\times \BA_0}) = D(H_\alpha|_{[0,1]\times \BA_0})$ and $D(\overline{G} \# H|_{[0,1]\times (\BA_1\cup \BA_2)})$. The former piece admits a $\partial$-tame packing by Example \ref{ex:boundary_tame_packing_rational_ellipsoid}. The latter piece is precisely of the form considered in Theorem \ref{thm:stratification_perturbed_affine_subgraph} and therefore admits a $\partial$-tame packing by Corollary \ref{cor:tame_packing_perturbed_affine_subgraph}. Applying again Lemma \ref{lem:boundary_tame_packing_stratification}, we see that $D(\overline{G} \# H)$ admits a $\partial$-tame packing.

Let us now turn to $D(G)$. Define $F$ by
\begin{equation*}
F\coloneqq \min \left\{ G, \frac{2\alpha}{3} \right\}.
\end{equation*}
We stratify $D(G)$ into $D(F)$ and $\operatorname{clos}(\operatorname{int}D(F,G))$. Note that the latter stratum is symplectomorphic to $D(H_\alpha|_{[0,1]\times \BA_0})$ and hence admits a $\partial$-tame packing. It remains to deal with $D(F)$. On a neighbourhood of $[0,1]\times \BA_2$, the Hamiltonian $F$ is autonomous. This is not necessarily the case outside of this neighbourhood, but it is possible to construct a Hamiltonian $\tilde{F}$ which is $C^\infty$ close to $F$, agrees with $F$ on a neighbourhood of $[0,1]\times \BA_2$, satisfies $\tilde{F}_t = \tilde{F}_0$ for all $t\in [0,1]$ sufficiently close to the endpoints of $[0,1]$, and has the property that $D(F)$ and $D(\tilde{F})$ are symplectomorphic. Since $\tilde{F}_t = \tilde{F}_0$ for $t$ close to $0$ and $1$, the Hamiltonian $\tilde{F}$ descends to a $1$-periodic Hamiltonian defined on $\BT\times \BA$. Let us write $\tilde{F}|_{\BT\times \BA}$ to refer to this $1$-periodic Hamiltonian.

We claim that $D(\tilde{F}|_{\BT \times \BA})$ admits a $\partial$-tame packing. Indeed, consider the symplectomorphism
\begin{equation*}
\sigma : (\BR \times \BT)^2  \rightarrow (\BR\times \BT)^2 \qquad (s,t,x,y) \mapsto (x,y,s,t)
\end{equation*}
swapping the two factors. The image of $D(\tilde{F}|_{\BT\times \BA})$ under $\sigma$ is of the form $D(E)$ for a $\partial$-admissible Hamiltonian $E:\BT\times \BA(2\alpha/3) \rightarrow \BR$ which is a $C^\infty$ small perturbation of
\begin{equation*}
e : \BT\times \BA(2\alpha/3) \rightarrow \BR \qquad e(t,x,y) = 1 - \frac{x}{2\alpha}.
\end{equation*}
By Corollary \ref{cor:tame_packing_perturbed_affine_subgraph}, $D(E|_{[0,1]\times \BA(2\alpha/3)})$ admits a $\partial$-tame packing. In view of the natural inclusion
\begin{equation*}
\operatorname{int}(D(E|_{[0,1]\times \BA(2\alpha/3)})) \subset D(E),
\end{equation*}
the same is true for $D(E)$ and hence for $D(\tilde{F}|_{\BT \times \BA})$.

It remains to deduce that $D(\tilde{F})$ admits a $\partial$-tame packing. This follows from an argument similar to the one explained in Example \ref{ex:boundary_tame_packing_rational_ellipsoid}. Set $r \coloneqq \frac{2\alpha}{3}$. Then $D(\tilde{F})$ is a subset of the symplectic product $\BA(r) \times \BA$ and $D(\tilde{F}|_{[0,1]\times \BA})$ is a subset of $\square(r,1)\times \BA$. As in Example \ref{ex:boundary_tame_packing_rational_ellipsoid}, we obtain a symplectomorphism $\varphi$ between $\operatorname{int}(\BA(r))$ and the complement of a line segment in $\square(r,1)$, see Figure \ref{fig:admissible_curve_family_rectangle}. The image of $\operatorname{int}(D(\tilde{F}))$ under $\varphi\times \operatorname{id}_\BA$ is an approximation of $D(\tilde{F}|_{[0,1]\times \BA})$. Let $Z \subset \operatorname{int}(D(\tilde{F}))$ be a domain conformally symplectomorphic to $D(\tilde{F})$ obtained by slightly pushing the boundary of $D(\tilde{F})$ inwards. Given a collection of disks in a suitable open and dense subset of the boundary of $D(\tilde{F}|_{[0,1]\times \BA})$, we can connect them to the image of $Z$ under $\varphi\times \operatorname{id}_\BA$ via symplectic ribbons. Since we already know that $Z$ admits a $\partial$-tame packing, we can conclude that the same is true for $D(\tilde{F}|_{[0,1]\times \BA})$.
\end{proof}

\section{Symplectic frusta}
\label{sec:symplectic_frusta}

In view of Lemma \ref{lem:boundary_tame_packing_stratification} and Corollary \ref{cor:from_tame_packing_to_ellipsoid_packing}, in order to show that ellipsoid embedding stability holds for a symplectic $4$-manifold $M$, we need to decompose $M$ into pieces which admit $\partial$-tame packings by balls and polydisks. We introduce the notion of a \textit{perturbed symplectic frustum with cuts}, which will be technically convenient for this purpose. The main  result of this section, Theorem \ref{thm:packing_perturbed_frusta_with_cuts}, states that each perturbed symplectic frustum with cuts admits a $\partial$-tame packing by balls and polydisks. This is a relatively straightforward consequence of Corollary \ref{cor:tame_packing_perturbed_affine_subgraph} and Theorem \ref{thm:tame_packing_perturbed_linear_subgraph}.\\

Let $a,b,c>0$ be positive real numbers and suppose that $b<c$. Define the Hamiltonian
\begin{equation*}
H_{a,b,c} : [0,1]\times B^2(c) \rightarrow \BR \qquad H_{a,b,c}(t,p) \coloneqq a \cdot \min \left\{ 1, \frac{c-\pi|p|^2}{c-b} \right\}.
\end{equation*}
The \textit{symplectic frustum} $F(a,b,c)$ is defined to be
\begin{equation*}
F(a,b,c) \coloneqq D(H_{a,b,c}) \subset \BR\times [0,1] \times B^2(c) \subset \BR^4.
\end{equation*}

Any symplectic frustum $F = F(a,b,c)$ is a manifold with boundary and corners. We abbreviate the five top-dimensional strata of the boundary $\partial F$ by
\begin{eqnarray*}
\partial_0 F &\coloneqq & \partial F \cap (\BR\times \left\{ 0 \right\} \times \BR^2)\\
\partial_1 F &\coloneqq &\partial F \cap (\BR\times \left\{ 1 \right\} \times \BR^2)\\
\partial_+ F &\coloneqq &\partial F \cap (\left\{ a \right\}\times \BR \times \BR^2)\\
\partial_- F &\coloneqq &\partial F \cap (\left\{ 0 \right\}\times \BR \times \BR^2)\\
\partial_\bullet F &\coloneqq &\operatorname{clos} (\partial F \cap (0,a)\times (0,1) \times \BR^2).
\end{eqnarray*}
See Figure \ref{fig:frustum} for an illustration. The $2$-dimensional strata of $\partial F$ are either symplectic, meaning that the ambient symplectic form on $\BR^4$ restricts to an area form, or Lagrangian, meaning that the symplectic form vanishes. The symplectic strata are given by $\partial_j F \cap \partial_\pm F$ and $\partial_j F \cap \partial_\bullet F$ for $j \in \left\{ 0,1 \right\}$. The Lagrangian strata are $\partial_\pm F \cap \partial_\bullet F$.

\begin{figure}[htpb]
\centering
\def\svgwidth{\textwidth}
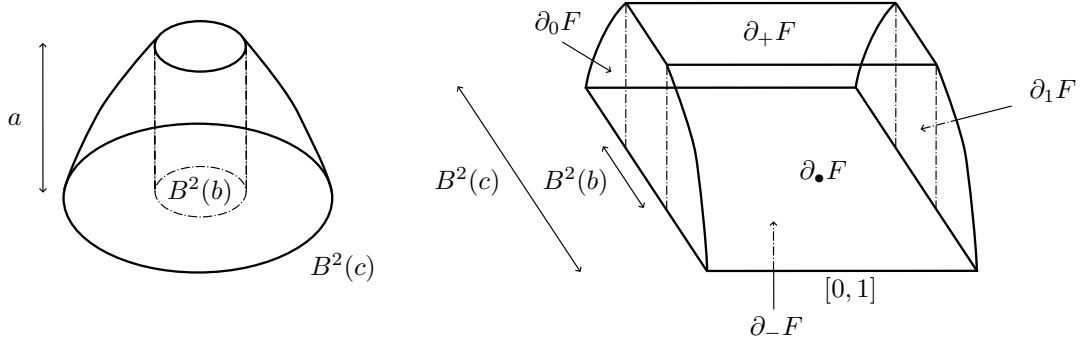
\caption{On the left the intersection $F \cap (\BR \times\left\{ 0 \right\}\times B^2(c))$. On the right a schematic of the boundary strata of $F$.}
\label{fig:frustum}
\end{figure}

A domain $X\subset \BR^4$ is called a \textit{perturbed symplectic frustum} if there exist a symplectic frustum $F\subset \BR^4$ and a diffeomorphism $\varphi$ of $\BR^4$ such that the following properties are satisfied:
\begin{itemize}
\item $\varphi(F) = X$
\item All $2$-dimensional boundary strata of $X$ are either symplectic or Lagrangian.
\item $\varphi$ maps symplectic $2$-dimensional boundary strata of $\partial F$ to symplectic $2$-dimensional boundary strata of $\partial X$. It preserves the orientations of these strata induced by the symplectic form, but not necessarily the symplectic form itself.
\item $\varphi$ maps Lagrangian $2$-dimensional boundary strata of $\partial F$ to Lagrangian $2$-dimensional boundary strata of $\partial X$.
\end{itemize}

If $X = \varphi(F)$ is a perturbed symplectic frustum, we abbreviate its top-dimensional boundary strata by $\partial_*X \coloneqq \varphi(\partial_* F)$ for $*\in \left\{ 0,1,+,-,\bullet \right\}$. We endow the space of perturbed symplectic frusta with the $C^\infty$ topology. A sequence of perturbed symplectic frusta $(X_k)_k$ converges to a perturbed symplectic frustum $X$ with respect to this topology if and only if there exists a sequence of diffeomorphisms $(\varphi_k)_k$ of $\BR^4$ converging to the identity in the $C^\infty$ topology such that $X_k = \varphi_k(X)$ for all sufficiently large $k$.

\begin{theorem}
\label{thm:packing_perturbed_frusta}
Let $F$ be a symplectic frustum. Then each perturbed symplectic frustum $X$ which is sufficiently $C^\infty$ close to $F$ admits a $\partial$-tame packing by balls and polydisks.
\end{theorem}

\begin{lemma}
\label{lem:bringing_perturbed_frusta_into_good_position}
Let $F = F(a,b,c)$ be a symplectic frustum and let $X$ be a perturbed symplectic frustum which is $C^\infty$ close to $F$. Then there exist positive real numbers $a',b',c'$ close to $a,b,c$, respectively, and Hamiltonians $G,H:[0,1]\times B^2(c')\rightarrow \BR$ such that
\begin{itemize}
\item both $G$ and $H$ vanish on $[0,1]\times \partial B^2(c')$.
\item $G$ is $C^\infty$ close to $0$.
\item $H$ agrees with the Hamiltonian $H_{a',b',c'}$ on the set $[0,1]\times B^2(b')$ and is $C^\infty$ close to $H_{a',b',c'}$ on the complement of this set.
\item $X$ is symplectomorphic to $D(G,H)$.
\end{itemize}
\end{lemma}

\begin{proof}
We prove the lemma in two steps.

\emph{Step 1:} We show that it suffices to prove Lemma \ref{lem:bringing_perturbed_frusta_into_good_position} under the additional assumption
\begin{equation}
\label{eq:bringing_perturbed_frusta_into_good_position_proof_a}
\partial_0 X \subset \BR \times \left\{ 0 \right\} \times \BR^2 \quad \text{and} \quad \partial_1 X \subset \BR \times \left\{ 1 \right\} \times \BR^2.
\end{equation}
Define the submanifold $V_1\coloneqq \BR \times \left\{ 0 \right\} \times \BR^2\subset \BR^4$. Let $V_0 \subset \BR^4$ be a $C^\infty$ small perturbation of $V_1$ which is supported in the region $(-1,a+1) \times \left\{ 0 \right\}\times B^2(c+1)$ and contains $\partial_0 X$. Let $(V_t)_{t\in [0,1]}$ be a smooth family of $C^\infty$ small perturbations of $V_1$ supported in this region and connecting $V_0$ to $V_1$. We define a family of diffeomorphisms $\alpha_t : V_0 \rightarrow V_t$ as follows. On the set $(-\infty,-1)\times \left\{ 0 \right\}\times \BR^2$, we define $\alpha_t$ to be the identity. Outside of this set, $\alpha_t$ is uniquely determined by the requirements that it maps leaves of the characteristic foliation on $V_0$ to leaves of the characteristic foliation on $V_t$ and that it preserves the coordinate of the first $\BR$ factor of $\BR^4$, i.e. that it is of the form $\alpha_t(s,t,p) = (s,*,*)$. Note that we have $\alpha_t^*\omega = \alpha_0^*\omega$ for all $t\in [0,1]$. This identity holds trivially in the region $(-\infty,-1)\times \left\{ 0 \right\}\times \BR^2$. On the complement of this region, it follows easily from the fact that $\alpha_t$ respects the characteristic foliations on $V_0$ and $V_t$. We define the vector field $X_t$ along $V_t$ by $X_t \coloneqq \partial_t\alpha_t \circ \alpha_t^{-1}$. Clearly, the family $(V_t)_{t\in[0,1]}$ of $C^\infty$ small perturbations of $V_1$ can be chosen such that the family of vector fields $(X_t)_{t\in [0,1]}$ is $C^\infty$ small. We will assume in the following that this is the case. It follows from the identity $\alpha_t^*\omega = \alpha_0^*\omega$ and Cartan's magic formula that $\iota_{X_t}\omega$ defines a closed $1$-form on $V_t$. It is therefore possible to find a $C^\infty$ small Hamiltonian $E:[0,1]\times \BR^4\rightarrow \BR$ with compact support contained in $[0,1]\times\BR\times (-1/2,1/2)\times \BR^2$ such that the restriction of the Hamiltonian vector field $X_{E_t}$ to $V_t \cap (-1,a+1)\times \BR\times \BR^2$ agrees with the restriction of $X_t$ to this set. Note that this implies that $\varphi_E^t|_{\partial_0 X} = \alpha_t|_{\partial_0X}$. After replacing $X$ by the symplectomorphic domain $\varphi_E^1(X)$, we can therefore assume that $\partial_0 X \subset \BR\times \left\{ 0 \right\}\times \BR^2$. Note that this replacement does not affect the boundary stratum $\partial_1 X$. An analogous construction therefore allows us to further reduce to the case $\partial_1 X\subset \BR\times \left\{ 1 \right\}\times \BR^2$.

\emph{Step 2:} Let us now assume that $X$ satisfies \eqref{eq:bringing_perturbed_frusta_into_good_position_proof_a}. If $X$ is sufficiently $C^\infty$ close to $F$, we may choose compactly supported and $C^\infty$ small Hamiltonians $E, E_\pm : [0,1]\times \BR^2\rightarrow \BR$ satisfying the following properties:
\begin{itemize}
\item $\partial_- X$ is contained in the graph of $E_-$.
\item $\partial_+ X$ is contained in the graph of $a+ E_+$.
\item $E$ agrees with $E_-$ on the set $[0,1] \times (\BR^2 \setminus B^2(b + 3(c-b)/5))$.
\item $E$ agrees with $E_+$ on the set $[0,1] \times B^2(b + 2(c-b)/5)$.
\end{itemize}
By Lemma \ref{lem:symplectomorphisms_of_subgraphs}, the domain $D(E_-,a+E_+)$ is symplectomorphic to $D(\overline{E}\# E_-,\overline{E}\# (a+E_+))$. Since $E$ is $C^\infty$ small, the proof of Lemma \ref{lem:symplectomorphisms_of_subgraphs} in fact yields a symplectomorphism $\psi$ of $\BR^4$ which is $C^\infty$ close to the identity and maps $D(E_-,a+E_+)$ to $D(\overline{E}\# E_-,\overline{E}\# (a+E_+))$. If $X$ is sufficiently $C^\infty$ close to $F$, then $\overline{E}\# E_-$ vanishes on the set $[0,1]\times (\BR^2 \setminus B^2(b + 4(c-b)/5))$ and $\overline{E} \# (a + E_+)$ is equal to $a$ on the set $[0,1]\times B^2(b + (c-b)/5)$. This implies that the Lagrangian annulus $\psi(\partial_-X\cap\partial_\bullet X)$ is of the form $\left\{ 0 \right\}\times [0,1]\times C$ where $C\subset \BR^2$ is a circle $C^\infty$ close to $\partial B^2(c)$. Moreover, it implies that the solid cylinder $\psi(\partial_+ X)$ is of the form $\left\{ a \right\} \times [0,1] \times B$ where $B\subset \BR^2$ is a disk $C^\infty$ close to $B^2(b)$. We may pick an area preserving diffeomorphism $\alpha$ of $\BR^2$ which is $C^\infty$ close to the identity and satisfies $\alpha(C) = \partial B^2(c')$ and $\alpha(B) = \partial B^2(b')$ with $b', c'$ close to $b,c$, respectively. The resulting domain $(\operatorname{id}_{\BR^2} \times \alpha) \circ \psi(X)$ is then of the form $D(G,H)$ where $G,H:[0,1]\times B^2(c')\rightarrow \BR$ satisfy all the properties listed in Lemma \ref{lem:bringing_perturbed_frusta_into_good_position}.
\end{proof}

\begin{proof}[Proof of Theorem \ref{thm:packing_perturbed_frusta}]
Fix a symplectic frustum $F=F(a,b,c)$ and let $X$ be a perturbed symplectic frustum $C^\infty$ close to $F$. Let the positive numbers $a',b',c'$ and the Hamiltonians $G,H$ satisfy all assertions in Lemma \ref{lem:bringing_perturbed_frusta_into_good_position}. Since $X$ is symplectomorphic to $D(G,H)$, it suffices to construct a $\partial$-tame packing of the latter domain.

Pick a number $d>0$ close to $a'/2$ such that $d/c'$ is not an integer. Define the Hamitonian
\begin{equation*}
E: [0,1]\times B^2(c') \rightarrow \BR \qquad E(t,p) \coloneqq d(1-\pi |p|^2/c')
\end{equation*}
We divide $D(G,H)$ into $D(G,E)$ and $D(E,H)$. By Lemma \ref{lem:boundary_tame_packing_stratification}, it suffices to show that both pieces admit $\partial$-tame packings.

The piece $D(G,E)$ is symplectomorphic to $D(\overline{G}\#E)$. Note that $\overline{G}\# E$ is a $C^\infty$ small perturbation of $E$. The time $1$-map $\varphi_E^1$ is a rotation by $2\pi d/c'$. Since $d/c'$ is not an integer, the center of the disk $B^2(c')$ is a non-degenerate fixed point. Therefore, the time-$1$-map $\varphi_{\overline{G}\# E}^1$ has a non-degenerate fixed point $p$ near the center. After conjugation by a area preserving diffeomorphism of $B^2(c')$ which is $C^\infty$-close to the identity, we can assume that $p = 0$. Moreover, it is not hard to construct a Hamiltonian $\tilde{E}$ which vanishes on $[0,1]\times \partial B^2(c')$, is $C^\infty$ close to $E$, has the property that the center is the unique maximizer of $\tilde{E}_t$ for all $t\in [0,1]$, and such that $D(\tilde{E})$ is symplectomorphic to $D(\overline{G}\# E)$. Consider the area preserving map
\begin{equation*}
\psi:\BA(c') \rightarrow B^2(c') \qquad (x,y) \mapsto \sqrt{\frac{x}{\pi}} e^{2\pi i \theta}.
\end{equation*}
Note that the pullback $\psi^* \tilde{E}$ is a Hamiltonian as in Theorem \ref{thm:tame_packing_perturbed_linear_subgraph} and therefore admits a $\partial$-tame packing. Since every embedding of a symplectic ribbon complex into $D(\psi^*\tilde{E})$ induces an embedding into $D(\tilde{E})$ via $\operatorname{id}_{\BR^2}\times \psi$, we deduce that $D(\tilde{E})$ admits a $\partial$-tame packing as well.

Next, consider the piece $D(E,H)$, which is symplectomorphic to $D(\overline{E}\# H)$. Again, let us consider the pullback $\psi^* (\overline{E}\# H)$ to a Hamiltonian on $\BA(a')$. It suffices to construct a $\partial$-tame packing of $D(\psi^*(\overline{E} \# H))$. Let us divide $\BA(c')$ into two smaller annuli
\begin{equation*}
\BA_1 \coloneqq [0,b']\times \BT \quad \text{and} \quad \BA_2 \coloneqq [b',c']\times \BT.
\end{equation*}
We stratify $D(\psi^*(\overline{E} \# H))$ into two pieces $D(\psi^*(\overline{E}\# H)|_{[0,1]\times \BA_1})$ and $D(\psi^*(\overline{E}\# H)_{[0,1]\times \BA_2})$. On the annulus $\BA_1$, the Hamiltonian $\psi^*(\overline{E}\#H)$ is strictly positive and affine. Therefore, the first stratum admits a $\partial$-tame packing as a special case of Corollary \ref{cor:tame_packing_perturbed_affine_subgraph}. On the annulus $\BA_2$, the Hamiltonian $\psi^*(\overline{E}\# H)$ is of the form considered in Theorem \ref{thm:tame_packing_perturbed_linear_subgraph}. Thus, the second stratum also admits a $\partial$-tame packing. We conclude that the same is true for $D(E,H)$.
\end{proof}

Let $X\subset \BR^4$ be a perturbed symplectic frustum. In the following, we introduce a cut operation which produces a new domain $\tilde{X}\subset \BR^4$. Let $\varepsilon>0$ and $a>0$ be positive real numbers. Suppose that $\varphi$ is a symplectomorphism of $\BR^4$ such that
\begin{equation}
\label{eq:cut_special_position}
\varphi(X) \enspace \cap \enspace \BR^2 \times (-\varepsilon,+\infty) \times (-\varepsilon,\varepsilon) = \left\{ 0\leq s \leq \min\left\{ a,a-ax \right\}, 0 \leq t \leq 1, -\varepsilon<x\leq 1, -\varepsilon<y<\varepsilon \right\}.
\end{equation}
Let $\psi: \BR^2\setminus ([0,+\infty)\times \left\{ 0 \right\}) \rightarrow \BR^2$ be an area preserving embedding. Let us assume that the restriction of $\psi$ to each of the open quadrants $\left\{ x>0, \pm y>0 \right\}$ extends to a smooth area preserving embedding of the closed quadrant $\psi_\pm : \left\{ x\geq 0, \pm y\geq 0 \right\} \rightarrow \BR^2$. Moreover, assume that the boundary of the image of $\psi$ is piecewise smooth; see Figure \ref{fig:cutting_map_examples}.

\begin{figure}[htpb]
\centering
\def\svgwidth{0.6\textwidth}
\begingroup%
  \makeatletter%
  \providecommand\color[2][]{%
    \errmessage{(Inkscape) Color is used for the text in Inkscape, but the package 'color.sty' is not loaded}%
    \renewcommand\color[2][]{}%
  }%
  \providecommand\transparent[1]{%
    \errmessage{(Inkscape) Transparency is used (non-zero) for the text in Inkscape, but the package 'transparent.sty' is not loaded}%
    \renewcommand\transparent[1]{}%
  }%
  \providecommand\rotatebox[2]{#2}%
  \newcommand*\fsize{\dimexpr\f@size pt\relax}%
  \newcommand*\lineheight[1]{\fontsize{\fsize}{#1\fsize}\selectfont}%
  \ifx\svgwidth\undefined%
    \setlength{\unitlength}{685.98425197bp}%
    \ifx\svgscale\undefined%
      \relax%
    \else%
      \setlength{\unitlength}{\unitlength * \real{\svgscale}}%
    \fi%
  \else%
    \setlength{\unitlength}{\svgwidth}%
  \fi%
  \global\let\svgwidth\undefined%
  \global\let\svgscale\undefined%
  \makeatother%
  \begin{picture}(1,0.81818182)%
    \lineheight{1}%
    \setlength\tabcolsep{0pt}%
    \put(0,0){\includegraphics[width=\unitlength,page=1]{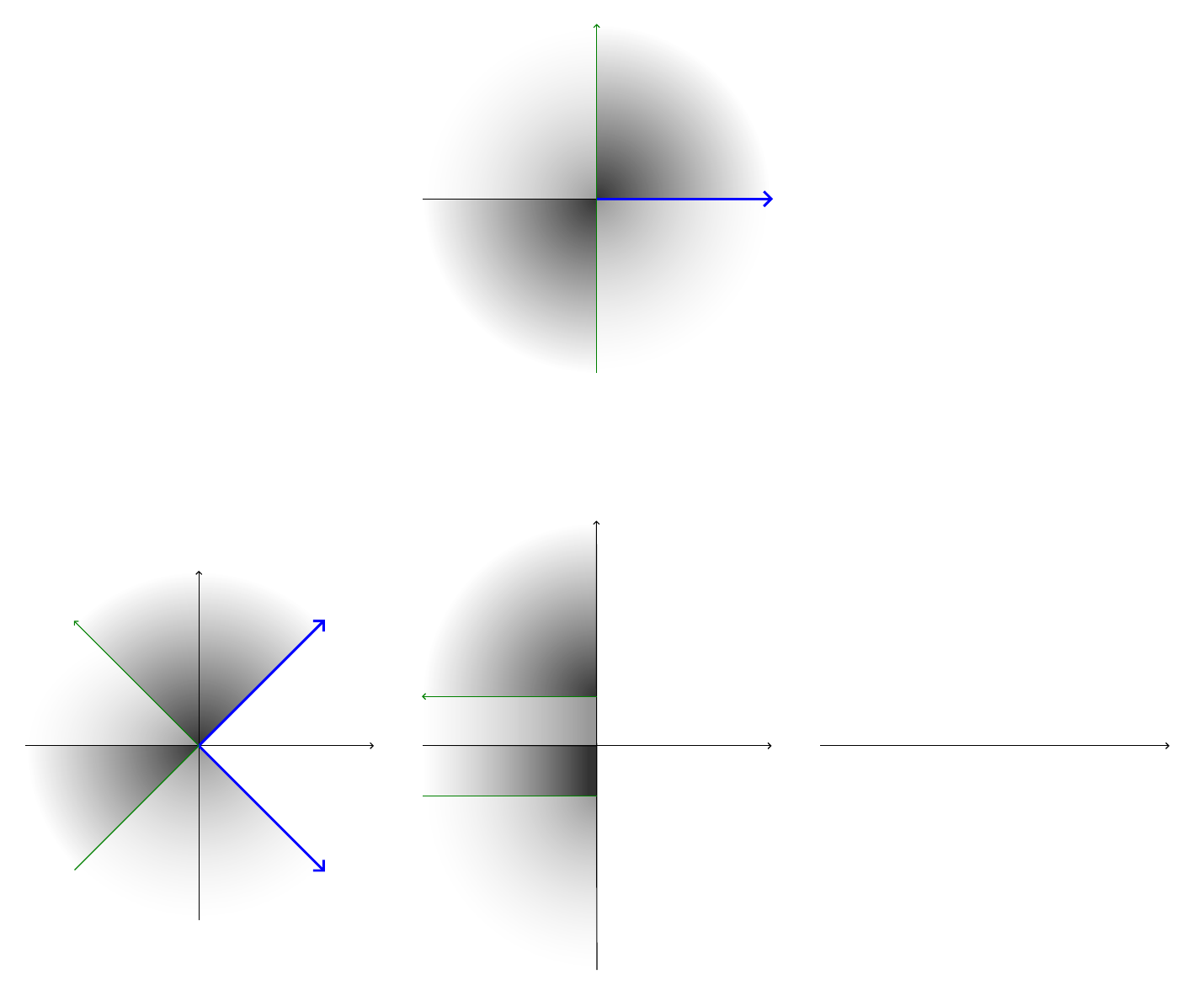}}%
    \put(0.23140497,0.47520662){\color[rgb]{0,0,0}\makebox(0,0)[lt]{\lineheight{1.25}\smash{\begin{tabular}[t]{l}$\psi$\end{tabular}}}}%
    \put(0.4586777,0.44214877){\color[rgb]{0,0,0}\makebox(0,0)[lt]{\lineheight{1.25}\smash{\begin{tabular}[t]{l}$\psi$\end{tabular}}}}%
    \put(0.74380167,0.47520662){\color[rgb]{0,0,0}\makebox(0,0)[lt]{\lineheight{1.25}\smash{\begin{tabular}[t]{l}$\psi$\end{tabular}}}}%
    \put(0,0){\includegraphics[width=\unitlength,page=2]{cutting_map_examples.pdf}}%
    \put(0.36025394,0.37307509){\color[rgb]{0,0,0}\makebox(0,0)[lt]{\lineheight{1.25}\smash{\begin{tabular}[t]{l}(B)\end{tabular}}}}%
    \put(0.69277078,0.3739163){\color[rgb]{0,0,0}\makebox(0,0)[lt]{\lineheight{1.25}\smash{\begin{tabular}[t]{l}(C)\end{tabular}}}}%
    \put(0,0){\includegraphics[width=\unitlength,page=3]{cutting_map_examples.pdf}}%
    \put(0.03199984,0.37269889){\color[rgb]{0,0,0}\makebox(0,0)[lt]{\lineheight{1.25}\smash{\begin{tabular}[t]{l}(A)\end{tabular}}}}%
  \end{picture}%
\endgroup%

\caption{An illustration of different possibilities for the map $\psi$.}
\label{fig:cutting_map_examples}
\end{figure}

The new domain $\tilde{X}$ obtained from $X$ by cutting is defined by
\begin{equation*}
\tilde{X} \coloneqq \operatorname{clos}((\operatorname{id}_{\BR^2}\times\psi) (\varphi(X)\setminus \BR^2\times [0,+\infty)\times \left\{ 0 \right\})),
\end{equation*}
see Figure \ref{fig:frustum_with_cut_examples}.

\begin{figure}[htpb]
\centering
\def\svgwidth{0.8\textwidth}
\begingroup%
  \makeatletter%
  \providecommand\color[2][]{%
    \errmessage{(Inkscape) Color is used for the text in Inkscape, but the package 'color.sty' is not loaded}%
    \renewcommand\color[2][]{}%
  }%
  \providecommand\transparent[1]{%
    \errmessage{(Inkscape) Transparency is used (non-zero) for the text in Inkscape, but the package 'transparent.sty' is not loaded}%
    \renewcommand\transparent[1]{}%
  }%
  \providecommand\rotatebox[2]{#2}%
  \newcommand*\fsize{\dimexpr\f@size pt\relax}%
  \newcommand*\lineheight[1]{\fontsize{\fsize}{#1\fsize}\selectfont}%
  \ifx\svgwidth\undefined%
    \setlength{\unitlength}{977.95275591bp}%
    \ifx\svgscale\undefined%
      \relax%
    \else%
      \setlength{\unitlength}{\unitlength * \real{\svgscale}}%
    \fi%
  \else%
    \setlength{\unitlength}{\svgwidth}%
  \fi%
  \global\let\svgwidth\undefined%
  \global\let\svgscale\undefined%
  \makeatother%
  \begin{picture}(1,0.75072464)%
    \lineheight{1}%
    \setlength\tabcolsep{0pt}%
    \put(0,0){\includegraphics[width=\unitlength,page=1]{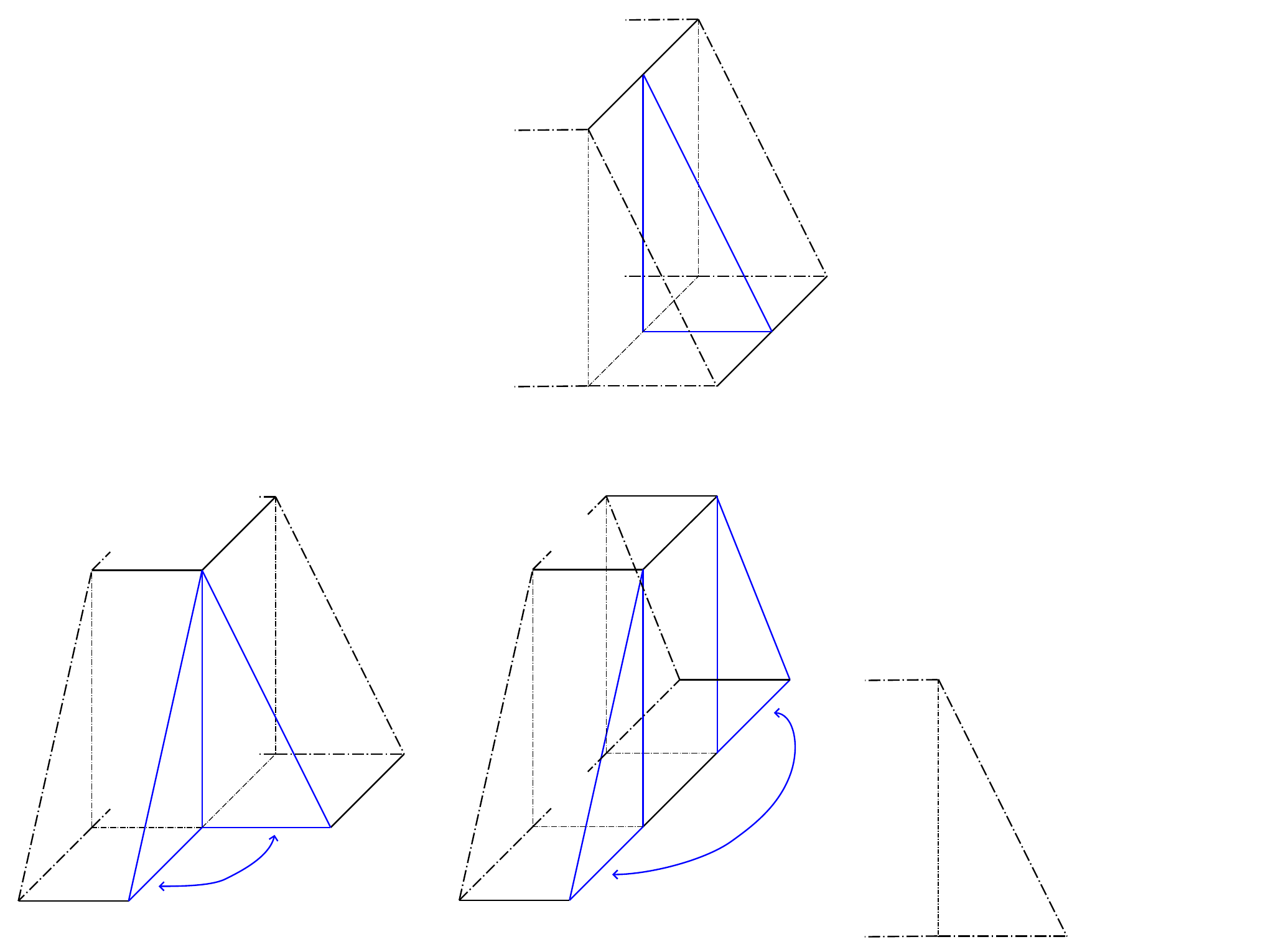}}%
    \put(0.02898552,0.342029){\color[rgb]{0,0,0}\makebox(0,0)[lt]{\lineheight{1.25}\smash{\begin{tabular}[t]{l}(A)\end{tabular}}}}%
    \put(0.37681158,0.342029){\color[rgb]{0,0,0}\makebox(0,0)[lt]{\lineheight{1.25}\smash{\begin{tabular}[t]{l}(B)\end{tabular}}}}%
    \put(0,0){\includegraphics[width=\unitlength,page=2]{frustum_with_cut_examples.pdf}}%
    \put(0.69565215,0.342029){\color[rgb]{0,0,0}\makebox(0,0)[lt]{\lineheight{1.25}\smash{\begin{tabular}[t]{l}(C)\end{tabular}}}}%
    \put(0,0){\includegraphics[width=\unitlength,page=3]{frustum_with_cut_examples.pdf}}%
  \end{picture}%
\endgroup%

\caption{For the different choices of $\psi$ illustrated in Figure \ref{fig:cutting_map_examples}, this figure shows the intersection of $\tilde{X}$ with $\BR \times\left\{ t \right\}\times \BR^2$ near the cut region for $0 \leq t \leq 1$.}
\label{fig:frustum_with_cut_examples}
\end{figure}

We call a domain $X\subset \BR^4$ a \textit{perturbed symplectic frustum with cuts} if, up to ambient symplectomorphism, it can be obtained from a perturbed symplectic frustum by finitely many disjoint cut operations.

\begin{theorem}
\label{thm:packing_perturbed_frusta_with_cuts}
Let $F$ be a symplectic frustum. Then each perturbed symplectic frustum with cuts $X$ obtained by cutting a perturbed symplectic frustum sufficiently $C^\infty$ close to $F$ admits a $\partial$-tame packing by balls and polydisks.
\end{theorem}

\begin{proof}
Let $X$ be a perturbed symplectic frustum $C^\infty$ close to $F$ and let $\tilde{X}$ be a domain obtained from $X$ by finitely many cuts. By Theorem \ref{thm:packing_perturbed_frusta}, the domain $X$ admits a $\partial$-tame packing. In order to deduce that the same is true for $\tilde{X}$, we follow the same strategy that we have already used repeatedly in Section \ref{sec:tame_packings}. Let $Z \subset \operatorname{int}(X)$ be a domain conformally symplectomorphic to $X$ obtained by slightly pushing the boundary of $X$ inwards. Our goal is to produce an explicit controlled symplectic embedding of $Z$ into $\operatorname{int}(\tilde{X})$ which has the property that symplectic disks in the smooth locus of the boundary of $\tilde{X}$ can be connected to the boundary of the image of the symplectic embedding of $Z$ via symplectic ribbons. Once we have this, the fact that $X$ and hence $Z$ admits a $\partial$-tame packing implies that $\tilde{X}$ also admits a $\partial$-tame packing.

For simplicity, let us consider the case that $\tilde{X}$ is obtained from $X$ by a single cut. The general case can be treated by iterating the construction explained below. Let $\varepsilon, a, \varphi, \psi$ be the parameters involved in the cut which transforms $X$ into $\tilde{X}$. In order to simplify notation, let us replace $X$ by the symplectomorphic domain $\varphi(X)$.

We carry out a variant of Schlenk's embedding construction \cite[\S 2]{sch03}, \cite[\S 4.1]{sch05} which we have used several times in Sections \ref{sec:symplectic_ribbons} and \ref{sec:tame_packings}. Consider the strip $(-\varepsilon,+\infty)\times (-\varepsilon,\varepsilon)$. Then there exists a symplectomorphism
\begin{equation*}
\alpha : (-\varepsilon,+\infty)\times (-\varepsilon,\varepsilon) \rightarrow (-\varepsilon,+\infty)\times (-\varepsilon,\varepsilon)\setminus [0,+\infty) \times \left\{ 0 \right\}
\end{equation*}
which maps maps vertical line segments of the form $\left\{ x \right\}\times (-\varepsilon,\varepsilon)$ to members of a family of line segments $\ML$ as depicted in Figure \ref{fig:line_segments_strip}.

\begin{figure}[htpb]
\centering
\def\svgwidth{0.8\textwidth}
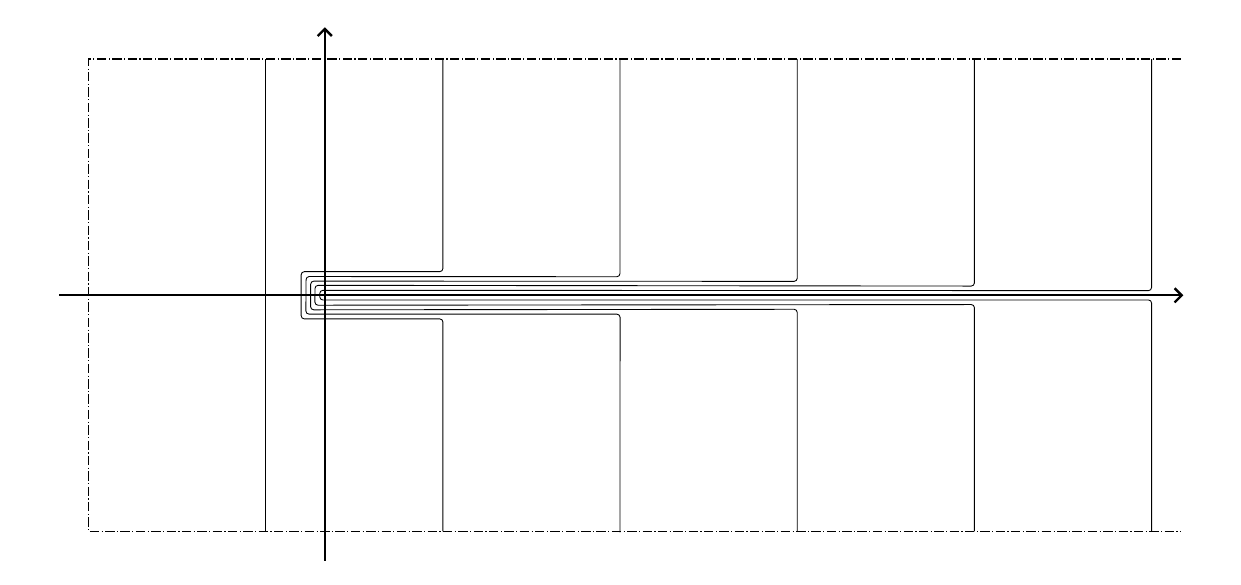
\caption{A family $\ML$ of line segments in $(-\varepsilon,+\infty)\times (-\varepsilon,\varepsilon) \setminus [0,+\infty) \times \left\{ 0 \right\}$.}
\label{fig:line_segments_strip}
\end{figure}

Near the boundary of the strip, the map $\alpha$ is close to the identity. After a small perturbation near the boundary, we can assume that $\alpha$ is exactly equal to the identity in this region. Extend $\alpha$ to a symplectomorphism between $\BR^2$ and $\BR^2\setminus [0,+\infty)\times \left\{ 0 \right\}$ by setting it equal to the identity outside the strip. Consider the symplectic embedding
\begin{equation*}
\rho : \BR^4 \rightarrow \BR^4 \qquad \rho \coloneqq \operatorname{id}_{\BR^2}\times (\psi \circ \alpha).
\end{equation*}
Note that the family of line segments $\ML$ can be arranged such that the image $\rho(X)$ approximates $\tilde{X}$ arbitrarily well. The desired embedding of $Z$ into $\operatorname{int}(\tilde{X})$ is simply given by the restriction $\rho|_Z$.
\end{proof}

\section{Proof of ellipsoid embedding stability}
\label{sec:proof_ellipsoid_embedding_stability}

The goal of this section is to prove Theorem \ref{thm:ellipsoid_embedding_stability}, i.e. that ellipsoid embedding stability holds for every compact connected symplectic $4$-manifold with smooth boundary.\\

We will make use of the following technical stratification lemma, whose proof we postpone to Section \ref{sec:proof_stratification_lemma}. Given positive numbers $a,b>0$, we define the symplectic cuboid $Q(a,b)\coloneqq \square(a,b)\times_L Q$; see Example \ref{ex:boundary_tame_packing_polydisk}.

\begin{lemma}
\label{lem:stratification_lemma}
There exist a compact set $\MF$ of symplectic frusta and a compact set $\MQ$ of symplectic cuboids such that the following is true for every $C^\infty$ open neighbourhood $\MU$ of $\MF$ inside the space of all perturbed symplectic frusta: Let $M$ be a symplectic $4$-manifold which is diffeomorphic to $[0,1]\times Y$ for some closed $3$-manifold $Y$. Then there exists $r>0$ such that, for all $0\leq \sigma_0 < \sigma_1 \leq 1$ satisfying $\sigma_1-\sigma_0 <r$, there exists a stratification of $[\sigma_0,\sigma_1]\times Y$ such that the closure of any top-dimensional stratum is conformally symplectomorphic to a symplectic cuboid contained in $\MQ$ or to a domain obtained from an element of $\MU$ by cutting.
\end{lemma}

\begin{proof}[Proof of Theorem \ref{thm:ellipsoid_embedding_stability}]

Our first step is to show that every symplectic $4$-manifold diffeomorphic to $[0,1]\times Y$ for some closed connected $3$-manifold $Y$ admits a $\partial$-tame packing. Let $M = [0,1]\times Y$ be such a symplectic $4$-manifold. Let $\MF$ be a compact set of symplectic frusta and $\MQ$ be a compact set of symplectic cuboids as in Lemma \ref{lem:stratification_lemma}. It follows from Theorem \ref{thm:packing_perturbed_frusta_with_cuts} that there exists an open neighbourhood $\MU$ of $\MF$ inside the space of perturbed symplectic frusta such that every domain obtained from an element of $\MU$ by cutting admits a $\partial$-tame packing. This clearly implies that every domain conformally symplectomorphic to a domain obtained from an element of $\MU$ by cutting also admits a $\partial$-tame packing. Moreover, recall from Example \ref{ex:boundary_tame_packing_polydisk} that every symplectic cuboid admits a $\partial$-tame packing.

Pick $r>0$ satisfying the assertion of Lemma \ref{lem:stratification_lemma}. Let $k>0$ be a positive integer such that $k^{-1} <r$. Divide $M$ into $k$ slices $M_i \coloneqq [(i-1)/k,i/k]\times Y$ for $1\leq i \leq k$. Let us stratify each slice $M_i$ into symplectic cuboids and perturbed symplectic frusta with cuts as in Lemma \ref{lem:stratification_lemma}. It follows from Lemma \ref{lem:boundary_tame_packing_stratification} that each $M_i$ admits a $\partial$-tame packing. Applying Lemma \ref{lem:boundary_tame_packing_stratification} again, we see that the same is true for $M$.

Let us now assume that $M$ is a general compact connected symplectic $4$-manifold with smooth boundary. Our goal is to show that $M$ admits a tame packing. Once we know this, it follows from Corollary \ref{cor:from_tame_packing_to_ellipsoid_packing} that ellipsoid embedding stability holds for $M$. Choose a Morse function $f:M\rightarrow \BR$ which is constant on each boundary component of $M$. We may arrange $f$ such that each critical level contains exactly one critical point. We divide $M$ into finitely many strata along regular level sets of $f$ such that each of the resulting strata contains at most one critical point. By Lemma \ref{lem:boundary_tame_packing_stratification}, it suffices to construct a $\partial$-tame packing for each of the strata. Since strata without critical points are of product form and admit a $\partial$-tame packing as proved above, we may reduce ourselves to the case that $f$ has exactly one critical point $p$. After replacing $f$ by $-f$ if necessary, we can further assume that $p$ has Morse index $0$, $1$ or $2$. In addition, we may assume that $f(p) = 0$.

Consider the case of Morse index $0$. We modify $f$ such that it is given by $f(z) = |z|^2$ in local Darboux coordinates $(z_1,z_2) = (x_1,y_1,x_2,y_2)$ near $p=0$. For $\delta>0$ sufficiently small, we divide $M$ into two strata along the level set $f^{-1}(\delta)$. One of the strata is a ball and the other one is of product form. Both strata admit a $\partial$-tame packing. For the balls this is obvious and for the stratum of product form this was proved above. Hence Lemma \ref{lem:boundary_tame_packing_stratification} implies that $M$ admits a $\partial$-tame packing as well.

Next, consider the case of Morse index $2$. We modify $f$ such that it is given by $f(z) = |z_1|^2 - |z_2|^2$ in local Darboux coordinates near $p=0$. By Lemma \ref{lem:boundary_tame_packing_stratification}, it suffices to show that both $f^{-1}( (-\infty,0])$ and $f^{-1} ( [0,\infty) )$ admit $\partial$-tame packings. Up to replacing $f$ by $-f$, it suffices to consider the sublevel set $X\coloneqq f^{-1}( (-\infty,0])$. Consider the standard moment map
\begin{equation*}
\mu : \BC^2 \rightarrow (\BR_{\geq 0})^2 \qquad (z_1,z_2) \mapsto (\pi|z_1|^2 , \pi|z_2|^2)
\end{equation*}
and the moment regions $\Omega,\Omega'\subset \BR_{\geq 0}^2$ defined by
\begin{equation*}
\Omega \coloneqq \left\{ \mu_2\geq \mu_1 \right\} \quad\text{and}\quad \Omega'\coloneqq \left\{ \mu_2\geq \mu_1 + 1 \right\}.
\end{equation*}
Note that the toric domain $\mu^{-1}(\Omega)$ is smooth away from the singular point $0$ and that a neighbourhood of $p$ in $X$ is symplectomorphic to a neighbourhood of $0$ in $\mu^{-1}(\Omega)$. The toric domain $\mu^{-1}(\Omega')$ is smooth everywhere. The moment region $\Omega$ is obtained from $\Omega'$ by translation along the vector $(0,-1)$. This translation lifts to a map
\begin{equation*}
T:\mu^{-1}(\Omega') \rightarrow \mu^{-1} (\Omega) \qquad (\sqrt{\frac{\mu_1}{\pi}}e^{2\pi i \theta_1}, \sqrt{\frac{\mu_2}{\pi}}e^{2\pi i \theta_2}) \mapsto (\sqrt{\frac{\mu_1}{\pi}}e^{2\pi i \theta_1}, \sqrt{\frac{\mu_2-1}{\pi}}e^{2\pi i \theta_2}).
\end{equation*}
which restricts to a symplectomorphism $T:\mu^{-1}(\Omega'\setminus\left\{ (0,1) \right\})\rightarrow \mu^{-1}(\Omega\setminus\left\{ (0,0) \right\})$. We construct a smooth symplectic manifold $X'$ out of $X$ by replacing a neighbourhood of $p$, which we symplectically identify with a neighbourhood $U$ of $0$ in $\mu^{-1}(\Omega)$, with $T^{-1}(U)$. Note that $X'$ is of product form and therefore admits a $\partial$-tame packing. This implies that the same is true for $X$. Indeed, any collection of symplectic disks $\MD$ in the smooth locus of the boundary of $X$ lifts to a collection of smooth disks $\MD'$ in the boundary of $X'$. Moreover, the map $T$ sends any symplectic ribbon complex in $X'$ with free ends at the disks $\MD'$ to a symplectic ribbon complex in $X$ with free ends at $\MD$.

Finally, let us consider the case that $p$ has index $1$. In this case, we modify $f$ such that there exist $\delta > 0$, $a>0$, and $0<b<\sqrt{a/\pi}$ such that the ball $B^4(a)$ is contained in a Darboux neighbourhood of $p$, the function $f$ takes the value $\delta$ on the set $\partial B^4(a) \cap \left\{ |x_1|\leq b \right\}$, and $B^4(a)\cap \left\{ |x_1|\leq b \right\}$ is contained in the sublevel set $f^{-1}( (-\infty,\delta])$. Let us divide $M$ into two strata $f^{-1}( (-\infty,\delta])$ and $f^{-1}( [\delta,\infty))$. Since the latter stratum is of product form, we may replace $M$ by $f^{-1}( (\infty,\delta])$. We cut $M$ along the $3$-dimensional ball $B^4(a) \cap \left\{ x_1=0 \right\}$. The result is a symplectic manifold $X$ with boundary and corners with either one or two connected components. Near the non-smooth regions of the boundary, $X$ looks like a half ball $B^4_+(a)\coloneqq B^4(a) \cap \left\{ x_1 \geq 0 \right\}$. Note that after smoothing the corners, $X$ is diffeomorphic to a manifold of product form.

Our goal is to show that each component of $X$ admits a $\partial$-tame packing. For simplicity, let us assume that $X$ is connected in the following. As in the case of an index $2$ critical point, our strategy is to construct a symplectic $4$-manifold $X'$ with smooth boundary such that the existence of a $\partial$-tame packing for $X'$ implies the existence of such a packing for $X$. The symplectic manifold $X'$ is obtained by carefully smoothing the corners of $X$.

Let us first explain the corner smoothing in the case of the half ball $X = B^4_+(a)$. In this case, $X'$ is given by the ellipsoid $E(a/2,a)$, which has the same volume. We need to explain how to obtain a $\partial$-tame packing of $B^4_+(a)$ from a $\partial$-tame packing of $E(a/2,a)$. For every $r \in (0,1)$, we construct a controlled symplectic embedding of the ellipsoid $E(ra/2,ra)$ into the half ball $B^4_+(a)$. Let us regard the ellipsoid $E(a/2,a)$ as a family of disks of varying sizes over $B^2(a/2)$. The area of the disk over $z \in B^2(a/2)$ is given by $f_1(z) \coloneqq 2(a/2-\pi|z|^2)$. Similarly, we regard the half ball $B^4_+(a)$ as a family of disks over the half disk $B^2_+(a)\coloneqq B^2(a)\cap \left\{ x_1\geq 0 \right\}$. The area of the disk over $z \in B^2_+(a)$ is $f_2(z)\coloneqq a - \pi|z|^2$. Observe that, for every $r\geq 0$, we have
\begin{equation}
\label{eq:weak_tame_embedding_into_half_ball_area_equality}
\operatorname{area}\left(\left\{ z \in B^2(a/2) \mid f_1(z) \geq r \right\}\right) = \operatorname{area}\left(\left\{ z \in B^2_+(a) \mid f_2(z) \geq r \right\}\right).
\end{equation}
We apply the embedding construction of Schlenk \cite[\S 2]{sch03}, \cite[\S 4.1]{sch05} that we have already used several times. To this end, consider the family of admissible loops $\ML$ in the half disk $\operatorname{int}(B^2_+(a))$ depicted in Figure \ref{fig:admissible_curve_family_half_disk}.
\begin{figure}[htpb]
\centering
\def\svgwidth{0.2\textwidth}
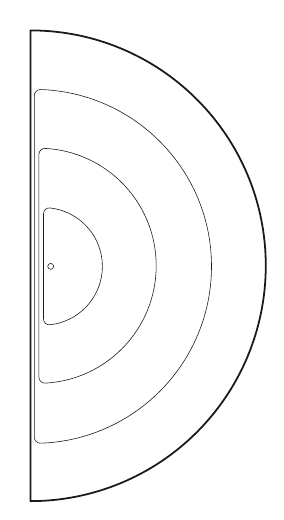
\caption{A special admissible family of loops $\ML$ in the half disk $B^2_+(a)$.}
\label{fig:admissible_curve_family_half_disk}
\end{figure}
Then by \cite[Lemma 2.5]{sch03}, there exists a symplectomorphism $\varphi : \operatorname{int}(B^2(a/2)) \rightarrow \operatorname{int}(B^2_+(a))$ which maps circles in $\operatorname{int}(B^2(a/2))$ centered at the origin to members of the family of loops $\ML$. We can arrange the family of loops $\ML$ such that the image of $E(ra/2,ra)$ unter $\varphi\times \operatorname{id}_\BC$ is contained in $\operatorname{int}(B^4_+(a))$. This makes important use of identity \eqref{eq:weak_tame_embedding_into_half_ball_area_equality}. Moreover, if the family of loops $\ML$ is chosen carefully, we can make sure that the top-dimensional boundary strata of $B^4_+(a)$ are approximated in $C^\infty_{\operatorname{loc}}$ by parts of the boundary of the image of $E(ra/2,ra)$ as $r$ tends to $1$. Symplectic disks in the boundary of $B^4_+(a)$ can therefore be connected to symplectic disks in the boundary of the image of $E(ra/2,ra)$ via symplectic ribbons, showing that the existence of a $\partial$-tame packing of $E(a/2,a)$ implies the existence of a $\partial$-tame packing of $B^4_+(a)$.

The adaptation of the corner smoothing from the special case $B^4_+(a)$ to the case of general $X$ is straightforward. The smoothing $X'$ is obtained by replacing an open neighbourhood $B^4_+(a) \cap \left\{ x_1 < b \right\}$ of the non-smooth locus of the boundary of $X$ by an appropriate open subset of the ellipsoid $E(a/2,a)$. The resulting symplectic manifold $X'$ has smooth boundary and is of product form. It therefore admits a $\partial$-tame packing and we can deduce that the same is true for $X$.
\end{proof}

\section{Proof of the stratification lemma}
\label{sec:proof_stratification_lemma}

The goal of this section is to prove the technical stratification result Lemma \ref{lem:stratification_lemma}.

\subsection{Stratifications compatible with a flow}
\label{subsec:strat_flow}

Let $Y$ be a closed $3$-manifold equipped with a nowhere vanishing vector field $R$. The main example of interest is that $Y$ is the boundary of a symplectic manifold and that $R$ is tangent to the characteristic foliation. Let $\varphi^t$ denote the flow generated by $R$. Let $\Sigma\subset Y$ be a compact embedded surface, possibly with boundary. We say that $\Sigma$ is a {\it transverse surface} if it is everywhere transverse to the vector field $R$. We emphasize that we also assume $\Sigma$ to be transverse to $R$ at the boundary $\partial\Sigma$. In particular, the boundary $\partial \Sigma$ does not consist of periodic orbits, which is a familiar assumption in the context of global surfaces of section. We say that $\Sigma$ is {\it exhaustive} if there exists a compact subset $K\subset \operatorname{int}\Sigma$ such that every flow line of $\varphi^t$ meets $K$ infinitely often forward and backward in time. It is not hard to see that an exhaustive transverse surface always exists. For example, one can construct such a surface by taking the disjoint union of sufficiently many small disks transverse to the flow.

Given an exhaustive transverse surface $\Sigma$, we define
\begin{equation*}
T = T_\Sigma :Y\rightarrow \BR_{>0}\qquad T(p)\coloneqq \inf\{t>0\mid \varphi^t(p)\in \Sigma\}.
\end{equation*}
Note that $T$ indeed takes values in $\BR_{>0}$. Since the flow line $\varphi^t(p)$ intersects $\Sigma$ forward in time and is transverse to $\Sigma$, the infimum $T(p)$ is finite and non-zero. Let us define
\begin{equation*}
\Phi:Y\rightarrow \Sigma \qquad \Phi(p)\coloneqq \varphi^{T(p)}(p).
\end{equation*}
We abbreviate $\tau\coloneqq T|_\Sigma$ and $\varphi\coloneqq \Phi|_{\Sigma}$. We observe that $\varphi$ is a bijection of $\Sigma$. In general, it is not continuous.

We say that an exhaustive transverse surface $\Sigma$ is {\it in good position} if the following is true for every point $p\in \partial\Sigma$. For $k\in \BZ$, set $p_k\coloneqq \varphi^k(p)$.
\begin{itemize}
\item Suppose that $p_1\in \partial\Sigma$. Then the immersion $\BR\times \partial\Sigma\rightarrow Y$ induced by the flow $\varphi^t$ intersects $\partial\Sigma$ transversely at $(\tau(p),p) \in \BR \times \partial\Sigma$. Moreover, the points $p_{-2}$, $p_{-1}$, $p_2$ and $p_3$ are contained in the interior of $\Sigma$.
\item Suppose that $p_1$ is not contained in $\partial\Sigma$. If $p_2$ is contained in $\partial\Sigma$, then the immersion $\BR\times \partial\Sigma\rightarrow Y$ induced by the flow $\varphi^t$ intersects $\partial\Sigma$ transversely at $(\tau(p)+\tau(p_1),p) \in \BR \times \partial\Sigma$.
\end{itemize}
Note that an exhaustive transverse surface can be brought into good position by a generic perburbation.

Given an exhaustive transverse surface $\Sigma\subset Y$ in good position, we construct a stratification $\MY$ of $Y$. Constructing $\MY$ amounts to constructing a filtration
\begin{equation*}
Y = Y_3 \supset Y_2 \supset Y_1 \supset Y_0
\end{equation*}
by closed subsets $Y_i\subset Y$ such that each complement $Y_i\setminus Y_{i-1}$ is a submanifold of dimension $i$. We define $Y_2$ to be the union of a subset $Y_2^\perp$ transverse to the vector field $R$ and a subset $Y_2^\top$ tangent to $R$. The transverse subset $Y_2^\perp$ is simply defined to be $Y_2^\perp\coloneqq \Sigma$. Let $\Sigma^+\subset Y$ be a surface obtained by slightly enlarging $\Sigma$, i.e. by slightly pushing $\partial\Sigma$ in the direction of the outward normal vector field. The tangent subset $Y_2^\top$ is constructed by flowing every point $p\in\partial\Sigma\cup \partial\Sigma^+$ forward and backward until $\Sigma$ is hit for the first time. In formulas this means
\begin{equation*}
Y_2^\perp\coloneqq \Sigma\qquad Y_2^\top\coloneqq \overline{\Phi^{-1}(\partial\Sigma)\cup \Phi^{-1}(\varphi(\partial\Sigma)) \cup \Phi^{-1}(\Phi(\partial\Sigma^+))}\qquad Y_2\coloneqq Y_2^\perp\cup Y_2^\top.
\end{equation*}
We define $Y_1\subset Y_2$ to be the set of non-smooth points of $Y_2$ and $Y_0\subset Y_1$ to be the set of non-smooth points of $Y_1$. It follows from the assumption that $\Sigma$ is in good position that this actually defines a stratification of $\MY$.

Let us now describe the structure of $\MY$ in more detail. Let $S\subset \Sigma$ be the set of all points $p\in \Sigma$ such that the first intersection of the forward flow line starting at $p$ with $\Sigma$ is contained in the boundary $\partial\Sigma$. In other words, we set $S\coloneqq \varphi^{-1}(\partial\Sigma)$. The map $\varphi$ is continuous on $\Sigma\setminus S$. Since $\Sigma$ is in good position, the set $S$ is a smooth $1$-dimensional manifold except for finitely many T-shaped singularities, i.e. singularities modeled on $\BR_{> 0}\cup i\BR\subset\BC$. Each smooth point $p\in S$ has a preferred normal orientation constructed as follows. The function $\tau$ has a jump discontinuity along $S$. We equip $p\in S$ with the normal orientation along which $\tau$ jumps up. Consider a T-shaped singularity of $S$ identified with $\BR_{>0}\cup i\BR$. It is not hard to check that the normal orientation at points in $i\BR$ must point in the direction of positive real part.

Let $S^+$ denote the set of all points $p\in \Sigma$ such that the forward flow line starting at $p$ meets $\partial\Sigma^+$ before it meets $\Sigma$ for the first time, i.e. $S^+\coloneqq \varphi^{-1}(\Phi(\partial\Sigma^+))$. If $\partial\Sigma^+$ is sufficiently close to $\partial\Sigma$, then $S^+$ is a small push-off of $S$ in the direction of the normal orientation. The model at a T-shaped singularity of $S$ is given by
\begin{equation*}
S= \BR_{>0}\cup i\BR\subset \BC \qquad \text{and}\qquad S^+ = (i\varepsilon + \BR_{>0}) \cup (\varepsilon+i\BR)
\end{equation*}
for $\varepsilon>0$. See Figure \ref{fig:stratification_S} for an illustration of the sets $S$ and $S^+$.

\begin{figure}[htpb]
\centering
\def\svgwidth{\textwidth}
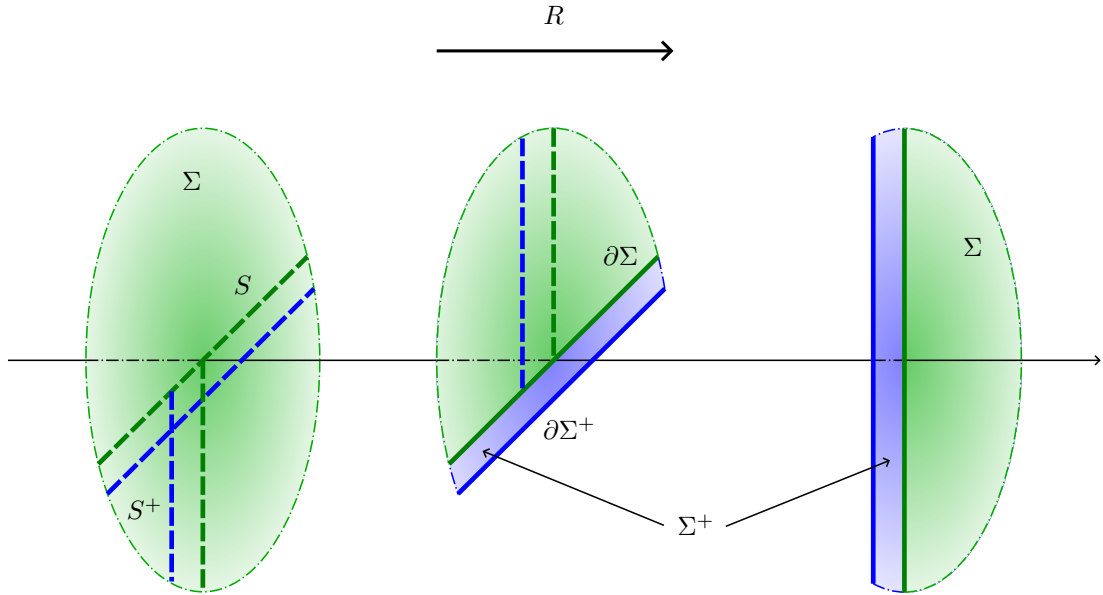
\caption{Intersection of the surfaces $\Sigma$ (green) and $\Sigma^+$ (blue) with a flow box. The sets $S$ and $S^+$ are highlighted by dashed lines.}
\label{fig:stratification_S}
\end{figure}

The sets $S$ and $S^+$ can be used to define a stratification $\MS$ of $\Sigma$. We define the $1$-skeleton of this stratification to be the union $\partial\Sigma \cup S\cup S^+$ and the $0$-skeleton to be the set of non-smooth points of $\partial\Sigma \cup S\cup S^+$. The set of top-dimensional strata of $\MS$ consists of the following (see Figure \ref{fig:stratification_S_disk}):
\begin{enumerate}[\indent (S1)]
\item \label{item:smooth_stratification_small_square} a small square at each T-shaped singularity of $S$;
\item \label{item:smooth_stratification_thin_rectangle} a thin rectangle along each smooth arc of $S$ and a thin annulus along each smooth circle component of $S$;
\item \label{item:smooth_stratification_rest} for each component $C$ of $\Sigma\setminus S$, a set $C^-$ which is obtained from $C$ by slightly pushing inwards some of its boundary arcs.
\end{enumerate}

\begin{figure}[htpb]
\centering
\def\svgwidth{0.5\textwidth}
\begingroup%
  \makeatletter%
  \providecommand\color[2][]{%
    \errmessage{(Inkscape) Color is used for the text in Inkscape, but the package 'color.sty' is not loaded}%
    \renewcommand\color[2][]{}%
  }%
  \providecommand\transparent[1]{%
    \errmessage{(Inkscape) Transparency is used (non-zero) for the text in Inkscape, but the package 'transparent.sty' is not loaded}%
    \renewcommand\transparent[1]{}%
  }%
  \providecommand\rotatebox[2]{#2}%
  \newcommand*\fsize{\dimexpr\f@size pt\relax}%
  \newcommand*\lineheight[1]{\fontsize{\fsize}{#1\fsize}\selectfont}%
  \ifx\svgwidth\undefined%
    \setlength{\unitlength}{283.46456693bp}%
    \ifx\svgscale\undefined%
      \relax%
    \else%
      \setlength{\unitlength}{\unitlength * \real{\svgscale}}%
    \fi%
  \else%
    \setlength{\unitlength}{\svgwidth}%
  \fi%
  \global\let\svgwidth\undefined%
  \global\let\svgscale\undefined%
  \makeatother%
  \begin{picture}(1,0.98)%
    \lineheight{1}%
    \setlength\tabcolsep{0pt}%
    \put(0,0){\includegraphics[width=\unitlength,page=1]{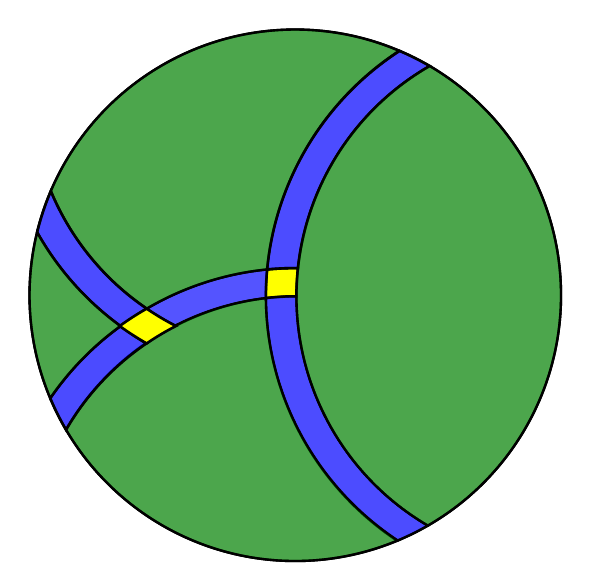}}%
    \put(0.51000001,0.48999999){\color[rgb]{0,0,0}\makebox(0,0)[lt]{\lineheight{1.25}\smash{\begin{tabular}[t]{l}(S1)\end{tabular}}}}%
    \put(0.61000001,0.74){\color[rgb]{0,0,0}\makebox(0,0)[lt]{\lineheight{1.25}\smash{\begin{tabular}[t]{l}(S2)\end{tabular}}}}%
    \put(0.74000002,0.48000002){\color[rgb]{0,0,0}\makebox(0,0)[lt]{\lineheight{1.25}\smash{\begin{tabular}[t]{l}(S3)\end{tabular}}}}%
  \end{picture}%
\endgroup%

\caption{Stratification $\MS$ of a disk component of $\Sigma$. Strata of type (S\ref{item:smooth_stratification_small_square}) are drawn in yellow, strata of type (S\ref{item:smooth_stratification_thin_rectangle}) in blue, and strata of type (S\ref{item:smooth_stratification_rest}) in green.}
\label{fig:stratification_S_disk}
\end{figure}
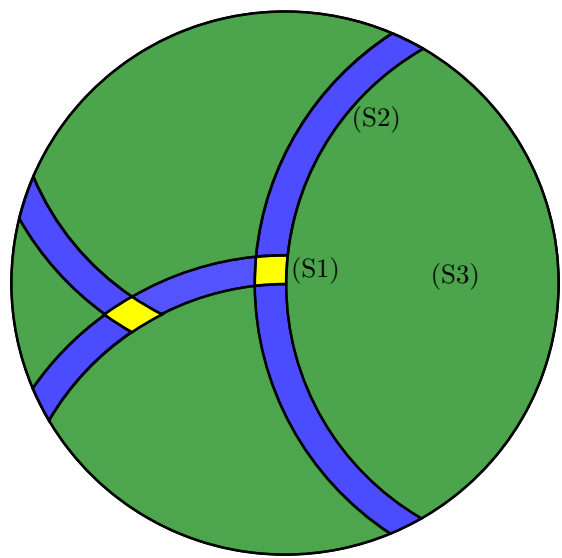

The top-dimensional strata of $\MS$ and $\MY$ are in bijection. Given a top-dimensional stratum $C$ of $\MS$, the corresponding top-dimensional stratum of $\MY$ is given by the flow box
\begin{equation*}
C_\MY \coloneqq \left\{ \varphi^t(p) \mid p \in C \enspace\text{and} \enspace t\in (0,\tau(p)) \right\}.
\end{equation*}

We conclude this subsection by constructing exhaustive transverse surfaces $\Sigma$ which are in good position in a precise quantitative sense. Let $g$ be an auxiliary Riemannian metric on $Y$ such that $R$ has unit norm. Given a small positive number $r>0$ and a point $p\in Y$, we define the transverse geodesic disk $D(p,r) \subset Y$ by
\begin{equation*}
D(p,r) \coloneqq \left\{ \exp_p(v) \mid v\in T_pY,\enspace v \perp R_p,\enspace |v| \leq r \right\}.
\end{equation*}
Suppose that $\rho >0$ is a small positive number. If $\rho$ is sufficiently small, we can build an exhaustive transverse surface $\Sigma$ by taking the disjoint union of finitely many transverse geodesic disks of radius $\rho$. By adding more disks if necessary, we can make sure that $\max \tau \leq \rho$. In the following we assume that this is the case.

Now consider an arbitrary point $p\in Y$ and choose an isometric identification $\eta:\BR^2 \cong \langle R_p\rangle^\perp$. Define the embedding
\begin{equation*}
\psi : (-10,10) \times B(0,10) \rightarrow Y \qquad \psi(t,z) \coloneqq \varphi^{\rho t}(\exp_p(\rho z))
\end{equation*}
where $B(0,10)\subset \BR^2$ is the ball of radius $10$ centered at $0$. Let $g_0$ be the Riemannian metric on $(-10,10)\times B(0,10)$ obtained by restricting the standard metric on $\BR^3$. Moreover, define the vector field $R_0\coloneqq \partial_t$, where $t$ is the coordinate of the factor $(-10,10)$. Note that we have
\begin{equation*}
\rho \psi^*R = R_0 \qquad \text{and} \qquad 
\operatorname{dist}_{C^\infty}(\rho^{-2}\psi^*g,g_0) = O(\rho) \quad (\rho \rightarrow 0).
\end{equation*}
Let $E\subset (-10,10)\times B(0,10)$ denote the preimage under $\psi$ of all disk components of $\Sigma$ which are fully contained in the image of $\psi$. Since $\rho$ is small, each component of $E$ is $C^\infty$ close to a disk of the form $\left\{ t \right\} \times \overline{B}(z,1)$ where $(t,z)\in (-10,10)\times B(0,10)$. This $C^\infty$ distance is controlled by $\rho$ and goes to zero as $\rho$ tends to zero.

Let $\delta>0$ be a positive number. We say that the surface $\Sigma$ is in \textit{$\delta$-good position} if the following is true:
\begin{enumerate}
\item $\min \tau \geq \delta\rho$
\item Let $\psi: (-10,10)\times B(0,10)\rightarrow Y$ be an embedding and $E\subset (-10,10)\times B(0,10)$ a surface as defined above for an arbitrary choice of point $p\in Y$ and identification $\BR^2\cong \langle R_p\rangle^\perp$. Let $\operatorname{pr}:(-10,10)\times B(0,10)\rightarrow B(0,10)$ be the projection onto the second factor. Consider the collection of circles $\operatorname{pr}(\partial D) \subset B(0,10)$, where $D$ ranges over all disk components of $E$. Then this collection of circles does not have any tangency or triple intersection points. Moreover, the distance between any two disjoint circles and the distance between any two distinct intersection points is at least $\delta$.
\end{enumerate}

It is not hard to see that if the centers of the disk components of $\Sigma$ are chosen generically, then $\Sigma$ is in $\delta$-good position for all sufficiently small $\delta>0$.

Suppose that $\Sigma$ is in $\delta$-good position. Given $\varepsilon >0$, let $\Sigma^+$ be the collection of transverse geodesic disks obtained from $\Sigma$ by increasing the radius of every disk component from $\rho$ to $(1+\varepsilon)\rho$ while leaving the center fixed. If $\varepsilon$ is sufficiently small, then the pair $(\Sigma,\Sigma^+)$ gives rise to a stratification $\MY$ of $Y$ as described above. We can choose $\varepsilon$ to only depend on $\delta$, for example $\varepsilon(\delta) \coloneqq \delta^2/100$.

\begin{lemma}
\label{lem:transverse_surface_in_delta_good_position}
There exists a universal constant $\delta>0$ with the following property. Let $Y$ be a closed $3$-manifold, $R$ a nowhere vanishing vector field on $Y$, and $g$ a Riemannian metric such that $R$ has unit norm. Then, for all sufficiently small $\rho>0$, there exists an exhaustive transverse surface $\Sigma\subset Y$ constructed as above as the disjoint union of finitely many transverse geodesic disks of radius $\rho$ such that $\Sigma$ is in $\delta$-good position and such that $\max \tau \leq \rho$.
\end{lemma}

\begin{proof}
The proof of the lemma is based on the following simple observation. For every $\delta>0$, there exists a constant $\delta' = \delta'(\delta)>0$ such that the following is true for every Riemannian metric $g$ on $(-10,10)\times B(0,10)$ which is sufficiently $C^\infty$ close to $g_0$. Consider a surface $E \subset (-10,10)\times B(0,10)$ consisting of finitely many disjoint, transverse (with respect to the vector field $R_0$) geodesic disks of radius $1$. Assume that $E$ is in $\delta$-good position. By this we mean that the flow line of the vector field $R_0$ starting at any point in $E$ can only meet $E$ again after time at least $\delta$ and that the projections of the boundary circles of $E$ to $B(0,10)$ satisfy the properties in item (2) above. Then there exists a surface $E'$ obtained from $E$ by adding disjoint transverse geodesic disks of radius $1$ such that $E'$ is in $\delta'$-good position and such that, for every point $(t,z) \in (-5,5)\times B(0,5)$, the flow line of the vector field $R_0$ starting at $(t,z)$ meets $E'$ within time strictly less than $1$.

Let $\rho >0$ be small. Note that there exist a universal postive integer $N>0$ independent of any of the other quantities involved, a finite set $\Psi$ of embeddings $\psi: (-10,10)\times B(0,10)\rightarrow Y$ defined as above for various choices of points $p\in Y$ and trivializations $\eta$, and an $N$-coloring of $\Psi$ such that the images of any two distinct embeddings of the same color are disjoint and such that the union over all $\psi \in \Psi$ of the sets $\psi( (-5,5)\times B(0,5))$ covers $Y$. Let $\Psi = \Psi_1 \sqcup \cdots \sqcup \Psi_N$ be the partition of $\Psi$ induced by the coloring.

Set $\delta_0 \coloneqq 1$. For $1\leq i \leq N$, recursively define $\delta_i \coloneqq \delta_{i-1}'$, where $\delta_{i-1}'$ is obtained from $\delta_{i-1}$ via the above observation. We claim that the assertion of the lemma holds with $\delta \coloneqq \delta_N$. Set $\Sigma_0 \coloneqq \emptyset$. Then $\Sigma_0$ is clearly in $\delta_0$-good position. Since the images of any two distinct embeddings in $\Psi_1$ are disjoint, we can apply the above observation for each such embedding separately and obtain a surface $\Sigma_1$ consisting of disjoint transverse geodesic disks of radius $\rho$ which is in $\delta_1$-good position and has the property that $T_{\Sigma_1}(p) \leq \rho$ for all points $p$ in the union $\bigcup_{\psi \in \Psi_1}\psi( (-5,5)\times B(0,5))$. We repeat the same step for all embeddings $\psi \in \Psi_2$ and obtain a surface $\Sigma_2$ in $\delta_2$-good position such that $T_{\Sigma_2}(p)\leq \rho$ for all points $p$ in $\bigcup_{\psi \in \Psi_1\cup \Psi_2}\psi( (-5,5)\times B(0,5))$. We iterate this step for all colors. Then the resulting surface $\Sigma \coloneqq \Sigma_N$ is in $\delta$-good position and satisfies $\max \tau \leq \rho$, as desired.
\end{proof}

\subsection{Stratification of thin products}

Let $(M,\omega)$ be a symplectic $4$-manifold. Assume that $M$ is diffeomorphic to $[0,1]\times Y$ for some closed $3$-manifold $Y$. Let $0\leq s_- < s_+ \leq 1$ and assume that the difference $\sigma\coloneqq s_+ - s_- \ll 1$ is small. In order to prove Lemma \ref{lem:stratification_lemma}, we need to construct a special stratification $\MX$ of $[s_-,s_+]\times Y$ whose top-dimensional strata are symplectic cuboids and perturbed symplectic frusta with cuts.

Let $R$ denote the Hamiltonian vector field induced by the Hamiltonian
\begin{equation*}
H: M = [0,1]\times Y \rightarrow \BR \qquad H(s,p) \coloneqq -s.
\end{equation*}
Let $J$ be a compatible almost complex structure on $(M,\omega)$ such that $J\partial_s = R$, where $s$ is the coordinate of the first factor of $[0,1]\times Y$. Let $g$ denote the Riemannian metric induced by $\omega$ and $J$.

Let us fix a constant $\delta>0$ as in Lemma \ref{lem:transverse_surface_in_delta_good_position}. If $\sigma$ is sufficiently small, we may apply Lemma \ref{lem:transverse_surface_in_delta_good_position} with $\rho\coloneqq \sigma/10$ to the $3$-manifold $\left\{ s_- \right\}\times Y$ equipped with the restrictions of $R$ and $g$. Let $\Sigma \subset \left\{ s_- \right\}\times Y$ be the resulting exhaustive transverse surface in $\delta$-good position. Moreover, let $\Sigma^+$ be the surface obtained from $\Sigma$ by increasing the radii of all disk components from $\rho$ to $(1+\varepsilon)\rho$ where $\varepsilon \coloneqq \varepsilon(\delta) = \delta^2/100$ as in Subsection \ref{subsec:strat_flow}. We emphasize that $\delta$ and $\varepsilon$ are fixed universal constants. In particular, they do not change if we shrink $\sigma$.

Let $\MS$ and $\MY$ denote the stratifications of $\Sigma$ and $\left\{ s_- \right\}\times Y$, respectively, induced by the pair of surfaces $(\Sigma,\Sigma^+)$. Recall that the sets of top-dimensional strata of $\MS$ and $\MY$ are in bijective correspondence. The set of top-dimensional strata of the stratification $\MX$ of $[s_-,s_+]\times Y$ we are about to construct will also be in bijective correspondence with these two sets.

Consider a point $p\in Y$ and a unitary trivialization
\begin{equation*}
\eta: \BC \cong \langle \partial_s,R_{(s_-,p)}\rangle^\perp \subset T_{(\sigma_-,p)}M.
\end{equation*}
Given such a tuple $(p,\eta)$, we define an embedding
\begin{equation*}
\psi : [0,1]\times (-1,1)\times B(0,1)\rightarrow [s_-,s_+]\times Y
\end{equation*}
as follows: First, we define $\psi(0,0,z) \coloneqq (s_-,\exp_{p}(\sigma \eta(z)))$, where $\exp$ denotes the exponential map of the restriction of the metric $g$ to $\left\{ s_- \right\}\times Y\cong Y$. Consider the hypersurface
\begin{equation*}
W\coloneqq [s_-,s_+]\times D(p,2\sigma) \subset [s_-,s_+]\times Y,
\end{equation*}
where $D(p,2\sigma)\subset Y$ is the transverse geodesic disk defined via the same exponential map. The symplectic form $\omega$ induces a characteristic foliation on $W$. At the point $(s_-,p)$, this characteristic foliation is tangent to $\partial_s$. If $\sigma$ is sufficiently small, the characteristic folation is close to being tangent to $\partial_s$ everywhere on $W$. In particular, the characteristic leaf starting at a point $(s_-,q)$ with $q \in D(p,\sigma)$ is a small perturbation of the line segment $[s_-,s_+]\times \left\{ q \right\}$. Let $s\in [0,1]$ and $z\in B(0,1)$. We define $\psi(s,0,z)$ to be the unique intersection point of the characteristic leaf starting at the point $\psi(0,0,z)$ and the slice $\left\{ s_- + s\sigma \right\}\times Y$. Finally, we define $\psi(s,t,z)\coloneqq \varphi^{\sigma t}(\psi(s,0,z))$, where $\varphi^t$ denotes the flow induced by $R$.

\begin{lemma}
\label{lem:pullback_symplectic_form}
Let $\Omega_0$ denote the standard area form on $B(0,1)$ and let us define the area form $\Omega$ by
\begin{equation*}
\Omega \coloneqq \sigma^{-2} \psi(0,0,\cdot) ^* \omega.
\end{equation*}
Then we have
\begin{equation*}
\sigma^{-2} \psi^*\omega = \Omega + ds \wedge dt \qquad \text{and} \qquad \operatorname{dist}_{C^\infty}(\Omega, \Omega_0) = O(\sigma) \quad (\sigma\rightarrow 0).
\end{equation*}
\end{lemma}

\begin{proof}
It is immediate from the construction of $\psi$ that the characteristic foliation on $[0,1]\times \left\{ 0 \right\}\times B(0,1)$ induced by $\psi^*\omega$ is tangent to the vector field $\partial_s$. This implies that
\begin{equation}
\label{eq:pullback_symplectic_form_proof_a}
\psi^*\omega|_{[0,1]\times \left\{ 0 \right\}\times B(0,1)} = \psi(0,0,\cdot)^*\omega.
\end{equation}
Similarly, it follows from the construction of $\psi$ that the characteristic foliation on $\left\{ s \right\}\times (-1,1)\times B(0,1)$ induced by $\psi^*\omega$ is tangent to the vector field $\partial_t$. We conclude that
\begin{equation*}
\psi^*\omega|_{\left\{ s \right\}\times (-1,1)\times B(0,1)} = \psi(0,0,\cdot)^*\omega.
\end{equation*}
This shows that we can write
\begin{equation*}
\sigma^{-2}\psi^*\omega = \Omega + ds \wedge \alpha_s
\end{equation*}
where $(\alpha_s)_{s\in [0,1]}$ is a family of $1$-forms on $(-1,1)\times B(0,1)$. Note that \eqref{eq:pullback_symplectic_form_proof_a} implies that the restriction of $\alpha_s$ to $\left\{ 0 \right\}\times B(0,1)$ vanishes. We have
\begin{equation*}
0 = d(\Omega + ds \wedge \alpha_s) = ds \wedge d\alpha_s.
\end{equation*}
This implies that $\alpha_s$ is a closed $1$-form. We can therefore write $\alpha_s = du_s$ for a function $u_s:(-1,1)\times B(0,1)\rightarrow \BR$. The function $u_s$ is uniquely determined after imposing the normalization $u_s(0,0) = 0$. Note that since the restriction of $\alpha_s$ to $\left\{ 0 \right\}\times B(0,1)$ vanishes, we then have $u_s(0,z) = 0$ for all $z$. Note that the Hamiltonian vector field of the Hamiltonian $\sigma^{-1}\psi^*H$ with respect to the symplectic form $\sigma^{-2}\psi^*\omega$ is given by $\sigma\psi^* R$. Moreover, note that
\begin{equation*}
(\sigma^{-1}\psi^*H)(s,t,z) = -s - \sigma^{-1}\sigma_- \qquad \text{and} \qquad \sigma\psi^*R = \partial_t.
\end{equation*}
This implies that $\alpha_s(\partial_t) = 1$. We conclude that $u_s(t,z) = t$.

It remains to prove that $\operatorname{dist}_{C^\infty}(\Omega,\Omega_0) = O(\sigma)$ as $\sigma$ tends to zero. Let us write $\Omega_\sigma$ to indicate the dependence of $\Omega$ on $\sigma$. For $0<r\leq 1$, define the map $S_r:B(0,1)\rightarrow B(0,1)$ by $S_r(z) \coloneqq rz$. A simple computation shows that
\begin{equation*}
\Omega_{r\sigma} = r^{-2} S_{r}^*\Omega_\sigma.
\end{equation*}
Since $\Omega_\sigma$ agrees with $\Omega_0$ at the center of $B(0,1)$, this implies that $\operatorname{dist}_{C^\infty}(\Omega_{r\sigma},\Omega_0) = O(r)$ as $r$ tends to zero.
\end{proof}

\begin{lemma}
\label{lem:transition_maps}
Let $\psi,\psi' : [0,1]\times (-1,1)\times B(0,1) \rightarrow [s_-,s_+]\times Y$ be the embeddings induced by two choices of points and unitary trivializations $(p,\eta)$ and $(p',\eta')$, respectively. Assume that the intersection $U \coloneqq \operatorname{im}(\psi) \cap \operatorname{im}(\psi')$ is non-empty and consider the transition map
\begin{equation*}
\chi \coloneqq \psi'^{-1}\circ \psi : \psi^{-1}(U) \rightarrow \psi'^{-1}(U).
\end{equation*}
Then there exist a translation $f:\BR\rightarrow \BR$ and an orientation preserving affine isometry $g:\BR^2\rightarrow \BR^2$ such that
\begin{equation*}
\operatorname{dist}_{C^\infty}(\chi, (\operatorname{id}_{\BR} \times f \times g)|_{\psi^{-1}(U)}) = O(\sigma) \quad (\sigma\rightarrow 0).
\end{equation*}
\end{lemma}

\begin{proof}
Let $g_0$ denote the Riemannian metric on $[0,1]\times (-1,1)\times B(0,1)$ obtained by restricting the standard metric on $\BR^4$. Let us show that $\operatorname{dist}_{C^\infty}(\sigma^{-2}\psi^*g, g_0) = O(\sigma)$ as $\sigma$ tends to zero. For $0<r\leq 1$, let $\psi_r$ denote the embedding associated to $(p,\eta)$, but with the interval $[s_-,s_+]$ replaced by $[s_-,s_- + r (s_+ - s_-)]$. Define a map
\begin{equation*}
S_r : [0,1]\times (-1,1)\times B(0,1) \rightarrow [0,1]\times (-1,1)\times B(0,1) \qquad S_r(s,t,z) \coloneqq (rs,rt,rz).
\end{equation*}
It is straightforward to check that $\psi_r = \psi \circ S_r$ and hence that
\begin{equation*}
(r\sigma)^{-2}\psi_r^*g = r^{-2} S_r^* (\sigma^{-2}\psi^*g).
\end{equation*}
Note that $\sigma^{-2}\psi^*g$ agrees with $g_0$ at the point $(0,0,0)$. This implies that $\operatorname{dist}_{C^\infty}((r\sigma)^{-2}\psi_r^*g,g_0) = O(r)$, as desired.

We observe that the transition map $\chi$ is an isometry with respect to the metric $\sigma^{-2}\psi^*g$ on $\psi^{-1}(U)$ and the metric $\sigma^{-2}(\psi')^*g$ on $(\psi')^{-1}(U)$. Both of these metrics have $C^\infty$ distance $O(\sigma)$ from $g_0$ by the above. Moreover, by definition of $\psi$ and $\psi'$, the map $\chi$ preserves the coordinate of the first factor of $[0,1]\times (-1,1)\times B(0,1)$ and we have $\chi_*\partial_t = \partial_t$. Combining these observations, we see that $\chi$ must have $C^\infty$ distance $O(\sigma)$ from a map of the form $\operatorname{id}_{\BR} \times f \times g$ for a translation $f$ and an affine isometry $g$.
\end{proof}

For each disk component $D = \left\{ s_- \right\}\times D(p,\rho)$ of $\Sigma$, pick an arbitrary unitary trivialization $\eta:\BC \cong \langle \partial_s, R_{(s_-,p)}\rangle^\perp$ and let $\psi_D: [0,1]\times (-1,1)\times B(0,1)\rightarrow [\sigma_-,\sigma_+]\times Y$ be the embedding associated to $(p,\eta)$. Note that $\psi_D^{-1}(D) = \left\{ 0 \right\}\times \left\{ 0 \right\}\times B(0,1/10)$. Moreover, it follows from Lemma \ref{lem:transition_maps} that, for every disk component $D'$ of $\Sigma$ which intersects the image of $\psi_D$, there exist a point $(t,z)\in \BR\times \BR^2$ and a perturbation $E$ of the disk $\left\{ 0 \right\}\times \left\{ t \right\}\times B(z,1/10)$ of $C^\infty$ size $O(\sigma)$ such that $\psi_D^{-1}(D') = E \cap ([0,1]\times (-1,1)\times B(0,1))$. Let us assume that $\sigma$ is sufficiently small such that the $C^\infty$ size of the perturbation is much smaller than $\varepsilon$.

Let us define the auxiliary hypersurface $W_D\coloneqq \psi_D([0,1]\times \left\{ 0 \right\}\times B(0,1))$. Again, it follows from Lemma \ref{lem:transition_maps} that, for each disk component $D'$ of $\Sigma$ such that $W_{D'}$ intersects the image of $\psi_D$, there exist $(t,z)\in \BR\times \BR^2$ and a perturbation $V$ of $[0,1]\times \left\{ t \right\} \times B(z,1)$ of $C^\infty$ size $O(\sigma)$ such that $\psi_D^{-1}(W_{D'}) = V \cap ([0,1]\times (-1,1)\times B(0,1))$. We assume that $\sigma$ is sufficiently small such that the $C^\infty$ size of the perturbation is much smaller than $\varepsilon$. In the following, we perturb the hypersurfaces $W_D$ while preserving this property. 

Consider a top-dimensional stratum $C$ of $\MS$ of type (S\ref{item:smooth_stratification_small_square}). A neighbourhood of the corresponding top-dimensional stratum $C_\MY$ of $\MY$ in $\left\{ s_- \right\}\times Y$ is depicted in Figure \ref{fig:neighbourhood_stratum_cy}. The closure $\overline{C}_\MY$ intersects four disk components of $\Sigma$. We abbreviate them by $D, D_1, D_2, D_3$ as indicated in Figure \ref{fig:neighbourhood_stratum_cy}.

\begin{figure}[htpb]
\centering
\def\svgwidth{0.8\textwidth}
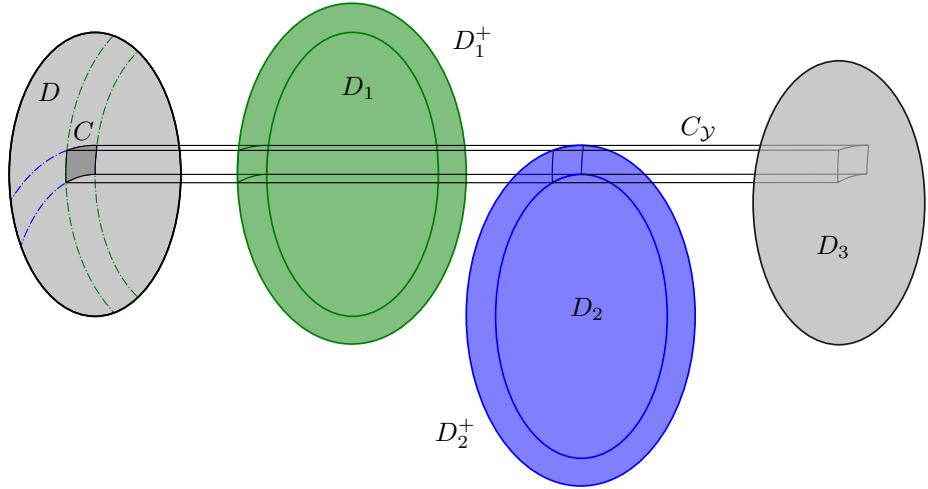
\caption{The stratum $C_\MY$ and the disks $D, D_1, D_2, D_3$.}
\label{fig:neighbourhood_stratum_cy}
\end{figure}

Recall from Subsection \ref{subsec:strat_flow} that the time $\tau(q)$ it takes a flow line of $\varphi^t$ starting at $q\in \Sigma$ to hit $\Sigma$ again is bounded from above by $\rho$. Hence the disks $D_i$ are contained in the image of $\psi_D$. As explained above, $\psi_D^{-1}(D_i)$ is a perturbation of $C^\infty$ size $O(\sigma)$ of the disk $\left\{ 0 \right\}\times \left\{ t_i \right\} \times B(z_i,1/10)$ for some $(t_i,z_i)$. It follows from the bound $\tau\leq \rho$ and the fact that the disks $D_i$ intersect $\overline{C}_\MY$ that we can take $(t_i,z_i) \in (0,1/10)\times B(0,2/10)$. Similarly, the preimage $\psi_D^{-1}(W_{D_i})$ is the intersection with $[0,1]\times (-1,1)\times B(0,1)$ of a perturbation of $[0,1]\times \left\{ t_i \right\} \times B(z_i,1)$ of $C^\infty$ size $O(\sigma)$.

Let $\operatorname{pr}_{23} : [0,1]\times (-1,1)\times B(0,1) \rightarrow \left\{ 0 \right\}\times (-1,1)\times B(0,1)$ denote the projection onto the second and third factor. For $i\in \left\{ 1,2,3 \right\}$, let us now perturb $W_{D_i}$ such that
\begin{equation*}
\psi_D^{-1}(W_{D_i}) \cap N_\varepsilon (\operatorname{pr}_{23}^{-1}\psi_D^{-1}(\overline{C}_\MY))
\end{equation*}
is contained in $[0,1]\times \left\{ t_i \right\}\times B(0,1)$. Here $N_\varepsilon$ denotes the $\varepsilon$-neighbourhood as measured with respect to the standard metric on $[0,1]\times (-1,1)\times B(0,1)$. We can arrange this perturbation to be supported inside $N_{2\varepsilon} (\operatorname{pr}_{23}^{-1}\psi_{D}^{-1}(\overline{C}_\MY))$. Moreover, we can make sure that the $C^\infty$ size of this perturbation measured in the coordinate chart $\psi_{D}$ is $O(\sigma)$ and in particular much smaller than $\varepsilon$. Let us repeat this step for all top-dimensional strata $C$ of $\MS$ of type (S\ref{item:smooth_stratification_small_square}). Note that for two distinct such strata $C$ and $C'$, the modifications take place in disjoint regions because $C_\MY$ and $C'_\MY$ are sufficiently far apart. We abbreviate the perturbed hypersurfaces by the same symbol $W_D$. Let us also perturb $(\Sigma,\Sigma^+)$ by flowing the disk components $(D,D^+)$ along the characteristic foliation on $\left\{ s_- \right\}\times Y$ such that they are contained in $W_D$. The modified surface pair and its induced stratifications are still denoted by $(\Sigma,\Sigma^+)$, $\MS$, and $\MY$.

For each disk component $D$ of $\Sigma$, let $V_D$ denote the union of all characteristic leaves on $W_D$ which intersect $D$. Let $V_\Sigma$ be the union of all $V_D$ as $D$ ranges over the disk components of $\Sigma$. We define a diffeomorphism $\alpha: [0,\sigma]\times \Sigma \rightarrow V_\Sigma$ by declaring $\alpha(s,p)$ to be the unique intersection point of the characteristic leaf starting at $p$ with the slice $\left\{ s_-+s \right\}\times Y$. Moreover, we define a map $\tau : V_{\Sigma} \rightarrow \BR_{>0}$ by
\begin{equation*}
\tau(q) \coloneqq \inf \left\{ t>0 \mid \varphi^t(q) \in V_\Sigma \right\}
\end{equation*}

We are now in a position to describe the top-dimensional strata of $\MX$. Let $C$ be a top-dimensional stratum of $\MS$ of type (S\ref{item:smooth_stratification_small_square}). We define the corresponding stratum $C_\MX$ of $\MX$ to be
\begin{equation*}
C_\MX \coloneqq \left\{ \varphi^t(p) \mid p \in \alpha([0,\sigma]\times C),\enspace 0<t<\tau(p) \right\}.
\end{equation*}
We claim that $\overline{C}_\MX$ is a symplectic cuboid. Indeed, let $D$ denote the disk component of $\Sigma$ containing $C$. Then it follows from our construction of the hypersurfaces $W_D$ that 
\begin{equation*}
\psi_D^{-1}(\overline{C}_\MX) = [0,1]\times [0,T] \times \psi_D^{-1}(\overline{C})
\end{equation*}
for some $T>0$. Since $\sigma^{-2}\psi_D^*\omega = \Omega + ds\wedge dt$ for some area form $\Omega$ on $B(0,1)$ by Lemma \ref{lem:pullback_symplectic_form}, we see that $\overline{C}_\MX$ is a symplectic cuboid. It is conformally symplectomorphic to the cuboid $Q(T,A)$ where $A$ denotes the $\Omega$-area of $\psi_D^{-1}(\overline{C})$. Note that we can bound $T$ and $A$ from above and below purely in terms of the universal constant $\delta$. This shows that we can find a universal compact set $\MQ$ of symplectic cuboids such that every stratum $\overline{C}_\MX$ associated to a stratum $C$ of type (S\ref{item:smooth_stratification_small_square}) is conformally symplectomorphic to an element of $\MQ$.

Next, consider a top-dimensional stratum $C$ of $\MS$ of type (S\ref{item:smooth_stratification_thin_rectangle}). Again, let $D$ denote the disk component of $\Sigma$ containing $C$ and let $A$ be the $\Omega$-area of $\psi_D^{-1}(C)$. Choose an area preserving parametrization $\beta: [0,\sigma A]\times [0,\sigma] \rightarrow \overline{C}$ such that the boundary edges of $C$ bordering strata of $\MS$ of type (S\ref{item:smooth_stratification_rest}) are given by $\beta( \left\{ 0 \right\}\times [0,\sigma])$ and $\beta( \left\{ \sigma A \right\}\times [0,\sigma])$. The boundary edges $\beta( [0,\sigma A]\times \left\{ 0 \right\})$ and $\beta( [0,\sigma A]\times \left\{ \sigma \right\})$ either border strata of $\MS$ of type (S\ref{item:smooth_stratification_small_square}) or are contained in the boundary of $D$.

Define a function $f:[0,\sigma A]\times [0,\sigma]\rightarrow [0,\sigma]$ by the formula
\begin{equation*}
f(x,y) \coloneqq
\begin{cases}
3x/A & \text{if $0 \leq x \leq \sigma A/3$} \\
\sigma & \text{if $\sigma A/3 \leq x \leq 2\sigma A/3$} \\
3\sigma - 3 x/A & \text{if $2\sigma A/3 \leq x \leq \sigma A$}.
\end{cases}
\end{equation*}
We define
\begin{equation*}
C_\MX \coloneqq \left\{ \varphi^t(p) \mid p \in \alpha\circ (\operatorname{id}_\BR\times \beta)( \operatorname{int}D(f) ), \enspace 0<t<\tau(p)  \right\}.
\end{equation*}
Here we recall that $D(f) \subset \BR \times [0,\sigma A] \times [0,\sigma]$ denotes the subgraph of $f$ truncated at zero. See Figure \ref{fig:stratification_X} for an illustration.

\begin{figure}[htpb]
\centering
\def\svgwidth{0.7\textwidth}
\begingroup%
  \makeatletter%
  \providecommand\color[2][]{%
    \errmessage{(Inkscape) Color is used for the text in Inkscape, but the package 'color.sty' is not loaded}%
    \renewcommand\color[2][]{}%
  }%
  \providecommand\transparent[1]{%
    \errmessage{(Inkscape) Transparency is used (non-zero) for the text in Inkscape, but the package 'transparent.sty' is not loaded}%
    \renewcommand\transparent[1]{}%
  }%
  \providecommand\rotatebox[2]{#2}%
  \newcommand*\fsize{\dimexpr\f@size pt\relax}%
  \newcommand*\lineheight[1]{\fontsize{\fsize}{#1\fsize}\selectfont}%
  \ifx\svgwidth\undefined%
    \setlength{\unitlength}{368.50393701bp}%
    \ifx\svgscale\undefined%
      \relax%
    \else%
      \setlength{\unitlength}{\unitlength * \real{\svgscale}}%
    \fi%
  \else%
    \setlength{\unitlength}{\svgwidth}%
  \fi%
  \global\let\svgwidth\undefined%
  \global\let\svgscale\undefined%
  \makeatother%
  \begin{picture}(1,0.79230769)%
    \lineheight{1}%
    \setlength\tabcolsep{0pt}%
    \put(0,0){\includegraphics[width=\unitlength,page=1]{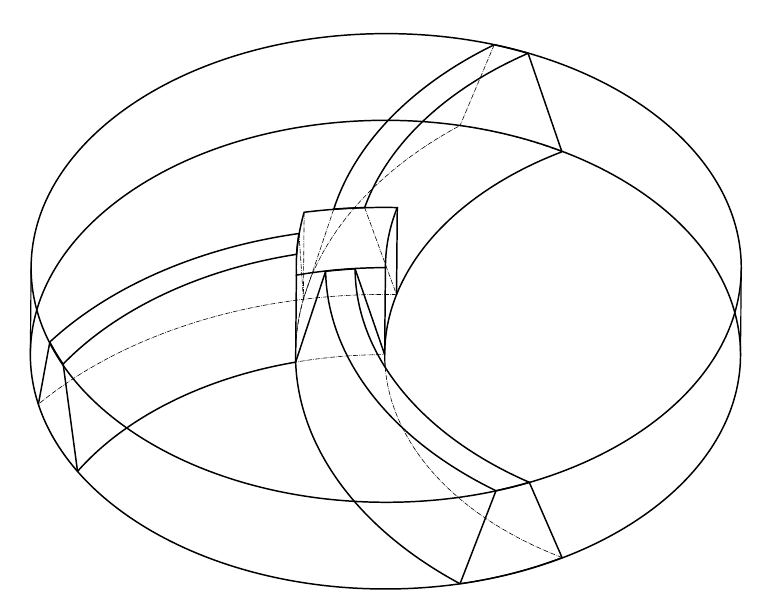}}%
    \put(0.89230771,0.11538463){\color[rgb]{0,0,0}\makebox(0,0)[t]{\lineheight{1.25}\smash{\begin{tabular}[t]{c}$D$\end{tabular}}}}%
  \end{picture}%
\endgroup%

\caption{For some disk component $D$ of $\Sigma$, the figure shows the intersection with $V_D$ of all closures $\overline{C}_\MX$ of strata of $\MX$ associated to strata of $\MS$ of types (S\ref{item:smooth_stratification_small_square}) and (S\ref{item:smooth_stratification_thin_rectangle}).}
\label{fig:stratification_X}
\end{figure}

We claim that $\overline{C}_\MX$ is a perturbed symplectic frustum with cuts. Indeed, it follows from the way we adjusted the hypersurfaces $W_D$ near the already constructed strata $C_\MX'$ with $C'$ of type (S\ref{item:smooth_stratification_small_square}) that, for $y \in \left\{ 0,\sigma \right\}$, a sufficiently small neighbourhood of
\begin{equation*}
\left\{ \varphi^t(p) \mid p \in \alpha\circ (\operatorname{id}_\BR\times \beta)( D(f|_{[0,\sigma A]\times \left\{ y \right\}})) \right\}
\end{equation*}
in $\overline{C}_\MX$ looks eactly like a neighbourhood of a cut region in a perturbed symplectic frustum with cuts; see (B) in Figure \ref{fig:frustum_with_cut_examples}. Moreover, after undoing these two cuts one obtains a perturbed symplectic frustum. Note that $\overline{C}_\MX$ is conformally symplectomorphic to $\psi_D^{-1}(\overline{C}_\MX)$ equipped with the symplectic form $\Omega + ds\wedge dt$. Since the hypersurfaces $\psi_D^{-1}(W_{D'})$ for other disk components $D'$ of $\Sigma$ are contained in perturbations of $C^\infty$ size $O(\sigma)$ of hypersurfaces of the form $[0,1]\times \left\{ t \right\}\times B(0,1)$, the perturbed symplectic frustum underlying $\psi_D^{-1}(\overline{C}_\MX)$ has, up to symplectomorphism, $C^\infty$ distance $O(\sigma)$ from an actual symplectic frustum $F = F(a,b,c)$. Again we can bound $a,b,c$ and also $c-b$ from above and below by universal constants only depending on $\delta$. This implies that $F$ is contained in some universal compact set $\MF$ of symplectic frusta.

For each disk component $D$ of $\Sigma$, let $D \subset D^0 \subset D^+$ be a disk whose boundary circle $\partial D^0$ lies halfway between the boundary circles $\partial D$ and $\partial D^+$. Moreover, let $V_{D^0}$ be the union of all characteristic leaves of $W_D$ intersecting $D^0$. Let $\Sigma^0$ denote the union of all disks $D^0$. Similarly, let $V_{\Sigma^0}$ denote the union of all $V_{D^0}$. Define $A$ to be the union of $V_{\Sigma^0}$ with all $\overline{C}_\MX$ where $C$ ranges over strata of $\MS$ of types (S\ref{item:smooth_stratification_small_square}) and (S\ref{item:smooth_stratification_thin_rectangle}). The remaining top-dimensional strata of $\MX$ are given by the connected components of $((0,\sigma)\times Y) \setminus A$. These connected components are in bijective correspondence with strata of $\MS$ of type (S\ref{item:smooth_stratification_rest}).

We claim that for each stratum $C$ of $\MS$ of type (S\ref{item:smooth_stratification_rest}), the closure of the corresponding stratum $\overline{C}_\MX$ is symplectomorphic to a perturbed symplectic frustum with cuts, sitting inside $[0,\sigma]\times Y$ upside down; see Figure \ref{fig:stratification_X}. Consider a stratum $C'$ of $\MS$ of type (S\ref{item:smooth_stratification_small_square}) such that $(\overline{C}_\MX' \cap \overline{C}_\MX) \setminus V_\Sigma$ is non-empty. Then the intersection of $\overline{C}_\MX$ with a neighbourhood of $\overline{C}_\MX'$ looks like a neighbourhood of a cut region in a perturbed symplectic frustum with cuts; see (A) and (C) in Figure \ref{fig:frustum_with_cut_examples}. Undoing all these cuts yields a perturbed symplectic frustum. As before, let $D$ be the disk component of $\Sigma$ containing $C$ and consider $\psi_D^{-1}(\overline{C}_\MX)$ equipped with $\Omega + ds\wedge dt$. As above, the perturbed symplectic frustum obtained by undoing all cuts of $\psi_D^{-1}(\overline{C}_\MX)$ has, up to symplectomorphism, $C^\infty$ distance $O(\sigma)$ from an actual symplectic frustum $F$, which can be taken to be contained in a universal compact set $\MF$ of symplectic frusta. This concludes the proof of Lemma \ref{lem:stratification_lemma}. \qed

\section{Bounding the subleading asymptotics}
\label{sec:bounding_subleading_asymptotics}

The goal of this section is to prove Theorems \ref{thm:subleading_asymptotics_ech}, \ref{thm:subleading_asymptotics_pfh}, and \ref{thm:subleading_asymptotics_link_invariants}.

\begin{proof}[Proof of Theorem \ref{thm:subleading_asymptotics_ech}]
The proof follows the same strategy as the original proof of the ECH Weyl law in \cite{hut11}. We proceed in several steps.

\emph{Step 1:} We claim that the error terms are bounded for balls, i.e. $e_k(B(a)) = O(1)$. Indeed, by \cite[Example 1.2]{hut22}, we have
\begin{equation*}
\liminf_{k\rightarrow \infty} e_k(B(a)) = -\frac{3}{2}a \qquad \text{and}\qquad \limsup_{k\rightarrow \infty} e_k(B(a)) = -\frac{1}{2}a,
\end{equation*}
which clearly implies that $e_k(B(a)) = O(1)$.

\emph{Step 2:} We show that the error terms are bounded for all finite disjoint unions of balls. In fact, we show the following more general statement: Suppose that $X_1,\dots,X_n$ is a finite collection of symplectic $4$-manifolds such that $e_k(X_i) = O(1)$ for all $i$. Let $X \coloneqq \sqcup_i X_i$ be the disjoint union. Then we also have $e_k(X) = O(1)$.

Recall from \cite[Equation 1.4]{hut22} that the ECH capacities of the disjoint union $X$ are given by
\begin{equation}
\label{eq:subleading_asymptotics_ech_proof_disjoint_sum_formula}
c_k (X) = \max_{\sum_i k_i = k} c_{k_i}(X_i).
\end{equation}
Here the maximum is taken over all tuples of non-negative integers $(k_1,\dots,k_n)$ whose sum is $k$. Set $V\coloneqq \operatorname{vol}(X)$ and $V_i \coloneqq \operatorname{vol}(X_i)$. The desired statement follows immediately if we can verify that
\begin{equation}
\label{eq:subleading_asymptotics_ech_proof_a}
\max_{\sum\limits_ik_i=k} \sum\limits_i \sqrt{V_ik_i} = \sqrt{Vk} + O(1)\qquad (k\rightarrow\infty).
\end{equation}
Consider the function
\begin{equation*}
f:\BR_{\geq 0}^n\rightarrow\BR\qquad f(x_1,\dots,x_n) \coloneqq \sum\limits_i\sqrt{V_ix_i}.
\end{equation*}
Let $k\in\BZ_{> 0}$. It follows from a direct computation that the maximum of $f$ subject to the constraint $\sum\limits_i x_i\leq k$ is attained at the point $k/V\cdot (V_1,\dots,V_n)$. In particular, this implies that
\begin{equation}
\label{eq:subleading_asymptotics_ech_proof_b}
\sum\limits_i \sqrt{V_i\left\lfloor \frac{V_ik}{V} \right\rfloor} \leq \max_{\sum\limits_ik_i=k} \sum\limits_i \sqrt{V_ik_i} \leq \sum\limits_i \sqrt{\frac{V_i^2k}{V}} = \sqrt{Vk}.
\end{equation}
We compute
\begin{equation}
\label{eq:subleading_asymptotics_ech_proof_c}
\sum\limits_i \sqrt{V_i\left\lfloor \frac{V_ik}{V} \right\rfloor} = \sum\limits_i\sqrt{\frac{V_i^2k}{V}} + O(1) = \sqrt{Vk}+ O(1).
\end{equation}
The desired identity \eqref{eq:subleading_asymptotics_ech_proof_a} is immediate from \eqref{eq:subleading_asymptotics_ech_proof_b} and \eqref{eq:subleading_asymptotics_ech_proof_c}.

\emph{Step 3:} Let $X \subset \BR^4$ be a compact domain with smooth boundary. We show that the sequence of error terms $e_k(X)$ is bounded from below. By Theorem \ref{thm:ball_packing_stability}, ball packing stability holds for $X$. Pick a positive integer $n$ such that $p_n(X) = 1$. Let $a>0$ be the unique real number such that $n\operatorname{vol}(B(a)) = \operatorname{vol}(X)$. Then for every $b<a$, the disjoint union of $n$ copies of $B(b)$ admits a symplectic embedding into $\operatorname{int}(X)$. By monotonicity of ECH capacities \cite[Equation 1.2]{hut22}, we obtain the inequality
\begin{equation*}
c_k(\bigsqcup\limits_{i=1}^n B(b)) \leq c_k(X)
\end{equation*}
for all $b < a$. Since the left hand side is continuous in $b$, the inequality continues to hold for $b = a$. Using that $\sqcup_i B(a)$ has the same volume as $X$ and bounded error terms by Step 2, we conclude that the error terms $e_k(X)$ are bounded from below.

\emph{Step 4:} We show that the sequence of error terms $e_k(X)$ is bounded from above for every compact domain $X\subset \BR^4$ with smooth boundary. Let $B(A)$ be a large ball containing $X$ in its interior. By Theorem \ref{thm:ball_packing_stability}, ball packing stability holds for $W \coloneqq B(A) \setminus \operatorname{int}(X)$. Pick $n\geq 1$ such that $p_n(W) = 1$. Let $a>0$ be the real number such that $n \operatorname{vol}(B(a)) = \operatorname{vol}(W)$. Then for every $b<a$, the disjoint union of $n$ copies of $B(b)$ embeds into $\operatorname{int}(W)$ and monotonicity implies the inequality
\begin{equation*}
c_k( X \sqcup \bigsqcup\limits_{i=1}^n B(b)) \leq c_k(B(A)).
\end{equation*}
Again, this inequality continues to hold for $b = a$. The disjoint union formula \eqref{eq:subleading_asymptotics_ech_proof_disjoint_sum_formula} then implies that
\begin{equation}
\label{eq:subleading_asymptotics_ech_proof_d}
c_k(X) + c_\ell(\sqcup_i B(a)) \leq c_{k+\ell}(B(A))
\end{equation}
for all $\ell \geq 0$. Set
\begin{equation*}
V_1 \coloneqq \operatorname{vol}(X), \quad V_2 \coloneqq \operatorname{vol}(\sqcup_i B(a)), \quad V \coloneqq \operatorname{vol}(B(A)) = V_1 + V_2.
\end{equation*}
Since the error terms of $B(A)$ and of $\sqcup_i B(a)$ are bounded by Steps 1 and 2, an upper bound on the error terms of $X$ follows easily from inequality \eqref{eq:subleading_asymptotics_ech_proof_d} if we can show that
\begin{equation*}
\inf_{\ell \geq 0} \sqrt{V (k+\ell)} - \sqrt{V_1 k} - \sqrt{V_2 \ell}
\end{equation*}
admits a finite upper bound independent of $k$. This can be checked via an elementary computation in the same spirit as the one given in Step 2. We omit the details.
\end{proof}

\begin{proof}[Proof of Theorem \ref{thm:subleading_asymptotics_pfh}]
We closely follow the strategy of the proof of \cite[Theorem 8.1]{eh} and adapt the notation from that paper.

We begin by observing that is suffices to show the inequality
\begin{equation}
\label{eq:subleading_asymptotics_pfh_proof_a}
c_{\sigma_i}(\phi,\gamma_i,H_+) - c_{\sigma_i}(\phi,\gamma_i,H_-)+\int_{\gamma_i} (H_+-H_-)dt \geq d_iA^{-1}\int_{Y_\phi} (H_+-H_-)dt\wedge\omega_\phi + O(1)
\end{equation}
for all Hamiltonians $H_\pm \in C^\infty(Y_\phi)$. Indeed, switching $H_+$ and $H_-$ in inequality \eqref{eq:subleading_asymptotics_pfh_proof_a} and multiplying the resulting inequality by $-1$ yields the reverse inequality. Next, we observe that inequality \eqref{eq:subleading_asymptotics_pfh_proof_a} is invariant under shifting the Hamiltonians $H_\pm$ by constants. More precisely, if $C_\pm\in\BR$ are real numbers, then inequality \eqref{eq:subleading_asymptotics_pfh_proof_a} holds for the pair $(H_+,H_-)$ if and only if it holds for the pair $(H_++C_+,H_-+C_-)$. After replacing $H_+$ by $H_++C$ for a sufficiently large constant $C\geq 0$, we can assume that $H_+>H_-$. Let $M$ denote the symplectic cobordism between the graphs of $H_+$ and $H_-$. By Theorem \ref{thm:ball_packing_stability}, ball packing stability holds for $M$. Pick $N \geq 1$ such that $p_N(M) = 1$. Let $a>0$ be the real number such that $N \operatorname{vol}(B(a)) =\operatorname{vol}(M)$. Then for all $b<a$, the disjoint union of $N$ copies of $B(b)$ admits a symplectic embedding into $\operatorname{int}(M)$. It follows from \cite[Lemma 5.2]{eh} that for all non-negative integers $k_i\geq 0$ and for all $b<a$ we have
\begin{equation}
\label{eq:subleading_asymptotics_pfh_proof_b}
c_{\sigma_i}(\phi,\gamma_i,H_+)-c_{U^{k_i}\sigma_i}(\phi,\gamma_i,H_-) + \int_{\gamma_i}(H_+-H_-)dt \geq c_{k_i}^{\operatorname{alt}}(\bigsqcup\limits_{j=1}^N B(b)).
\end{equation}
Here $c_k^{\operatorname{alt}}$ denote the elementary alternatives to ECH capacities defined in \cite{hut22b}, see Remark \ref{rem:elementary_ech_capacities}. It is a consequence of \cite[Theorem 6]{hut22b} that the alternative capacities $c_k^{\operatorname{alt}}$ agree with the usual ECH capacities $c_k$ on balls and disjoint unions of balls. Note that inequality \eqref{eq:subleading_asymptotics_pfh_proof_b} continues to hold for $b=a$ by continuity.

Let $m\geq 1$ be an integer such that all classes $\sigma_i$ are $U$-cyclic of order $m$. Such $m$ exists by \cite{cppz}. We abbreviate
\begin{equation*}
V\coloneqq \operatorname{vol}(\sqcup_j B(a)) = \operatorname{vol}(M) = \int_{Y_\phi}(H_+-H_-)dt\wedge\omega_\phi.
\end{equation*}
and set
\begin{equation*}
n_i\coloneqq \left\lfloor \frac{d_i^2V}{mA^2(d_i-g+1)}\right\rfloor\qquad\text{and}\qquad k_i=mn_i(d_i-g+1).
\end{equation*}
By \cite[Proposition 4.2(a)]{eh}, we have
\begin{equation*}
c_{U^{k_i}\sigma_i}(\phi,\gamma_i,H_-) = c_{\sigma_i}(\phi,\gamma_i,H_-) - mn_iA.
\end{equation*}
Using this, inequality \eqref{eq:subleading_asymptotics_pfh_proof_b} with $b=a$ can be rewritten as
\begin{equation}
\label{eq:subleading_asymptotics_pfh_proof_c}
c_{\sigma_i}(\phi,\gamma_i,H_+)-c_{\sigma_i}(\phi,\gamma_i,H_-) + \int_{\gamma_i}(H_+-H_-)dt
\geq c_{k_i}(\sqcup_j B(a))-mn_iA.
\end{equation}
Using Theorem \ref{thm:subleading_asymptotics_ech}, we compute
\begin{align*}
c_{k_i}(\sqcup_j B(a)) &= 2 \sqrt{V k_i} + O(1) \\
&= 2\sqrt{V} \sqrt{m(d_i-g+1)\left\lfloor\frac{d_i^2V}{mA^2(d_i-g+1)}\right\rfloor} + O(1) \\
&= 2\sqrt{V}\sqrt{\frac{d_i^2V}{A^2}+O(d_i)} + O(1) \\
&= 2A^{-1}Vd_i\sqrt{1+O(d_i^{-1})} + O(1) \\
&= 2A^{-1}Vd_i+O(1).
\end{align*}
Moreover, we compute
\begin{align*}
mn_iA &= mA\left\lfloor\frac{d_i^2V}{mA^2(d_i-g+1)}\right\rfloor \\
&= A^{-1}Vd_i\frac{d_i}{d_i-g+1} + O(1) \\
&= A^{-1}Vd_i+O(1).
\end{align*}
Combining these two computations with inequality \eqref{eq:subleading_asymptotics_pfh_proof_c} yields the desired inequality \eqref{eq:subleading_asymptotics_pfh_proof_a}.
\end{proof}

\begin{proof}[Proof of Theorem \ref{thm:subleading_asymptotics_link_invariants}]
We begin by observing that it suffices to establish the inequality
\begin{equation}
\label{eq:subleading_asymptotics_link_invariants_proof_a}
c_{\underline{L}_d}(H) \leq A^{-1} \int_{[0,1]\times \Sigma} H dt\wedge\omega + O(d^{-1})
\end{equation}
for every Hamiltonian $H \in C^\infty([0,1]\times \Sigma)$. Indeed, substituting $\overline{H}$ for $H$ in inequality \eqref{eq:subleading_asymptotics_link_invariants_proof_a} yields
\begin{equation}
\label{eq:subleading_asymptotics_link_invariants_proof_b}
c_{\underline{L}_d}(\overline{H}) \leq A^{-1} \int_{[0,1]\times \Sigma} \overline{H} dt\wedge \omega + O(d^{-1}) = - A^{-1} \int_{[0,1]\times \Sigma} H dt\wedge \omega + O(d^{-1}).
\end{equation}
By the subadditivity property in \cite[Theorem 1.13]{chmss22}, we have
\begin{equation}
\label{eq:subleading_asymptotics_link_invariants_proof_c}
0 = c_{\underline{L}_d}(0) = c_{\underline{L}_d}(H \# \overline{H}) \leq c_{\underline{L}_d}(H) + c_{\underline{L}_d}(\overline{H}).
\end{equation}
Combining inequalities \eqref{eq:subleading_asymptotics_link_invariants_proof_b} and \eqref{eq:subleading_asymptotics_link_invariants_proof_c} yields the desired reverse inequality
\begin{equation*}
c_{\underline{L}_d}(H) \geq A^{-1} \int_{[0,1]\times \Sigma} H dt\wedge\omega + O(d^{-1}).
\end{equation*}

It remains to prove inequality \eqref{eq:subleading_asymptotics_link_invariants_proof_a}. This involves the relationship between PFH and link spectral invariants established by Chen \cite{chea,cheb}. He considers PFH spectral invariants $c_\sigma(\phi,\gamma,H)$ in the special case $\phi = \operatorname{id}_\Sigma$. In this case, the mapping torus $Y_{\phi}$ is simply given by the product $\BT \times \Sigma$. For every degree $d>0$, Chen considers the special reference cycle $\gamma_d \coloneqq d \cdot (\BT \times \left\{ * \right\})$ for some point $*\in \Sigma$. He constructs a distinguished non-zero PFH class $e_d \in HP(\operatorname{id}_{\Sigma},\gamma_d)$ called the ``PHF unit'' and shows the inequality
\begin{equation}
\label{eq:subleading_asymptotics_link_invariants_proof_d}
d c_{\underline{L}_d}(H) \leq c_{e_d}(\operatorname{id}_{\Sigma},\gamma_d,H) + \int_{\gamma_d} H dt.
\end{equation}
Here we follow the notational conventions for link and PFH spectral invariants from \cite{chmss22} and \cite{eh}, which differ from Chen's conventions. In particular, Chen absorbs the factor $d$ into the link spectral invariant $c_{\underline{L}_d}(H)$. By Theorem \ref{thm:subleading_asymptotics_pfh}, we have
\begin{equation}
\label{eq:subleading_asymptotics_link_invariants_proof_e}
c_{e_d}(\operatorname{id}_\Sigma,\gamma_d,H) - c_{e_d}(\operatorname{id}_\Sigma,\gamma_d,0) + \int_{\gamma_d}H dt = d A^{-1} \int_{[0,1]\times \Sigma} H dt \wedge \omega + O(1).
\end{equation}
It can be seen from the description given in \cite{chea,cheb} of the class $e_d$ in the PFH chain complex of $\varphi_G^1$ for a small Morse function $G$ that the spectral invariant $c_{e_d}(\operatorname{id}_\Sigma,\gamma_d,0)$ vanishes. The desired inequality \eqref{eq:subleading_asymptotics_link_invariants_proof_a} hence follows from \eqref{eq:subleading_asymptotics_link_invariants_proof_d} and \eqref{eq:subleading_asymptotics_link_invariants_proof_e}.
\end{proof}

\section{\texorpdfstring{$C^{2-\varepsilon}$}{TEXT} failure of packing stability}
\label{sec:packing_failure}

In this section we prove Theorem \ref{thm:failure_packing_stability}, i.e. we construct a star-shaped domain $X\subset \BR^4$ which is $C^1$ close to the unit ball, whose boundary is $C^{1,\alpha}$ for all $\alpha \in (0,1)$ and smooth on the complement of a single point, and for which packing stability fails.\\

We begin by recalling a construction from \cite{edt24} which, given a $1$-periodic Hamiltonian $H$ on the unit disk $\BD\subset \BC$, yields a domain $X(H)\subset \BC^2$ such that the characteristic flow on the boundary of $X(H)$ lifts the Hamiltonian flow on $\BD$ induced by $H$. It will be convenient to follow the exposition in \cite[\S 1]{abe}. Recall that $\BT \coloneqq \BR/\BZ$ denotes the circle. Define the domain
\begin{equation*}
U \coloneqq \left\{ (s,t,z) \in \BR \times \BT \times \BC \mid s> \pi(|z|^2-1) \right\} \subset \BR \times \BT \times \BC
\end{equation*}
and consider the diffeomorphism
\begin{equation*}
\Phi: U \rightarrow \BC^*\times \BC \qquad \Phi(s,t,z) \coloneqq e^{2\pi i t}\left(\sqrt{1+\frac{s}{\pi}-|z|^2},z\right).
\end{equation*}
We equip $U$ with the symplectic form $ds\wedge dt + \omega$, where $\omega$ denotes the standard area form on $\BC$. Moreover, we equip $\BC^*\times \BC$ with the restriction of the standard symplectic form $\omega_0$ on $\BC^2$. A simple computation shows that $\Phi$ is a symplectomorphism with respect to these symplectic forms. Note that the image under $\Phi$ of a set of the form $\left\{ (s,t,z) \in U \mid s<s_0 \right\}$ is an open ball with the complex line $\left\{ z_1=0 \right\}$ removed. Let $H: \BT \times \BD \rightarrow \BR$ be a compactly supported Hamiltonian such that
\begin{equation*}
H(t,z) > \pi(|z|^2 - 1) \qquad \text{for all $(t,z) \in \BT \times \BD$.}
\end{equation*}
We define
\begin{equation*}
U(H) \coloneqq \left\{ (s,t,z) \in U \mid z \in \BD, s<H(t,z) \right\} \subset U
\end{equation*}
and set
\begin{equation*}
X(H) \coloneqq \operatorname{clos}(\Phi(U(H))) \subset \BC^2.
\end{equation*}
The volume of $X(H)$ is given by
\begin{equation}
\label{eq:volume_of_domain_associated_to_hamiltonian}
\operatorname{vol}(X(H)) = \operatorname{vol}(B^4(\pi)) + \int_{[0,1]\times \BD} H dt\wedge \omega.
\end{equation}
If the Hamiltonian $H$ has regularity $C^{k,\alpha}$, then the same is true for the boundary of $X(H)$. If $H$ is $C^{k,\alpha}$ small, then $X(H)$ is $C^{k,\alpha}$ close to the unit ball in $\BC^2$. The desired domain $X\subset \BC^2$ for which packing stability fails will be of the form $X(H)$ for a carefully constructed $C^1$ small Hamiltonian $H$ which is $C^{1,\alpha}$ for all $\alpha \in (0,1)$.

We turn to the definition of $H$. Fix a positive parameter $\lambda>0$. For $n\geq 10$, consider the annulus $\BA_n \subset \BD$ with the same center as $\BD$, inner radius $R_n^-\coloneqq \frac{2}{\ln(n)}$, and outer radius $R_n^+\coloneqq \frac{2}{\ln(n-1)}$. Note that the interiors of the annuli $\BA_n$ are pairwise disjoint and that the width $R_n^+ - R_n^-$ of $\BA_n$ is strictly bigger than $\frac{2}{n\ln(n)^2}$. For each $n$, we choose $n$ pairwise disjoint disks $D_n^1,\dots,D_n^n$ of radius $r_n \coloneqq \frac{1}{n\ln(n)^2}$ inside the interior of $\BA_n$. In order to see that this is possible, note that the diameter $2r_n = \frac{2}{n\ln(n)^2}$ of each of the disks $D_n^i$ is strictly smaller than the width of $\BA_n$ and the sum of the diameters of the $n$ disks $D_n^i$, which is given by $2r_nn = \frac{2}{\ln(n)^2}$, is strictly smaller than the circumference of the inner boundary circle of $\BA_n$, which is given by $2\pi R_n^- = \frac{4\pi}{\ln(n)}$. For $r>0$, let $\BD_r\subset \BC$ denote the disk of radius $r$ centered at the origin. Fix a smooth radially symmetric function $f:\BD\rightarrow \BR$ which is non-increasing in the radial direction, satisfies $f(0) = 1$, and is supported inside $\BD_{\frac{1}{10}}$. For each $n$, we set
\begin{equation*}
f_n : \BD_{r_n} \rightarrow \BR \qquad f_n(z) \coloneqq \frac{\ln(n)}{n^2} f(\frac{z}{r_n}).
\end{equation*}
For $1\leq i \leq n$, let $f_n^i$ be the function obtained by pushing $f_n$ forward via the translation moving $\BD_{r_n}$ to $D_n^i$. Fix a compactly supported smooth function $g: (-\frac{1}{2},\frac{1}{2}) \rightarrow [0,\infty)$ with integral $\int g(t) dt = 1$ and set
\begin{equation*}
g_n(t) \coloneqq \ln(n) g(\ln(n) t).
\end{equation*}
Note that the support of $g_n$ is contained in $(-\frac{1}{2\ln(n)},\frac{1}{2\ln(n)})$ and that we have $\int g_n(t) dt = 1$. We regard each $g_n$ as a function $g_n : \BT \rightarrow [0,\infty)$. Now define
\begin{equation*}
H_k : \BT \times \BD \rightarrow \BR \qquad H_k(t,z) \coloneqq \lambda \sum\limits_{n=10}^k g_n(t) \sum\limits_{i=1}^n f_n^i(z)
\end{equation*}
and
\begin{equation*}
H : \BT \times \BD \rightarrow \BR \qquad H(t,z) \coloneqq \lim_{k\rightarrow \infty} H_k(t,z).
\end{equation*}
The limit is uniform, implying that $H$ is continuous. Each compact subset of $\BT \times \BD$ which does not contain the point $(0,0)$ intersects the support of $g_n(t)f_n^i(z)$ for only finitely many values of $n$. Thus $H$ is smooth on the complement of the point $(0,0)$. The sequence of Hamiltonian isotopies $(\varphi_{H_k}^t)_{t\in [0,1]}$ converges in the $C^0$ topology to an isotopy in $\operatorname{Homeo}_c(\BD,\omega)$, the space of compactly supported area preserving homeomorphisms of $\BD$. We abbreviate this isotopy by $(\varphi_H^t)_{t\in [0,1]}$. By definition, this means that $\varphi_H^1$ is contained in the group of hameomorphisms $\operatorname{Hameo}_c(\BD,\omega)$. This group was introduced by Oh an M\"{u}ller in \cite{om07}, see also \cite[Definition 2.2]{chmss}. It will also be convenient to have an autonomous Hamiltonian $\tilde{H}$ generating the same hameomorphism $\varphi_H^1$. We set
\begin{equation*}
\tilde{H}_k : \BD \rightarrow \BR \qquad \tilde{H}_k(z) \coloneqq \lambda\sum\limits_{n=10}^k \sum\limits_{i=1}^n f_n^i(z)
\end{equation*}
and
\begin{equation*}
\tilde{H} : \BD \rightarrow \BR \qquad \tilde{H}(z) \coloneqq \lim_{k\rightarrow \infty} \tilde{H}_k(z).
\end{equation*}
Then the Hamiltonian isotopies $(\varphi_{\tilde{H}_k}^t)_{t\in [0,1]}$ converge in the $C^0$ topology to an isotopy $(\varphi_{\tilde{H}}^t)_{t\in [0,1]}$ in $\operatorname{Homeo}_c(\BD,\omega)$. Since $\varphi_{H_k}^1 = \varphi_{\tilde{H}_k}^1$ for all $k$, we have $\varphi_H^1 = \varphi_{\tilde{H}}^1$.

\begin{lemma}
\label{lem:C1alpha_regularity}
The Hamiltonians $H$ and $\tilde{H}$ are contained in the H\"{o}lder space $C^{1,\alpha}$ for every $\alpha \in (0,1)$, but not in $C^{1,1}$.
\end{lemma}

\begin{proof}
For $x\neq y \in \BD_{r_n}$ and $\alpha \in (0,1]$, we compute
\begin{align*}
\frac{|df_n(x) - df_n(y)|}{|x-y|^{\alpha}} &= \frac{\ln(n)}{n^2}\cdot \frac{|\frac{1}{r_n}df(\frac{x}{r_n}) - \frac{1}{r_n}df(\frac{y}{r_n})|}{r_n^\alpha |\frac{x}{r_n}-\frac{y}{r_n}|^\alpha} \\
&= n^{\alpha-1} \ln(n)^{3+2\alpha} \frac{|df(\frac{x}{r_n}) - df(\frac{y}{r_n})|}{|\frac{x}{r_n} - \frac{y}{r_n}|^{\alpha}}.
\end{align*}
This implies that
\begin{equation*}
[df_n]_{C^{0,\alpha}} = n^{\alpha-1} \ln(n)^{3+2\alpha} [df]_{C^{0,\alpha}},
\end{equation*}
where $[\cdot]_{C^{0,\alpha}}$ denotes the $C^{0,\alpha}$ seminorm. If $\alpha\in (0,1)$, then $n^{\alpha-1} \ln(n)^{3+2\alpha}$ converges to zero as $n$ tends to infinity. Using this, we see that $\tilde{H}$ is $C^{1,\alpha}$ for $\alpha<1$. If $\alpha = 1$, then $n^{\alpha-1} \ln(n)^{3+2\alpha}$ diverges, which implies that $\tilde{H}$ is not $C^{1,1}$.

The proof that $H$ is in $C^{1,\alpha}$ for $\alpha<1$ but not for $\alpha = 1$ is similar. A short computation shows that $[dg_n]_{C^{0,\alpha}} = \ln(n)^{2+\alpha} [dg]_{C^{0,\alpha}}$. Since logarithmic terms are dominated by negative powers of $n$, we see that $[d(g_n f_n)]_{C^{0,\alpha}}$ converges to zero for $\alpha<1$ and diverges for $\alpha = 1$.
\end{proof}

Lemma \ref{lem:C1alpha_regularity} implies that the domain $X(H)$ is $C^{1,\alpha}$ for all $\alpha<1$, but not $C^{1,1}$. Since $H$ is smooth on the complement of a single point, the same is true for the domain $X(H)$. By choosing the parameter $\lambda>0$ small, we can make sure that $X(H)$ is arbitrarily $C^1$ close to the unit ball in $\BC^2$. Our goal is to show that the error terms in the Weyl law for the alternative ECH capacities (see Remark \ref{rem:elementary_ech_capacities}) of $X(H)$ diverge to $-\infty$. More precisely, let us write
\begin{equation*}
c_k^{\operatorname{alt}}(X(H)) = 2\sqrt{\operatorname{vol}(X(H)) k} + e_k^{\operatorname{alt}}(X(H)).
\end{equation*}
We will show:

\begin{proposition}
\label{prop:divergent_error_terms_ech}
We have $\lim_{k\rightarrow \infty} e_k^{\operatorname{alt}}(X(H)) = -\infty$.
\end{proposition}

\begin{remark}
It should also be the case that the error terms in the Weyl law for the usual ECH capacities of $X(H)$ diverge to $-\infty$. The reason we formulate Proposition \ref{prop:divergent_error_terms_ech} for the alternative ECH capacities instead of the original ones is that we rely on certain inequalities between PFH spectral invariants and capacities from \cite{eh}. These inequalities are stated for the alternative ECH capacities, but also expected to hold for the usual ones.
\end{remark}

\begin{remark}
The domain $X(H)$ is the limit of smooth domains $X(H_k)$. A direct computation shows that the Ruelle invariants $\operatorname{Ru}(X(H_k))$ of these domains diverge to $+\infty$. Therefore, Proposition \ref{prop:divergent_error_terms_ech} is expected in view of Hutchings' conjecture, see equation \eqref{eq:hutchings_conjecture}.
\end{remark}

Theorem \ref{thm:failure_packing_stability} is an easy consequence of Theorem \ref{thm:subleading_asymptotics_ech}, Remark \ref{rem:elementary_ech_capacities}, and Proposition \ref{prop:divergent_error_terms_ech}. Indeed, let $W\subset \BR^4$ be a compact domain with smooth boundary and assume by contradiction that $p^W(X(H))= 1$. After rescaling, we can assume that $W$ has the same volume as $X(H)$. The identity $p^W(X(H))=1$ then implies that $rW$ admits a symplectic embedding into $X(H)$ for all $r\in (0,1)$. By monotonicity of capacities, we obtain $c_k^{\operatorname{alt}}(W) \leq c_k^{\operatorname{alt}}(X(H))$. Since $\operatorname{vol}(W) = \operatorname{vol}(X(H))$, we deduce that $e_k^{\operatorname{alt}}(W) \leq e_k^{\operatorname{alt}}(X(H))$ for all $k$. But $e_k^{\operatorname{alt}}(W) = O(1)$ by Theorem \ref{thm:subleading_asymptotics_ech} and Remark \ref{rem:elementary_ech_capacities}, contradicting Proposition \ref{prop:divergent_error_terms_ech}.

It remains to prove Proposition \ref{prop:divergent_error_terms_ech}. We obtain the necessary estimates on the error terms $e_k^{\operatorname{alt}}(X(H))$ in a slightly roundabout way which involves both link spectral invariants and PFH spectral invariants. The reason for this detour is that while we do not know how to compute the (alternative) ECH capacities of $X(H)$, the link spectral invariants of the Hamiltonian $H$ are easy to compute for appropriate choices of links. PFH is used as an intermediary between link spectral invariants and alternative ECH capacities. We make use of the work of Chen \cite{chea,cheb}, which relates link spectral invariants and PFH spectral invariants. We also rely on methods from \cite{eh} to relate PFH spectral invariants and alternative ECH capacities.

We begin by estimating the link spectral invariants of $H$. Let $(S^2,\omega)$ be the sphere of area $A\coloneqq \pi$ obtained from $\BD$ by collapsing the boundary circle to a point. Any compactly supported Hamiltonian $G$ on $\BD$ descends to a Hamiltonian on $S^2$. In the following, we will use the same letter $G$ for this Hamiltonian. Recall from \cite{chmss22} that a Lagrangian link $\underline{L}\subset S^2$ is the disjoint union of finitely many smoothly embedded circles. We call a link $\underline{L}$ monotone if all components of $S^2\setminus \underline{L}$ have equal area. This area must be equal to $\frac{A}{d+1}$, where $d$ denotes the number of components of $\underline{L}$. Given a monotone link $\underline{L}\subset S^2$, the paper \cite{chmss22} introduces a spectral invariant
\begin{equation*}
c_{\underline{L}} : C^{\infty}([0,1]\times S^2) \rightarrow \BR
\end{equation*}
called link spectral invariant. The main properties of link spectral invariants are summarized in \cite[Theorem 1.1 \& 1.13]{chmss22}. By the Hofer Lipschitz property, each $c_{\underline{L}}$ is Lipschitz continuous with respect to the $C^0$ norm on $C^{\infty}([0,1]\times S^2)$. This implies that $c_{\underline{L}}$ admits a continuous extension to $C^0([0,1]\times S^2)$. In particular, $c_{\underline{L}}(H)$ is defined for every link $\underline{L}$.

\begin{lemma}
\label{lem:divergent_error_term_link_invariants}
For each $d\geq 1$, let $\underline{L}_d$ be a monotone $d$-component link in $(S^2,\omega)$. Then
\begin{equation}
\label{eq:divergent_error_term_link_invariants}
d c_{\underline{L}_d}(H) - d A^{-1} \int_{[0,1]\times S^2} H dt \wedge \omega \rightarrow - \infty \qquad (d\rightarrow + \infty).
\end{equation}
\end{lemma}

\begin{remark}
The factor $A^{-1}$ is not present in the Weyl law \cite[Theorem 1.1]{chmss22} because there the surface is normalized to have area $1$.
\end{remark}

\begin{proof}
We claim that
\begin{equation}
\label{eq:bounded_difference_different_links}
|c_{\underline{L}_d}(H) - c_{\underline{L}_d'}(H)| \leq \frac{2A}{d}
\end{equation}
for any two monotone $d$-component links $\underline{L}_d$ and $\underline{L}_d'$. This inequality implies that if \eqref{eq:divergent_error_term_link_invariants} holds for one sequence of monotone links, then it holds for every other sequence of monotone links as well. This means that once we know inequality \eqref{eq:bounded_difference_different_links}, we can reduce ourselves to verifying \eqref{eq:divergent_error_term_link_invariants} for one single carefully constructed sequence of links. We prove inequality \eqref{eq:bounded_difference_different_links}. Let $\underline{L}_d$ be a monotone $d$-component link. If $G \in C^\infty_c([0,1]\times \BD)$ is a smooth Hamiltonian compactly supported in the interior of the disk, then $c_{\underline{L}_d}(G)$ only depends on the time-$1$ map $\varphi_G^1$ by the homotopy invariance property in \cite[Theorem 1.13]{chmss22}. The resulting map
\begin{equation*}
f_{\underline{L}_d}: \operatorname{Ham}_c(\BD) \rightarrow \BR \qquad \varphi_G^1 \mapsto c_{\underline{L}_d}(G)
\end{equation*}
is uniformly continuous with respect to the $C^0$ topology on $\operatorname{Ham}_c(\BD)$ and continuously extends to a map $f_{\underline{L}_d}: \operatorname{Homeo}_c(\BD) \rightarrow \BR$, see \cite[Proposition 3.4]{chmss22}. We have
\begin{equation}
\label{eq:equality_f_and_c}
f_{\underline{L}_d}(\varphi_H^1) = \lim_{k\rightarrow \infty} f_{\underline{L}_d}(\varphi_{H_k}^1) = \lim_{k\rightarrow \infty} c_{\underline{L}_d}(H_k) = c_{\underline{L}_d}(H).
\end{equation}
The map $f_{\underline{L}_d}$ is a quasimorphism with defect at most $\frac{A}{d}$, see \cite[Proof of Proposition 7.9]{chmss22}. It is proved in \cite[Theorem 7.7]{chmss22} that the homogenization
\begin{equation*}
f_d : \operatorname{Homeo}_c(\BD) \rightarrow \BR \qquad f_d(\varphi) \coloneqq \lim_{n\rightarrow \infty} \frac{f_{\underline{L}_d}(\varphi^{n})}{n}
\end{equation*}
is independent of the monotone $d$-component link $\underline{L}_d$. A simple computation using the fact that $f_{\underline{L}_d}$ has defect at most $\frac{A}{d}$ shows that $|f_d - f_{\underline{L}_d}|\leq \frac{A}{d}$. We deduce that $|f_{\underline{L}_d} - f_{\underline{L}_d'}| \leq \frac{2A}{d}$ for any other monotone $d$-component link $\underline{L}_d'$. In combination with \eqref{eq:equality_f_and_c}, this implies inequality \eqref{eq:bounded_difference_different_links}.

Note that since $\varphi_{H}^1 = \varphi_{\tilde{H}}^1$, we have
\begin{equation*}
c_{\underline{L}_d}(H) = f_{\underline{L}_d}(\varphi_H^1) = f_{\underline{L}_d}(\varphi_{\tilde{H}}^1) = c_{\underline{L}_d}(\tilde{H}).
\end{equation*}
Moreover,
\begin{equation*}
\int_{[0,1]\times S^2}H dt\wedge\omega = \int_{[0,1]\times S^2}\tilde{H} dt\wedge\omega. 
\end{equation*}
When checking \eqref{eq:divergent_error_term_link_invariants}, we may therefore replace $H$ by the autonomous Hamiltonian $\tilde{H}$.

It remains to construct a special sequence of monotone links $\underline{L}_d$ for which we can show \eqref{eq:divergent_error_term_link_invariants} to hold. Monotonicity means that each component of $S^2 \setminus \underline{L}_d$ must have area $\frac{A}{d+1}$. Let $N = N(d)>0$ be the maximal positive integer such that the area of $D_N^i$ is at least $\frac{A}{d+1}$. For each of the disks $D_n^i$ with $n\leq N$, we place $\lfloor \operatorname{area}(D_n^i) / \frac{A}{d+1} \rfloor$ components of $\underline{L}_d$ inside $D_n^i$ such that they form a nested sequence of circles centered at the center of $D_n^i$. We adjust the radii of these circles such that the disk bounded by the innermost circle and the annuli bounded by two consecutive circles have area $\frac{A}{d+1}$. Note that since $f$ was chosen to be supported inside $\BD_{\frac{1}{10}}$, the outermost circle component of $\underline{L}_d$ in $D_n^i$ is disjoint from the support of $\tilde{H}$.

We show that we can insert the remaining components of the link $\underline{L}_d$ to be disjoint from the support of $\tilde{H}$. Let $\Sigma$ denote the surface obtained from $S^2$ be removing, for every disk $D_n^i$ with $n \leq N$, the disk bounded by the outermost component of $\underline{L}_d$ inside $D_n^i$. We place the remaining components of $\underline{L}_d$ inside $\Sigma$ such that each of them bounds a topological disk $D$ of area $\frac{A}{d+1}$. As a first step, we arrange the components such that for each of the resulting disks $D$, the total area of all disks $D_n^i$ with $n>N$ intersecting $D$ is at most $\frac{16A}{d+1}$. This is not particularly subtle. For example, it is enough to make sure that each of the topological disks $D$ is contained in a metric disk of area $\frac{4A}{d+1}$, where we use the standard metric on $\BD$. We can also make sure that each component of $\underline{L}_d$ only intersects finitely many of the disks $D_n^i$.

Our second step is to move the link $\underline{L}_d$ via a compactly supported Hamiltonian diffeomorphism of $\Sigma$ to make it disjoint from the support of $\tilde{H}$. For each disk $D_n^i$, let $\tilde{D}_n^i$ be the disk with the same center as $D_n^i$ and with radius $\frac{r_n}{10}$. Note that $\tilde{H}$ is supported in the closure of the union of all $\tilde{D}_n^i$. Moreover, note that for each component $L$ of $\underline{L}_d$, the total area of all disks $\tilde{D}_n^i$ with $n>N$ intersecting the disk bounded by $L$ is strictly less than $\frac{A}{d+1}$. We can therefore find a compactly supported Hamiltonian diffeomorphism $\psi$ of $\Sigma$ which moves every disk $\tilde{D}_n^i$ with $n>N$ intersecting $\underline{L}_d$ into the interior of the disk bounded by some component $L$ of $\underline{L}_d$ it intersects. Moreover, we can choose $\psi$ such that it does not move any disk $\tilde{D}_n^i$ which does not intersect $\underline{L}_d$. After replacing $\underline{L}_d$ by $\psi^{-1}(\underline{L}_d)$, we can assume that $\underline{L}_d$ is disjoint from the support of $\tilde{H}$.

We observe that each component of $\underline{L}_d$ is contained in a level set of $\tilde{H}$. One can therefore easily compute $c_{\underline{L}_d}(\tilde{H})$ via the Lagrangian control property in \cite[Theorem 1.13]{chmss22}. For each component $L$ of $\underline{L}_d$, let $\tilde{H}|_L\in \BR$ be the value $\tilde{H}$ takes on $L$. Then
\begin{equation}
\label{eq:support_control_property}
c_{\underline{L}_d}(\tilde{H}) = \frac{1}{d}\sum_{L\subset \underline{L}_d} \tilde{H}|_L.
\end{equation}
Let $\underline{L}_d^+$ be the collection of all components of $\underline{L}_d$ which are contained in a disk $D_n^i$ for $n\leq N$. By construction of $\underline{L}_d$, only components $L\subset \underline{L}_d^+$ contribute to the sum \eqref{eq:support_control_property}.

Consider a disk $D_n^i$ with $n\leq N$. Let $L_1, \dots, L_k$ denote the components of $\underline{L}_d$ inside $D_n^i$. Suppose that they are ordered such that $L_i$ is inside $L_j$ for $i<j$. We claim that
\begin{equation}
\label{eq:existence_special_link_with_divergent_error_terms_proof_a}
\frac{A}{d+1} \sum\limits_{i=1}^k \tilde{H}|_{L_i} \leq \int_{D_n^i}\tilde{H}\omega.
\end{equation}
Indeed, let $G$ be the function on $D_n^i$ whose value on the disk bounded by $L_1$ is equal to $\tilde{H}|_{L_1}$, whose value on the annulus between $L_i$ and $L_{i+1}$ is equal to $\tilde{H}|_{L_{i+1}}$, and which vanishes outside of $L_k$. Since the function $f$ was chosen to be non-increasing in the radial direction, we have $G\leq \tilde{H}$. Moreover, the left hand side of \eqref{eq:existence_special_link_with_divergent_error_terms_proof_a} is simply the integral of $G$ over $D_n^i$. Thus identity \eqref{eq:existence_special_link_with_divergent_error_terms_proof_a} is an immediate consequence.

Let $D_d^+$ denote the union of all disks $D_n^i$ with $n\leq N$ and $D_d^-$ the union of all disks $D_n^i$ with $n>N$. Then we can estimate using \eqref{eq:existence_special_link_with_divergent_error_terms_proof_a}:
\begin{align*}
dc_{\underline{L}_d}(\tilde{H}) - d A^{-1} \int_\BD \tilde{H}\omega &= \sum\limits_{L\subset \underline{L}_d^+} \tilde{H}|_L - d A^{-1}\int_{D_d^+}\tilde{H}\omega - d A^{-1}\int_{D_d^-} \tilde{H}\omega \\
&\leq (d+1)A^{-1} \int_{D_d^+}\tilde{H}\omega - dA^{-1}\int_{D_d^+}\tilde{H}\omega - d A^{-1}\int_{D_d^-} \tilde{H}\omega \\
&\leq A^{-1}\int_\BD \tilde{H}\omega - d A^{-1}\int_{D_d^-}\tilde{H}\omega.
\end{align*}
Therefore, it suffices to show that $d \int_{D_d^-}\tilde{H}\omega$ diverges to $+\infty$ as $d$ tends to infinity. By construction of $\tilde{H}$ we have
\begin{align*}
d \int_{D_d^-} \tilde{H}\omega &= d\sum\limits_{n>N(d)} n \int_{\BD_{r_n}}f_n \omega \\
&= d \sum\limits_{n>N(d)} \frac{n\ln(n)r_n^2}{n^2}  \int_{\BD}f\omega \\
&= d \sum\limits_{n>N(d)} \frac{1}{n^3 \ln(n)^3} \int_\BD f\omega.
\end{align*}
We estimate
\begin{align*}
d\sum\limits_{n>N(d)}\frac{1}{n^3 \ln(n)^3} & \geq d\int_{N(d)+1}^\infty \frac{1}{x^3 \ln(x)^3} dx \\
&= d\int_{\ln(N(d)+1)}^\infty u^{-3} e^{-2u} du \\
&= d(N(d)+1)^{-2}\int_0^\infty \frac{1}{(v + \ln(N(d)+1))^3} e^{-2v} dv \\
&= d(N(d)+1)^{-2} \ln(N(d)+1)^{-3} \int_0^\infty \frac{1}{(1 + \frac{v}{\ln(N(d)+1)})^3} e^{-2v} dv \\
&\geq C dN(d)^{-2} \ln(N(d))^{-3}
\end{align*}
for some constant $C>0$ which can be chosen uniformly among all sufficiently large $d$. Recall that $N(d)$ is the maximal integer $n$ such that the area of $D_n^i$ is at least $\frac{A}{d+1}$. This means that
\begin{equation*}
\pi N(d)^{-2} \ln(N(d))^{-4} \geq \frac{A}{d+1}.
\end{equation*}
We can therefore further estimate
\begin{equation*}
d N(d)^{-2} \ln(N(d))^{-3} \geq C \ln(N(d)).
\end{equation*}
Since $N(d)$ tends to infinity as $d$ goes to infinity, we conclude that $d \int_{D_d^-} \tilde{H}\omega$ diverges to $+\infty$ as $d$ goes to infinity.
\end{proof}

Our next step is to use Chen's work \cite{chea,cheb} relating PFH and link spectral invariants to obtain information on the subleading asymptotics of the PFH spectral invariants of $H$. Recall from \cite{eh} that for every area preserving map $\phi \in \operatorname{Diff}(S^2,\omega)$, every reference cycle $\gamma\subset Y_\phi$ in the mapping torus $Y_\phi$ of $\phi$, and every non-zero PFH class $\sigma \in HP(\phi,\gamma)$, we have a spectral invariant
\begin{equation*}
C^{\infty}(Y_\phi) \rightarrow \BR \qquad G \mapsto c_{\sigma}(\phi,\gamma,G).
\end{equation*}
Similarly to the case of link spectral invariants, this map is Lipschitz continuous with respect to the $C^0$ norm on $C^{\infty}(Y_\phi)$ and continuously extends to $C^0(Y_\phi)$. Chen considers the special case $\phi = \operatorname{id}_{S^2}$. In this case, the mapping torus $Y_\phi$ is simply given by $\BT\times S^2$. For every $d\geq 1$, Chen considers the special reference cycle $\gamma_d \coloneqq d \cdot (\BT \times \left\{ * \right\})$, i.e. $d$ copies of the circle $\BT \times \left\{ * \right\}$ for some point $*\in S^2$. For each $d\geq 1$, he constructs a distinguished non-zero class $e_d \in HP(\operatorname{id}_{S^2},\gamma_d)$ called the ``PFH unit''. Let $\underline{L}_d\subset S^2$ be a monotone $d$-component link. Suppose that every component of $\underline{L}_d$ bounds a disk in $S^2$ of area $\frac{A}{d+1}$. For such links $\underline{L}_d$, Chen shows that
\begin{equation}
\label{eq:pfh_link_inequality}
c_{e_d}(\operatorname{id}_{S^2},\gamma_d,G) + \int_{\gamma_d}G dt - A \leq d c_{\underline{L}_d}(G) \leq c_{e_d}(\operatorname{id}_{S^2},\gamma_d,G) + \int_{\gamma_d}G dt
\end{equation}
for every $1$-periodic Hamiltonian $G$ on $S^2$.

\begin{remark}
In \eqref{eq:pfh_link_inequality} we follow the notations and conventions for PFH and link spectral invariants introduced in \cite{eh} and \cite{chmss22}. These differ slightly from the ones used by Chen. Most notably, Chen's definition of link spectral invariants differs from the one used in \cite{chmss22} by a factor $d$. For this reason, the factor $d$ in front of the link spectral invariant in \eqref{eq:pfh_link_inequality} does not show up in Chen's work. Moreover, Chen normalizes surfaces to have area $1$. Thus the summand $-A$ becomes $-1$ in Chen's work.
\end{remark}

\begin{lemma}
\label{lem:divergent_error_term_pfh}
We have
\begin{equation*}
c_{e_d}(\operatorname{id}_{S^2},\gamma_d,H) + \int_{\gamma_d}H - d A^{-1}\int_{[0,1]\times S^2} H dt\wedge \omega \rightarrow - \infty \quad (d\rightarrow \infty).
\end{equation*}
\end{lemma}

\begin{proof}
This is an immediate consequence of Lemma \ref{lem:divergent_error_term_link_invariants} and Chen's inequalities \eqref{eq:pfh_link_inequality}.
\end{proof}

Consider $\BR \times \BT \times S^2$ equipped with the symplectic form $ds\wedge dt + \omega$. Let $M \coloneqq D(H+A) \subset \BR \times \BT \times S^2$ be the symplectic cobordism between the graph of $H + A$ and the graph of the Hamiltonian which is constant equal to $0$. It follows from \eqref{eq:volume_of_domain_associated_to_hamiltonian} that the volume of $M$ agrees with the volume of the disjoint union of $X(H)$ and $B^4(A)$.

\begin{lemma}
\label{lem:weak_embedding_into_pfh_cobordism}
Every compact subset of the disjoint union of the interiors of $X(H)$ and $B^4(A)$ admits a symplectic embedding into the interior of $M$.
\end{lemma}

\begin{remark}
Using the notation introduced in Theorem \ref{thm:failure_packing_stability}, this statement is equivalent to the assertion that $p^{X(H)\sqcup B^4(A)}(M) = 1$.
\end{remark}

\begin{proof}
Note that the domain $U(H)$ showing up in the definition of $X(H)$ is contained in $\BR \times \BT \times \operatorname{int}(\BD)$ and can therefore be naturally regarded as a subset of $\BR \times \BT \times S^2$. It is straightforward to check that translating $U(H)$ by $A$ along the $\BR$ factor of $\BR\times \BT \times S^2$ yields a domain contained in $\operatorname{int}(M)$. In other words, $U(H)$ symplectically embeds into $\operatorname{int}(M)$. The complement of the image of this symplectic embedding contains the set
\begin{equation*}
W\coloneqq \left\{ (s,t,z) \in \BR \times \BT \times \operatorname{int}(\BD) \mid z\in \operatorname{int}(\BD), 0<s<A|z|^2 \right\}.
\end{equation*}
We claim that the interior of the ball $B(A)$ admits a symplectic embedding into $W$. Indeed, let $T$ denote the triangle in the first quadrant of $\BR^2$ with vertices $(0,0)$, $(A,0)$, and $(0,A)$. Let
\begin{equation*}
X_{\operatorname{int}(T)} \coloneqq \left\{ (z_1,z_2)\in \BC^2 \mid \pi(|z_1|^2,|z_2|^2) \in \operatorname{int}(T) \right\}
\end{equation*}
denote the toric domain associated to the interior of $T$. It is well-known that the interior of $B(A)$ symplectically embeds into $X_{\operatorname{int}(T)}$, see e.g. \cite{tra95}. The claim follows in view of the symplectomorphism
\begin{equation*}
W \rightarrow X_{\operatorname{int}(T)} \qquad (s,t,r e^{2\pi i \theta}) \mapsto \left(\sqrt{\frac{s}{\pi}}e^{2\pi i t}, \sqrt{1-r^2}e^{-2\pi i\theta}\right).
\end{equation*}

It remains to show that every compact subset $K$ of the interior of $X(H)$ admits a symplectic embedding into $U(H)$. Let $L\subset \BC^2$ denote the complex axis $\left\{ z_1 = 0 \right\}$. By construction, $U(H)$ is symplectomorphic to $\operatorname{int}(X(H)) \setminus L$. This means that if we can find a compactly supported symplectomorphism $\varphi$ of $\operatorname{int}(X(H))$ which disjoins $L$ from $K$, then $K$ symplectically embeds into $U(H)$. One can construct $\varphi$ as follows: Consider the $3$-dimensional half space $Q\coloneqq \BR_{\geq 0}\times \BC\subset \BC^2$. The boundary of $Q$ is precisely $L$ and the intersection of $Q$ with $X(H)$ is of the form $\left\{ (g(z),z) \mid z\in \BD \right\}$ for a suitable function $g:\BD \rightarrow \BR_{\geq 0}$. The leaves of the characteristic foliation on $Q$ induced by the symplectic form are rays of the form $\BR_{\geq 0} \times \left\{ * \right\}$. Now consider an isotopy of $L$ supported inside $\operatorname{int}(X(H))$ which pushes $L$ along the characteristic rays on $Q$ towards the boundary of $X(H)$. Any such isotopy can be realized by a Hamiltonian flow compactly supported in the interior of $X(H)$ and it is clearly possible to disjoin $L$ from $K$ via such an isotopy. This concludes the proof of the lemma.
\end{proof}

We are finally in a position to prove Proposition \ref{prop:divergent_error_terms_ech}, i.e. that the error terms $e_k^{\operatorname{alt}}(X(H))$ diverge to $-\infty$ as $k$ tends to $\infty$.

\begin{proof}[Proof of Proposition \ref{prop:divergent_error_terms_ech}]
It follows from Lemma \ref{lem:weak_embedding_into_pfh_cobordism} and \cite[Lemma 5.2]{eh} that
\begin{equation}
\label{eq:divergent_error_terms_ech_proof_a}
c_{e_d}(\operatorname{id}_{S^2},\gamma_d,H+A) - c_{U^k e_d}(\operatorname{id}_{S^2},\gamma_d,0) + \int_{\gamma_d}(H+A)dt \geq c_{k}^{\operatorname{alt}}(X(H)\sqcup B(A))
\end{equation}
for all $k$. Set $V \coloneqq \operatorname{vol}(M) = \int_{[0,1]\times S^2} (H+A) dt\wedge \omega$. Moreover, abbreviate
\begin{equation*}
n_d \coloneqq \left\lfloor \frac{d^2V}{A^2(d+1)}\right\rfloor \qquad \text{and}\qquad k_d\coloneqq n_d(d+1).
\end{equation*}
Since our underlying surface is the sphere $S^2$, every PFH class $\sigma$ is $U$-cyclic of order $1$. Combining inequality \eqref{eq:divergent_error_terms_ech_proof_a} with \cite[Proposition 4.2(a)]{eh}, we see that
\begin{equation*}
c_{e_d}(\operatorname{id}_{S^2},\gamma_d,H+A) - c_{e_d}(\operatorname{id}_{S^2},\gamma_d,0) + \int_{\gamma_d}(H+A)dt \geq c_{k_d}^{\operatorname{alt}}(X(H)\sqcup B(A)) - n_dA.
\end{equation*}
Set
\begin{equation*}
i_d \coloneqq \left\lfloor \frac{\operatorname{vol}(X(H))k_d}{V}\right\rfloor \qquad \text{and}\qquad j_d \coloneqq k_d - i_d.
\end{equation*}
By the disjoint union property for the capacities $c_k^{\operatorname{alt}}$, we have
\begin{equation*}
c_{k_d}^{\operatorname{alt}}(X(H) \sqcup B(A)) \geq c_{i_d}^{\operatorname{alt}}(X(H)) + c_{j_d}^{\operatorname{alt}}(B(A)) = 2\sqrt{\operatorname{vol}(X(H))i_d} + 2\sqrt{\operatorname{vol}(B(A))j_d} + e_{i_d}^{\operatorname{alt}}(X(H)) + O(1)
\end{equation*}
An elementary computation shows that
\begin{equation*}
2\sqrt{\operatorname{vol}(X(H))i_d} + 2\sqrt{\operatorname{vol}(B(A))j_d} - n_dA = dA^{-1}V + O(1).
\end{equation*}
We deduce that
\begin{equation}
\label{eq:divergent_error_terms_ech_proof_b}
c_{e_d}(\operatorname{id}_{S^2},\gamma_d,H+A) - c_{e_d}(\operatorname{id}_{S^2},\gamma_d,0) + \int_{\gamma_d}(H+A)dt - dA^{-1}V \geq e_{i_d}^{\operatorname{alt}}(X(H)) + O(1).
\end{equation}
Using \cite[Remark 4.4]{eh} and Lemma \ref{lem:divergent_error_term_pfh}, we see that
\begin{multline*}
c_{e_d}(\operatorname{id}_{S^2},\gamma_d,H+A)+ \int_{\gamma_d}(H+A)dt - dA^{-1}V \\ = c_{e_d}(\operatorname{id}_{S^2},\gamma_d,H) + \int_{\gamma_d}Hdt - dA^{-1} \int_{[0,1]\times S^2}Hdt\wedge\omega \quad \longrightarrow \quad -\infty \qquad (d\rightarrow \infty).
\end{multline*}
Moreover, as already observed in the proof of Theorem \ref{thm:subleading_asymptotics_link_invariants}, the spectral invariant $c_{e_d}(\operatorname{id}_{S^2},\gamma_d,0)$ vanishes. Combining these observations with inequality \eqref{eq:divergent_error_terms_ech_proof_b}, we see that $e_{i_d}^{\operatorname{alt}}(X(H))$ tends to $-\infty$ as $d$ goes to infinity. From this one can further deduce that the entire sequence of error terms $e_k^{\operatorname{alt}}(X(H))$ diverges to $-\infty$. Indeed, since the sequence $c_k^{\operatorname{alt}}(X(H))$ is non-decreasing in $k$, we have
\begin{equation*}
e_{k}^{\operatorname{alt}}(X(H)) - e_{\ell}^{\operatorname{alt}}(X(H)) \leq 2\sqrt{\operatorname{vol}(X(H))} (\sqrt{k} - \sqrt{\ell}) \qquad \text{for all $k\geq \ell$}.
\end{equation*}
Since $i_d$ grows quadratically in $d$, this estimate is enough to control all error terms $e_k^{\operatorname{alt}}(X(H))$ just in terms of the subsequence $e_{i_d}^{\operatorname{alt}}(X(H))$.
\end{proof}

\medskip
\medskip
\medskip
\noindent Oliver Edtmair\\
\noindent ETH-ITS, ETH Z\"{u}rich, Scheuchzerstrasse 70, 8006 Z\"{u}rich, Switzerland.\\
{\it Email address:} {\tt oliver.edtmair@eth-its.ethz.ch}


\begin{thebibliography}{}

\bibitem[ABE]{abe}
A.~Abbondandolo, G.~Benedetti and O.~Edtmair, \emph{Symplectic capacities of domains close to the ball and Banach–Mazur geodesics in the space of contact forms}, Duke Math. J. {\bf 174} (2025), no.~8, 1567--1646.

\bibitem[ABHS18]{abhs18}
A.~Abbondandolo, B.~Bramham, U.~Hryniewicz and P.~Salom\~{a}o, \emph{Sharp systolic inequalities for Reeb flows on the three-sphere}, Invent. Math. {\bf 211} (2018), no.~2, 687--778.

\bibitem[Aim]{aim}
AimPL: Floer theory of symmetric products and Hilbert schemes , available at \url{http://aimpl.org/floerhilbert}.

\bibitem[AO14]{ao14}
S.~Artstein-Avidan and Y.~Ostrover, \emph{Bounds for Minkowski billiard trajectories in convex bodies}, Int. Math. Res. Not. IMRN {\bf 2014} (2014), no.~1, 165--193.

\bibitem[Ban78]{ban78}
A.~Banyaga, \emph{Sur la structure du groupe des diff\'{e}omorphismes qui pr\'{e}servent une forme symplectique}, Comment. Math. Helv. {\bf 53} (1978), no.~2, 174--227.

\bibitem[Ban97]{ban97}
A. Banyaga, \emph{The structure of classical diffeomorphism groups}, Mathematics and its Applications, 400, Kluwer Acad. Publ., Dordrecht, 1997.

\bibitem[BP94]{bp94}
M.~Bialy and L.~Polterovich, \emph{Geodesics of Hofer's metric on the group of Hamiltonian diffeomorphisms}, Duke Math. J. {\bf 76} (1994), no.~1, 273--292.

\bibitem[Bir97]{bir97}
P.~Biran, \emph{Symplectic packing in dimension 4}, Geom. Funct. Anal. {\bf 7} (1997), no.~3, 420--437.

\bibitem[Bir99]{bir99}
P.~Biran, \emph{A stability property of symplectic packing}, Invent. Math. {\bf 136} (1999), no.~1, 123--155.

\bibitem[Bir01]{bir01}
P.~Biran, \emph{Lagrangian barriers and symplectic embeddings}, Geom. Funct. Anal. {\bf 11} (2001), no.~3, 407--464.

\bibitem[BK13]{bk13}
M.~Brandenbursky and J.~Kedra, \emph{On the autonomous metric on the group of area-preserving diffeomorphisms of the 2-disc}, Algebr. Geom. Topol. {\bf 13} (2013), no.~2, 795--816.

\bibitem[BKS18]{bks18}
M.~Brandenbursky, J.~Kedra and E.~Shelukhin, \emph{On the autonomous norm on the group of Hamiltonian diffeomorphisms of the torus}, Commun. Contemp. Math. {\bf 20} (2018), no.~2, 1750042, 27 pp.

\bibitem[Bro41]{bro41}
R.~Brooks, \emph{On colouring the nodes of a network}. Mathematical Proceedings of the Cambridge Philosophical Society {\bf 37} (1941), no.~2, 194--197.

\bibitem[BS]{bs}
L.~Buhovsky and M.~Stoki\'{c}, \emph{Flexibility of the adjoint action of the group of Hamiltonian diffeomorphisms}, Ann. ENS, to appear.

\bibitem[BH11]{bh11}
O.~Buse and R.~Hind, \emph{Symplectic embeddings of ellipsoids in dimension greater than four}, Geom. Topol. {\bf 15} (2011), no.~4, 2091--2110.

\bibitem[BH13]{bh13}
O.~Buse and R.~Hind, \emph{Ellipsoid embeddings and symplectic packing stability}, Compos. Math. {\bf 149} (2013), no.~5, 889--902.

\bibitem[BHO16]{bho16}
O.~Buse, R.~Hind and E.~Opshtein, \emph{Packing stability for symplectic 4-manifolds}, Trans. Amer. Math. Soc. {\bf 368} (2016), no.~11, 8209--8222.

\bibitem[CG23]{cg23}
Y.~Canzani and J.~Galkowski, \emph{Weyl remainders: an application of geodesic beams}, Invent. Math. {\bf 232} (2023), no.~3, 1195--1272.

\bibitem[CE22]{ce22}
J.~Chaidez and O.~Edtmair, \emph{3D convex contact forms and the Ruelle invariant}, Invent. Math. {\bf 229} (2022), no.~1, 243--301.

\bibitem[Chea]{chea}
G.~Chen, \emph{Closed-open morphisms on periodic Floer homology}, {\tt arXiv:2111.11891}.

\bibitem[Cheb]{cheb}
G.~Chen, \emph{On PFH and HF spectral invariants}, {\tt arXiv:2209.11071}.

\bibitem[CW13]{cw13}
C.~Cheng and L.~Wang, \emph{Destruction of Lagrangian torus for positive definite Hamiltonian systems}, Geom. Funct. Anal. {\bf 23} (2013), no.~3, 848--866.

\bibitem[Cie97]{cie97}
K.~Cieliebak, \emph{Symplectic boundaries: creating and destroying closed characteristics}, Geom. Funct. Anal. {\bf 7} (1997), no.~2, 269--321.

\bibitem[CFHW96]{cfhw96}
K.~Cieliebak, A.~Floer, H.~Hofer and K.~Wysocki, \emph{Applications of symplectic homology. II. Stability of the action spectrum}, Math. Z. {\bf 223} (1996), no.~1, 27--45.

\bibitem[CHLS07]{chls07}
K.~Cieliebak, H.~Hofer, J.~Latschev, and F.~Schlenk, \emph{Quantitative symplectic geometry}, in {\it Dynamics, ergodic theory, and geometry}, 1--44, Math. Sci. Res. Inst. Publ., 54, Cambridge Univ. Press, Cambridge, 2007.

\bibitem[Cou20]{cou20}
R.~Courant, \emph{\"{U}ber die Eigenwerte bei den Diﬀerentialgleichungen
der mathematischen Physik}, Math. Z. {\bf 7} (1920), 1--57.

\bibitem[Cri19]{cri19}
D.~Cristofaro-Gardiner, \emph{Symplectic embeddings from concave toric domains into convex ones}, J. Differential Geom. {\bf 112} (2019), no.~2, 199--232.

\bibitem[CHa]{ch}
D.~Cristofaro-Gardiner and R.~Hind, \emph{Boundaries of open symplectic manifolds and the failure of packing stability}, J. Eur. Mat. Soc. (JEMS) (2025), published online first.

\bibitem[CHb]{chb}
D.~Cristofaro-Gardiner and R.~Hind, \emph{On the large-scale geometry of domains in an exact symplectic 4-manifold}, {\tt arXiv:2311.06421}.

\bibitem[CHMP25]{chmp25}
D.~Cristofaro-Gardiner, T.~Holm, A.~Mandini and A.~Pires, \emph{On infinite staircases in toric symplectic four-manifolds}, J. Differential Geom. {\bf 129} (2025), no.~2, 335--413.

\bibitem[CHS24a]{chs24}
D.~Cristofaro-Gardiner, V.~Humili\`{e}re and S.~Seyfaddini, \emph{Proof of the simplicity conjecture}, Ann. of Math. (2) {\bf 199} (2024), no.~1, 181--257.

\bibitem[CHS24b]{chs24b}
D.~Cristofaro-Gardiner, V.~Humili\`{e}re and S.~Seyfaddini, \emph{PFH spectral invariants on the two-sphere and the large scale geometry of Hofer's metric}, J. Eur. Math. Soc. (JEMS) {\bf 26} (2024), no.~12, 4537--4584.

\bibitem[CHMSS22]{chmss22}
D.~Cristofaro-Gardiner, V.~Humili\`{e}re, C.~Mak, S.~Seyfaddini, and I.~Smith, \emph{Quantitative Heegaard Floer cohomology and the Calabi invariant}, Forum Math. Pi {\bf 10} (2022), Paper No. e27, 59 pp.

\bibitem[CHMSS]{chmss}
D.~Cristofaro-Gardiner, V.~Humili\`{e}re, C.~Mak, S.~Seyfaddini, and I.~Smith, \emph{Subleading asymptotics of link spectral invariants and homeomorphism groups of surfaces}, Ann. ENS, to appear.

\bibitem[CHR15]{chr15}
D.~Cristofaro-Gardiner, M.~Hutchings and V.~Ramos, \emph{The asymptotics of ECH capacities}, Invent. Math. {\bf 199} (2015), no.~1, 187--214.

\bibitem[CMM]{cmm}
D.~Cristofaro-Gardiner, N.~Magill and D.~McDuff, \emph{Curvy points, the perimeter, and the complexity of convex toric domains}, {\tt arXiv:2506.23498}.

\bibitem[CPPZ]{cppz}
D.~Cristofaro-Gardiner, D.~Pomerleano, R.~Prasad, and B.~Zhang, \emph{A note on the existence of $U$-cyclic elements in periorid Floer homology}, Proc. AMS, to appear.

\bibitem[CPZ]{cpz}
D.~Cristofaro-Gardiner, R.~Prasad and B.~Zhang, \emph{Periodic Floer homology and the smooth closing lemma for area-preserving surface diffeomorphisms}, {\tt arXiv:2110.02925}.

\bibitem[CS20]{cs20}
D.~Cristofaro-Gardiner and N.~Savale, \emph{Sub-leading asymptotics of ECH capacities}, Selecta Math. (N.S.) {\bf 26} (2020), no.~5, Paper No. 65, 34 pp.

\bibitem[Don96]{don96}
S.~Donaldson, \emph{Symplectic submanifolds and almost-complex geometry}, J. Differential Geom. {\bf 44} (1996), no.~4, 666--705.

\bibitem[DG75]{dg75}
J.~Duistermaat and V.~Guillemin, \emph{The spectrum of positive elliptic operators and periodic bicharacteristics}, Invent. Math. {\bf 29} (1975), no.~1, 39--79.

\bibitem[Edt24]{edt24}
O.~Edtmair, \emph{Disk-like surfaces of section and symplectic capacities}, Geom. Funct. Anal. {\bf 34} (2024), no.~5, 1399--1459.

\bibitem[Edta]{edta}
O.~Edtmair, \emph{An elementary alternative to PFH spectral invariants}, J. Symplectic Geom. {\bf 23} (2025), no.~3, 511--574.

\bibitem[Edtb]{edt}
O.~Edtmair, \emph{Smooth perfectness of Hamiltonian diffeomorphism groups}.

\bibitem[EH]{eh}
O.~Edtmair and M.~Hutchings, \emph{PFH spectral invariants and $C^\infty$ closing lemmas}, {\tt arXiv:2110.02463}.

\bibitem[EH92]{eh92}
Y.~Eliashberg and H.~Hofer, \emph{Towards the definition of symplectic boundary},
Geom. Funct. Anal. {\bf 2} (1992), 211--220.

\bibitem[Eps70]{eps70}
D.~Epstein, \emph{The simplicity of certain groups of homeomorphisms}, Compositio Math. {\bf 22} (1970), 165--173.

\bibitem[GG04]{gg04}
J.~Gambaudo and \'{E}~Ghys, \emph{Commutators and diffeomorphisms of surfaces}, Ergodic Theory Dynam. Systems {\bf 24} (2004), no.~5, 1591--1617.

\bibitem[Got82]{got82}
M.~Gotay, \emph{On coisotropic imbeddings of presymplectic manifolds}, Proc. Amer. Math. Soc. {\bf 84} (1982), no.~1, 111--114.

\bibitem[Gro85]{gro85}
M.~Gromov, \emph{Pseudo holomorphic curves in symplectic manifolds}, Invent. Math. {\bf 82} (1985), no.~2, 307--347.

\bibitem[GT02]{gt02}
E.~Gutkin and S.~Tabachnikov, \emph{Billiards in Finsler and Minkowski geometries}, J. Geom. Phys. {\bf 40} (2002), no.~3-4, 277--301.

\bibitem[HRT13]{hrt13}
S.~Haller, T.~Rybicki and J.~Teichmann, \emph{Smooth perfectness for the group of diffeomorphisms}, J. Geom. Mech. {\bf 5} (2013), no.~3, 281--294.

\bibitem[Her70]{her70}
M.~Herman, \emph{Sur la conjugaison diﬀ\'{e}rentiable des diﬀ\'{e}omorphismes du cercle \`{a} des rotations}, Pub. Math. de l’IH
ES, Tome 49 (1970), 5–233.

\bibitem[Her83]{her83}
M.~Herman, \emph{Sur les courbes invariantes par les diff\'{e}omorphismes de l’anneau}, Ast\'{e}risque, tome 103-104 (1983), p. 1-221.

\bibitem[Hut02]{hut02}
M.~Hutchings, \emph{An index inequality for embedded pseudoholomorphic curves in symplectizations}, J. Eur. Math. Soc. (JEMS) {\bf 4} (2002), no.~4, 313--361.

\bibitem[Hut11a]{hut11}
M.~Hutchings, \emph{Quantitative embedded contact homology}, J. Differential Geom. {\bf 88} (2011), no.~2, 231--266.

\bibitem[Hut11b]{hut11b}
M.~Hutchings, \emph{Recent progress on symplectic embedding problems in four dimensions}, Proc. Natl. Acad. Sci. USA {\bf 108} (2011), no.~20, 8093--8099.

\bibitem[Hut14]{hut14}
M.~Hutchings, \emph{Lecture notes on embedded contact homology}, in {\it Contact and symplectic topology}, 389--484, Bolyai Soc. Math. Stud., 26, J\'anos Bolyai Math. Soc., Budapest, 2014.

\bibitem[Hut22a]{hut22}
M.~Hutchings, \emph{ECH capacities and the Ruelle invariant}, J. Fixed Point Theory Appl. {\bf 24} (2022), no.~2, Paper No. 50, 25 pp.

\bibitem[Hut22b]{hut22b}
M.~Hutchings, \emph{An elementary alternative to ECH capacities}, Proc. Natl. Acad. Sci. U.S.A. {\bf 119} (35) e2203090119, https://doi.org/10.1073/pnas.2203090119 (2022).

\bibitem[HS05]{hs05}
M.~Hutchings and M.~Sullivan, \emph{The periodic Floer homology of a Dehn twist}, Algebr. Geom. Topol. {\bf 5} (2005), 301--354.

\bibitem[Iri15]{iri15}
K.~Irie, \emph{Dense existence of periodic Reeb orbits and ECH spectral invariants}, J. Mod. Dyn. {\bf 9} (2015), 357--363.

\bibitem[Ivr80]{ivr80}
V.~Ivrii, \emph{Second term of the spectral asymptotic expansion of the Laplace - Beltrami operator on manifolds with boundary}, Funct. Anal. Appl. {\bf 14} (1980), 98--106.

\bibitem[LM95]{lm95}
F.~Lalonde and D.~McDuff, \emph{Hofer's $L^\infty$-geometry: energy and stability of Hamiltonian flows. I, II}, Invent. Math. {\bf 122} (1995), no.~1, 1--33, 35--69.

\bibitem[Lev52]{lev52}
B.~Levitan, \emph{On the asymptotic behavior of the spectral function of a self-adjoint differential equation of the second order}, Izv. Akad. Nauk SSSR Ser. Mat. {\bf 16} (1952), 325--352.

\bibitem[LMP23]{lmp23}
M.~Levitin, D.~Mangoubi and I.~Polterovich, {\it Topics in spectral geometry}, Graduate Studies in Mathematics, 237, Amer. Math. Soc., Providence, RI, (2023).

\bibitem[MSS]{mss}
C.~Mak, S.~Seyfaddini and I.~Smith, \emph{Orbifold Floer theory for global quotients and Hamiltonian dynamics}, {\tt arXiv:2502.11290}.

\bibitem[MT]{mt}
C.~Mak and I.~Trifa, \emph{Hameomorphism groups of positive genus surfaces}, Comment. Math. Helv. (2025), published online first.

\bibitem[McD09]{mcd09}
D.~McDuff, \emph{Symplectic embeddings of 4-dimensional ellipsoids}, J. Topol. {\bf 2} (2009), no.~1, 1--22.

\bibitem[MO15]{mo15}
D.~McDuff and E.~Opshtein, \emph{Nongeneric $J$-holomorphic curves and singular inflation}, Algebr. Geom. Topol. {\bf 15} (2015), no.~1, 231--286.

\bibitem[MP94]{mp94}
D.~McDuff and L.~Polterovich, \emph{Symplectic packings and algebraic geometry}, Invent. Math. {\bf 115} (1994), no.~3, 405--434.

\bibitem[MS12]{ms12}
D.~McDuff and F.~Schlenk, \emph{The embedding capacity of 4-dimensional symplectic ellipsoids}, Ann. of Math. (2) {\bf 175} (2012), no.~3, 1191--1282.

\bibitem[OM07]{om07}
Y.~Oh and S.~M\"{u}ller, \emph{The group of Hamiltonian homeomorphisms and $C^0$-symplectic topology}, J. Symplectic Geom., {\bf 5} (2007), no.~2, 167--219.

\bibitem[Ops13]{ops13}
E.~Opshtein, \emph{Polarizations and symplectic isotopies}, J. Symplectic Geom. {\bf 11} (2013), no.~1, 109--133.

\bibitem[PV15]{pv15}
\'A.~Pelayo and S.~V\~u~Ng\d oc, \emph{Hofer's question on intermediate symplectic capacities}, Proc. Lond. Math. Soc. (3) {\bf 110} (2015), no.~4, 787--804.

\bibitem[PS23]{ps23}
L.~Polterovich and E.~Shelukhin, \emph{Lagrangian configurations and Hamiltonian maps}, Compos. Math. {\bf 159} (2023), no.~12, 2483--2520.

\bibitem[Rue85]{rue85}
D.~Ruelle, \emph{Rotation numbers for diffeomorphisms and flows}, Ann. Inst. H. Poincar\'e{} Phys. Th\'eor. {\bf 42} (1985), no.~1, 109--115.

\bibitem[Ryb98]{ryb98}
T.~Rybicki, \emph{Commutators of diffeomorphisms of a manifold with boundary}, Annales Polonici Mathematici {\bf 68} (1998), no.~3, 199--210.

\bibitem[Sch]{sch}
F.~Schlenk, \emph{Symplectic embedding problems old and new}, Survey at \url{http://members.unine.ch/felix.schlenk/Daejeon18/Survey.embeddings.pdf}.

\bibitem[Sch03]{sch03}
F.~Schlenk, \emph{Symplectic embeddings of ellipsoids}, Israel J. Math. {\bf 138} (2003), 215--252.

\bibitem[Sch05a]{sch05}
F.~Schlenk, \emph{Packing symplectic manifolds by hand}, J. Symplectic Geom. {\bf 3} (2005), no.~3, 313--340.

\bibitem[Sch05b]{sch05b}
F.~Schlenk, {\it Embedding problems in symplectic geometry}, De Gruyter Expositions in Mathematics, 40, Walter de Gruyter, Berlin, 2005.

\bibitem[Sun19]{sun19}
W.~Sun, \emph{An estimate on energy of min-max Seiberg-Witten Floer generators}, Math. Res. Lett. {\bf 26} (2019), no.~6, 1807--1827.

\bibitem[Tra95]{tra95}
L.~Traynor, \emph{Symplectic packing constructions}, J. Differential Geom. {\bf 42} (1995), no.~2, 411--429.

\bibitem[Wor23]{wor23}
B.~Wormleighton, \emph{Towers of Looijenga pairs and asymptotics of ECH capacities}, Manuscripta Math. {\bf 172} (2023), no.~1-2, 499--530.

\end{thebibliography}
\end{document}